\documentclass[aos]{imsart}

\RequirePackage{amsthm,amsmath,amsfonts,amssymb}
\RequirePackage[authoryear]{natbib}
\RequirePackage[colorlinks,citecolor=blue,urlcolor=blue,breaklinks=true]{hyperref}
\usepackage[hyphenbreaks]{breakurl}
\usepackage{float}
\usepackage{multirow}

\startlocaldefs
\theoremstyle{plain}
\newtheorem{theorem}{Theorem}[section]
\newtheorem{proposition}[theorem]{Proposition}
\newtheorem{lemma}[theorem]{Lemma}

\theoremstyle{remark}
\newtheorem{remark}[theorem]{Remark}
\newtheorem{definition}[theorem]{Definition}
\newtheorem*{example}{Example}

\newcommand{\E}{{\mathbb E}}
\newcommand{\R}{{\mathbb R}}
\renewcommand{\P}{{\mathbb P}}
\newcommand{\A}{{\mathcal{A}}}
\newcommand{\D}{{\mathcal{D}}}
\newcommand{\F}{{\mathcal{F}}}
\newcommand{\G}{{\mathcal{G}}}
\renewcommand{\H}{{\mathcal{H}}}
\newcommand{\I}{{\mathcal{I}}}
\newcommand{\K}{{\mathcal{K}}}

\newcommand{\infmars}{\mathcal{F}_{\infty-\text{mars}}} 

\newcommand{\Vmars}{V_{\text{mars}}}
\newcommand{\Vhk}{V_{\text{HK0}}}
\newcommand{\Sbeta}{T}
\newcommand{\ind}{\mathbf{1}}
\newcommand{\zerovec}{\mathbf{0}}
\newcommand{\onevec}{\mathbf{1}}
\newcommand{\talpha}{t^{(\alpha)}}
\newcommand{\salpha}{s^{(\alpha)}}
\newcommand{\xalpha}{x^{(\alpha)}}
\newcommand{\yalpha}{y^{(\alpha)}}
\newcommand{\sumoveralpha}{\sum_{\substack{\alpha \in \{0, 1\}^d \setminus \{\zerovec\} \\ |\alpha| \leq s}}} 
\newcommand{\sumoveralphafull}{\sum_{\alpha \in \{0, 1\}^d \setminus \{\zerovec\}}} 
\newcommand{\sumoverbeta}{\sum_{\substack{\beta \in \{0, 1, 2\}^d \\ \max_j \beta_j = 2}}} 
\newcommand{\fazeronu}{f_{a_{\zerovec}, \{\nu_{\alpha}\}}}
\newcommand{\fazeromu}{f_{a_{\zerovec}, \{\mu_{\alpha}\}}}
\newcommand{\intoper}{\mathcal{T}}
\newcommand{\genmalpha}{\mathcal{G}}

\newcommand{\floor}[1]{\lfloor #1 \rfloor}
\newcommand{\ceil}[1]{\lceil #1 \rceil}
\newcommand\inner[2]{\langle #1, #2 \rangle}

\def\qt#1{\qquad\text{#1}}
\def\argmin{\mathop{\rm argmin}}

\endlocaldefs

\begin{document}

\begin{frontmatter}
\title{MARS via LASSO}
\runtitle{MARS via LASSO}

\begin{aug}
\author[A]{\fnms{Dohyeong}~\snm{Ki}\ead[label=e1]{dohyeong\_ki@berkeley.edu}},
\author[B]{\fnms{Billy}~\snm{Fang}\ead[label=e2]{blfang@berkeley.edu}}
\and
\author[A]{\fnms{Adityanand}~\snm{Guntuboyina}\ead[label=e3]{aditya@stat.berkeley.edu}}
\address[A]{Department of Statistics, University of California, Berkeley,
\printead[presep={,\ }]{e1,e3}}

\address[B]{Google LLC,
\printead[presep={,\ }]{e2}}
\end{aug}

\begin{abstract}
Multivariate adaptive regression splines (MARS) is a popular method for nonparametric regression introduced by Friedman in 1991. 
MARS fits simple nonlinear and non-additive functions to regression data. 
We propose and study a natural lasso variant of the MARS method. 
Our method is based on least squares estimation over a convex class of functions obtained by considering infinite-dimensional linear combinations of functions in the MARS
basis and imposing a variation based complexity constraint. 
Our estimator can be computed via finite-dimensional convex optimization, although it is defined as a solution to an infinite-dimensional optimization problem.
Under a few standard design assumptions, we prove that our estimator achieves a rate of convergence that depends only logarithmically on dimension and thus avoids the usual curse of dimensionality to some extent. 
We also show that our method is naturally connected to nonparametric estimation techniques based on smoothness constraints.
We implement our method with a cross-validation scheme for the selection of the involved tuning parameter and compare it to the usual MARS method in various simulation and real data settings.
\end{abstract}  


\end{frontmatter}

\section{Introduction} \label{sec:intro}
We study a natural lasso variant of the
multivariate adaptive regression splines (MARS) method (see
\cite{friedman1991multivariate} or 
\cite{hastie2009elements}, Section 9.4) for nonparametric
regression. To understand the relationship between a response
variable $y$ and $d$ explanatory variables $x_1, \dots, x_d$ based on
observed data $(x^{(1)}, y_1), \dots, (x^{(n)}, y_n)$ with $x^{(i)}
\in \R^d$ and $y_i \in \R$, MARS
fits a function $y = \hat{f}_{\text{mars}}(x_1, \dots, x_d)$ where
$\hat{f}_{\text{mars}}$  is a sparse linear combination of functions of the
form:
\begin{equation}\label{basismars}
  \prod_{j=1}^d (b_j(x_j))^{\alpha_j} =  \prod_{j: \alpha_j
    =1} b_j(x_j)
\end{equation}
with $\alpha = (\alpha_1, \dots, \alpha_d) \in \{0, 1\}^d$ and
\begin{equation*}
b_j(x_j) = (x_j - 
t_j)_+ \text{ or }  (t_j - x_j)_+ \text{ for some real number }t_j.   
\end{equation*}
Here $\cdot_+ := \max \{ \cdot, 0 \}$ indicates the ReLU function. 
As a concrete example, to understand the relationship between the logarithm of Weekly Earnings ($y$) and the
two variables, Years of Education ($x_1$) and Years of Experience
($x_2$), from a standard dataset (\textsf{ex1029} in the R library \textsf{Sleuth3}) collected from 25,437 full-time
male workers in 1987, the default implementation of MARS from the R
package \textsf{earth} (with the maximum degree of interaction  set to two) fits the function 
\begin{equation}\label{expfuncfit}
\begin{split}
    &5.83 + 0.0695 (x_1 - 7)_+ - 0.0370 (11 - x_1)_+ \\
    &\ \ + 0.0155 (x_2 - 13)_+ - 0.0600 (13 - x_2)_+ - 0.0164 (x_2 - 30)_+ \\
    &\ \ - 0.0114 (x_1 - 11)_+ (x_2 - 40)_+ + 0.00148 (x_1 - 11)_+ (40 - x_2)_+,  
\end{split}
\end{equation}
which is clearly a linear combination of the eight functions each of the
form \eqref{basismars}. 

MARS fits a nonlinear function to the observed data that is simple enough to be
interpretable because it is built from the basic ReLU functions $(x_j
- t_j)_+$ and $(t_j - x_j)_+$. Furthermore, MARS fits non-additive
functions because of the presence of products in \eqref{basismars}, which enables interactions between the explanatory variables $x_1,
\dots, x_d$. Indeed, the term \eqref{basismars} can be interpreted as
an interaction term of order $|\alpha|$ between the variables in the
set $S(\alpha)$. Here and in the rest of the paper, we use the notation
\begin{equation*}
  S(\alpha) :=  \{j \in [d] : \alpha_j = 1\}, \qt{where $[d] := \{1, \dots, d\}$,}
\end{equation*}
and
\begin{equation*}
   |\alpha| := |S(\alpha)| = \sum_{j=1}^d \ind\{\alpha_j = 1\}. 
\end{equation*}
The exact methodology that MARS uses involves a greedy
algorithm similar to stepwise regression methods. Specifically, one
adds in basis functions of the form \eqref{basismars} starting with a 
constant function using a goodness of fit criterion. Typically, one
only considers terms \eqref{basismars} for which the interaction order
$|\alpha|$ is smaller than a pre-chosen integer $s \leq d$ (most
commonly $s =  1$ or $s = 2$). Once a reasonably large number of basis
functions with $|\alpha| \leq s$ 
  are added, a backward deletion procedure is applied  to
  remove superfluous basis functions. We refer the reader to
  \cite{hastie2009elements}, Section 9.4 for more details on MARS.    

Our goal in this paper is to propose and study a lasso variant of the
MARS method where we consider all the basis functions of the form
\eqref{basismars} with $|\alpha| \leq s$ and apply the lasso method of
\cite{tibshirani1996regression}. As is well-known, lasso is an
attractive alternative to stepwise regression methods in usual linear
models. In order to apply lasso in the MARS setting, we first
assume that the explanatory variables $x_1, \dots, x_d$ all take
values in the interval $[0, 1]$. In other words, the domain of the
regression function is assumed to be $[0, 1]^d$. In practical
settings, this can be achieved by subtracting the minimum possible
value and dividing by the range for each explanatory variable. This
scaling puts all the variables on a comparable footing
enabling the application of lasso. Without such a scaling,
$x_1, \dots, x_d$  might be on very different scales in which case
penalizing or constraining the sum of the absolute values of the
coefficients corresponding to the terms $(x_j  - t_j)_+$ would be
unnatural (e.g., think of the setting where $x_1$ is \textit{years} of
education and $x_2$ is \textit{days} of experience). After
fitting a function in the transformed domain $[0, 1]^d$, we can simply
invert the transformation 
to find the equation of the fitted function in the original domain
(see Section \ref{experiment} for some real data applications). 

In the rest of the paper, we assume that the explanatory variables
$x_1, \dots, x_d$ all belong to $[0, 1]$. The observed data is
$(x^{(1)}, y_1), \dots, (x^{(n)}, y_n)$ where $x^{(i)} \in [0, 1]^d$
and $y_i \in \R$. To this data, we fit functions of the
form $y = \hat{f}(x_1, \dots, x_d)$ where $\hat{f} : [0, 1]^d
\rightarrow \R$ via the application of lasso with the MARS basis
functions. The restriction $x_j \in [0, 1]$ allows us to make two
simplifications to the usual MARS setup:
\begin{longlist}
\item Instead of considering both kinds of functions $(x_j -
t_j)_+$ and $(t_j - x_j)_+$, we only take into account $(x_j -
t_j)_+$, because as each $x_j$ is assumed to be in $[0, 1]$,
we can write
\begin{equation*}
  (t_j - x_j)_+ = (x_j - t_j)_+ - x_j + t_j = (x_j - t_j)_+ - (x_j - 0)_+ + t_j,
\end{equation*}
which implies that every linear combination of functions of the
form \eqref{basismars} is also a linear combination of functions of the same form \eqref{basismars} where $b_j(x_j) = 
(x_j - t_j)_+$ for some $t_j$.
\item We assume that $t_j \in [0, 1)$ for each $j$. This is because
  when $t_j \geq 1$, the function $(x_j - t_j)_+$ becomes 0 as $x_j \in
  [0, 1]$, and for $t_j < 0$, the function $(x_j - t_j)_+ = x_j - t_j
  = (x_j - 0)_+ - t_j$ is a linear combination of $(x_j - 0)_+$ and
  the constant function 1. 
\end{longlist}
Because there are an uncountable number of functions of the form
\eqref{basismars} (as $t_j$ can be any number in $[0, 1)$), we work
with an infinite-dimensional version of
lasso. 
Infinite-dimensional lasso  
formulations have been used in many papers including
\cite{rosset2007L1}, \cite{de2012exact}, \cite{bredies2013inverse}, \cite{candes2014towards}, \cite{duval2015exact}, \cite{de2016exact}, \cite{denoyelle2019sliding}, and \cite{condat2020atomic},
which studied various inverse problems in spaces of measures. The main idea is to  
consider infinite linear combinations of basis functions that are
parametrized by signed measures and to measure complexity in
terms of the variations of the involved signed measures. In the MARS
context, infinite linear combinations of the basis functions
\eqref{basismars} with $|\alpha| \leq s$ are 
\begin{equation}\label{form-of-functions}
    \fazeronu (x_1, \dots, x_d) := a_{\zerovec} + \sumoveralpha
    \int_{[0, 1)^{|\alpha|}} \prod_{j \in S(\alpha)} (x_j -
    t_j)_+ \, d\nu_{\alpha}(\talpha),     
\end{equation}
where $\zerovec := (0, \dots, 0)$, $a_{\zerovec}
  \in \R$,  $\nu_{\alpha}$ is a finite (Borel) signed measure on $[0,
  1)^{|\alpha|}$, and $\talpha$ indicates the vector $(t_j, j \in
  S(\alpha))$ for each binary vector $\alpha \in \{0, 1\}^d
  \setminus \{\zerovec\}$ with $|\alpha| \leq s$.  
We will denote the collection of all such functions 
$\fazeronu$ by $\infmars^{d,s}$ (the subscript
$\infty$ highlights the fact that $\infmars^{d, s}$ contains
\textit{infinite} linear combinations of the functions
\eqref{basismars}). 
The usual MARS functions are special cases of
\eqref{form-of-functions} corresponding to discrete signed measures $\nu_{\alpha}$. Indeed,
when each $\nu_{\alpha}$ is supported on a finite set $\{(t_{lj}^{(\alpha)},
j \in S(\alpha)) : l  = 1, \dots, k_{\alpha} \}$ with   
\begin{equation}\label{disc.sign}
  \nu_{\alpha}\big(\big\{\big(t_{lj}^{(\alpha)}, j \in
        S(\alpha)\big)\big\}\big) = b_l^{(\alpha)} \qt{for $l
    = 1, \dots, k_{\alpha}$},  
\end{equation}
the function $f_{a_{\zerovec}, \{\nu_{\alpha}\}}$ becomes
\begin{equation}\label{genmars}
  (x_1, \dots, x_d) \mapsto a_{\zerovec} +\sumoveralpha \sum_{l=1}^{k_{\alpha}} 
  b_l^{(\alpha)}  \prod_{j \in S(\alpha)} \big(x_j - t_{lj}^{(\alpha)} \big)_+. 
\end{equation}

Our infinite-dimensional lasso estimator minimizes the least squares
criterion over $f_{a_{\zerovec}, \{\nu_\alpha\}} \in
\infmars^{d, s}$ with a constraint on the complexity of $f_{a_{\zerovec},
  \{\nu_\alpha\}}$. The complexity measure involves the sum of the
variations of the underlying signed measures  
$\nu_{\alpha}$ and is an infinite-dimensional analogue of the
usual $L^1$ norm of the coefficients used in finite-dimensional
lasso. Recall that, for a signed measure $\nu$ on $\Omega$ and a 
measurable subset $E \subseteq \Omega$, the variation of $\nu$ on $E$
is denoted by $|\nu|(E)$ and is
defined as the supremum of $\sum_{A \in \pi} |\nu(A)|$ over all
partitions $\pi$ of $E$ into a countable number of disjoint measurable
subsets. Using the variation of the involved signed measures, we define our complexity measure for functions $f = f_{a_{\zerovec},
  \{\nu_{\alpha}\}} \in \infmars^{d, s}$ by  
\begin{equation}\label{vardef}
      \Vmars(\fazeronu) = \sumoveralpha |\nu_{\alpha}|\big([0, 1)^{|\alpha|} \setminus \{\zerovec\}\big). 
\end{equation}
We are excluding $\zerovec = (0, \dots, 0)$ in the variation of
$\nu_{\alpha}$ because we want to only penalize those basis functions
that include at least one nonlinear term and leave unpenalized basis
functions that are products of linear functions (note 
that $(x_j - t_j)_+ = x_j$ is linear when $t_j = 0$ because $x_j \in
[0, 1]$). In Appendix \ref{pf:uniqueness-of-representation}, we show that this complexity measure 
is well-defined by proving the uniqueness of the representation $f =
f_{a_{\zerovec}, \{\nu_{\alpha}\}}$. 

To see why \eqref{vardef} is a generalization
(to infinite linear combinations) of the notion of $L^1$ norm over the
coefficients of the finite linear combination \eqref{genmars}, just
note that when $\nu_{\alpha}$ is the discrete signed measure given by
\eqref{disc.sign}, we have 
\begin{equation*}
  \Vmars(f_{a_{\zerovec}, \{\nu_{\alpha}\}}) = \sumoveralpha \sum_{l=1}^{k_{\alpha}} |b_l^{(\alpha)}| \cdot
  \ind\big\{\big(t_{lj}^{(\alpha)}, j \in S(\alpha)\big) \neq \zerovec \big\},
\end{equation*}
which is simply the sum of the absolute values of the coefficients in
\eqref{genmars} corresponding to the basis functions that have at least
one nonlinear term in their product. 

Our estimator is thus given by
\begin{equation}\label{our-problem-restated}
  \hat{f}^{d, s}_{n, V} \in \underset{f}{\argmin} \bigg\{\sum_{i=1}^n \big(y_i -
      f(x^{(i)}) \big)^2 : f \in \infmars^{d, s} \text{ and } \Vmars(f) \leq V \bigg\}
\end{equation}
for a tuning parameter $V > 0$. We prove that $\hat{f}_{n, V}^{d, s}$
exists and can be computed by applying finite-dimensional lasso
algorithms to the finite basis of functions obtained by placing the
following restrictions on the knots $t_j$ in \eqref{basismars}:   
\begin{equation}\label{fbasis}
    t_j \in  \{0\} \cup \big\{x^{(i)}_{j} : i \in [n]\big\}. 
\end{equation} 
Here we use the notation $x^{(i)} = (x^{(i)}_{1},
\dots, x^{(i)}_{d})$ for the $i^{th}$ design point $x^{(i)}$. As the
finite-dimensional lasso estimation procedure usually zeros out many
regression coefficients, it enables us to obtain $\hat{f}_{n, V}^{d, s}$ that is a sparse linear combination of 
\eqref{basismars}. Therefore, our estimation procedure can be seen as
an alternative to the usual MARS procedure.  It is interesting to note
that the  usual MARS algorithm also works with the restriction
\eqref{fbasis}  on the knots, although typically no  theoretical
justification is provided for this reduction. We also introduce a
computationally more efficient approximate version 
$\tilde{f}_{n, V}^{d, s}$ of $\hat{f}_{n, V}^{d, s}$ which seems to
work nearly as well in practice.  
The approximate version $\tilde{f}_{n, V}^{d, s}$ is
obtained by restricting the knots $t_j$ as
\begin{equation*}
t_j \in \Big\{0, \frac{1}{N_j}, \frac{2}{N_j}, \dots, 1 \Big\}   
\end{equation*} 
for some pre-selected positive integers $N_1, \dots, N_d$. 
For large $n$, $\tilde{f}_{n, V}^{d, s}$ can be computed much more efficiently
than $\hat{f}_{n, V}^{d, s}$. 

We study the theoretical accuracy of these estimators for an unknown
regression function $f^*$ under the standard regression model: 
\begin{equation*}
  y_i = f^*(x^{(i)}) + \xi_i
\end{equation*}
where $\xi_i$ are mean zero errors whose distributions satisfy certain
restrictions. We work with both the fixed design setting
where $x^{(1)}, \dots, x^{(n)}$ form a lattice in $[0, 1]^d$, as well
as the random design setting where $x^{(1)}, \dots, x^{(n)}$ are
assumed to be realizations of i.i.d. random variables. In the former
lattice design setting, which is restrictive but standard in nonparametric
function estimation (see, e.g., \cite{nemirovski2000}), 
we analyze the non-asymptotic accuracy of $\hat{f}_{n, V}^{d, s}$ and $\tilde{f}_{n, V}^{d, s}$. 
In the latter random design setting, we study their accuracy asymptotically.

Our theoretical results show that these estimators achieve rates of
convergence of the form $n^{-4/5} (\log n)^{a s + b}$ for some
constants $a$ and $b$.
It is already known that in
the univariate case ($d = s = 1$), the estimator $\hat{f}_{n, V}^{1, 1}$
achieves the rate $n^{-4/5}$ (see, e.g., \cite{mammen1997locally}, Theorem 10 or \cite{guntuboyina2020adaptive}, Theorem 2.1). Thus, our results imply that in going
from the univariate to the multivariate setting, the rate of
convergence only deteriorates by a logarithmic multiplicative
factor. This suggests that our lasso method for MARS fitting avoids
the usual curse of dimensionality to some extent and can thus be an
effective function estimation technique in higher dimensions.

We can see why our estimators achieve the dimension-free rates (up to the logarithmic multiplicative factors) in part from an alternative characterization of $\hat{f}_{n, V}^{d, s}$.
We can characterize $\hat{f}_{n, V}^{d, s}$  alternatively as a least
squares estimator over a class of functions whose smoothness order, in
a certain sense, grows with the dimension $d$. A key role in this
characterization is played by mixed partial derivatives of
order $2$. For an integer $k \ge 1$ and a real-valued function $f$ defined on $[0, 1]^m$, by the mixed partial derivatives of $f$ of order $k$, we mean
\begin{equation}\label{partial-derivative}
f^{(\beta)} := \frac{\partial^{\beta_1 + \dots + \beta_m}
    f}{\partial x_1^{\beta_1} \cdots \partial x_m^{\beta_m}} 
\end{equation}
where $\beta$ is an $m$-dimensional nonnegative integer vector with
$\max_j \beta_j = k$. 
Whenever we use the notation $f^{(\beta)}$, we inherently assume that $f$ is sufficiently smooth, so that the right-hand side of \eqref{partial-derivative} is irrespective of the order of differentiation and $f^{(\beta)}$ is well-defined. 
Using mixed partial derivatives, we prove the
following alternative characterization of $\hat{f}_{n, V}^{d, s}$:
\begin{equation}\label{altcharintro}
\begin{split}
    &\hat{f}_{n, V}^{d, s} \in \argmin_{f} \bigg\{\sum_{i=1}^n \big(y_i - f(x^{(i)}) 
    \big)^2 : \sum_{\substack{\beta \in \{0, 1, 2\}^d \\ \max_{j} \beta_j = 2}}
    \int_{\bar{T}^{(\beta)}} |f^{(\beta)} | \leq V \\ 
    &\qquad \qquad \qquad \qquad \qquad \qquad \quad \text{ and } f^{(\alpha)} = 0 \text{ for every } \alpha \in \{0, 1\}^d \text{ with } |\alpha| > s \bigg\}
\end{split}
\end{equation}
where
\begin{align}\label{tbeta}
    \bar{\Sbeta}^{(\beta)} := \bar{\Sbeta}^{(\beta)}_1  \times \dots \times
    \bar{\Sbeta}^{(\beta)}_d   ~\text{ where }~     \bar{\Sbeta}^{(\beta)}_k =  
    \begin{cases}
        [0, 1] &\mbox{if } \beta_k = \max_j \beta_j  \\
        \{0\} &\mbox{otherwise}.
    \end{cases}
\end{align}
The main condition here is that the sum of the $L^1$ norms of mixed partial
derivatives of order $2$ is at most $V$. The set
$\bar{\Sbeta}^{(\beta)}$ appearing in the integral signifies that the
integral of the mixed partial derivative $f^{(\beta)}$ is only over
those coordinates $x_l$ for which $\beta_l = \max_j \beta_j$ (the
remaining coordinates are set to zero). 
Also, the condition $f^{(\alpha)} = 0$ for $|\alpha| > s$ rules out interactions of order greater than
$s$. This characterization shows that the maximum total order
$\beta_1 + \dots + \beta_d$ of the mixed
partial derivatives appearing in the constraint equals $2d$. In this
sense, the smoothness order of the constraint can be taken to be $2d$,
which increases with the dimension $d$ and explains the dimension-free
(up to the logarithmic multiplicative factors) rates of convergence.
It should be noted however that not all (in fact, only one) mixed partial derivatives of total order $2d$ are considered in the constraint, and this keeps the function class being too small or restrictive. 
Also, it should be mentioned that it is well-known from approximation theory that $L^p$
norm constraints on mixed partial derivatives are advantageous
and allow one to overcome the curse of dimensionality to some extent from the perspective of metric entropy, approximation, and
interpolation (see, e.g., \cite{temlyakov2018multivariate, dung2018hyperbolic, bungartz2004sparse}).

In fact, the smoothness characterization \eqref{altcharintro} is not fully rigorous. 
Functions of the form \eqref{expfuncfit} 
clearly belong to the constraint set in \eqref{our-problem-restated}, but
they do not belong to the constraint set in \eqref{altcharintro} because mixed partial derivatives of order $2$ do not exist for these functions.
We fix this problem by
interpreting the $L^1$ norms of mixed partial derivatives of order 2
in terms of the Hardy--Krause variations of particular derivatives that we will define in Section \ref{smoothness}. 
Hardy--Krause variation (see, e.g., \cite{aistleitner2015functions,   owen2005multidimensional}) is a
multivariate generalization of total variation of univariate
functions (we review the definition of Hardy--Krause variation and its properties in Appendix \ref{hkreview}). 
Thus, even though \eqref{altcharintro} is not
fully rigorous because mixed partial derivatives of order 2 do not exist for many important MARS
functions, it is still helpful for understanding how the curse of dimensionality can be avoided by our estimators.

The characterization \eqref{altcharintro} also connects our estimators
to other related methods from the literature. In the univariate case
($d = s = 1$), we have 
\begin{equation*}
    \hat{f}_{n, V}^{1, 1} \in \argmin_{f} \bigg\{\sum_{i=1}^n \big(y_i - f(x^{(i)})
    \big)^2 : \int_0^1 |f''| \leq V \bigg\}
  \end{equation*}
which is a constrained analogue of the 
locally adaptive regression splines estimator of
\cite{mammen1997locally} when the order $k$ (in their notation) equals 2. Hence, our estimator
$\hat{f}_{n, V}^{d, s}$ can be seen as a multivariate generalization
of this univariate estimator of \cite{mammen1997locally}. 
Furthermore, if $s = d$ and the condition $\max_j \beta_j = 2$ is replaced with $\max_j \beta_j = 1$ in \eqref{altcharintro}, then one obtains the Hardy--Krause variation denoising
estimator of \cite{fang2021multivariate}. 
Therefore, we can also view $\hat{f}_{n, V}^{d, d}$ as a second-order Hardy--Krause variation
denoising estimator. Further connections to related work are detailed
in Section \ref{relwork}.

We would like to point out here that the theoretical rates of
convergence as well as the smoothness characterization have been made
possible due to our infinite-dimensional lasso formulation of
MARS. In contrast, to the best of our knowledge, no rates of
convergence are known for the usual MARS method. Also, there exist
no prior connections between the usual MARS method and nonparametric
regression methods based on smoothness assumptions.

In addition to the above theoretical contributions, we also implement
our method with a cross-validation scheme for the selection of the
tuning parameter $V$ and compare our estimators to the usual MARS
estimator using simulated and real data. 

The rest of the paper is organized as follows. In Section \ref{existcompute}, 
we present results on the existence and computation of $\hat{f}_{n, V}^{d, s}$ and also introduce the approximate version $\tilde{f}_{n, V}^{d, s}$.
Theoretical accuracy results for $\hat{f}_{n, V}^{d, s}$ and $\tilde{f}_{n, V}^{d, s}$ are in Section \ref{riskresults}. 
Section \ref{smoothness} is devoted to the alternative characterization \eqref{altcharintro} based on smoothness. 
In Section \ref{relwork}, we discuss connections between our method and other related methods.
In Section \ref{experiment}, we illustrate 
the performance of our method in simulated and real data settings and
compare its performance to that of the usual MARS algorithm.

\section{Existence, Computation, and Approximation} \label{existcompute} 
In this section, we prove the existence of our infinite-dimensional
lasso estimator $\hat{f}_{n, V}^{d, s}$ (defined in
\eqref{our-problem-restated}) and show that it can be computed via
finite-dimensional lasso algorithms. We also introduce a
computationally more efficient approximate version of our estimator. 

We start with the observation that the objective
function of the optimization problem defined in
\eqref{our-problem-restated} only depends on the function $f$ through
its values at the design points $x^{(i)}, i \in [n]$. As proved
in the next lemma, this observation allows us to restrict our attention to
the finite-dimensional subclass of $\infmars^{d, s}$ consisting of 
the functions \eqref{form-of-functions} where each 
$\nu_{\alpha}$ is a discrete signed measure supported on the
lattice generated by the design points. For each $k \in [d]$,
let $\mathcal{U}_k$ denote the finite subset of $[0, 1]$ consisting of
the points $0, x_k^{(1)}, \dots, x_k^{(n)}, 1$ (recall here that
$x_k^{(i)}$ denotes the $k^{th}$ coordinate of the $i^{th}$ design
point $x^{(i)} = (x_1^{(i)}, \dots, x_d^{(i)})$). As there could be
ties among $0, x_k^{(1)}, \dots, x_k^{(n)}, 1$, we will write, for
some $n_k \in [n + 1]$,
\begin{equation*}
  \mathcal{U}_k = \big\{u^{(k)}_0, u^{(k)}_1, \dots, u^{(k)}_{n_k}\big\}
  \qt{where $0 = u^{(k)}_0 < \dots < u^{(k)}_{n_k} = 1$}. 
\end{equation*}
Note specially that the cardinality of $\mathcal{U}_k$ is $n_k+1$,
that $u_0^{(k)}$ is always 0, and that $u_{n_k}^{(k)}$ is always
1. The next lemma (proved in Appendix
\ref{pf:reduction-to-discrete-measures}) implies that, for the
optimization problem \eqref{our-problem-restated}, we can restrict
to the functions of the form \eqref{form-of-functions} where each
$\nu_{\alpha}$ is a discrete signed measure supported on the finite set
$(\prod_{k \in S(\alpha)} \mathcal{U}_k) \cap [0, 1)^{|\alpha|}$.

\begin{lemma}\label{lem:reduction-to-discrete-measures}
    Suppose we are given a real number $a_{\zerovec}$ and a collection of finite signed measures $\{\nu_{\alpha}\}$ where $\nu_{\alpha}$ is defined on $[0, 1)^{|\alpha|}$ for each $\alpha \in \{0, 1\}^d \setminus \{\zerovec\}$ with $|\alpha| \le s$.
    Then, there exists a collection of discrete signed measures $\{\mu_{\alpha}\}$ where $\mu_\alpha$ is concentrated on $(\prod_{k \in S(\alpha)} \mathcal{U}_k) \cap [0, 1)^{|\alpha|}$ for each $\alpha \in \{0, 1\}^d \setminus \{\zerovec\}$ with $|\alpha| \le s$ such that 
    \begin{longlist}
        \item $\fazeromu (x^{(i)}) = \fazeronu (x^{(i)})$ for all $i \in [n]$, and    
        \item $\Vmars(\fazeromu) \le \Vmars(\fazeronu)$.
    \end{longlist}
  \end{lemma}

  When $\nu_\alpha$ is concentrated on $(\prod_{k \in S(\alpha)}
\mathcal{U}_k) \cap [0, 1)^{|\alpha|}$ for each $\alpha$, the function
$\fazeronu$ can be written as  
\begin{equation}\label{discnu}
   a_{\zerovec} + \sumoveralpha
    \sum_{l \in \prod\limits_{k \in S(\alpha)}[0: (n_k-1)]} 
      \nu_{\alpha}\big(\big\{\big(u^{(k)}_{l_k}, k \in
            S(\alpha)\big)\big\}\big) \cdot \prod_{k \in
        S(\alpha)} \big(x_k - u^{(k)}_{l_k}\big)_{+} ,  
    \end{equation}
where we use the notation $[p: q] := \{p, p + 1, \dots, q\}$ for two
integers $p \le q$. Also, its complexity measure becomes 
\begin{equation*}
    \Vmars(\fazeronu) = \sumoveralpha \sum_{l \in \prod\limits_{k \in S(\alpha)}[0: (n_k - 1)] \setminus \{\zerovec\}} \Big|\nu_{\alpha}\big(\big\{\big(u^{(k)}_{l_k}, k \in S(\alpha)\big)\big\}\big)\Big|.
  \end{equation*}
The above function \eqref{discnu} is a linear
combination of the basis functions \eqref{basismars} whose knots $t_k$ are chosen from $\mathcal{U}_k \setminus \{1\} = \{0, x_k^{(1)}, \dots, x_k^{(n)}\}$, and its complexity measure equals the absolute sum of
the coefficients of the involved basis functions with at least one nonlinear term. 
Thus, if we
additionally assume that $f$ in the problem
\eqref{our-problem-restated} is constructed from discrete signed
measures as above, then \eqref{our-problem-restated} reduces to 
a finite-dimensional lasso problem.  Lemma
\ref{lem:reduction-to-discrete-measures} then implies that every solution to this
finite-dimensional lasso problem is also a solution to \eqref{our-problem-restated}. A precise statement is given in
the following result, which we prove in Appendix
\ref{pf:reduction-to-lasso}.

\begin{proposition}\label{prop:reduction-to-lasso}
  Let 
    \begin{equation*}
    J = \bigg\{(\alpha, l): \alpha \in \{0, 1\}^d \setminus \{\zerovec\}, |\alpha| \le s, \mbox{ and } l \in \prod_{k \in S(\alpha)}[0: (n_k - 1)]\bigg\}
  \end{equation*}
  and let $M$ be the $n \times |J|$ matrix with columns indexed by
  $(\alpha, l) \in J$ such that 
    \begin{equation*}
M_{i, (\alpha, l)} = \prod_{k \in S(\alpha)} \big(x^{(i)}_k -
  u^{(k)}_{l_k}\big)_{+} \qt{for $i \in [n]$ and $(\alpha, l) \in
  J$}. 
\end{equation*}
Also, let $(\hat{a}_{\zerovec},
\hat{\gamma}_{n, V}^{d, s}) \in \R \times \R^{|J|}$ 
be a solution to the following finite-dimensional lasso problem
\begin{align}\label{finite-dimensional-lasso}
\begin{split}
    &(\hat{a}_{\zerovec}, \hat{\gamma}^{d, s}_{n, V}) \in
    \underset{a_{\zerovec} \in \R, \gamma \in \R^{|J|}}{\argmin}
    \Bigg\{\|y - a_{\zerovec} \onevec - M\gamma\|_2^2: \sum_{\substack{(\alpha, l) \in
    J \\ l \neq \zerovec}} |\gamma_{\alpha, l}| \le
    V \Bigg\},
\end{split}
\end{align}
where $\onevec := (1, \dots, 1)$ and $y = (y_i, i \in [n])$ is the vector of observations. Then,
the function $f$ on $[0, 1]^d$ defined by
\begin{equation}\label{eq:how-to-construct-solution}
    f(x_1, \dots, x_d) = \hat{a}_{\zerovec} + \sum_{(\alpha, l) \in J} \big(\hat{\gamma}^{d, s}_{n, V}\big)_{\alpha, l} \cdot \prod_{k \in S(\alpha)} \big(x_k - u^{(k)}_{l_k}\big)_{+}
\end{equation}
is a solution to the problem \eqref{our-problem-restated}. 
The problem \eqref{our-problem-restated} can have multiple solutions, but every solution $\hat{f}^{d, s}_{n, V}$ satisfies
\begin{equation*}
    \hat{f}^{d, s}_{n, V}(x^{(i)}) = \hat{a}_{\zerovec} + \big(M
    \hat{\gamma}^{d, s}_{n, V}\big)_i = \hat{a}_{\zerovec} + \sum_{(\alpha, l) \in J} \big(\hat{\gamma}^{d, s}_{n, V}\big)_{\alpha, l} \cdot \prod_{k \in S(\alpha)} \big(x^{(i)}_k - u^{(k)}_{l_k}\big)_{+} 
\end{equation*}
for every $i \in [n]$.
\end{proposition}

Because the set
\begin{equation*}
    \Bigg\{a_0 \onevec + M \gamma: a_0 \in \R, \gamma \in \R^{|J|}, \mbox{ and } \sum_{\substack{(\alpha, l) \in J \\ l \neq \zerovec}} |\gamma_{\alpha, l}| \le V \Bigg\}
\end{equation*}
is closed and convex, there exists a solution to the finite-dimensional lasso
problem \eqref{finite-dimensional-lasso}.
Hence, the existence of solutions to our estimation problem \eqref{our-problem-restated} is guaranteed by Proposition \ref{prop:reduction-to-lasso}.  
Also, once we find a solution to the
problem \eqref{finite-dimensional-lasso} via any optimization algorithms, we can construct a solution $\hat{f}^{d, s}_{n, V}$ to
the problem \eqref{our-problem-restated} through the equation
\eqref{eq:how-to-construct-solution}.

However, solving the finite-dimensional lasso problem
\eqref{finite-dimensional-lasso} can be computationally intensive if
$n$ is large because the number of columns of $M$ equals
\begin{equation*}
  |J| = \sum_{\substack{\alpha \in \{0, 1\}^d \setminus \{\zerovec\} \\ |\alpha| \le s}} \prod_{k \in S(\alpha)} n_k,
\end{equation*} 
which is of order $O(n^s)$ (ignoring a multiplicative factor in $d$)
in the worst case when each $n_k = O(n)$.
In the current implementation of our method, we utilize the optimization software \textsf{MOSEK} as a black-box tool for solving the problem \eqref{finite-dimensional-lasso} (see Section \ref{experiment} for more details). 
Using this black-box tool involves creating the whole matrix $M$, and thus, when $n$ is large, our current implementation not only requires a large amount of space for this matrix but also often consumes most of running time constructing it.

This limitation motivates us to come up with the following approximate method. 
As we have seen above, 
Lemma \ref{lem:reduction-to-discrete-measures} ensures that we only need to consider discrete signed measures $\nu_{\alpha}$ supported on the lattices $(\prod_{k \in S(\alpha)}
\mathcal{U}_k) \cap [0, 1)^{|\alpha|}$ for our estimation problem \eqref{our-problem-restated}.
In the approximate method, we instead restrict our attention to discrete signed measures $\nu_{\alpha}$ supported on the lattices generated by 
\begin{equation*}
  \tilde{\mathcal{U}}_k = \Big\{0, \frac{1}{N_k}, \frac{2}{N_k}, \dots, 1 \Big\} 
\end{equation*}
for some pre-selected positive integers $N_1, \dots, N_d$, and we only take into consideration the basis functions corresponding to those signed measures.
Note that in contrast to $\mathcal{U}_k$ whose cardinality are of order $O(n)$ in the worst case, the cardinality of each set $\tilde{\mathcal{U}}_k$ is always $N_k + 1$ regardless of the design points $x^{(1)}, \dots, x^{(n)}$.

We then consider the finite-dimensional optimization problem to which the problem \eqref{our-problem-restated} reduces when we additionally impose such restrictions on signed measures $\nu_{\alpha}$. 
We call this problem the approximate (finite-dimensional optimization) problem. 
The approximate problem has the same form as \eqref{finite-dimensional-lasso} but with different $M$ and $J$. 
Here
\begin{equation*}
    J = \bigg\{(\alpha, l): \alpha \in \{0, 1\}^d \setminus \{\zerovec\}, |\alpha| \le s, \mbox{ and } l \in \prod_{k \in S(\alpha)}[0: (N_k - 1)]\bigg\},
\end{equation*}
and $M$ is the $n \times |J|$ matrix with columns indexed by
$(\alpha, l) \in J$ such that 
\begin{equation*}
M_{i, (\alpha, l)} = \prod_{k \in S(\alpha)} \Big(x^{(i)}_k - \frac{l_k}{N_k}
  \Big)_{+} \qt{for $i \in [n]$ and $(\alpha, l) \in
  J$}. 
\end{equation*}
As opposed to the original finite-dimensional problem \eqref{finite-dimensional-lasso}, the number of columns of $M$ in this problem is always fixed and not affected by the design points $x^{(1)}, \dots, x^{(n)}$.
Hence, this approximate problem can be solved much efficiently than \eqref{finite-dimensional-lasso}, especially when $n$ is large. 
Once we find a solution to the approximate problem, we can construct an estimator of the true underlying function $f^*$ through the equation \eqref{eq:how-to-construct-solution} as before.
We denote this estimator by $\tilde{f}^{d, s}_{n, V}$ and call it an approximate version of $\hat{f}^{d, s}_{n, V}$. 
In the next section, we study the theoretical accuracy of $\tilde{f}^{d, s}_{n, V}$ along with $\hat{f}^{d, s}_{n, V}$. 
We will see that if we choose $N_k$ appropriately, the approximate method is as accurate as the original method, while it significantly improves computational efficiency.

\section{Risk Analysis} \label{riskresults}
This section is dedicated to the study of the theoretical accuracy of $\hat{f}^{d, s}_{n, V}$ and $\tilde{f}^{d, s}_{n, V}$ as an estimator for unknown regression functions.
We first consider the non-asymptotic accuracy of $\hat{f}^{d, s}_{n, V}$ and $\tilde{f}^{d, s}_{n, V}$ in the fixed design setting and then study their asymptotic accuracy in the random design setting. 
The proofs of all the results in this section are provided in Appendix \ref{subsec:proof-of-risk-results}.

\subsection{Fixed Design} 
Here we assume that $x^{(1)}, \dots,
x^{(n)}$ form a lattice
\begin{equation}\label{lattice-design}
    \{x^{(1)}, \dots, x^{(n)} \} = \prod_{k=1}^{d}
    \Big\{u^{(k)}_{i_k}: i_k \in [0: (n_k - 1)] \Big\}
\end{equation}
where for every $k \in [d]$, we have $n_k \ge 2$, $0 = u^{(k)}_0 < u^{(k)}_1 < \cdots < u^{(k)}_{n_k - 1} \le 1$, and 
\begin{equation*}
     u^{(k)}_{i_k} - u^{(k)}_{i_k - 1} \ge \frac{\rho}{n_k} \qt{for all $i_k \in [n_k - 1]$}
\end{equation*}
for some constant $\rho > 0$.
We also assume that $y_1, \dots, y_n$ are generated according to the regression model 
\begin{equation}\label{eq:model-fixed}
  y_i = f^*(x^{(i)}) + \xi_i
\end{equation}
where $f^* : [0, 1]^d \rightarrow \R$ is an unknown regression function and $\xi_i$ are independent sub-Gaussian errors with mean zero and with a sub-Gaussian parameter $\sigma$, i.e.,
\begin{equation*}
    \E [e^{\lambda \xi_i}] \le e^{\frac{\sigma^2 \lambda^2}{2}}
\end{equation*}
for all $\lambda \in \R$. 
We measure the accuracy of an
estimator $\hat{f}_n$ of $f^*$ via the squared empirical $L^2$ norm 
\begin{equation}\label{squared-empirical-norm}
\|\hat{f}_{n} - f^* \|_n^2 := \frac{1}{n} \sum_{i =
        1}^{n} \big(\hat{f}_n(x^{(i)}) - f^{*}(x^{(i)})\big)^2 
\end{equation}
and define its risk as 
\begin{equation*}
    \mathcal{R}_F(\hat{f}_n, f^{*}) = \E
    \|\hat{f}_{n} - f^* \|_n^2 
\end{equation*}
where the expectation is taken over $y_1, \dots, y_n$.

Our first result states an upper bound of the risk of $\hat{f}^{d, s}_{n,
  V}$ under the assumption $f^{*} \in \infmars^{d, s}$ and $\Vmars(f^{*})
\le V$. 
\begin{theorem}\label{thm:risk-upper-bound-fixed}
    Suppose $f^{*} \in \infmars^{d, s}$ and $\Vmars(f^{*}) \le V$ and
    assume the lattice design \eqref{lattice-design}. The estimator $\hat{f}^{d, s}_{n, V}$
    then satisfies that
    \begin{equation}\label{upper-bound-of-risk-fixed-design}
        \mathcal{R}_F\big(\hat{f}^{d, s}_{n, V}, f^{*}\big) \le C_{\rho, d}  \Big(\frac{\sigma^2 V^{\frac{1}{2}}}{n}\Big)^{\frac{4}{5}} \bigg[\log\Big(2 + \frac{V n^{\frac{1}{2}}}{\sigma}\Big)\bigg]^{\frac{3(2s - 1)}{5}} + C_{\rho, d} \frac{\sigma^2}{n} [\log n]^2
    \end{equation}
    for some positive constant $C_{\rho, d}$ depending on $\rho$ and $d$.
\end{theorem}
Note that for fixed $\rho, d, \sigma$, and $V$ and sufficiently large $n$, the first term is the dominant term on the right-hand side
of \eqref{upper-bound-of-risk-fixed-design}, so that  
\begin{equation*}
    \mathcal{R}_F\big(\hat{f}^{d, s}_{n, V}, f^{*}\big) = O \big(n^{-\frac{4}{5}} (\log n)^{\frac{3(2s-1)}{5}} \big),  
\end{equation*}
where the multiplicative constant underlying $O(\cdot)$ depends on $\rho, d, \sigma$, and $V$.
In the univariate case, we can deduce from \cite{guntuboyina2020adaptive}, Theorem 2.1 that
\begin{equation*}
    \mathcal{R}_F\big(\hat{f}^{1, 1}_{n, V}, f^{*}\big) \le C_{\rho} \Big(\frac{\sigma^2 V^{\frac{1}{2}}}{n}\Big)^{\frac{4}{5}} + C_{\rho} \frac{\sigma^2}{n} \log n,
\end{equation*}
where $C_{\rho}$ is a positive constant depending on $\rho$.
In other words,
\begin{equation*}
    \mathcal{R}_F\big(\hat{f}^{1, 1}_{n, V}, f^{*}\big) = O (n^{-\frac{4}{5}}),
\end{equation*}
where the multiplicative constant underlying $O(\cdot)$ depends on $\rho, \sigma$, and $V$.
Thus, what Theorem \ref{thm:risk-upper-bound-fixed} tells us is that for general $d$ and $s$,
$\hat{f}^{d, s}_{n, V}$ can achieve the same rate $n^{-4/5}$, although it
slightly deteriorates by a logarithmic multiplicative factor depending
on $s$.  
This suggests that our lasso method for MARS fitting can avoid the
curse of dimensionality to some extent and be a useful estimation technique in higher
dimensions.

The key step of our proof of Theorem \ref{thm:risk-upper-bound-fixed} is to
find an upper bound of the metric entropy of $\D_m$ (under
the $L^2$ norm), which is defined as the collection of all the functions
of the form  
\begin{equation*}
    (x_1, \dots, x_m) \mapsto \int (x_1 - t_1)_+ \cdots (x_m - t_m)_+ \, d\nu(t), 
\end{equation*}
where $m \in [d]$ and $\nu$ is a signed
measure on $[0, 1]^m$ with variation $|\nu|([0, 1]^m) \le 1$.
The following theorem contains our result on the metric entropy of $\D_m$. 

\begin{theorem}\label{thm:d-metric-entropy-main-text}
    There exist positive constants $C_m$ and $\epsilon_m$ depending on $m$ such that
    \begin{equation*}
        \log N(\epsilon, \D_m, \|\cdot\|_2) \le C_m \epsilon^{-\frac{1}{2}}\Big[\log\frac{1}{\epsilon}\Big]^{\frac{3(2m - 1)}{4}}
    \end{equation*}
    for every $0 < \epsilon < \epsilon_m$. 
    The logarithmic multiplicative factor can be omitted when $m = 1$.
\end{theorem}

\begin{remark}
If the class $\D_m$ is altered by replacing $(x - t)_+$ with
$\ind\{x \ge t\}$ and restricting $\nu$ to probability measures, 
one obtains the collection of all the functions of the form 
\begin{equation*}
    (x_1, \dots, x_m) \mapsto \int \ind\{x_1 \ge t_1\} \cdots \ind\{x_m \ge
    t_m\} \, d\nu(t) = \nu([\zerovec, x]). 
\end{equation*}
This class of functions is indeed the collection of all probability distributions on $[0,
1]^m$, whose upper bounds on the metric entropy were derived in \cite{blei2007metric}. Thus, we are basically extending the argument in \cite{blei2007metric} from $\ind\{x \ge t\}$ to $(x - t)_+$.
\end{remark}

Theorem \ref{thm:d-metric-entropy-main-text} is novel to the best of
our knowledge even though we use standard tools and techniques for
proving it. We first connect upper bounds of the metric entropy of
$\D_m$ to lower bounds of the small ball probability of
integrated Brownian sheet based on ideas from \cite{blei2007metric}, Section 3 and \cite{gao2008entropy}, Section 3 and 
results from \cite{Li1999}, Theorem 1.2 and \cite{Artstein2004b}, Theorem 5.  The small ball probability of integrated Brownian
sheet here refers to the quantity  
\begin{equation*}
    \P\Big(\sup_{t \in [0, 1]^m} |X_m(t)| \le \epsilon \Big),
\end{equation*}
where $\epsilon > 0$ and $X_m$ is an $m$-dimensional integrated
Brownian sheet (a description of integrated Brownian sheet is given in
Appendix \ref{pf:d-metric-entropy}). Required bounds on this small ball
probability are then obtained using results from
\cite{Dunker1999}, Theorem 6 and \cite{Chen2003}, Theorem 1.2. Specifically, we show that there exist positive
constants $c_m$ and $\epsilon_m$ depending on $m$ such that    
\begin{equation*}
    \log \P\Big(\sup_{t \in [0, 1]^m} |X_m(t)| \le
      \epsilon\Big) \ge -c_m \epsilon^{-\frac{2}{3}}\Big[\log
      \frac{1}{\epsilon}\Big]^{2m - 1} 
\end{equation*}
for every $0 < \epsilon < \epsilon_m$. This result, along with the
connection between the metric entropy and the small ball probability, leads to
Theorem \ref{thm:d-metric-entropy-main-text}, which is the main
ingredient in our proof of Theorem \ref{thm:risk-upper-bound-fixed}.

Now, we turn to the result for the approximate version $\tilde{f}^{d, s}_{n, V}$. 
The next theorem presents an upper bound of the risk of $\tilde{f}^{d, s}_{n, V}$ under the same assumption as in Theorem \ref{thm:risk-upper-bound-fixed}.
Recall that $N_k$ are the pre-selected integers for the approximate method. 

\begin{theorem}\label{thm:risk-upper-bound-fixed-approx}
    Suppose $f^{*} \in \infmars^{d, s}$ and $\Vmars(f^{*}) \le V$ and
    assume the lattice design \eqref{lattice-design}. The estimator $\tilde{f}^{d, s}_{n, V}$
    then satisfies that
    \begin{equation*}
        \mathcal{R}_F\big(\tilde{f}^{d, s}_{n, V}, f^{*}\big) \le \frac{8 V^2}{N^2} + C_{\rho, d}  \Big(\frac{\sigma^2 V^{\frac{1}{2}}}{n}\Big)^{\frac{4}{5}} \bigg[\log\Big(2 + \frac{V n^{\frac{1}{2}}}{\sigma}\Big)\bigg]^{\frac{3(2s - 1)}{5}} + C_{\rho, d} \frac{\sigma^2}{n} [\log n]^2
    \end{equation*}
    for some positive constant $C_{\rho, d}$ depending on $\rho$ and $d$, where $N = \min_k N_k$.
\end{theorem}

Theorem \ref{thm:risk-upper-bound-fixed-approx} shows that $\tilde{f}^{d, s}_{n, V}$ has almost the same risk upper bound as $\hat{f}^{d, s}_{n, V}$. 
The only difference is the existence of the approximation error term $8V^2/N^2$, which converges to 0 as $N$ goes to infinity. 
Hence, for sufficiently large $N$, $\tilde{f}^{d, s}_{n, V}$ achieves the same rate as $\hat{f}^{d, s}_{n, V}$.
Indeed, if we set each $N_k$ to be of order at least $n^{2/5}$, then 
\begin{equation}\label{upper-bound-O-notation-fixed-approximate}
    \mathcal{R}_F\big(\tilde{f}^{d, s}_{n, V}, f^{*}\big) = O \big(n^{-\frac{4}{5}} (\log n)^{\frac{3(2s-1)}{5}} \big)  
\end{equation}
where the multiplicative constant underlying $O(\cdot)$ depends on $\rho, d, \sigma$, and $V$.

\subsection{Random Design}
Here we assume that $x^{(1)}, \dots,
x^{(n)}$ are realizations of i.i.d. random variables $X^{(1)}, \dots,
X^{(n)}$ with a probability density function $p_0$ on $[0, 1]^d$
that is bounded by some constant $B \ge 1$, i.e., $\|p_0 \|_{\infty} \le B$. 
Also, we assume that $(X^{(1)}, y_1), \dots, (X^{(n)}, y_n)$ are generated according to the regression model 
\begin{equation}\label{eq:model-random}
  y_i = f^*(X^{(i)}) + \xi_i
\end{equation}
where $\xi_i$ are i.i.d. errors independent of
$X^{(1)}, \dots, X^{(n)}$ with mean zero and with finite $L^{5, 1}$ norm; that is,
\begin{equation}\label{finite-5-1-norm}
  \| \xi_i \|_{5, 1} := \int_{0}^{\infty} (\P(|\xi_i| > t))^{\frac{1}{5}} \, dt < \infty.
\end{equation}
Note that the condition \eqref{finite-5-1-norm} is stronger than the finite fifth-moment condition $\| \xi_i \|_{5} < \infty$, but weaker than the finite $(5 + \epsilon)^{\text{th}}$-moment condition $\| \xi_i \|_{5 + \epsilon} < \infty$ for every $\epsilon > 0$ (see, e.g., \cite{ledoux1991probability}, Chapter 10).
In this setting, we measure the accuracy of an
estimator $\hat{f}_n$ of $f^*$ by
\begin{equation}\label{prediction-error}
    \|\hat{f}_{n} - f^* \|_{p_0, 2}^2 :=
    \int \big(\hat{f}_{n}(x) -
    f^{*}(x)\big)^2 p_0(x) \, dx.
\end{equation}

The next theorem presents the rate of convergence of $\hat{f}^{d, s}_{n, V}$ under the assumption $f^{*} \in \infmars^{d, s}$ and $\Vmars(f^{*}) \le V$. 
Note that $\hat{f}^{d, s}_{n, V}$ still achieves the rate $n^{-4/5}$ as in the fixed lattice design setting, although the exponent of the logarithmic multiplicative factor is slightly bigger when $s > 2$.

\begin{theorem}\label{thm:rate-of-convergence-random}
    If $f^{*} \in \infmars^{d, s}$ and $\Vmars(f^{*}) \le V$, then we have
    \begin{equation}\label{upper-bound-of-risk-random-design}
        \| \hat{f}^{d, s}_{n, V} - f^* \|_{p_0, 2}^2 = O_p \big(n^{-\frac{4}{5}} (\log n)^{\frac{8(s - 1)}{5}}\big).
    \end{equation}
\end{theorem}

As the metric entropy of $\D_m$ played a central role in our proof of Theorem \ref{thm:risk-upper-bound-fixed}, 
the bracketing entropy of $\D_m$ is the key ingredient of our proof of Theorem \ref{thm:rate-of-convergence-random}.
The following theorem states an upper bound of the bracketing entropy of $\D_m$. 

\begin{theorem}\label{thm:d-bracket-entropy}
    There exists a positive constant $C_m$ depending on $m$ such that
    \begin{equation*}
        \log N_{[ \ ]}(\epsilon, \D_m, \|\cdot\|_2) \le C_m \Big(\frac{4}{\epsilon}\Big)^{\frac{1}{2}}\Big|\log \frac{4}{\epsilon}\Big|^{2(m - 1)}
    \end{equation*}
    for every $\epsilon > 0$, 
    where $N_{[ \ ]}(\epsilon, \D_m, \|\cdot\|_2)$ is the $\epsilon$-bracketing
    number of $\D_m$ under the $L^2$ norm.
\end{theorem}

\begin{remark}
Theorem \ref{thm:d-bracket-entropy} also provides an upper bound of the metric entropy of $\D_m$ (under the $L^2$ norm). 
Since 
\begin{equation*}
    N(\epsilon, \D_m, \|\cdot\|_2) \le N_{[ \ ]}(2\epsilon, \D_m, \|\cdot\|_2),
\end{equation*}
we can derive from Theorem \ref{thm:d-bracket-entropy} that
\begin{equation*}
        \log N(\epsilon, \D_m, \|\cdot\|_2) \le C_m \Big(\frac{2}{\epsilon}\Big)^{\frac{1}{2}}\Big|\log \frac{2}{\epsilon}\Big|^{2(m - 1)}
\end{equation*}
for every $\epsilon > 0$.
However, this upper bound is weaker than the one we achieved in Theorem \ref{thm:d-metric-entropy-main-text}.
Although it has the same order for $\epsilon$, the exponent of the logarithmic multiplicative factor is bigger. 
We can obtain from this result an upper bound of the risk of $\hat{f}^{d, s}_{n, V}$ under the fixed lattice design setting, but it will lead to a bound looser than the one in Theorem \ref{thm:risk-upper-bound-fixed}.
\end{remark}

We can prove a similar result as in Theorem \ref{thm:rate-of-convergence-random} for the approximate version $\tilde{f}^{d, s}_{n, V}$. 
As we state in the following theorem, 
$\tilde{f}^{d, s}_{n, V}$ achieves the same rate of convergence as $\hat{f}^{d, s}_{n, V}$ if $N_1, \dots, N_d$ are sufficiently large. 
Together with \eqref{upper-bound-O-notation-fixed-approximate}, this result suggests that the approximate method with appropriately chosen $N_1, \dots, N_d$ can be as accurate as the original method.

\begin{theorem}\label{thm:rate-of-convergence-random-approx}
    Suppose $f^{*} \in \infmars^{d, s}$ and $\Vmars(f^{*}) \le V$.
    Also, assume that $N = \min_k N_k = \Omega(n^{4/15})$, i.e., there exists a positive constant $c_{B, d, V}$ possibly depending on $B, d$, and $V$ such that
    \begin{equation*}
        N \ge c_{B,d,V} \cdot n^{\frac{4}{15}}.
    \end{equation*}
    Then, the estimator $\tilde{f}^{d, s}_{n, V}$ satisfies that
    \begin{equation*}
        \| \tilde{f}^{d, s}_{n, V} - f^* \|_{p_0, 2}^2 = O_p\big(n^{-\frac{4}{5}} (\log n)^{\frac{8(s - 1)}{5}}\big).
    \end{equation*}
\end{theorem}

Our next result shows that the logarithmic multiplicative factor in 
\eqref{upper-bound-of-risk-random-design} can not be completely
removed in the minimax sense. Specifically, we bound the minimax
risk defined as
\begin{equation*}
    \mathfrak{M}^{d, s}_{n, V} = \inf_{\hat{f}_n} \sup_{\substack{f^* \in \infmars^{d, s} \\ \Vmars(f^*) \le V}} \E_{f^*} \|\hat{f}_{n} - f^* \|_{p_0, 2}^2,
\end{equation*}
where the expectation is taken over $(X^{(1)}, y_1), \dots, (X^{(n)}, y_n)$ of
\eqref{eq:model-random} and $\inf_{\hat{f}_n}$ denotes the infimum over all
estimators $\hat{f}_n$ of $f^*$ based on $(X^{(1)}, y_1), \dots, (X^{(n)}, y_n)$. 
Here we further restrict that $\xi_i$ in the model \eqref{eq:model-random} are independent Gaussian errors with mean zero and variance $\sigma^2$ and that the probability density function $p_0$ of $X^{(i)}$ is bounded below by some positive constant $b$, i.e., $\| p_0 \|_{\infty} \ge b$.
Our result shows that the supremum risk of every 
estimator indeed requires a logarithmic multiplicative
factor depending on $s$ in addition to the $n^{-4/5}$ term. Note
though that there is still a gap between the exponent $8(s-1)/5$ of 
$\log n$ in the rate of convergence of $\hat{f}_{n, V}^{d, s}$ and the
exponent $4(s - 1)/5$ of $\log n$ in the minimax lower bound.

\begin{theorem}\label{thm:lower-bound}
There exist positive constants $C_{b, B, s}$ depending on $b, B$, and $s$ and $c_{B, s}$ depending on $B$ and $s$ such that
\begin{equation*}
    \mathfrak{M}^{d, s}_{n, V} \ge C_{b, B, s} \Big(\frac{\sigma^2 V^{\frac{1}{2}}}{n}\Big)^{\frac{4}{5}} \bigg[\log\Big(\frac{V n^{\frac{1}{2}}}{\sigma}\Big)\bigg]^{\frac{4(s - 1)}{5}}
\end{equation*} 
provided $n \ge c_{B, s} \cdot (\sigma^2/V^2)$. 
\end{theorem}

Our proof of Theorem \ref{thm:lower-bound} is based on Assouad's
lemma with a finite set of functions in $\{f^* \in \infmars^{d, s} :
\Vmars(f^*) \leq V \}$  that is constructed by an extension of the ideas in
\cite{blei2007metric}, Section 4.   
Results similar to Theorem \ref{thm:lower-bound} can be proved under the fixed design setting, but we do not go into detail in this paper.

\section{Characterization in terms of Smoothness}\label{smoothness}
In this section, we provide alternative characterizations of
$\infmars^{d, s}, \Vmars(\cdot)$, and $\hat{f}^{d, s}_{n, V}$ in terms of
smoothness. To motivate the results for general $d$ and $s$, let us
first consider the univariate case $d = s = 1$. 
We include the proofs of all the results in this section in Appendix \ref{pf:smoothness}.

\subsection{Smoothness Characterization for $d = s = 1$}
For $d = s = 1$, $\infmars^{1, 1}$ consists of all the functions
$f: [0, 1] \rightarrow \R$ of the form 
  \begin{equation}\label{infmars-d=1}
    f(x) = a_0 + \int_{[0, 1)} (x - t)_+ \, d\nu(t) 
  \end{equation}
where $a_0$ is a real number and $\nu$ is a finite signed measure on
$[0, 1)$, and the complexity measure of $f$ is given by the variation
of $\nu$ on $(0, 1)$; that is, $\Vmars(f) = |\nu|((0, 1))$.

The following simple arguments show that $\infmars^{1, 1}$ can be
characterized in terms of smoothness.  
First, by replacing $(x - t)_+$ with $\int_0^1
\ind\{t \leq s \leq x\} ds$ in the integral in \eqref{infmars-d=1} and
changing the order of integration, we obtain 
\begin{equation}\label{equivalent-formulation-1dim}
  f(x) = a_0 + \int_0^x g(t) \, dt,
\end{equation}
where the function $g: [0, 1] \rightarrow \R$ is given by
\begin{equation}\label{gdef1}
  g(t) = \nu([0, t] \cap [0, 1)). 
\end{equation}
It can be readily verified that the function $g$ in \eqref{gdef1} is right-continuous on
$[0, 1]$ and left-continuous at 1, 
and the total variation
$V(g)$ of $g$ is finite and can be represented as
\begin{equation}\label{varvar1}
  V(g) = |\nu|((0, 1)). 
\end{equation}
Here the total variation of a function $h: [0, 1] \rightarrow \R$
is defined by  
\begin{equation*}
  V(h) = \sup_{0 = u_0 < u_1 < \dots < u_k = 1} \sum_{i=0}^{k-1}
  \big|h(u_{i+1}) - h(u_i) \big|
\end{equation*}
where the supremum is over all integers $k \geq 1$ and partitions $0 = u_0
< u_1 < \dots < u_k = 1$ of $[0, 1]$. 
Conversely, every function $g: [0, 1] \rightarrow \R$ that is
right-continuous on $[0, 1]$, left-continuous at 1, and has finite
total variation can be written as \eqref{gdef1} for a unique signed
measure $\nu$ on $[0, 1)$ (see, e.g.,
\cite{aistleitner2015functions}, Theorem 3). 
Putting these observations together, we can argue that $\infmars^{1, 1}$ has the following alternative
characterization:
\begin{align}\label{smoo1}
    \begin{split}
      &\infmars^{1, 1} =\Big\{f : [0, 1] \rightarrow \R:
        \exists a_0 \in \R \text{ and } g: [0, 1] \rightarrow
        \R \text{ s.t. } \\ 
      & \qquad \qquad \qquad \qquad \quad g \text{ is right-continuous on }[0,
        1], \text{
      left-continuous at 1, } \\
      & \qquad \qquad \qquad \qquad \quad V(g) < \infty, \text{ and } f(x) = a_0 + \int_{0}^{x} g(t) \, dt \text{ for all } x \in [0, 1] \Big\}.
\end{split}
\end{align}
Moreover, we can see that the complexity measure $\Vmars(f)$ for $f \in \infmars^{1, 1}$ is
equal to the total variation $V(g)$ of the function $g$ appearing in 
\eqref{smoo1}.

For every function $f \in \infmars^{1, 1}$, we can show that $g$ satisfying the conditions in \eqref{smoo1} is unique, and thus, we can consider such $g$ as a particular derivative of $f$ satisfying \eqref{equivalent-formulation-1dim}. 
If we denote it by $D^{(1)} f$, the estimator $\hat{f}_{n, V}^{1, 1}$ then can be alternatively written as 
\begin{equation*}
  \hat{f}_{n, V}^{1, 1} \in \argmin_{f} \bigg\{\sum_{i=1}^n \big(y_i - f(x^{(i)})
    \big)^2 :  V(D^{(1)} f) \leq V \bigg\}. 
\end{equation*}
The representation \eqref{equivalent-formulation-1dim} implies, by the
Lebesgue differentiation theorem (see, e.g., \cite{Rudin:87book}, Theorem 7.10), that $f'$ exists and is equal to $D^{(1)}f$ almost
everywhere (with respect to the Lebesgue measure) on $[0, 1]$. 
Hence, we can also describe $\hat{f}_{n, V}^{1, 1}$ somewhat loosely as
\begin{equation*}
  \hat{f}_{n, V}^{1, 1} \in \argmin_{f} \bigg\{\sum_{i=1}^n \big(y_i - f(x^{(i)})
    \big)^2 : V(f') \leq V \bigg\}. 
\end{equation*}
The corresponding penalized version
\begin{equation}\label{penver}
  \argmin_{f} \bigg\{\sum_{i=1}^n \big(y_i - f(x^{(i)})  \big)^2 +
    \lambda V(f') \bigg\}
\end{equation}
was proposed by \cite{mammen1997locally} as part of the class of
estimators collectively called locally adaptive regression
splines. In the univariate case, $\hat{f}_{n, V}^{1, 1}$ can thus be seen as a
constrained analogue of the locally adaptive regression spline
estimator of \cite{mammen1997locally} when the order $k$ (in their notation) equals
2. \cite{steidl2006splines} used the terminology 
\textit{second-order total variation regularization}, and
\cite{kim2009ell_1} and 
\cite{tibshirani2014adaptive} used the 
terminology \textit{first-order trend filtering} for \eqref{penver}. Therefore, our estimator $\hat{f}_{n, V}^{d, s}$ can be considered as a multivariate generalization of piecewise linear (second-order) locally adaptive regression splines, second-order total variation regularization, or first-order trend filtering.  

From the alternative characterization of $\infmars^{1, 1}$ given above, it
follows that 
every sufficiently smooth function $f: [0, 1] \rightarrow \R$
belongs to $\infmars^{1, 1}$. Indeed, if $f'$ and $f''$ exist everywhere
and are continuous on $[0, 1]$, then we have 
\begin{equation*}
f(x) = f(0) + \int_{0}^{x} f'(t) \, dt
\end{equation*}
for all $x \in [0, 1]$, and 
\begin{equation*}
    V(f') = \int_{0}^{1} |f''| < \infty.
\end{equation*}
Thus, in this case, $f$ belongs to $\infmars^{1, 1}$ and the complexity measure of $f$ can be written as  
\begin{equation}\label{2ndder1}
    \Vmars(f) = \int_{0}^{1} |f''|.
\end{equation}
This formula highlights the role of the second derivative $f''$
in the determination of $\Vmars(f)$ for each sufficiently smooth function $f$.

\subsection{Smoothness Characterization for General $d$ and $s$}
As we have seen in the previous subsection, in the univariate case,
$\infmars^{1, 1}$ consists of all the functions $f$ satisfying
\eqref{equivalent-formulation-1dim} with some function $g$ having finite total variation and some one-sided continuity. An analogous characterization holds for general $d$ and $s$. For general $d$ and $s$, 
the role of total variation in the univariate case is played by Hardy--Krause
variation, which is an extension of total variation of 
univariate functions to higher dimensions.
In Appendix \ref{hkreview}, we review the definition of Hardy--Krause variation and its properties that we will use for proving the results in this subsection. 
Standard references for Hardy--Krause variation are \cite{aistleitner2015functions}
and \cite{owen2005multidimensional}.
Here we use Hardy--Krause variation anchored at $\zerovec$, which we denote by $\Vhk(\cdot)$.

The following result provides an alternative characterization of
$\infmars^{d, s}$ and $\Vmars(\cdot)$ in terms of smoothness. 
Recall that we use the notation $\talpha$ to indicate the vector $(t_j, j \in
S(\alpha))$ for each $\alpha \in \{0, 1\}^d
\setminus \{\zerovec\}$ with $|\alpha| \le s$. 

\begin{proposition}\label{prop:equivalent-formation} 
    The function class $\infmars^{d, s}$ consists precisely of all the functions of the form 
    \begin{equation}\label{equivalent-formation}
         f(x_1, \dots, x_d) = a_{\zerovec} + \sumoveralpha
         \int_{[\zerovec, \xalpha]}
         g_{\alpha}(\talpha) \, d\talpha
    \end{equation}
    for some $a_{\zerovec} \in \R$ and some collection
    of functions $\{g_{\alpha}: \alpha \in \{0, 1\}^d \setminus
    \{\zerovec\} \mbox{ and } |\alpha| \le s\}$, where for each $\alpha \in \{0, 1\}^d \setminus
    \{\zerovec\}$ with $|\alpha| \le s$,
    \begin{longlist}
    \item $g_{\alpha}$ is a real-valued function on $[0, 1]^{|\alpha|}$,
    \item $\Vhk(g_{\alpha}) < \infty$, 
    \item $g_{\alpha}$ is coordinate-wise right-continuous on $[0, 1]^{|\alpha|}$, and
    \item $g_{\alpha}$ is coordinate-wise
    left-continuous at each point $\xalpha = (x_j, j \in S(\alpha)) \in [0,
    1]^{|\alpha|} \setminus [0, 1)^{|\alpha|}$ with respect to
    all the $j^{\text{th}}$ coordinates where $x_j = 1$.  
    \end{longlist}
    Furthermore, the complexity of $f$ in \eqref{equivalent-formation}
    can be written in terms of the Hardy--Krause variations of
    $g_{\alpha}$ as
    \begin{equation}\label{vmars-in-vhk}
        \Vmars(f) = \sumoveralpha \Vhk(g_{\alpha}).
    \end{equation}
  \end{proposition}

Proposition \ref{prop:equivalent-formation} is completely analogous to
\eqref{smoo1} for the case $d = s = 1$. Specifically, the condition
\eqref{equivalent-formation} is analogous to the univariate condition \eqref{equivalent-formulation-1dim}. The
condition $\Vhk(g_{\alpha}) < \infty$ for each $\alpha \in \{0, 1\}^d
\setminus \{\zerovec\}$ with $|\alpha| \le s$ corresponds to the univariate condition
$V(g) < \infty$. The coordinate-wise right-continuity of each $g_{\alpha}$
on $[0, 1]^{|\alpha|}$ is matched with the univariate right-continuity
on $[0, 1]$. Lastly, the coordinate-wise left-continuity of each
$g_{\alpha}$ at each $\xalpha \in [0,
    1]^{|\alpha|} \setminus [0, 1)^{|\alpha|}$ (with respect to
    all the $j^{\text{th}}$ coordinates where $x_j = 1$) is a counterpart of the univariate left-continuity at 1. It is also interesting to
    note that $\Vmars(f)$ equals the \textit{sum} of the Hardy--Krause
    variations of $g_{\alpha}$ over $\alpha \in \{0, 1\}^d \setminus
    \{\zerovec\}$ with $|\alpha| \le s$.

For every function $f \in \infmars^{d, s}$, it can be easily checked that $g_{\alpha}$ appearing in Proposition \ref{prop:equivalent-formation} are uniquely determined by $f$.
As in the case $d = s = 1$, we can thus consider such $g_{\alpha}$ as particular derivatives of $f$ satisfying \eqref{equivalent-formation}. 
Let us denote them by $D^{(\alpha)} f$ for $\alpha \in \{0, 1\}^d \setminus \{\zerovec\}$.
We can then write our estimator $\hat{f}_{n, V}^{d, s}$ alternatively as 
\begin{equation}\label{althkd}
  \hat{f}^{d, s}_{n, V} \in \underset{f}{\argmin} \Bigg\{\sum_{i=1}^n \big(y_i - f(x^{(i)}) \big)^2 : \sumoveralpha \Vhk(D^{(\alpha)} f)\le V \Bigg\}.
\end{equation}
Hence, our estimator can be viewed as a least squares estimator under a specific smoothness constraint involving the sum of the Hardy--Krause variations of the particular derivatives defined via $D^{(\alpha)}$.

Recall that the univariate condition \eqref{equivalent-formulation-1dim} implies that $f'$ exists and equals to $D^{(1)}f$ almost everywhere.
Similarly, the condition \eqref{equivalent-formation} imposes a certain kind of smoothness on $f$ and characterizes the corresponding derivatives in terms of $D^{(\alpha)} f$. 
For each $\alpha \in \{0, 1\}^d \setminus \{\zerovec\}$ and for $x^{(\alpha)} = (x_k, k \in S(\alpha))$, let 
\begin{equation*}
    D_{\alpha} f(x^{(\alpha)}) = \lim_{\epsilon \rightarrow 0} \frac{1}{\epsilon^{|\alpha|}} \cdot \sum_{\delta \in \prod\limits_{k \in S(\alpha)}{\{0, 1\}}} (-1)^{\sum\limits_{k \in S(\alpha)} \delta_k} f\big(\widetilde{x(x + \epsilon)}_{\delta}^{(\alpha)}\big),
\end{equation*}
if the limit exists, where 
\begin{equation*}
    \big(\widetilde{x(x + \epsilon)}_{\delta}^{(\alpha)}\big)_k = 
    \begin{cases}
        \delta_k x_k + (1 - \delta_k) (x_k + \epsilon) &\mbox{if } k \in S(\alpha)  \\
        0 &\mbox{otherwise}
    \end{cases}
\end{equation*}
for $k \in [d]$. 
For example, if $d = 3$ and $\alpha = (1, 1, 0)$, 
$D_{1, 1, 0} f$ is defined as 
\begin{align*}
    &D_{1, 1, 0} f (x_1, x_2) = \lim_{\epsilon \rightarrow 0} \frac{1}{\epsilon^2} \cdot \big(f(x_1 + \epsilon, x_2 + \epsilon, 0) - f(x_1, x_2 + \epsilon, 0) \\
    &\qquad \qquad \qquad \qquad \qquad \qquad \qquad \qquad \qquad - f(x_1 + \epsilon, x_2, 0) + f(x_1, x_2, 0)\big),
\end{align*}
if the limit exists.
Note that in contrast to mixed partial derivatives $f^{(\alpha)}$ (defined in \eqref{partial-derivative}), in which partial derivatives $\partial/\partial x_j$ are taken sequentially, here all the $j^{\text{th}}$ coordinates where $\alpha_j = 1$ are considered simultaneously.
Also, note that the remaining coordinates where $\alpha_j = 0$ are set to zero for $D_{\alpha} f$.

As in the case $d = s = 1$, we can show that $D_{\alpha}f$ exist and equal to $D^{(\alpha)}f$ almost everywhere (with respect to the Lebesgue measure) on $[0, 1]^{|\alpha|}$. 
The precise statement is given in the following result.

\begin{proposition}\label{galpha-chracterization}
Suppose that the condition \eqref{equivalent-formation} holds.
Then, for each $\alpha \in \{0, 1\}^d \setminus \{\zerovec\}$, $D_{\alpha} f = 0$ if $|\alpha| > s$, and $D_{\alpha} f = D^{(\alpha)} f$ almost everywhere (with respect to the Lebesgue measure) on $[0, 1]^{|\alpha|}$ if $|\alpha| \le s$. 
\end{proposition}

Proposition \ref{prop:equivalent-formation} also implies that every
sufficiently smooth function belongs to $\infmars^{d, d}$. This is proved
in the next result, which also gives an expression for $\Vmars(f)$ in
terms of the $L^1$ norms of the mixed partial derivatives of $f$, for sufficiently smooth functions $f$.

The following notation will be used below. For each $\alpha \in \{0, 1\}^d \setminus
\{\zerovec\}$, we let $J_{\alpha}$ be the set of all $\beta \in
\{0, 1, 2\}^d$ such that 
\begin{equation}\label{definition-of-J-alpha}
\max_j \beta_j = 2
    ~\text{ and }~ \beta_j =
  \begin{cases} 
 0   & \text{if } \alpha_j = 0 \\
  1 \text{ or } 2       & \text{if } \alpha_j = 1.
 \end{cases}
\end{equation}
Also, recall the notation $\bar{\Sbeta}^{(\beta)}$ from \eqref{tbeta}.

\begin{lemma}\label{lem:smooth-functions}
  Suppose $f: [0, 1]^d \rightarrow \R$ is smooth in the sense
  that
  \begin{longlist}
  \item $f^{(\alpha)}$ exists and is continuous on $[0, 1]^d$ for every $\alpha
    \in \{0, 1\}^d$, and
  \item   $f^{(\beta)}$ exists and is continuous on
    $\bar{\Sbeta}^{(\alpha)}$ for every $\beta \in J_{\alpha}$, for
    every $\alpha \in \{0, 1\}^d \setminus \{\zerovec\}$.
  \end{longlist}
Then, $f \in \infmars^{d, d}$ and 
    \begin{equation}\label{vsmooth}
    \Vmars(f) = \sumoverbeta \int_{\bar{\Sbeta}^{(\beta)}} |f^{(\beta)}|. 
    \end{equation}
Furthermore, if $f^{(\alpha)} = 0$ for all $\alpha \in \{0, 1\}^d$ with $|\alpha| > s$ in addition, then $f \in \infmars^{d, s}$ and
\begin{equation}\label{vsmooth-restricted}
    \Vmars(f) = \sumoveralpha \sum_{\beta \in J_{\alpha}} \int_{\bar{\Sbeta}^{(\beta)}} |f^{(\beta)}|. 
\end{equation}
\end{lemma}
Note that the integrals on the right-hand side of
\eqref{vsmooth} and \eqref{vsmooth-restricted} are only over those coordinates $x_l$ for which
$\beta_l = \max_j \beta_j$ (the remaining coordinates are set to
zero).

The formula \eqref{vsmooth} is a multivariate generalization of the univariate formula \eqref{2ndder1}, stating that for sufficiently smooth functions $f$, $\Vmars(f)$ is the sum of
the $L^1$ norms of the mixed partial derivatives of $f$ of order 2, where we take at most two partial derivatives along each coordinate and exactly two partial derivatives along at least one coordinate. 
Note that mixed partial derivatives of total order ($\beta_1 + \cdots + \beta_d$) up to $2d$ appear in
\eqref{vsmooth}. 
From this perspective, we can think of the smoothness order of our complexity measure $\Vmars(\cdot)$ as $2d$, which is proportional to the dimension $d$. 
This gives an intuitive explanation on why our estimators achieve the dimension-free rate $n^{-4/5}$ (up to the logarithmic multiplicative factors), as we observed in Section \ref{riskresults}.
However, we should also note that the maximum total order $2d$ is solely achieved by the mixed partial derivative $f^{(2, \dots, 2)}$,
which prevents our function class from being too small and restrictive. 
For these reasons, we can say that our complexity measure is an effective constraint that leads to estimators avoiding the curse of dimensionality (to some extent) while keeping the corresponding function class reasonably large.

There is also an 
interesting connection between $\Vmars(\cdot)$ and $\Vhk(\cdot)$ via \eqref{vsmooth}. 
Specifically, if the condition $\max_{j} \beta_j = 2$ in
\eqref{vsmooth} is replaced with $\max_{j} \beta_j = 1$, then one obtains the formula for
$\Vhk(f)$ for sufficiently smooth functions $f$ 
(see Lemma \ref{lem:hardy-krause}).
Hence, we can also view
$\Vmars(\cdot)$ as \textit{second-order} Hardy--Krause variation (anchored at $\zerovec$).

In Appendix \ref{alt_char_ex}, we describe the results of this section by specializing to
the case $d = s = 2$. 
We encourage readers who find the results in this section too abstract to refer to Appendix \ref{alt_char_ex} for more explicit formulae.

\section{Related Work}\label{relwork}
Here are some connections between our paper and existing works
on nonparametric regression.

As mentioned earlier, our method can be viewed as a multivariate
generalization of the piecewise linear locally adaptive regression
spline estimator of \cite{mammen1997locally} (see also
\cite{steidl2006splines}). There are other ways of generalizing the
piecewise linear locally adaptive regression splines estimator to the
multivariate setting as well (see, e.g., \cite{parhi2021banach, parhi2022near} for some recent work).
Also, we have seen that our method can be considered as a multivariate extension of first-order trend filtering (see, e.g., \cite{kim2009ell_1, tibshirani2014adaptive}). 
\cite{ortelli2022tensor} and \cite{sadhanala2021multivariate} recently studied different multivariate extensions of trend filtering.
Although they covered all orders of trend filtering in contrast to our method, their methods were however restricted to lattice designs.
Moreover, they imposed weaker penalties on their models, which resulted in their estimators converging to the true underlying function at dimension-dependent rates.
In Appendix \ref{alt_fin_opt}, we describe the method of \cite{ortelli2022tensor} and compare their estimator to the discrete formation of our estimator in the equally spaced lattice design setting (see Remark \ref{comp_ortelli} for more details).

We have also mentioned that $\Vmars(\cdot)$ can be viewed as second-order Hardy--Krause variation (anchored at $\zerovec$). Our
estimation strategy can thus be seen as second-order Hardy--Krause variation
denoising. In \cite{fang2021multivariate}, first-order Hardy--Krause variation
denoising (i.e., least squares estimation over
functions with bounded Hardy--Krause variation) was studied. First-order
Hardy--Krause variation denoising leads to piecewise constant fits while
our method leads to MARS fits (linear combinations of
products of ReLU functions of individual variables). 
\cite{fang2021multivariate} also proved that their estimator achieves a dimension-free (up to a logarithmic multiplicative factor) rate of convergence. 
However, it should be noted that their result is only proved in the fixed lattice design setting.   
Moreover, unlike our method, interaction order restriction is not considered in \cite{fang2021multivariate}.

As is clear from the form of our functions \eqref{form-of-functions} and our complexity measure \eqref{vardef}, our method can also be considered as a multivariate ANOVA modeling method based on total variation constraints. 
There are a few works that utilize total variation penalties in multivariate ANOVA modeling. 
\cite{petersen2016fused}, \cite{yang2018backfitting}, and \cite{sadhanala2019additive} utilized total variation of univatiate functions in additive modeling, which can be seen as a special case of ANOVA modeling where the interaction between covariates is not allowed.
Also, \cite{yang2021hierarchical} used  for multivariate ANOVA modeling a class of penalties characterized in terms of certain hierarchical notions of total variation. 
Their hierarchical total variations are defined using a pre-fixed grid of points.  Interestingly, for functions $f$ that are sufficiently smooth, one of their hierarchical total variations (corresponding to $m = 2$ in their notation) converges to $\Vmars(f)$ as the grid resolution becomes arbitrarily
small.

It should be mentioned that \cite{lin1998thesis, lin2000tensor} also
studied a multivariate ANOVA modeling method, but instead of $L^1$
norms as in our paper, they worked with penalties that
are related to the squared $L^2$-Sobolev norms. Relevance of their
works to our paper is therefore not from the type of penalties but
from tensor product structures on their basis functions. As our basis
functions \eqref{basismars} are the tensor products of univariate ReLU
functions, their basis functions are also the tensor products of
univariate functions whose smoothness is constrained by the
$L^2$-Sobolev norms.  It is notable that the multivariate function spaces
considered in \cite{lin1998thesis, lin2000tensor} are defined as an
appropriate completion of the pre-Hilbert space given by the tensor
product of the univariate 
$L^2$-Sobolev spaces. We are curious whether we can also view our
function classes (e.g., $\infmars^{d, d}$) as an appropriate completion
of a tensor product space. However, it is unclear to us at this point
what norm  should be chosen for completion as a counterpart of the
$L^2$-Sobolev norms of \cite{lin1998thesis, lin2000tensor}. We believe
this is an interesting direction to extend our work, which can
provide a new perspective on our function spaces.  

In addition, \cite{van2023efficient} discussed 
function classes similar to $\infmars^{d, d}$ and norms similar 
to $\Vmars(\cdot)$ and used them for estimation in settings that are
different from our classical nonparametric regression framework.

\section{Numerical Experiments}\label{experiment}
In this section, we provide the results of some numerical experiments
illustrating the performance of our estimators from either the original or the approximate method.
The performance of our estimators is compared to the performance of the usual MARS estimator in our experiments.
The results for simulated data are presented first, and those for real data follow next.
Our methods are implemented in the R package \textsf{regmdc}, which is available at \url{https://github.com/DohyeongKi/regmdc}. 
Our R package \textsf{regmdc} employs the R package \textsf{Rmosek} (based on interior point convex optimization) to solve the finite-dimensional lasso problem \eqref{finite-dimensional-lasso}.
For the usual MARS estimator, we use the R package \textsf{earth} based on \cite{friedman1991multivariate} and \cite{friedman1993fast}.
The code for all the results in this section is available at \url{https://github.com/DohyeongKi/mars-lasso-paper}.

\subsection{Simulation Studies}
\cite{mammen1997locally} and \cite{tibshirani2014adaptive} demonstrated that their locally adaptive regression splines and trend filtering excel at estimating functions with locally varying smoothness.
Considering that our estimator is a multivariate generalization of the piecewise (second-order) locally adaptive regression splines estimator and the first-order trend filtering estimator, it is natural to examine whether our methods also excel at adapting to the local variation of functions.
To test the local adaptivity of our methods and compare it to the local adaptivity of the usual MARS method, we exploit in our simulation studies the following four functions (Function 1, Function 2, Function 3, and Function 4) whose smoothness varies significantly over the domain.

The definitions of the four functions (Function 1, Function 2, Function 3, and Function 4) are presented below. 
For each function, we consider the uniform design where $X^{(i)}$ are uniformly distributed on $[0, 1]^d$ and $y_i$ are generated according to the model \eqref{eq:model-random}.
For Function 1 and Function 3, we also consider the equally spaced lattice design where
\begin{equation*}
\{x^{(1)}, \dots, x^{(n)} \} = \prod_{k=1}^{d}
\Big\{\frac{i_k}{n_k}: i_k \in [0: (n_k - 1)] \Big\} 
\end{equation*}
for some integers $n_k \ge 2$.
In all cases, independent Gaussian errors with standard deviation 1 are added to function evaluations. 
Simulations are performed at various sample sizes for each function as presented in Table \ref{tab:Simulation-results}.

\begin{itemize}
  \item \textit{Function 1 (L1)}.
  The first function $f^*: [0, 1]^{2} \rightarrow \R$ is defined by
  \begin{align*}
    f^*(x_1, x_2) = 10 \exp(-5 \cdot r(x_1, x_2)) 
    \cdot \cos(10\pi \cdot r(x_1, x_2))
  \end{align*}
  for $x_1, x_2 \in [0, 1]$, where $r(x_1, x_2) = \sqrt{(x_1 - 0.3)^2 + (x_2 - 0.4)^2}$. 
  This function represents a two-dimensional damped sinusoidal wave.
  \\
  \item \textit{Function 2 (L2)}. 
  The second function $f^*: [0, 1]^{5} \rightarrow \R$ is defined as in Function 1 but additionally includes three dummy variables $x_3, x_4$, and $x_5$.
  \\
  \item \textit{Function 3 (L3)}.
  The third function $f^*: [0, 1]^{2} \rightarrow \R$ is defined by
  \begin{align*}
    f^*(x_1, x_2) = 5 \sin\Big(\frac{4}{\sqrt{x_1^2 + x_2^2} + 0.001}\Big) + 7.5
  \end{align*}
  for $x_1, x_2 \in [0, 1]$.
  This function is a (scaled) two-dimensional version of the 
  Doppler function used for simulation studies in \cite{mammen1997locally} and \cite{tibshirani2014adaptive}. 
  We add 0.001 to avoid division by zero.
  \\
  \item \textit{Function 4 (L4)}. 
  The fourth function $f^*: [0, 1]^{5} \rightarrow \R$ is defined as in Function 3 but additionally includes three dummy variables $x_3, x_4$, and $ x_5$.
  \\
\end{itemize}

We also use for comparison the following four Friedman's functions (see \cite{friedman1991multivariate}, Section 4.3 and 4.4), which have been frequently utilized to measure the performance of nonparametric regression methods
(see, e.g., \cite{meyer2003support, potts2021interpretable, potts2022learning}).
The four functions (Function 5, Function 6, Function 7, and Function 8) are defined as below, and for each function, we generate data $(X^{(1)}, y_1), \dots, (X^{(n)}, y_n)$ according to the model \eqref{eq:model-random} where $X^{(i)}$ are uniformly distributed on $[0, 1]^d$. 
For Function 5 and Function 6, we consider Gaussian errors with standard deviation 1 as in \cite{friedman1991multivariate}, Section 4.3. 
For Function 7 and Function 8, we consider Gaussian errors with standard deviation 125 and 0.1, which yield a 3:1 signal to noise ratio as designed in \cite{friedman1991multivariate}, Section 4.4 (see also \cite{meyer2003support}). 
Also, for each function, we conduct experiments on three different sample sizes as in \cite{friedman1991multivariate}: 50, 100, and 200 for Function 5 and Function 6; 100, 200, and 400 for Function 7 and Function 8.  

\begin{itemize}
  \item \textit{Function 5 (F1)}.
  The first function $f^*: [0, 1]^{10} \rightarrow \R$ is defined by
  \begin{equation*}
    f^*(x_1, \dots, x_{10}) = 10 \sin(\pi x_1 x_2) + 20 (x_3 - 0.5)^2 + 10 x_4 + 5x_5
  \end{equation*}
  for $x_1, \dots, x_{10} \in [0, 1]$.
  \\
  \item \textit{Function 6 (F2)}.
  The second function $f^*: [0, 1]^{5} \rightarrow \R$ is defined as in Function 5, but here the five variables $x_6, \dots, x_{10}$ are removed.
  \\
  \item \textit{Function 7 (F3)}.
  The third function $f^*: [0, 1]^{4} \rightarrow \R$ is defined by
  \begin{equation*}
    f^*(x_1, x_2, x_3, x_4) = \sqrt{(T_1(x_1))^2 + \big(T_2(x_2) \cdot x_3 - (T_2(x_2) \cdot T_4(x_4))^{-1} \big)^2}
  \end{equation*}
  for $x_1, x_2, x_3, x_4 \in [0, 1]$, where $T_1, T_2$, and $T_4$ are the linear maps defined by 
  $T_1(x_1) = 100x_1$, $T_2(x_2) = 2 \pi (260 x_2 + 20)$, and $T_4(x_4) = 10x_4 + 1$.
  \\
  \item \textit{Function 8 (F4)}.
  The fourth function $f^*: [0, 1]^{4} \rightarrow \R$ is defined by
  \begin{equation*}
    f^*(x_1, x_2, x_3, x_4) = \arctan \Big( \frac{T_2(x_2) \cdot x_3 - (T_2(x_2) \cdot T_4(x_4))^{-1}}{T_1(x_1)} \Big)
  \end{equation*}
  for $x_1, x_2, x_3, x_4 \in [0, 1]$, where $T_1, T_2$, and $T_4$ are defined as in Function 7.
  \\
\end{itemize}

In all simulations, we compare the estimator $\hat{f}_{n, V}^{d, 2}$ and its approximate version $\tilde{f}_{n, V}^{d, 2}$ to the usual MARS estimator whose order of interaction is restricted to 2. 
The approximate method is considered only when explanatory variables are uniformly distributed. 
The positive integers $N_k$ for the approximate method is set to 25 for each coordinate.
Also, the tuning parameter $V$ is selected by 10-fold cross-validation in all cases.
For each function, we repeat the above data generation and estimator construction processes 25 times and average out computed losses.
We use the squared empirical $L^2$ norm \eqref{squared-empirical-norm} for the fixed lattice design, and for the uniform design, we generate 1000 new samples and approximate the prediction error \eqref{prediction-error}.

Table \ref{tab:Simulation-results} presents the average loss of our estimators and the usual MARS estimator over 25 repetitions for each function.
In Table \ref{tab:Simulation-results}, we can first observe that our estimators outperform the usual MARS estimator in capturing the local variation of the first four functions, L1, L2, L3, and L4, under both the lattice and the uniform design.
The only exceptions are when we estimate the function L2 with 100 samples and the function L4 with 100, 200, and 400 samples. 
Even for these functions, our methods start performing better when they are given more samples. 
It turns out that our methods usually benefit more from increases in sample sizes compared to the usual MARS method. 
By comparing L1 with L2 and L3 with L4, we can also see that our methods however suffer more from the addition of dummy variables.

Table \ref{tab:Simulation-results} also shows that our methods estimate the Friedman's functions F2, F3, and F4 much better, while the usual MARS method do better in the estimation of the function F1.
However, the performance gap in estimating F1 narrows as the sample size increases, which reaffirms that increases in sample sizes are often more favorable for our methods than the usual MARS method.
Also, comparing the results of F2 with those of F1, again we can see that our methods can benefit more from removing unnecessary covariates in advance. 
This hints that having appropriate variable selection as a pre-processing step can greatly improve the performance of our methods.
We think it can be a promising avenue for future work.

Moreover, we can observe from Table \ref{tab:Simulation-results} that the original and approximate method yield almost the same outputs in most cases. 
Only for the Doppler function L3, which has extremely wild fluctuation around the origin, the two methods show a notable difference.
This tells us there may not be much degradation in practice from using the approximate method in place of the original one, and there even can be some gains, as we observed in the case of L3.

From these simulation results, we can expect that although the usual MARS estimator might perform better on small sample sizes, our estimators usually outdo the usual MARS estimator when given enough samples. 
At what sample sizes such a transition occurs will definitely vary across functions, but in the above examples, they were reasonably small.

\begingroup
\begin{table}\label{simulation-results}
\centering
\caption{\label{tab:Simulation-results} The average and standard error of squared empirical $L^2$ norms \eqref{squared-empirical-norm} (lattice design) or prediction errors \eqref{prediction-error} (uniform design) for each experimental setting. 
The number below the name of each function  (if any) is the scale that should be multiplied to the corresponding averages and standard errors.}
\begin{tabular}{ c | c | c | c  c  c }
  \hline
  Function  & Design & Number of Data & Usual Method & Original Method & Approximate Method \\ \hline
  \multirow{6}{*}{L1} & \multirow{2}{*}{Lattice} &  256 & 3.66 (0.07) & \textbf{1.06 (0.07)} & -- \\ 
  && 1024 & 3.34 (0.04) & \textbf{1.09 (0.05)} & -- \\ \cline{2-6}
  & \multirow{4}{*}{Uniform} & 100 & 3.30 (0.15) & 3.21 (0.09) & \textbf{3.20 (0.08)} \\ 
  && 200 & 3.18 (0.09) & \textbf{3.00 (0.05)} & 3.02 (0.05) \\ 
  && 400 & 2.90 (0.05) & \textbf{2.79 (0.05)} & 2.80 (0.04) \\
  && 800 & 2.90 (0.03) & -- & \textbf{2.62 (0.03)} \\\hline
  \multirow{4}{*}{L2} & \multirow{4}{*}{Uniform} & 100 & \textbf{3.29 (0.08)} & 3.59 (0.07) & 3.59 (0.07) \\ 
  && 200 & 3.15 (0.08) & \textbf{3.13 (0.05)} & \textbf{3.13 (0.05)} \\ 
  && 400 & 3.12 (0.05) & \textbf{3.09 (0.04)} 
  & \textbf{3.09 (0.04)} \\
  && 800 & 2.92 (0.06) & -- & \textbf{2.91 (0.05)} \\\hline
  \multirow{6}{*}{L3} & \multirow{2}{*}{Lattice} & 256 & 3.76 (0.08) & \textbf{1.83 (0.07)} & -- \\ 
  && 1024 & 3.43 (0.07) & \textbf{1.57 (0.05)} & -- \\ \cline{2-6}
  & \multirow{4}{*}{Uniform} & 100 & 5.37 (0.16) & 5.20 (0.23) & \textbf{5.11 (0.21)} \\ 
  && 200 & 4.06 (0.15) & 3.59 (0.08) & \textbf{3.43 (0.09)} \\
  && 400 & 3.54 (0.07) & 3.06 (0.17) & \textbf{2.27 (0.08)} \\
  && 800 & 3.28 (0.07) & -- & \textbf{1.51 (0.06)} \\\hline
  \multirow{4}{*}{L4} & \multirow{4}{*}{Uniform} & 100 & \textbf{6.25 (0.23)} & 7.62 (0.14) & 7.62 (0.14) \\ 
  && 200 & \textbf{4.47 (0.17)} & 6.03 (0.12) & 5.94 (0.10) \\
  && 400 & \textbf{3.63 (0.07)} & 4.30 (0.08) & 4.25 (0.09) \\
  && 800 & 3.28 (0.05) & --  & \textbf{3.10 (0.06)} \\\hline
  \multirow{3}{*}{F1} & \multirow{3}{*}{Uniform} & 50 & \textbf{4.05 (0.37)} & 8.60 (0.35) & 8.89 (0.36) \\ 
  && 100 & \textbf{1.03 (0.09)} & 4.39 (0.36) & 4.49 (0.38) \\ 
  && 200 & \textbf{0.43 (0.02)} & 0.87 (0.03) & 0.87 (0.03) \\ \hline
  \multirow{3}{*}{F2} & \multirow{3}{*}{Uniform} & 50 & 3.25 (0.28) & \textbf{2.45 (0.31)} & 2.48 (0.33) \\ 
  && 100 & 1.13 (0.10) & 0.77 (0.12) & \textbf{0.72 (0.07)} \\ 
  && 200 & 0.41 (0.03) & 0.37 (0.03) & \textbf{0.35 (0.02)} \\ \hline
  \multirow{3}{*}{\shortstack[1]{F3 \\ $(\cdot 10^{3})$}}  & \multirow{3}{*}{Uniform} & 100 & 3.62 (0.37) & \textbf{2.10 (0.23)} & \textbf{2.10 (0.23)} \\ 
  && 200 & 2.16 (0.22) & \textbf{0.95 (0.07)} & \textbf{0.95 (0.07)} \\ 
  && 400 & 1.20 (0.15) & \textbf{0.55 (0.04)} & \textbf{0.55 (0.04)} \\ \hline
  \multirow{3}{*}{\shortstack[1]{F4 \\ $(\cdot 10^{-3})$}} & \multirow{3}{*}{Uniform} & 100 & 14.0 (1.04) & \textbf{11.5 (0.60)} & 11.6 (0.63) \\ 
  && 200 & 9.12 (0.78) & \textbf{7.40 (0.34)} & 7.43 (0.34) \\ 
  && 400 & 5.51 (0.35) & 5.25 (0.24) & \textbf{5.23 (0.23)} \\ \hline 
\end{tabular}
\end{table}
\endgroup

\subsection{Real Datasets}
Here we use a few standard real datasets (Earnings, Airfoil Self-Noise, Abalone, Concrete, Ozone, Red Wine, and White Wine dataset) to compare our estimators with the usual MARS estimator.
Brief descriptions of each dataset is presented in Appendix \ref{datasets}.

For each dataset, we first linearly transform each explanatory variable into $[0, 1]$. 
In general, there can be multiple options for linear transformations. 
If the domain of an explanatory variable is known as $[m, M]$, then it is natural to subtract $m$ from the variable and then divide it by $M - m$. 
In case the two extreme values of the domain, $m$ and $M$, are unknown, we can consider the same linear transformation after simply setting $m$ as the minimum and $M$ as the maximum among observed values.
Here we choose the latter option for all datasets.

We also split each dataset into a training set and a test set and use the mean squared error on the test set for comparison. 
We use $80\%$ observations as training data and the remaining $20\%$ observations as test data. 
Also, for every dataset, we focus on interactions between explanatory variables of order up to 2. 
In other words, we use $\hat{f}_{n, V}^{d, 2}$ or its approximate version $\tilde{f}_{n, V}^{d, 2}$ for estimating regression functions and compare it to the usual MARS estimator whose order of interaction is restricted to 2.
For the Earnings and the Airfoil Self-Noise datasets, we employ the original method, and for the other datasets, we employ the approximate method, with setting each $N_k$ to 25.
Furthermore, the tuning parameter $V$ is chosen by 10-fold cross-validation in all cases.

In Table \ref{tab:real_data_sets-summary}, we present for each dataset the number of explanatory variables, the number of data, and the type of our method we use (whether the original or the approximate method). 
Table \ref{tab:real-data-sets-results} shows the average mean squared error of our estimator and the usual MARS estimator over 25 random training and test set splits for each dataset.
In Table \ref{tab:real-data-sets-results}, we can see that our method outperforms the usual MARS method in almost all examples, while showing great improvement on the Airfoil Self-Noise, Concrete, Ozone, and White Wine dataset.
The Red Wine dataset is the only exception for which the usual MARS method produces better or equally good fits.
However, for its sibling (the White Wine dataset), which has about three times as many observations, our method is clearly a better option than the usual MARS method.
This again proves that the usual MARS method may beat our method on small sample sizes, but our method reclaims the throne once enough data are provided.

We can conclude from all the results in this section that not only can our estimators be an alternative to the usual MARS estimator, but also they can supplement each other.

\begingroup
\begin{table}\label{real_data_sets-summary}
\centering
\caption{\label{tab:real_data_sets-summary} The number of explanatory variables, the number of data, and the type of our method employed for each dataset. Original and Approx here stand for the original and  approximate method, respectively.}
\begin{tabular}{ c | c | c | c }
  \hline
  Dataset & Dimension & Number of Data & Method \\ \hline
  Earnings & 2 & 25437 & Original \\ \hline
  Airfoil Self-Noise & 5 & 1503 & Original \\ \hline
  Abalone & 7 & 4177 & Approx \\ \hline
  Concrete & 8 & 1030 & Approx \\ \hline
  Ozone & 9 & 330 & Approx \\ \hline
  Red Wine & 11 & 1599 & Approx \\ \hline
  White Wine & 11 & 4898 & Approx \\ \hline
\end{tabular}
\end{table}
\endgroup

\begingroup
\begin{table}\label{real-data-sets-results}
\centering
\caption{\label{tab:real-data-sets-results} The average and standard error of mean squared errors on test sets for each dataset. The number next to the name of each dataset is the scale that should be multiplied to the corresponding average and standard error.}
\begin{tabular}{ c | c  c }
  \hline
  Dataset & Usual Method & Our Method \\ \hline
  Earnings $(\cdot 10^{-1})$ & 2.68  (0.01) & $\textbf{2.65 (0.01)}$ \\ \hline
  Airfoil Self-Noise $(\cdot 10^{0})$ & 9.16 (0.31) & \textbf{3.67 (0.13)} \\ \hline
  Abalone $(\cdot 10^{0})$ & 4.64 (0.08) & \textbf{4.59 (0.08)} \\ \hline
  Concrete $(\cdot 10^{1})$ & 4.01 (0.08) & \textbf{2.57 (0.17)} \\ \hline
  Ozone $(\cdot 10^{1})$ & 1.68 (0.08) & \textbf{1.45 (0.06)} \\ \hline
  Red Wine $(\cdot 10^{-1})$ & \textbf{4.17 (0.08)} & \textbf{4.17 (0.07)} \\ \hline
  White Wine $(\cdot 10^{-1})$ & 5.19 (0.05) & \textbf{4.92 (0.07)} \\ \hline
\end{tabular}
\end{table}
\endgroup

\begin{acks}[Acknowledgments]
We are immensely grateful to the anonymous referees for their constructive comments and suggestions, which significantly improved the quality of the paper.
We also thank Prof. Trevor Hastie for helpful comments and discussion.

\end{acks}

\begin{funding}
The first author was supported by NSF Grant DMS-2023505, NSF CAREER
Grant DMS-1654589, and NSF Grant DMS-2210504.  
The third author was supported by NSF CAREER Grant DMS-1654589 and NSF
Grant DMS-2210504.
\end{funding}


\bibliographystyle{imsart-nameyear}
\bibliography{main}       

\newpage 

\appendix

In Appendix \ref{hkreview}, we provide the definition of Hardy--Krause variation and introduce some of its properties that we use in this paper.
In Appendix \ref{alt_char_ex}, we describe the results of Section \ref{smoothness} for the case $d = s = 2$.
In Appendix \ref{alt_fin_opt}, we consider the fixed lattice design setting where the design points $x^{(1)}, \dots, x^{(n)}$ form an equally
spaced lattice in $[0, 1]^d$ and compare our estimator under this setting to the multivariate trend filtering estimator proposed in \cite{ortelli2022tensor}.
Appendix \ref{datasets} contains descriptions of the datasets used for our numerical experiments in Section \ref{experiment}.
Proofs of all our results are included in Appendix \ref{proofs}.

\section{Hardy--Krause Variation}\label{hkreview}
Hardy--Krause variation is based on another variation called Vitali variation. 
We first recall Vitali variation and some related concepts
and then define Hardy--Krause variation using them. We will use the
following notation below. For two $m$-dimensional
vectors $u = (u_1, \dots, u_m)$ 
and $v = (v_1, \dots, v_m)$ with $u_i \le v_i$ for all $i \in [m]$,
we let  $[u, v]$ denote the axis-aligned closed rectangle 
\begin{equation*}
    [u, v] := \{x \in \R^m: u_i \le x_i \le v_i \text{ for all
    } i \in [m] \}. 
  \end{equation*}
The \textit{quasi-volume} of a function $g$ on $[u, v]$ is then defined by 
\begin{equation*}
    \Delta(g; [u, v]) = \sum_{\delta \in \{0, 1\}^m} (-1)^{\delta_1 +
      \dots + \delta_m} g(v_1 + \delta_1 (u_1 - v_1), \dots, v_m +
    \delta_m (u_m - v_m)). 
\end{equation*}
Also, we say that a collection $\mathcal{P}$ of subsets of $[0, 1]^m$ is an
axis-aligned split if   
\begin{equation*}
    \mathcal{P} = \bigg\{\prod_{k = 1}^{m} \big[u^{(k)}_{l_k - 1},
        u^{(k)}_{l_k}\big]: l_k \in [n_k] \text{ for } k \in
      [m]\bigg\} 
\end{equation*}
for some $0 = u^{(k)}_0 < \dots < u^{(k)}_{n_k} = 1$ for $k \in
[m]$. Observe that every element of $\mathcal{P}$ is an axis-aligned
closed rectangle of the form  
\begin{equation*}
    \Big[\big(u^{(1)}_{l_1 - 1}, \dots, u^{(m)}_{l_m - 1}\big),
      \big(u^{(1)}_{l_1}, \dots, u^{(m)}_{l_m}\big)\Big] 
\end{equation*}
for some $l_k \in [n_k]$ for each $k \in [m]$.

Vitali variation is then defined as follows. 
Note that it coincides with the usual total variation when $m = 1$. 

\begin{definition}[Vitali variation]
For a real-valued function $g$ on $[0, 1]^m$, we define the Vitali variation of $g$ by
\begin{equation*}
    V^{(m)}(g) = \sup_{\mathcal{P}} \sum_{A \in \mathcal{P}}
    |\Delta(g; A)|, 
\end{equation*}
where the supremum is taken over all axis-aligned splits $\mathcal{P}$
of $[0, 1]^m$. 
\end{definition}

We now turn to Hardy--Krause variation.
Recall that for each $m$-dimensional vector $\beta$ with nonnegative integer
components, we let
\begin{align*}
\bar{\Sbeta}^{(\beta)} = \bar{\Sbeta}^{(\beta)}_1  \times \dots \times
    \bar{\Sbeta}^{(\beta)}_m   ~\text{ where }~     \bar{\Sbeta}^{(\beta)}_k =  
        \begin{cases}
            [0, 1] &\mbox{if } \beta_k = \max_j \beta_j  \\
            \{0\} &\mbox{otherwise}.
        \end{cases}
\end{align*}

\begin{definition}[Hardy--Krause variation]\label{def:hardy-krause}
For a real-valued function $g$ on $[0, 1]^m$, we define the Hardy--Krause variation of $g$ by
\begin{equation}\label{hkdef}
    \Vhk(g) = \sum_{\alpha \in \{0, 1\}^m \setminus \{\zerovec\}}
    V^{(|\alpha|)}(g|_{\bar{\Sbeta}^{(\alpha)}}). 
\end{equation}
Here for each $\alpha \in \{0, 1\}^m \setminus \{\zerovec\}$,
$g|_{\bar{\Sbeta}^{(\alpha)}}$ is the restriction of $g$ to  
$\bar{\Sbeta}^{(\alpha)}$, and
$V^{(|\alpha|)}(g|_{\bar{\Sbeta}^{(\alpha)}})$  is the Vitali
variation of $g|_{\bar{\Sbeta}^{(\alpha)}}$  when treated as a
function on $[0, 1]^{|\alpha|}$.   
\end{definition}
Observe that if $m = 1$, there is only one term in \eqref{hkdef}, so that 
$\Vhk(g) = V^{(1)}(g) = V(g)$.  

The sets $\bar{\Sbeta}^{(\alpha)}$, $\alpha \in \{0, 1\}^m
\setminus \{\zerovec\}$ in Definition \ref{def:hardy-krause} are
the faces of $[0, 1]^m$ that are adjacent to the point
$\zerovec$. Thus, this Hardy--Krause variation is referred to
as Hardy--Krause variation anchored at $\zerovec$.  
It is also possible to define Hardy--Krause variation using different anchors (see, e.g., \cite{aistleitner2015functions}), but we only need Hardy--Krause variation anchored at $\zerovec$ for our purpose.
To highlight the special role of $\zerovec$ in our Hardy--Krause
variation, we abbreviate it to HK0 variation and use the notation $\Vhk(\cdot)$.

There is a well-known close connection between functions
with finite HK0 variation and the cumulative distribution
functions of finite signed measures. The following
result is a version of this connection, and we will use it for proving Proposition \ref{prop:equivalent-formation}. This result follows essentially from \cite{aistleitner2015functions}, Theorem 3 with minor modification. We provide our
argument for the modification in Appendix
\ref{pf:aistleitner-variant}. 

\begin{proposition}\label{prop:aistleitner-variant}
Suppose a real-valued function $g$ on $[0, 1]^m$ has finite
HK0 variation and is coordinate-wise right-continuous on
$[0, 1]^m$. Assume also that $g$ is coordinate-wise left-continuous at
each point $x \in [0, 1]^{m} \setminus [0, 1)^{m}$ with respect to all
the $j^{\text{th}}$ coordinates where $x_j = 1$.   
Then, there exists a unique finite signed measure $\nu$ on $[0, 1)^m$
such that
\begin{equation}\label{eq:aistleitner-variant}
    g(x) = \nu([\zerovec, x] \cap [0, 1)^m) \qt{for every $x \in [0,
      1]^m$}. 
  \end{equation}
Also, if a real-valued function $g$ defined on $[0, 1]^m$ is of the
form \eqref{eq:aistleitner-variant} for some finite signed measure
$\nu$ on $[0, 1)^m$, then $g$ has finite HK0
variation and
\begin{equation*}
    |\nu|([0, 1)^m \setminus \{\zerovec\}) = \Vhk(g).
\end{equation*}
\end{proposition}

Lastly, let us note that when $g$ is sufficiently smooth, its Hardy--Krause
variation $\Vhk(g)$ can be written as the sum of the $L^1$ norms of the
mixed partial derivatives of $g$ of order 1. 
The precise statement is given in the following lemma (proved in
Appendix \ref{pf:hardy-krause}).

\begin{lemma}
\label{lem:hardy-krause}
    Suppose a real-valued function $g$ defined on $[0, 1]^m$ is smooth in the sense that $g^{(\alpha)}$ exists and is continuous on $[0, 1]^m$ for every $\alpha \in \{0, 1\}^m$. 
    Then, $g$ has finite HK0 variation and 
    \begin{equation}\label{hksmooth}
        \Vhk(g) = \sum_{\alpha \in \{0, 1\}^m \setminus \{\zerovec\}} \int_{\bar{\Sbeta}^{(\alpha)}} |g^{(\alpha)}|.
    \end{equation}
\end{lemma}

\section{Results of Section \ref{smoothness} for the case $d = s = 2$}\label{alt_char_ex}
According to the definition in Section
\ref{sec:intro}, $\infmars^{2, 2}$ consists of all the functions of the form
\begin{align*}
    &f(x_1, x_2) = a_{0, 0} + \int_{[0, 1)} (x_1 - t_1)_+ \, d\nu_{1,
    0}(t_1) + \int_{[0, 1)} (x_2 - t_2)_+ \, d\nu_{0, 1}(t_2) \\
    &\qquad \qquad \quad+ \int_{[0,
  1)^2} (x_1 - t_1)_+ (x_2 - t_2)_+ \, d\nu_{1, 1}(t_1, t_2)
\end{align*}
  for $a_{0, 0} \in \R$ and signed measures $\nu_{1,
    0}$, $\nu_{0, 1}$, and $\nu_{1, 1}$ (on $[0, 1), [0, 1)$, and $[0,
  1)^2$ respectively).
  On the other hand, according to the characterization in Proposition
  \ref{prop:equivalent-formation}, $\infmars^{2, 2}$ consists of all the
  functions of the form
  \begin{equation}\label{abscont2}
   f(x_1, x_2) = a_{0, 0} + \int_0^{x_1} g_{1, 0}(t_1) \, dt_1 +
   \int_0^{x_2} g_{0, 1}(t_2) \, dt_2 + \int_{[0, x_1] \times [0, x_2]} g_{1,1}(t_1, t_2) \, d(t_1, t_2)
  \end{equation}
  for $a_{0, 0} \in \R$ and real-valued functions $g_{1, 0}, g_{0, 1}$, and
  $g_{1, 1}$ (on $[0, 1], [0, 1]$, and $[0, 1]^2$ respectively) which
  have finite Hardy--Krause variation and satisfy the one-sided
  continuity conditions stated in Proposition
  \ref{prop:equivalent-formation}. 
  The functions $g_{1, 0}, g_{0, 1}$, and $g_{1, 1}$ can be considered as particular derivatives of $f$, and they are named as $D^{(1, 0)} f$, $D^{(0, 1)} f$, and $D^{(1, 1)} f$.
  Hence, we can write the estimator
  $\hat{f}_{n, V}^{2, 2}$ alternatively as 
\begin{align*}
  \hat{f}_{n, V}^{2, 2} \in \argmin_{f} \bigg\{\sum_{i=1}^n \big(y_i - f(x^{(i)}) \big)^2 : \Vhk (D^{(1, 1)}f) + V (D^{(1, 0)}f) + V (D^{(0, 1)}f) \leq V \bigg\}. 
\end{align*}
Also, according to Proposition \ref{galpha-chracterization}, the functions $g_{1, 0}, g_{0, 1}$, and $g_{1, 1}$ coincide with another type of derivatives $D_{1, 0} f$, $D_{0, 1} f$, and $D_{1, 1} f$ of $f$ almost everywhere (with respect to the Lebesgue measure).
For smooth functions $f$, the formula \eqref{vsmooth} for $d = s = 2$ becomes
    \begingroup
    \allowdisplaybreaks
    \begin{align*}
      \Vmars(f) &= \int_0^1 \int_0^1 \big|f^{(2, 2)}(x_1, x_2) \big|
      \, dx_1 \, dx_2 \\
      &\quad + \int_0^1 \big|f^{(2, 1)}(x_1, 0) \big| \, dx_1 + 
      \int_0^1 \big|f^{(2, 0)}(x_1, 0) \big| \, dx_1 \\
      &\quad + \int_0^1 \big|f^{(1, 2)}(0, x_2) \big| \, dx_2 +
      \int_0^1 \big|f^{(0, 2)}(0, x_2) \big| \, dx_2. 
    \end{align*}
    \endgroup
This can be viewed as a second-order version of the Hardy--Krause
variation: 
\begin{align*}
  \Vhk(f) &= \int_0^1 \int_0^1 \big|f^{(1, 1)}(x_1, x_2) \big|
            \, dx_1 \, dx_2 \\ 
    &\quad + \int_0^1 \big|f^{(1, 0)}(x_1, 0) \big| \, dx_1 +
                  \int_0^1 \big|f^{(0, 1)}(0, x_2) \big| \, dx_2. 
\end{align*}

\section{Alternative Optimization Problem for Equally Spaced Lattice Designs}\label{alt_fin_opt}
Here we describe an alternative finite-dimensional
optimization problem for solving \eqref{our-problem-restated} when $s = d$ and the design points $x^{(1)}, \dots, x^{(n)}$ form an equally
spaced lattice in $[0, 1]^d$. 
The finite-dimensional optimization problem in this case can be considered as a discrete analogue of our problem \eqref{our-problem-restated} for the case $s = d$. 
We will compare this discrete analogue with another multivariate trend filtering method proposed in \cite{ortelli2022tensor}.

To motivate the general result, let us first
consider the univariate case $d = s = 1$. Let us write the data as
$(x^{(i)}, y_i), i = 0, 1, \dots, n-1$ with $x^{(i)} = i/n$ (note that
the index $i$ runs from $0$ to $n-1$ as opposed to $1$ to $n$). In
this case, it is well-known that a solution $\hat{f}^{1, 1}_{n, V}$ to
\eqref{our-problem-restated} can be obtained by solving the 
finite-dimensional optimization problem
\begin{equation} \label{alternative-form-lattice-d=1}
   \hat{\theta} = \underset{\theta \in \R^n}{\argmin} \bigg\{\sum_{i = 0}^{n
        - 1} (y_{i} - \theta_i)^2: n \cdot \sum_{i = 2}^{n - 1}
      \big|(D^{(2)} \theta)_i \big| \le V \bigg\} 
\end{equation}
where 
\begingroup
\allowdisplaybreaks
\begin{equation*}
    (D^{(1)} \theta)_i := \theta_i - \theta_{i - 1}, \ i \in [n-1] ~\text{ and }~ (D^{(2)} \theta)_i := (D^{(1)} \theta)_{i} -
    (D^{(1)} \theta)_{i - 1}, \ i \in [2: (n - 1)]
\end{equation*}
\endgroup
and taking
\begin{equation*}
    \hat{f}_{n, V}^{1, 1}(x) := \hat{\theta}_0 + n \big(\hat{\theta}_1 -
      \hat{\theta}_0 \big) x + \sum_{l=1}^{n-2}
    \big(\hat{\theta}_{l+1} -  2 \hat{\theta}_l + \hat{\theta}_{l-1}
    \big) (n x - l)_+. 
\end{equation*}
The connection between \eqref{our-problem-restated} and
\eqref{alternative-form-lattice-d=1} can be intuitively seen 
by noting that 
\begin{equation*}
  \Vmars(f) = \int_0^1 |f''(x)| \, dx \approx n \sum_{i=2}^{n-1} |(D^{(2)} \theta)_i| 
\end{equation*}
for $\theta_i = f(i/n), i \in [0: (n-1)]$ provided that $n$ is sufficiently
large and $f$ is sufficiently smooth.  The second-order difference vector $D^{(2)}
\theta$ (multiplied by $n^2$) therefore plays the role of the second
derivative in the continuous case.

We now turn to the case of general $d = s$ where we assume
that the design points $x^{(1)}, \dots, x^{(n)}$ are given by the
equally spaced lattice
\begin{equation*}
\{x^{(1)}, \dots, x^{(n)} \} = \prod_{k=1}^{d}
\Big\{\frac{i_k}{n_k}: i_k \in [0: (n_k - 1)] \Big\}, 
\end{equation*}
where $n_k \ge 2$ for all $k \in [d]$. It makes sense then to index
the data with the set
\begin{equation}\label{I0}
  I_0 := [0 : (n_1 - 1)] \times \dots \times [0 : (n_d - 1)]. 
\end{equation}
In this setting, our next result (Proposition
\ref{prop:reduction-to-discrete-analogue}) describes a finite-dimensional 
optimization problem for solving \eqref{our-problem-restated} that is
analogous to \eqref{alternative-form-lattice-d=1}. Recall from
\eqref{vsmooth} that the complexity $\Vmars(\cdot)$ for smooth functions now involves the $L^1$ norms of
the mixed partial derivatives of order 2. The finite-dimensional optimization problem in Proposition
\ref{prop:reduction-to-discrete-analogue} can be viewed as a discrete analogue of \eqref{our-problem-restated} for the case $s = d$, in which these mixed partial derivatives are
replaced by appropriate discrete differences. Let us first define
these discrete difference operators.

We first define a basic difference operator $D_d$ for each positive
integer $d$ as follows.  We use the notation $[p: q] =
\{p, p + 1, \dots, q\}$ for two integers $p \le q$. 
\begin{definition} \label{defnDd}
    Let $n_1, \dots, n_d$ be positive integers and $a_1, \dots, a_d$
    be nonnegative integers. For an $n_1 \cdots
    n_d$-dimensional vector $\theta$ that is indexed by 
    $i \in [a_1: (a_1 + n_1 - 1)] \times \dots \times [a_d: (a_d + n_d - 1)]$,
    we define $D_d\theta$ to be the $(n_1 - 1) \cdots (n_d -
    1)$-dimensional vector that is indexed by $i \in [(a_1 + 1): (a_1 + n_1 -
    1)] \times \dots \times [(a_d + 1): (a_d + n_d - 1)]$ and has entries   
    \begin{equation*}
        (D_d\theta)_i = \sum_{\delta \in \{0, 1\}^d} (-1)^{\delta_1 +
          \dots + \delta_d} \theta_{i - \delta} 
    \end{equation*}
    for $i \in [(a_1 + 1): (a_1 + n_1 - 1)] \times \dots \times [(a_d
    + 1): (a_d + n_d - 1)]$.  
\end{definition}

\begin{example}[$d = 1$ and $d = 2$] 
For $\theta \in \R^{n}$ indexed by $i \in [a: (a + n - 1)]$, $D_1 \theta$ is defined as the $(n - 1)$-dimensional vector indexed by $i \in [(a + 1): (a + n - 1)]$ for which
\begin{equation*}
    (D_1 \theta)_i = \theta_i - \theta_{i-1} 
\end{equation*}
for $i \in [(a + 1): (a + n - 1)]$.

For $\theta \in \R^{n_1 n_2}$ indexed by $(i_1, i_2) \in [a_1: (a_1 + n_1 - 1)] \times [a_2: (a_2 + n_2 - 1)]$, $D_2 \theta$ is defined as the $(n_1 - 1)(n_2 - 1)$-dimensional vector indexed by $(i_1, i_2) \in [(a_1 + 1): (a_1 + n_1 - 1)] \times [(a_2 + 1): (a_2 + n_2 - 1)]$ for which
\begin{equation*}
    (D_2 \theta)_{i_1, i_2} = \theta_{i_1, i_2} - \theta_{i_1 - 1, i_2} - \theta_{i_1, i_2 - 1} + \theta_{i_1 - 1, i_2 - 1}
\end{equation*}
for $(i_1, i_2) \in [(a_1 + 1): (a_1 + n_1 - 1)] \times [(a_2 + 1):
(a_2 + n_2 - 1)]$. 
\end{example}

Based on the basic difference operators, we define a general
difference operator $D^{(\alpha)}$ for each vector
$\alpha$ with nonnegative integer components.   
As in \eqref{partial-derivative}, we call $D^{(\alpha)}$ a (discrete) difference operator of order $\max_k \alpha_k$.  

\begin{definition}
    Let $n_1, \dots, n_d$ be positive integers and suppose that
    $\theta \in \R^{n_1 \cdots n_d}$ is indexed by
    $i \in [a_1: (a_1 + n_1 - 1)] \times \dots \times [a_d: (a_d + n_d -
    1)]$.   
    Then, we define $D^{(\alpha)} \theta$ in the following recursive way. 
    \begin{itemize}
        \item If $\alpha = \zerovec$, then we define $D^{(\zerovec)} \theta = \theta$.
        \item If $\alpha_k > 0$ for all $k$, then we recursively define $D^{(\alpha)} \theta  = D^{(\alpha - \onevec)} (D_d \theta)$.
        \item Otherwise, let $\theta'$ be the lower dimensional slice of $\theta$ consisting of $\theta_i$ where $i_k = a_k$ for all $k \in S_0(\alpha) := \{j: \alpha_j = 0\}$.
        Index $\theta'$ with the set $\prod_{k \not\in S_0(\alpha)} [a_k: (a_k + n_k - 1)]$. 
        Next, let $\alpha'$ be the vector with nonnegative integer components obtained by removing the zero components of $\alpha$. 
        For example, if $\alpha = (2, 0, 1, 0)$, then $\alpha'$ will be $(2, 1)$.
        Then, we can consider $D^{(\alpha')} \theta'$, which can be defined recursively. 
        Let $I'$ be the index set of $D^{(\alpha')} \theta'$ and construct $I$ from $I'$ by adding the singleton $\{a_k\}$ to the component where $[a_k: (a_k + n_k - 1)]$ was previously dropped, for every $k \in S_0(\alpha)$.
        We then define $D^{(\alpha)} \theta$ as follows: 
        \begin{itemize}
            \item $D^{(\alpha)} \theta$ is indexed by $i \in I$.
            \item For each $i \in I$, 
            \begin{equation*}
                (D^{(\alpha)} \theta)_i = (D^{(\alpha')} \theta')_{i'}, 
            \end{equation*}
            where $i'$ is obtained from $i$ by dropping out the $k^{th}$ component for every $k \in S_0(\alpha)$. 
            For instance, if $i = (3, 0, 1, 0)$, $i'$ will be $(3, 1)$.
        \end{itemize}
    \end{itemize}
\end{definition}

\begin{example} 
Suppose $\alpha = (2, 1)$ and $\theta \in \R^{n_1 n_2}$ is indexed by $(i_1, i_2) \in [0: (n_1 - 1)] \times [0: (n_2 - 1)]$. 
Since $\alpha_k > 0$ for all $k$, $D^{(\alpha)} \theta = D^{(\alpha - \onevec)} (D_2 \theta) = D^{(1, 0)} (D_2 \theta)$.
Let $\eta = D_2 \theta$ and note (from Definition \ref{defnDd}) that
$\eta$ is indexed by $(i_1, i_2) \in [1: (n_1 - 1)] \times [1: (n_2 - 1)]$. 
Next, let $\eta'$ be the lower dimensional slice of $\eta$, i.e., 
$\eta'$ is indexed by $i_1 \in [1: (n_1 - 1)]$ and $\eta'_{i_1} = \eta_{i_1, 1}$ for $i_1 \in [1: (n_1 - 1)]$. 
Then, since $D^{(1)} \eta' = D^{(0)}(D_1 \eta') = D_1 \eta'$ is indexed by $i_1 \in [2: (n_1 - 1)]$, 
$D^{(1, 0)} \eta$ is indexed by $(i_1, i_2) \in [2: (n_1 - 1)] \times \{1\}$ and 
$(D^{(1, 0)} \eta)_{i_1, 1} = (D^{(1)} \eta')_{i_1} = \eta'_{i_1} - \eta'_{i_1 - 1} = \eta_{i_1, 1} - \eta_{i_1 - 1, 1}$ for $i_1 \in [2: (n_1 - 1)]$. 
Consequently, $D^{(\alpha)} \theta$ is indexed by $(i_1, i_2) \in [2: (n_1 - 1)] \times \{1\}$ and 
\begin{align*}
    (D^{(\alpha)} \theta)_{i_1, 1} &= (D^{(1, 0)} \eta)_{i_1, 1} = \eta_{i_1, 1} - \eta_{i_1 - 1, 1} = (D_2 \theta)_{i_1, 1} - (D_2 \theta)_{i_1 - 1, 1} \\
    &= (\theta_{i_1, 1} - \theta_{i_1 - 1, 1} - \theta_{i_1, 0} + \theta_{i_1 - 1, 0}) - (\theta_{i_1 - 1, 1} - \theta_{i_1 - 2, 1} - \theta_{i_1 - 1, 0} + \theta_{i_1 - 2, 0}) \\
    &= \theta_{i_1, 1} - \theta_{i_1, 0} - 2\theta_{i_1 - 1, 1} + 2\theta_{i_1 - 1, 0} + \theta_{i_1 - 2, 1} - \theta_{i_1 - 2, 0}
\end{align*}
for $i_1 \in [2: (n_1 - 1)]$.
\end{example}

Let $\theta$ be an $n$-dimensional vector indexed with the set $I_0$ (given by
\eqref{I0}) where $n = n_1 \cdots n_d$.  From the definition of the
general difference operators, we can observe that for a $d$-dimensional vector $\alpha = (\alpha_1, 
\dots, \alpha_d)$ with nonnegative integer components, $D^{(\alpha)} \theta$ is indexed with the set 
$I_0^{(\alpha)} := I_1^{(\alpha)} \times \dots \times
I_d^{(\alpha)}$ where 
\begin{align*}
    I_k^{(\alpha)} = 
    \begin{cases}
        \{\alpha_k\} &\mbox{if } \alpha_k \neq \max_j \alpha_j \\
        [\max_j \alpha_j: (n_k - 1)] &\mbox{if } \alpha_k = \max_j \alpha_j 
    \end{cases}
\end{align*}
for $k \in [d]$.
Thus,  
\begin{equation*}
    I_0^{(\alpha)} \cap I_0^{(\beta)} \neq \emptyset \quad \mbox{if and only if} \quad \alpha = \beta
\end{equation*}
for all $\alpha, \beta$ with $\max_j \alpha_j = \max_j \beta_j$,
and  
\begin{equation*}
    \biguplus_{\substack{\alpha \in [0: r]^d \\ \max_j \alpha_j = r}} I_0^{(\alpha)} = I_0 \setminus [0: (r - 1)]^d.
\end{equation*}
Here $\biguplus$ indicates disjoint union.
Also, we can check that for each $\alpha \in \{0, 1\}^d \setminus
\{\zerovec\}$, 
\begin{equation*}\label{decomposition-of-I-alpha}
    I_0^{(\alpha)} = \{\alpha\} \uplus \biguplus_{\beta \in J_{\alpha}} I_0^{(\beta)},
\end{equation*}
where $J_{\alpha}$ is as defined in
\eqref{definition-of-J-alpha}.

We are now ready to state Proposition
\ref{prop:reduction-to-discrete-analogue}. 
We prove this result in Appendix
\ref{pf:reduction-to-discrete-analogue}. 

\begin{proposition}\label{prop:reduction-to-discrete-analogue}
Let $\hat{\theta}^{d}_{n, V} \in \R^{|I_0|}$ be the unique solution to the problem 
\begin{equation}\label{discrete-analogue}
    \hat{\theta}^{d}_{n, V} = \underset{\theta \in \R^{|I_0|}}{\argmin} \big\{\|y - \theta\|_2^2: \theta \in C(V)\big\},
\end{equation}
where 
\begin{equation*}
    C(V) := \{\theta \in \R^{|I_0|}: V_2(\theta) \le V\} 
\end{equation*}
and 
\begin{equation}\label{v2theta}
    V_2(\theta) := \sumoverbeta \bigg[\bigg(\prod_{k = 1}^{d}
        n_k^{\beta_k - \ind\{\beta_k = 2\}}\bigg) \cdot \sum_{i
        \in I_0^{(\beta)}} \big|(D^{(\beta)}
        \theta)_i\big|\bigg]. 
  \end{equation}
Then, a solution to the original optimization problem
\eqref{our-problem-restated} for the case $s = d$ is the function $f$
defined on $[0, 1]^d$ by
\begin{equation}\label{discrete-analogue-to-original}
    f(x_1, \dots, x_d) = \big(\hat{\theta}^{d}_{n, V}\big)_{\zerovec} + \sum_{(\alpha, l) \in \tilde{J}} \big(H^{(2)} \hat{\theta}^{d}_{n, V}\big)_{\tilde{l} + \alpha} \cdot \prod_{k \in S(\alpha)} (n_k x_k - l_k)_{+},
\end{equation}
where $H^{(2)} \theta$ (for $\theta \in \R^{|I_0|}$) is the $n$-dimensional vector that is
indexed by $i \in I_0$ and has components
\begin{align*}
    (H^{(2)} \theta)_i = 
    \begin{cases}
        (D^{(\beta)} \theta )_i  &\mbox{if } i \in I_0^{(\beta)} \mbox{ for some } \beta \in \{0,1, 2\} \mbox{ with } \max_j \beta_j = 2 \\
        (D^{(i)} \theta )_i &\mbox{if } i \in \{0, 1\}^d,
    \end{cases}
\end{align*}
$\tilde{J}$ is the set
\begin{equation*}
    \tilde{J} = \bigg\{(\alpha, l): \alpha \in \{0, 1\}^d \setminus \{\zerovec\} \mbox{ and } l \in \prod_{k \in S(\alpha)}[0: (n_k - 2)]\bigg\},
\end{equation*}
and $\tilde{l}$ is the $d$-dimensional vector with 
\begin{equation*}
    \tilde{l}_k = 
    \begin{cases}
        l_k  &\mbox{if } k \in S(\alpha) \\
        0 &\mbox{otherwise}. 
    \end{cases}
  \end{equation*}
  Furthermore, for every solution $\hat{f}^{d, d}_{n, V}$ to
  \eqref{our-problem-restated}, the values of $\hat{f}^{d, d}_{n, V}$ at
  the design points equal the components of $\hat{\theta}^{d}_{n, V}$, i.e., 
\begin{equation*}
    \hat{f}^{d, d}_{n, V}\Big(\Big(\frac{i_k}{n_k}, k \in
        [d]\Big)\Big) = \big(\hat{\theta}^{d}_{n, V}\big)_i
    \qt{for all $i \in I_0$}. 
\end{equation*}
\end{proposition}

We can consider the finite-dimensional optimization problem \eqref{discrete-analogue} as a discrete analogue of the original problem \eqref{our-problem-restated} for the case $s = d$.  
Also, the formula \eqref{v2theta} for $V_2(\theta)$ can be seen as a
discrete analogue of the formula \eqref{vsmooth} for $\Vmars(f)$ for
smooth functions $f$. Indeed, for $\theta_{i} = f(i_1/n_1, \dots, i_d/n_d)$ for
$i \in I_0$, we have
\begin{equation*}
\bigg(\prod_{k = 1}^{d}
        n_k^{\beta_k - \ind\{\beta_k = 2\}}\bigg) \cdot \sum_{i
        \in I_0^{(\beta)}} \big|(D^{(\beta)}
        \theta)_i\big| \approx \int_{\bar{T}^{(\beta)}} |f^{(\beta)}| 
    \end{equation*}
when $f$ is smooth and $n_k, k \in [d]$ are large.

Denote by $\theta^*$ the $n$-dimensional vector with components $\theta^*_i =f^*(i_1/n_1, \dots, i_d/n_d)$, $i \in I_0$. 
Then, combining Theorem \ref{thm:risk-upper-bound-fixed} and Proposition \ref{prop:reduction-to-discrete-analogue}, we can show that the solution $\hat{\theta}^{d}_{n, V}$ to the problem \eqref{discrete-analogue} satisfies that  
\begin{equation}\label{discrete-risk-result}
    \E\Big[\frac{1}{n} \cdot \lVert \hat{\theta}^{d}_{n, V} - \theta^* \rVert_2\Big] = O \big(n^{-\frac{4}{5}} (\log n)^{\frac{3(2d-1)}{5}} \big)
\end{equation}
if $V_2(\theta^*) \le V$. 
Here the multiplicative constant underlying $O(\cdot)$ depends on $d, \sigma$, and $V$.

\begin{remark}\label{comp_ortelli}
The second-order Vitali trend filtering estimator of
\cite{ortelli2022tensor} is similar to our estimator
$\hat{\theta}^{d}_{n, V}$ defined in \eqref{discrete-analogue}. 
\cite{ortelli2022tensor} studied the estimator that can described in our notation as
\begin{equation}\label{ortelli-problem}
    \hat{\theta}_{\mathcal{N}_2^{\perp}} = \underset{\theta \in \R^{|I_0|}}{\argmin} \bigg\{\frac{1}{n} \big\lVert (y - \theta)_{\mathcal{N}_2^{\perp}} \big\rVert_2^2 + 2\lambda n \cdot \sum_{i \in I_0^{(2, 2, \dots, 2)}} \big|(D^{(2, 2, \dots, 2)} \theta)_i\big|\bigg\},
\end{equation}
where $\lambda > 0$ is a tuning parameter and $(y -
\theta)_{\mathcal{N}_2^{\perp}}$ denotes the orthogonal projection of
$(y - \theta)$ onto the orthogonal complement $\mathcal{N}_2^{\perp}$
of 
\begin{equation*}
    \mathcal{N}_2 := \big\{\theta \in \R^{|I_0|}: (D^{(2, 2,
        \dots, 2)} \theta)_i = 0 \mbox{ for all } i \in I_0^{(2, 2,
        \dots, 2)}\big\}. 
\end{equation*}
The optimization problem \eqref{ortelli-problem} is similar to
\eqref{discrete-analogue} in the sense that it also aims to find a
solution $\theta$ that is close to the observation $y$ and has small
second-order differences at the same time. However,
\eqref{ortelli-problem} measures the closeness to the observation via
$(y - \theta)_{\mathcal{N}_2^{\perp}}$, instead of $y - \theta$ as in
\eqref{discrete-analogue}.  
Also, \eqref{ortelli-problem} only penalizes $D^{(2, 2, \dots, 2)}
\theta$, while  \eqref{discrete-analogue} constrains $D^{(\beta)}
\theta$ for every $\beta \in \{0, 1, 2\}^d$ with $\max_j \beta_j =
2$. 

The estimator $\hat{\theta}_{\mathcal{N}_2^{\perp}}$ estimates
$\theta^*_{\mathcal{N}_2^{\perp}}$ but not the whole $\theta^*$. 
In contrast, our estimator $\hat{\theta}_{n, V}^d$ effectively estimates
$\theta^*$. 
\cite{ortelli2022tensor} provided a theoretical result on the accuracy of $\hat{\theta}_{\mathcal{N}_2^{\perp}}$ (as an estimator of
$\theta^*_{\mathcal{N}_2^{\perp}}$) that can be compared to our result \eqref{discrete-risk-result}.     
\cite{ortelli2022tensor}, Theorem 3.2 showed that with appropriately chosen $\lambda$, the estimator
$\hat{\theta}_{\mathcal{N}_2^{\perp}}$ (defined in
\eqref{ortelli-problem}) achieves the rate of convergence
\begin{equation}\label{ortellirate}
    n^{-\frac{H(d) + 3}{2H(d) + 3}} (\log n)^{\frac{H(d)}{2H(d) + 3}}
\end{equation}
for estimating $\theta^*_{\mathcal{N}_2^{\perp}}$.  
Here $H(d)$ is 
the harmonic number defined by $H(d) = 1 + 1/2 + \cdots + 1/d$.  
Note that in \eqref{ortellirate}, the exponent of $n$
in the rate depends on $d$, and it increases as $d$ does.
For example, it is $-4/5$ if $d = 1$, $-3/4$ if $d = 2$, and $-29/40$ if $d = 3$.
On the other hand, the exponent of $n$ in the rate in our result \eqref{discrete-risk-result} is $-4/5$ regardless of the value of $d$. 
\end{remark}

\section{Brief Descriptions of Datasets}\label{datasets}
\begin{itemize}
  \item \textit{Earnings Dataset}.   
  This dataset is from \textsf{ex1029} in the R library \textsf{Sleuth3} (see also \cite{bierens2001integrated}).
  It includes weekly earnings of 25,437 full-time male workers in 1987, along with their years of education and years of experience. 
  The original dataset also has a few more indicator variables, but we do not include them in our analysis.
  We predict the logarithm of weekly earnings from years of education and years of experience. \\
  \item \textit{Airfoil Self-Noise Dataset}.
  This dataset is available at \url{https://archive.ics.uci.edu/dataset/291/airfoil+self+noise} (see also \cite{brooks1989airfoil}).
  It contains the results of aerodynamic and acoustic tests on (NACA 0012) airfoils of different size, conducted in an anechoic wind tunnel. 
  The tests measured the scaled sound pressure level of the self-noise of the airfoils at various wind tunnel speeds and angles of attack. 
  This dataset has 1503 observations and six variables including the scaled sound pressure level. 
  Here the scaled sound pressure level is predicted from the other five variables describing test settings.
  \\
  \item \textit{Abalone Dataset}. 
  This dataset is from \textsf{abalone} in the R library \textsf{AppliedPredictiveModeling}, and it is also available at \url{https://archive.ics.uci.edu/dataset/1/abalone}.
  This dataset is composed of the physical measurements of 4177 abalones. 
  We aim to predict the number of rings each abalone has from seven physical measurements: length, diameter, height, whole weight, and the weight of meat, gut, and shell. 
  The number of rings is of interest because it can be used for estimating age.  
  Here we exclude from our analysis the indicator variable for the sex of each abalone. \\
  \item \textit{Concrete Dataset}. 
  This dataset is originally from \cite{yeh1998modeling} and  currently available at \url{https://archive.ics.uci.edu/dataset/165/concrete+compressive+strength}. 
  It contains the compressive strength of 1030 concrete samples together with their ages and the quantity of seven components, such as cement, fly ash, and water. 
  In our analysis, we have the compressive strength as a response variable and the other eight variables as explanatory variables. \\
  \item \textit{Ozone Dataset}. 
  This dataset is from \textsf{ozone} in the R library \textsf{faraway} (see also \cite{breiman1985estimating}). 
  It consists of 330 observations of atmospheric ozone concentration level and other nine meteorological variables, including temperature, wind speed, humidity, and visibility, measured in the Los Angeles Basin in 1976. 
  Here we predict atmospheric ozone concentration level from the other meteorological variables. \\
  \item \textit{Wine Dataset}. 
  This dataset was first introduced in \cite{cortez2009modeling}, and it is currently available at \url{https://archive.ics.uci.edu/dataset/186/wine+quality}. 
  It comprises the quality of wines (\textit{vinho verde}) from the Minho region of Portugal and other eleven attributes of wines exhibiting their physicochemical properties; for example, alcohol content, density, pH, and the quantity of citric acid. 
  The quality of wines was graded on a 10-point scale, with 0 indicating very bad and 10 indicating excellent. 
  This dataset has two parts: one for red wines and the other for white wines. 
  There are 1599 red wine data and 4898 white wine data.
  In our analysis, we fit models separately for red and white wines, predicting the quality of wines from the  physicochemical variables. \\
\end{itemize}

\section{Proofs}\label{proofs}
\subsection{Uniqueness of Representation
  \eqref{form-of-functions}}\label{pf:uniqueness-of-representation}
Here we prove the following lemma, which shows that
the constant $a_{\zerovec}$ and the signed measures $\nu_{\alpha}$ appearing in the
representation \eqref{form-of-functions} of $f \in \infmars^{d, s}$ are
uniquely determined by $f$.

\begin{lemma}\label{lem:uniqueness-of-representation}
    Suppose $a_{\zerovec}$ and $b_{\zerovec}$ are real numbers and $\{\nu_{\alpha}\}$ and $\{\mu_{\alpha}\}$ 
    are collections of finite signed measures as in \eqref{form-of-functions}.   
    If $f_{a_{\zerovec}, \{\nu_{\alpha}\}} \equiv f_{b_{\zerovec}, \{\mu_{\alpha}\}}$ on $[0, 1]^d$,
    then $a_{\zerovec} = b_{\zerovec}$ and
    $\nu_{\alpha} = \mu_{\alpha}$ for all
    $\alpha \in \{0, 1\}^d \setminus \{\zerovec\}$ with $|\alpha| \le s$.  
\end{lemma}    

\begin{proof}[Proof of Lemma \ref{lem:uniqueness-of-representation}]
By the assumption, 
\begin{align*}
    &a_{\zerovec} + \sumoveralpha
    \int_{[0, 1)^{|\alpha|}} \prod_{j \in S(\alpha)} (x_j -
    t_j)_+ \, d\nu_{\alpha}(\talpha) \\
    &\qquad \quad= b_{\zerovec} + \sumoveralpha
    \int_{[0, 1)^{|\alpha|}} \prod_{j \in S(\alpha)} (x_j -
    t_j)_+ \, d\mu_{\alpha}(\talpha)     
\end{align*}
for all $(x_1, \dots, x_d) \in [0, 1]^d$.
Substituting $x_1 = \dots = x_d = 0$, we get $a_{\zerovec} = b_{\zerovec}$. 
Suppose for some positive integer $m \le s$, we have $\nu_{\alpha} = \mu_{\alpha}$ for all $\alpha \in \{0, 1\}^d \setminus \{\zerovec\}$ with $|\alpha| \le m - 1$.
Substituting $x_{m +1} = \dots = x_{d} = 0$ and eliminating from both sides the components that are already known to be equal, we obtain
\begin{equation}\label{eq:measure-equation}
    \int_{[0, 1)^{m}} \prod_{j = 1}^{m} (x_j -
    t_j)_+ \, d\nu_{\onevec_m}(t) = \int_{[0, 1)^{m}} \prod_{j = 1}^{m} (x_j -
    t_j)_+ \, d\mu_{\onevec_m}(t). 
\end{equation}
Here $\onevec_m$ indicates the $d$-dimensional vector which has 1 in the first $m$ components and 0 in the rest.
Note that 
\begingroup
\allowdisplaybreaks
\begin{align*}
    &\int_{[0, 1)^{m}} \prod_{j = 1}^{m} (x_j -                                                                                                 
    t_j)_+ \, d\nu_{\onevec_m}(t) = \int_{\prod\limits_{j=1}^{m} [0, x_j)} \prod_{j=1}^{m} (x_j - t_j) \, d\nu_{\onevec_m}(t) \\
    &\qquad = \int_{\prod\limits_{j=1}^{m} [0, x_j)} \int_{\prod\limits_{j=1}^{m} (t_j, x_j)} 1 \, ds \, d\nu_{\onevec_m}(t) = \int_{\prod\limits_{j=1}^{m} (0, x_j)} \int_{\prod\limits_{j=1}^{m} [0, s_j)} 1 \, d\nu_{\onevec_m}(t) \, ds.
\end{align*}
\endgroup
Thus, the equation \eqref{eq:measure-equation} can be written as 
\begin{equation*}
    \int_{\prod\limits_{j=1}^{m} (0, x_j)} \int_{\prod\limits_{j=1}^{m} [0, s_j)} \, d\nu_{\onevec_m}(t) \, ds = \int_{\prod\limits_{j=1}^{m} (0, x_j)} \int_{\prod\limits_{j=1}^{m} [0, s_j)} \, d\mu_{\onevec_m}(t) \, ds.
\end{equation*}
By the Lebesgue differentiation theorem (see, e.g., \cite{Rudin:87book}, Theorem 7.10), it follows that  
\begin{equation}\label{eq:integration-of-measures}
    \int_{\prod\limits_{j=1}^{m} [0, x_j)} \, d\nu_{\onevec_m}(t) = \int_{\prod\limits_{j=1}^{m} [0, x_j)} \, d\mu_{\onevec_m}(t)
\end{equation}
almost everywhere with respect to the Lebesgue measure.

For a subset $R$ of $[0, 1)^m$, we call it rectangle if it is of the form 
\begin{equation*}
    R = \prod_{j = 1}^{m} I_j,
\end{equation*}
where for each $j$, $I_j = (u_j, v_j), [u_j, v_j], [u_j, v_j)$, or $(u_j, v_j]$ for some real numbers $u_j$ and $v_j$. 
Now, let $\mathcal{P}$ be the collection of all rectangles in $[0, 1)^m$ and $\mathcal{Q}$ be the collection of all subsets $E$ of $[0, 1)^m$ for which 
\begin{equation*}
    \nu_{\onevec_m}(E) = \mu_{\onevec_m}(E).
\end{equation*}
Using the equation \eqref{eq:integration-of-measures}, we can show that $\mathcal{P} \subseteq \mathcal{Q}$. 
Also, it is clear that $\mathcal{P}$ is a $\pi$-system, and $\mathcal{Q}$ is a $\lambda$-system (or Dynkin system) on $[0, 1)^m$.
Therefore, by Dynkin's $\pi$-$\lambda$ theorem, we have $\sigma(\mathcal{P}) \subseteq \mathcal{Q}$. 
Since $\sigma(\mathcal{P})$ is indeed the collection of all Borel sets in $[0, 1)^m$, we can conclude that $\nu_{\onevec_m} = \mu_{\onevec_m}$. 
Consequently, induction argument completes the proof.
\end{proof}

\subsection{Proofs of Lemmas and Propositions in Section \ref{existcompute}}\label{pf:existcompute}
\subsubsection{Proof of Lemma \ref{lem:reduction-to-discrete-measures}}\label{pf:reduction-to-discrete-measures}

Although the ideas of our proof of Lemma \ref{lem:reduction-to-discrete-measures} are simple, the proof is somewhat technical because of the intensive handling of indices.
Hence, to deliver the underlying ideas clearly, we first prove the lemma when $d = s = 1$ and then consider general $d$ and $s$.

\begin{proof}[Proof of Lemma \ref{lem:reduction-to-discrete-measures}]
First, assume that $d = s = 1$. 
For $l = 0, \dots, n_1 - 1 $, let $\tilde{\mu}_l$ be the discrete signed measure on $[0, 1]$ that is concentrated on $\{u^{(1)}_{l}, u^{(1)}_{l + 1}\}$ and has weights
\begin{equation*}
    \tilde{\mu}_{l}\big(\big\{u^{(1)}_l \big\}\big) = \int_{[u^{(1)}_{l}, u^{(1)}_{l + 1})} \frac{u_{l + 1}^{(1)} - t}{u_{l + 1}^{(1)} - u_{l}^{(1)}} \,  d\nu_{1}(t)
\end{equation*}
and 
\begin{equation*}
    \tilde{\mu}_{l}\big(\big\{u^{(1)}_{l + 1} \big\}\big) = \int_{[u^{(1)}_{l}, u^{(1)}_{l + 1})} \frac{t - u_{l}^{(1)}}{u_{l + 1}^{(1)} - u_{l}^{(1)}} \,  d\nu_{1}(t).
\end{equation*}
Then, for $x \ge u^{(1)}_{l + 1}$,  
\begingroup
\allowdisplaybreaks
\begin{align*}
    &\int_{[u^{(1)}_{l}, u^{(1)}_{l + 1})} (x - t)_{+} \, d\nu_{1}(t) \\
    &\qquad = \int_{[u^{(1)}_{l}, u^{(1)}_{l + 1})} \big(x - u^{(1)}_{l}\big) \cdot \frac{u^{(1)}_{l + 1} - t}{u^{(1)}_{l + 1} - u^{(1)}_{l}} + \big(x - u^{(1)}_{l + 1}\big) \cdot \frac{t - u^{(1)}_{l}}{u^{(1)}_{l + 1} - u^{(1)}_{l}} \, d\nu_{1} (t) \\
    &\qquad = \big(x - u^{(1)}_{l}\big) \cdot \tilde{\mu}_{l}\big(\big\{u^{(1)}_l \big\}\big) + \big(x - u^{(1)}_{l + 1}\big) \cdot \tilde{\mu}_{l}\big(\big\{u^{(1)}_{l + 1} \big\}\big) = \int_{[0, 1]} (x - t)_{+} \, d\tilde{\mu}_{l}(t).
\end{align*}
\endgroup
We also have that
\begingroup
\allowdisplaybreaks
\begin{align*}
    &|\tilde{\mu}_{l}|([0, 1]) = |\tilde{\mu}_{l}|\big(\big\{u^{(1)}_l \big\}\big) + |\tilde{\mu}_{l}|\big(\big\{u^{(1)}_{l + 1}\big\}\big) \\
    &\quad\le \int_{[u^{(1)}_{l}, u^{(1)}_{l + 1})} \frac{u_{l + 1}^{(1)} - t}{u_{l + 1}^{(1)} - u_{l}^{(1)}} \,  d|\nu_{1}|(t) + 
    \int_{[u^{(1)}_{l}, u^{(1)}_{l + 1})} \frac{t - u_{l}^{(1)}}{u_{l + 1}^{(1)} - u_{l}^{(1)}} \,  d|\nu_{1}|(t) = |\nu_{1}|\big(\big[u^{(1)}_{l}, u^{(1)}_{l + 1}\big)\big),
\end{align*}
\endgroup
and that
\begin{equation*}
     |\tilde{\mu}_{0}|((0, 1]) = |\tilde{\mu}_{0}|\big(\big\{u^{(1)}_1 \big\}\big) \le \int_{[0, u^{(1)}_{1})} \frac{t }{u_{1}^{(1)}} \,  d|\nu_{1}|(t) \le |\nu_{1}|\big(\big(0, u^{(1)}_1 \big)\big).
\end{equation*}
Let $\tilde{\mu}$ be the signed measure on $[0, 1]$ defined by
\begin{equation*}
    \tilde{\mu} = \sum_{l = 0}^{n_1 - 1} \tilde{\mu}_{l}
\end{equation*}
and let $\mu$ be the signed measure on $[0, 1)$ defined by $\mu(E) = \tilde{\mu}(E)$.
Note that $\mu$ is a discrete signed measure concentrated on $\mathcal{U}_1 \cap [0, 1)$.
Then, for each $u_{i}^{(1)} \in \mathcal{U}_1$, we have
\begingroup
\allowdisplaybreaks
\begin{align*}
    f_{a_0, \{\nu_1\}} \big(u_{i}^{(1)}\big) &= a_0 + \int_{[0, 1)} \big(u_{i}^{(1)} - t \big)_{+} \, d\nu_1(t) = a_0 + \sum_{l = 0}^{n_1 - 1} \int_{[u^{(1)}_{l}, u^{(1)}_{l + 1})}  \big(u_{i}^{(1)} - t\big)_{+} \, d\nu_1(t) \\
    &= a_0 + \sum_{l = 0}^{i - 1} \int_{[u^{(1)}_{l}, u^{(1)}_{l + 1})}  \big(u_{i}^{(1)} - t\big)_{+} \, d\nu_1(t) = a_0 +  \int_{[0, 1]} \big(u_{i}^{(1)} - t \big)_{+} \, d\tilde{\mu}(t) \\
    &= a_0 +  \int_{[0, 1)} \big(u_{i}^{(1)} - t \big)_{+} \, d\mu(t) = f_{a_0, \{\mu\}} \big(u_{i}^{(1)}\big),
\end{align*}
\endgroup
which means that $f_{a_0, \{\mu\}}$ coincides with $f_{a_0, \{\nu_1\}}$ at every point of $\mathcal{U}_1$.
Also, 
\begingroup
\allowdisplaybreaks
\begin{align*}
    |\mu|((0,1)) &\le |\tilde{\mu}|((0,1]) \le \sum_{l = 0}^{n_1 - 1} |\tilde{\mu}_{l}| ((0,1]) = |\tilde{\mu}_{0}|((0,1]) + \sum_{l = 1 }^{n_1 - 1} |\tilde{\mu}_{l}|([0,1]) \\
    &\le |\nu_{1}|\big(\big(0, u^{(1)}_1 \big)\big) + \sum_{l = 1}^{n_1 - 1} |\nu_1|\big(\big[u^{(1)}_{l}, u^{(1)}_{l + 1}\big)\big) = |\nu_1|((0, 1)),
\end{align*}
\endgroup
from which we can conclude that
\begin{equation*}
    \Vmars(f_{a_0, \{\mu\}}) \le \Vmars(f_{a_0, \{\nu_1\}}).
\end{equation*}

Now, we consider general $d$ and $s$.
For $\alpha \in \{0, 1\}^d \setminus \{\zerovec\}$ with $|\alpha| \le s$ and $l \in \prod_{k \in S(\alpha)} [0: (n_k - 1)]$, let $\tilde{\mu}_{\alpha, l}$ be the discrete signed measure on $[0, 1]^{|\alpha|}$ that is concentrated on $\prod_{k \in S(\alpha)} \{u^{(k)}_{l_k}, u^{(k)}_{l_k + 1}\}$ and has weights 
\begin{align*}
    &\tilde{\mu}_{\alpha, l}\big(\big\{\big(u^{(k)}_{l_k + \delta_k}, k \in S(\alpha)\big)\big\}\big) \\
    &\qquad \quad = \int_{R^{(\alpha)}_l} \prod_{k \in S(\alpha)} \bigg[\bigg(\frac{u_{l_k + 1}^{(k)} - t_k}{u_{l_k + 1}^{(k)} - u_{l_k}^{(k)}}\bigg)^{1 - \delta_k} \cdot \bigg(\frac{t_k - u_{l_k}^{(k)}}{u_{l_k + 1}^{(k)} - u_{l_k}^{(k)}}\bigg)^{\delta_k}\bigg] \,  d\nu_{\alpha}(\talpha)
\end{align*}
for each $\delta = (\delta_k, k \in S(\alpha)) \in \prod_{k \in S(\alpha)} \{0, 1\}$, where $R^{(\alpha)}_l := \prod_{k \in S(\alpha)} [u^{(k)}_{l_k}, u^{(k)}_{l_k + 1})$. 
Then, if $x_k \ge u^{(k)}_{l_k + 1}$ for all $k \in S(\alpha)$, it follows that 
\begingroup
\allowdisplaybreaks
\begin{align*}
    &\int_{R^{(\alpha)}_l} \prod_{k \in S(\alpha)} (x_k - t_k)_{+} \, d\nu_{\alpha}(\talpha) \\
    &\quad = \int_{R^{(\alpha)}_l} \prod_{k \in S(\alpha)} \bigg[ \big(x_k - u^{(k)}_{l_k}\big) \cdot \frac{u^{(k)}_{l_k + 1} - t_k}{u^{(k)}_{l_k + 1} - u^{(k)}_{l_k}} + \big(x_k - u^{(k)}_{l_k + 1}\big) \cdot \frac{t_k - u^{(k)}_{l_k}}{u^{(k)}_{l_k + 1} - u^{(k)}_{l_k}}\bigg] \, d\nu_{\alpha} (\talpha) \\
    &\quad = \int_{R^{(\alpha)}_l} \sum_{\delta \in \prod\limits_{k \in S(\alpha)} \{0, 1\}} \prod_{k \in S(\alpha)} \bigg[ \big(x_k - u^{(k)}_{l_k + \delta_k}\big) \cdot \bigg(\frac{u^{(k)}_{l_k + 1} - t_k}{u^{(k)}_{l_k + 1} - u^{(k)}_{l_k}}\bigg)^{1 - \delta_k}\\ 
    &\qquad \qquad \qquad \qquad \qquad \qquad \qquad \qquad \qquad \qquad \quad \cdot \bigg(\frac{t_k - u^{(k)}_{l_k}}{u^{(k)}_{l_k + 1} - u^{(k)}_{l_k}}\bigg)^{\delta_k}\bigg] \, d\nu_{\alpha}(\talpha) \\
    &\quad = \sum_{\delta \in \prod\limits_{k \in S(\alpha)} \{0, 1\}} \bigg( \prod_{k \in S(\alpha)} \big(x_k - u^{(k)}_{l_k + \delta_k}\big) \bigg) \\
    &\qquad \qquad \qquad \qquad \cdot \int_{R^{(\alpha)}_l} \prod_{k \in S(\alpha)} \bigg[\bigg(\frac{u^{(k)}_{l_k + 1} - t_k}{u^{(k)}_{l_k + 1} - u^{(k)}_{l_k}}\bigg)^{1 - \delta_k} \cdot \bigg(\frac{t_k - u^{(k)}_{l_k}}{u^{(k)}_{l_k + 1} - u^{(k)}_{l_k}}\bigg)^{\delta_k}\bigg] \, d\nu_{\alpha}(\talpha) \\ 
    &\quad = \sum_{\delta \in \prod\limits_{k \in S(\alpha)} \{0, 1\}} \bigg(\prod_{k \in S(\alpha)} \big(x_k - u^{(k)}_{l_k + \delta_k}\big)\bigg) \cdot \tilde{\mu}_{\alpha, l}\big(\big\{\big(u^{(k)}_{l_k + \delta_k}, k \in S(\alpha)\big)\big\}\big) \\
    &\quad = \int_{[0, 1]^{|\alpha|}} \prod_{k \in S(\alpha)} (x_k - t_k)_{+} \, d\tilde{\mu}_{\alpha, l}(\talpha).
\end{align*}
\endgroup
Also, 
\begingroup
\allowdisplaybreaks
\begin{align*}
    &|\tilde{\mu}_{\alpha, l}|\big([0, 1]^{|\alpha|}\big) \\
    &\qquad= \sum_{\delta \in \prod\limits_{k \in S(\alpha)} \{0, 1\}} \Bigg|\int_{R^{(\alpha)}_l} \prod_{k \in S(\alpha)} \bigg[\bigg(\frac{u_{l_k + 1}^{(k)} - t_k}{u_{l_k + 1}^{(k)} - u_{l_k}^{(k)}}\bigg)^{1 - \delta_k} \cdot  \bigg(\frac{t_k - u_{l_k}^{(k)}}{u_{l_k + 1}^{(k)} - u_{l_k}^{(k)}}\bigg)^{\delta_k}\bigg] \, d\nu_{\alpha}(\talpha)\Bigg| \\
    &\qquad \le \sum_{\delta \in \prod\limits_{k \in S(\alpha)} \{0, 1\}} \int_{R^{(\alpha)}_l} \prod_{k \in S(\alpha)} \bigg[\bigg(\frac{u_{l_k + 1}^{(k)} - t_k}{u_{l_k + 1}^{(k)} - u_{l_k}^{(k)}}\bigg)^{1 - \delta_k} \cdot \bigg(\frac{t_k - u_{l_k}^{(k)}}{u_{l_k + 1}^{(k)} - u_{l_k}^{(k)}}\bigg)^{\delta_k}\bigg] \, d|\nu_{\alpha}|(\talpha) \\
    &\qquad= \int_{R^{(\alpha)}_l} \prod_{k \in S(\alpha)} \bigg(\frac{t_k - u_{l_k}^{(k)}}{u_{l_k + 1}^{(k)} - u_{l_k}^{(k)}} + \frac{u_{l_k + 1}^{(k)} - t_k}{u_{l_k + 1}^{(k)} - u_{l_k}^{(k)}}\bigg) \, d|\nu_{\alpha}|(\talpha) = |\nu_{\alpha}|\big(R^{(\alpha)}_l\big), 
\end{align*}
\endgroup
and 
\begingroup
\allowdisplaybreaks
\begin{align*}
     &|\tilde{\mu}_{\alpha, \zerovec}|\big([0, 1]^{|\alpha|} \setminus \{\zerovec\}\big) \\
     &\qquad \quad= \sum_{\delta \in \prod\limits_{k \in S(\alpha)} \{0, 1\} \setminus \{\zerovec\}} \Bigg|\int_{R^{(\alpha)}_\zerovec} \prod_{k \in S(\alpha)} \bigg[\bigg(\frac{u_{1}^{(k)} - t_k}{u_{1}^{(k)} - 0}\bigg)^{1 - \delta_k} \cdot \bigg(\frac{t_k - 0}{u_{1}^{(k)} - 0}\bigg)^{\delta_k}\bigg] \, d\nu_{\alpha}(\talpha)\Bigg| \\
     &\qquad \quad= \sum_{\delta \in \prod\limits_{k \in S(\alpha)} \{0, 1\} \setminus \{\zerovec\}} \Bigg|\int_{R^{(\alpha)}_\zerovec \setminus \{\zerovec\}} \prod_{k \in S(\alpha)} \bigg[\bigg(\frac{u_{1}^{(k)} - t_k}{u_{1}^{(k)} - 0}\bigg)^{1 - \delta_k} \cdot \bigg(\frac{t_k - 0}{u_{1}^{(k)} - 0}\bigg)^{\delta_k}\bigg] \,  d\nu_{\alpha}(\talpha)\Bigg| \\
     &\qquad \quad\le |\nu_{\alpha}|\big(R^{(\alpha)}_{\zerovec} \setminus \{\zerovec\}\big).
\end{align*}
\endgroup

Now, let $\tilde{\mu}_{\alpha}$ be the signed measure on $[0, 1]^{|\alpha|}$ defined by 
\begin{equation*}
    \tilde{\mu}_{\alpha} = \sum_{l \in \prod\limits_{k\in S(\alpha)} [0: (n_k - 1)]} \tilde{\mu}_{\alpha, l}
\end{equation*}
and let $\mu_{\alpha}$ be the signed measure on $[0, 1)^{|\alpha|}$ defined by 
\begin{equation*}
    \mu_{\alpha}(E) = \tilde{\mu}_{\alpha}(E).
\end{equation*}
Note that $\mu_{\alpha}$ is a discrete signed measure concentrated on $(\prod_{k \in S(\alpha)} \mathcal{U}_k) \cap [0, 1)^{|\alpha|}$.
Then, for each $(u_{i_1}^{(1)}, \dots, u_{i_d}^{(d)}) \in \prod_{k = 1}^{d}\mathcal{U}_k$, 
\begingroup
\allowdisplaybreaks
\begin{align*}
    &\fazeronu \big(u_{i_1}^{(1)}, \dots, u_{i_d}^{(d)}\big) = a_{\zerovec} + \sumoveralpha \int_{[0, 1)^{|\alpha|}} \prod_{k \in S(\alpha)} \big(u_{i_k}^{(k)} - t_k\big)_{+} \, d\nu_{\alpha}(\talpha) \\
    &\qquad = a_{\zerovec} + \sumoveralpha \sum_{l \in \prod\limits_{k\in S(\alpha)} [0: (n_k - 1)]} \int_{R^{(\alpha)}_l} \prod_{k \in S(\alpha)} \big(u_{i_k}^{(k)} - t_k\big)_{+} \, d\nu_{\alpha}(\talpha) \\
    &\qquad = a_{\zerovec} + \sumoveralpha \sum_{l \in \prod\limits_{k\in S(\alpha)} [0: (i_k - 1)]} \int_{R^{(\alpha)}_l} \prod_{k \in S(\alpha)} \big(u_{i_k}^{(k)} - t_k\big)_{+} \, d\nu_{\alpha}(\talpha) \\
    &\qquad = a_{\zerovec} + \sumoveralpha \int_{[0, 1]^{|\alpha|}} \prod_{k \in S(\alpha)} \big(u_{i_k}^{(k)} - t_k\big)_{+} \, d\tilde{\mu}_{\alpha}(\talpha) \\
    &\qquad = a_{\zerovec} + \sumoveralpha \int_{[0, 1)^{|\alpha|}} \prod_{k \in S(\alpha)} \big(u_{i_k}^{(k)} - t_k\big)_{+} \, d\mu_{\alpha}(\talpha) 
    = \fazeromu \big(u_{i_1}^{(1)}, \dots, u_{i_d}^{(d)}\big),
\end{align*}
\endgroup
which means that $\fazeromu$ agrees with $\fazeronu$ at every point of $\prod_{k = 1}^{d}\mathcal{U}_k$.
Also, for every $\alpha \in \{0, 1\}^d \setminus \{\zerovec\}$ with $|\alpha| \le s$, 
\begingroup
\allowdisplaybreaks
\begin{align*}
    &|\mu_{\alpha}|\big([0,1)^{|\alpha|} \setminus \{\zerovec\}\big) \le |\tilde{\mu}_{\alpha}|\big([0,1]^{|\alpha|} \setminus \{\zerovec\}\big) \le \sum_{l \in \prod\limits_{k\in S(\alpha)} [0 : (n_k - 1)]} |\tilde{\mu}_{\alpha, l}| \big([0,1]^{|\alpha|} \setminus \{\zerovec\}\big) \\
    &\qquad= |\tilde{\mu}_{\alpha, \zerovec}|\big([0,1]^{|\alpha|} \setminus \{\zerovec\}\big) + \sum_{l \in \prod\limits_{k\in S(\alpha)} [0: (n_k - 1)] \setminus \{\zerovec\}} |\tilde{\mu}_{\alpha, l}|\big([0,1]^{|\alpha|}\big) \\
    &\qquad\le |\nu_{\alpha}|\big(R^{(\alpha)}_{\zerovec} \setminus \{\zerovec\}\big) + \sum_{l \in \prod\limits_{k\in S(\alpha)} [0: (n_k - 1)] \setminus \{\zerovec\}} |\nu_{\alpha}|\big(R^{(\alpha)}_l\big) = |\nu_{\alpha}|\big([0,1)^{|\alpha|} \setminus \{\zerovec\}\big).
\end{align*}
\endgroup
From this result, we can derive that  
\begin{equation*}
    \Vmars(\fazeromu) \le \Vmars(\fazeronu).
\end{equation*}
\end{proof}

\subsubsection{Proof of Proposition \ref{prop:reduction-to-lasso}}\label{pf:reduction-to-lasso}
\begin{proof}[Proof of Proposition \ref{prop:reduction-to-lasso}]
Assume $a_{\zerovec}$ is a real number and $\{\nu_{\alpha}\}$ is a collection of finite signed measures, where $\nu_{\alpha}$ is defined on $[0, 1)^{|\alpha|}$ for each $\alpha \in \{0, 1\}^d \setminus \{\zerovec\}$ with $|\alpha| \le s$. 
Recall that under this assumption,  
\begingroup
\allowdisplaybreaks
\begin{align*}
    &\fazeronu(x_1, \dots, x_d) \\
    &\quad \quad = a_{\zerovec} + \sumoveralpha \sum_{l \in \prod\limits_{k \in S(\alpha)}[0: (n_k - 1)]} \nu_{\alpha}\big(\big\{\big(u^{(k)}_{l_k}, k \in S(\alpha)\big)\big\}\big) \cdot \prod_{k \in S(\alpha)} \big(x_k - u^{(k)}_{l_k}\big)_{+} 
\end{align*}
\endgroup
and 
\begin{equation*}
    \Vmars(\fazeronu) = \sumoveralpha \sum_{l \in \prod\limits_{k \in S(\alpha)}[0: (n_k - 1)] \setminus \{\zerovec\}} \Big|\nu_{\alpha}\big(\big\{\big(u^{(k)}_{l_k}, k \in S(\alpha)\big)\big\}\big)\Big|.
\end{equation*}
Now, let $\gamma$ be the $|J|$-dimensional vector indexed by $(\alpha, l) \in J$ for which 
\begin{equation*}
     \gamma_{\alpha, l} = \nu_{\alpha}\big(\big\{\big(u^{(k)}_{l_k}, k \in S(\alpha)\big)\big\}\big)
\end{equation*}
for $(\alpha, l) \in J$.
Then, it follows that 
\begin{equation*}
    \big(\fazeronu (x^{(i)}), i \in [n]\big) = a_0 + M \gamma
\end{equation*}
and 
\begin{equation*}
    \Vmars(\fazeronu) = \sum_{\substack{(\alpha, l) \in J \\ l \neq \zerovec}} |\gamma_{\alpha, l}|.
\end{equation*}
Thus, under the assumption, our estimation problem reduces to the finite-dimensional lasso problem \eqref{finite-dimensional-lasso}.
Recall that by Lemma \ref{lem:reduction-to-discrete-measures}, a solution to \eqref{our-problem-restated} obtained under the assumption is still a solution to the same problem without the assumption.
Therefore, for a solution $(\hat{a}_{\zerovec}, \hat{\gamma}^{d, s}_{n, V})$ to  the problem \eqref{finite-dimensional-lasso}, the function $f$ on $[0, 1]^d$ of the form
\begin{equation*}
    f(x_1, \dots, x_d) = \hat{a}_{\zerovec} + \sum_{(\alpha, l) \in J} \big(\hat{\gamma}^{d, s}_{n, V}\big)_{\alpha, l} \cdot \prod_{k \in S(\alpha)} \big(x_k - u^{(k)}_{l_k}\big)_{+}
\end{equation*}
is indeed a solution to the original problem \eqref{our-problem-restated}.
Moreover, since the set 
\begin{equation*}
    \Bigg\{a_0 \onevec + M \gamma: a_0 \in \R, \gamma \in \R^{|J|}, \mbox{ and } \sum_{\substack{(\alpha, l) \in J \\ l \neq \zerovec}} |\gamma_{\alpha, l}| \le V \Bigg\}
\end{equation*}
is closed and convex, even though there can exist multiple solutions to the problem \eqref{finite-dimensional-lasso}, they all should share the same value of $a_0 \onevec + M \gamma$, which is indeed the projection of $y$ onto the closed convex set.
For these reasons, Lemma \ref{lem:reduction-to-discrete-measures} also implies that for every solution $\hat{f}^{d, s}_{n, V}$ to the problem \eqref{our-problem-restated}, the values of $\hat{f}^{d, s}_{n, V}$ at the design points should be
\begin{equation*}
    \hat{f}^{d, s}_{n, V}(x^{(i)}) = \hat{a}_{\zerovec} + \big(M \hat{\gamma}^{d, s}_{n, V}\big)_i = \hat{a}_{\zerovec} + \sum_{(\alpha, l) \in J} \big(\hat{\gamma}^{d, s}_{n, V}\big)_{\alpha, l} \cdot \prod_{k \in S(\alpha)} \big(x^{(i)}_k - u^{(k)}_{l_k}\big)_{+} 
\end{equation*}
for each $i \in [n]$.
\end{proof}

\subsection{Proofs of Theorems in Section \ref{riskresults}}\label{subsec:proof-of-risk-results}
Throughout the proofs of theorems and the proofs of lemmas therein, $c$ and $C$ indicate universal constants, which can be different line by line. 
Just for convenience, we often use those constants without specifying their values. 
When constants depend on some variables, we specify those dependencies by subscripts.
\subsubsection{Proof of Theorem \ref{thm:risk-upper-bound-fixed}}\label{pf:risk-upper-bound}

We exploit the following general result for nonparametric least squares estimators.
This result is a straightforward generalization of \cite{VandegeerBook}, Theorem 9.1 to the misspecified case, in which the true underlying function may not belong to the function class where estimators are searched over. 
Although we only need to consider the well-specified case for Theorem \ref{thm:risk-upper-bound-fixed}, here we state a generalized version for future purpose. 
We will use this theorem again for our proof of Theorem \ref{thm:risk-upper-bound-fixed-approx}.

\begin{theorem}\label{thm:least-squares-estimator}
    Suppose that data $(x^{(1)}, y_1), \dots, (x^{(n)}, y_n)$ are generated according to the model 
    \begin{equation}
        y_i = f^{*}(x^{(i)}) + \xi_i 
    \end{equation}
    where $\xi_i$ are independent sub-Gaussian errors with mean zero and with a sub-Gaussian parameter $\sigma$, and
    $f^{*}$ is an unknown real-valued regression function defined on $[0, 1]^d$.
    Let $\F$ be a collection of real-valued functions defined on $[0, 1]^d$ and let $f_0$ be an element of $\F$.
    Also, let $\hat{f}$ be a least squares estimator of $f^{*}$ over $\F$ defined by 
    \begin{equation*}
        \hat{f} \in \argmin_{f \in \F} \bigg\{\sum_{i = 1}^{n} \big(y_i - f(x^{(i)})\big)^2 \bigg\}. 
    \end{equation*}
    Assume that there exists a real-valued function $\Psi$ defined on $\R^+$ such that $t \mapsto \Psi(t)/t^2$ is monotonically decreasing on $\R^+$ and 
    \begin{equation}\label{risk-bound-condition}
        \Psi(t) \ge \max\bigg\{t, \int_{0}^{t}\sqrt{\log N\big(\epsilon, B_{\F}(t, f_0, \|\cdot\|_n), \|\cdot\|_n \big)} d\epsilon \bigg\},
    \end{equation}
    where $B_{\F}(t, f_0, \|\cdot\|_n) := \{f \in \F: \|f - f_0\|_n \le t\}$ and $N(\epsilon, \G, \|\cdot\|_n)$ is the $\epsilon$-covering number of a function class $\G$ under the norm $\|\cdot\|_n$.
    Then, there exist universal positive constants $c$ and $C$ such that
    \begin{equation}\label{probability-bound-fixed-design}
        \P (\|\hat{f} - f^* \|_n^2 - \| f_0 - f^* \|_n^2 > t^2) \le C \exp\Big(-\frac{n t^2}{C \sigma^2}\Big) \qt{for all $t \ge t_n$,}
    \end{equation}
    for every $t_n > 0$ satisfying 
    \begin{equation*}
        \sqrt{n} t_n^2 \ge c \sigma \Psi(t_n) 
        \ \mbox{ and } \
        t_n \ge \| f_0 - f^* \|_n.
    \end{equation*}
\end{theorem}

\begin{remark}\label{rmk:least-squares-estimator-fixed-design}
The probability bound \eqref{probability-bound-fixed-design} in Theorem \ref{thm:least-squares-estimator} directly leads to an upper bound of the risk $\mathcal{R}_F(\hat{f}, f^{*})$. 
First, we can rewrite the bound \eqref{probability-bound-fixed-design} as 
\begin{equation*}
    \P (\|\hat{f} - f^* \|_n^2 - \| f_0 - f^* \|_n^2 > t) \le C \exp\Big(-\frac{n t}{C \sigma^2}\Big) \qt{for all $t \ge t_n^2$.}
\end{equation*}
If we integrate both sides of the above inequality from $t_n^2$ to infinity, we obtain that
\begin{equation*}
    \E[(\|\hat{f} - f^* \|_n^2 - \| f_0 - f^* \|_n^2 - t_n^2)_+] \le C \int_{t_n^2}^{\infty} \exp\Big(-\frac{n t}{C \sigma^2}\Big) \, dt \le \frac{C \sigma^2}{n}.
\end{equation*}
Therefore, using the inequality $x \le (x - y)_+ + y$, we can derive that
\begin{equation}\label{risk-upper-bound-fixed-misspecified}
    \mathcal{R}_F(\hat{f}, f^{*}) = \E \|\hat{f} - f^* \|_n^2 \le \| f_0 - f^* \|_n^2 + \frac{C \sigma^2}{n} + t_n^2. 
\end{equation}
\end{remark}

\begin{remark}[Well-specified Case]\label{rmk:least-squares-estimator-well-specified}
    If $f^*$ belongs to the function class $\F$, then we can choose $f_0$ as $f^{*}$. 
    In this case, the condition $t_n \ge \| f_0 - f^* \|_n$ is vacuous, and \eqref{risk-upper-bound-fixed-misspecified} becomes 
    \begin{equation*}
    \mathcal{R}_F(\hat{f}, f^{*}) \le \frac{C \sigma^2}{n} + t_n^2.
    \end{equation*}
    We will use this result for our proof of Theorem \ref{thm:risk-upper-bound-fixed}.
\end{remark}

\begin{remark}
    Here we briefly discuss how \cite{VandegeerBook}, Theorem 9.1 can be generalized to the misspecified case. 
    We describe what needs to be modified in the proof of \cite{VandegeerBook}, Theorem 9.1 to achieve the generalized result in Theorem \ref{thm:least-squares-estimator}.

    Fix $t \ge t_n$ and let 
    \begin{equation*}
        \F_j = \{f \in \F: 2^{2j - 2} t^2 < \| f - f^* \|_{n}^2 - \| f_0 - f^* \|_{n}^2 \le 2^{2j} t^2\}
    \end{equation*} 
    for each positive integer $j$.
    It then follows that
    \begin{equation}\label{eq:decomposition-into-fj-fixed}
        \P (\|\hat{f} - f^* \|_n^2 - \| f_0 - f^* \|_n^2 > t^2) = \sum_{j = 1}^{\infty} \P(\hat{f} \in \F_j),
    \end{equation}
    and it suffices to bound $\P(\hat{f} \in \F_j)$ for each positive integer $j$.
    Observe that for every $f \in \F_j$, 
    \begin{align*}
        2^{2j} t^2 &\ge \| f - f^* \|_{n}^2 - \| f_0 - f^* \|_{n}^2 = \| f - f_0 \|_{n}^2 + \frac{2}{n} \sum_{i = 1}^{n} (f - f_0)(x^{(i)}) \cdot (f_0 - f^*)(x^{(i)}) \\
        &\ge \| f - f_0 \|_{n}^2 - 2 \| f - f_0 \|_{n} \cdot \| f_0 - f^* \|_{n},  
    \end{align*}
    and thus, 
    \begin{equation*}
        \| f - f_0 \|_{n} \le \| f_0 - f^* \|_{n} + \sqrt{\| f_0 - f^* \|_{n}^2 + 2^{2j} t^2} \le 2 \| f_0 - f^* \|_{n} + 2^j t \le 2^{j + 1} t.
    \end{equation*}
    Also, since $\hat{f}$ is a least squares estimator over $\F$ and $f_0$ is an element of $\F$, 
    \begin{equation*}
        \sum_{i = 1}^{n} \big(y_i - \hat{f}(x^{(i)})\big)^2 \le \sum_{i = 1}^{n} \big(y_i - f_0(x^{(i)})\big)^2,
    \end{equation*}
    from which one can easily derive that
    \begin{equation*}
        \|\hat{f} - f^* \|_n^2 - \| f_0 - f^*  \|_n^2 \le \frac{2}{n} \sum_{i = 1}^{n} \xi_i (\hat{f} - f_0)(x^{(i)}).
    \end{equation*}
    For these reasons, for each positive integer $j$, we have 
    \begin{align*}
        \P(\hat{f} \in \F_j) &\le \P\bigg(\sup_{f \in \F_j} \frac{1}{n} \sum_{i = 1}^{n} \xi_i (f - f_0)(x^{(i)}) \ge 2^{2j - 3} t^2 \bigg) \\
        &\le \P\bigg(\sup_{\substack{f \in \F \\ \|f - f_0\|_{n} \le 2^{j + 1}t }} \frac{1}{n} \sum_{i = 1}^{n} \xi_i (f - f_0)(x^{(i)}) \ge 2^{2j - 3} t^2 \bigg).
    \end{align*}
    From this point, we can simply follow the proof of \cite{VandegeerBook}, Theorem 9.1. 
    By repeating their argument, we can show that 
    \begin{equation*}
        \P(\hat{f} \in \F_j) \le 4 \exp\Big(-\frac{2^j n t^2}{C \sigma^2}\Big)
    \end{equation*}
    for each positive integer $j$. 
    Combining this result with \eqref{eq:decomposition-into-fj-fixed}, we can conclude that
    \begin{equation*}
        \P (\|\hat{f} - f^* \|_n^2 - \| f_0 - f^* \|_n^2 > t^2) \le \sum_{j = 1}^{\infty} 4 \exp\Big(-\frac{2^j n t^2}{C \sigma^2}\Big) \le C \exp\Big(-\frac{n t^2}{C \sigma^2}\Big).
    \end{equation*}
\end{remark}

\begin{proof} [Proof of Theorem \ref{thm:risk-upper-bound-fixed}]
Let $\F_{\text{disc}}(V)$ be the collection of all the functions $\fazeronu \in \infmars^{d, s}$ with $\Vmars(\fazeronu) \le V$ where $\nu_\alpha$ is concentrated on $(\prod_{k \in S(\alpha)} \mathcal{U}_k) \cap [0, 1)^{|\alpha|}$ for each $\alpha \in \{0, 1\}^d \setminus \{\zerovec\}$ with $|\alpha| \le s$.
By Lemma \ref{lem:reduction-to-discrete-measures}, there exist $\hat{f}_{\text{disc}}, f^*_{\text{disc}} \in \F_{\text{disc}}(V)$ such that
\begin{equation*}
    \hat{f}_{\text{disc}} (x^{(i)}) = \hat{f}^{d, s}_{n, V} (x^{(i)}) \mbox{ and } f^*_{\text{disc}} (x^{(i)}) = f^* (x^{(i)})
\end{equation*}
for all $i \in [n]$.
Observe that 
\begin{itemize}
    \item $y_i = f^{*}(x^{(i)}) + \xi_i = f^{*}_{\text{disc}} (x^{(i)}) + \xi_i$ for $i \in [n]$, and \\
    \item $\mathcal{R}_F(\hat{f}^{d, s}_{n, V}, f^{*}) = \mathcal{R}_F(\hat{f}_{\text{disc}}, f^{*}_{\text{disc}})$.
\end{itemize}
Also, we have 
\begin{equation*}
    \hat{f}_{\text{disc}} \in \argmin_{f \in \F_{\text{disc}}(V)} \bigg\{\sum_{i = 1}^{n} \big(y_i - f(x^{(i)})\big)^2 \bigg\}. 
\end{equation*}
Thus, once we bound the metric entropy of $B_{\F_{\text{disc}}(V)}(t, f^*_{\text{disc}}, \|\cdot\|_n)$ (under the norm $\|\cdot\|_n$) and its integral as in \eqref{risk-bound-condition}, we can use Theorem \ref{thm:least-squares-estimator} (and Remark \ref{rmk:least-squares-estimator-well-specified}) to obtain an upper bound of the risk $\mathcal{R}_F(\hat{f}^{d, s}_{n, V}, f^{*})$.

For $V > 0$ and $t > 0$, let 
\begin{equation*}
    S(V, t) = \big\{f \in \F_{\text{disc}}(V): \|f\|_n \le t \big\}.
\end{equation*}
Then, clearly, 
\begin{equation*}
    B_{\F_{\text{disc}}(V)}(t, f^*_{\text{disc}}, \|\cdot\|_n) - \{f^*_{\text{disc}}\} \subseteq S(2V, t)
\end{equation*}
because 
\begin{equation*}
    \Vmars(f - f^*_{\text{disc}}) \le \Vmars(f) + \Vmars(f^*_{\text{disc}}) \le 2V
\end{equation*}
for every $f \in \F_{\text{disc}}(V)$.
Here for two function classes $\G_1$ and $\G_2$, $\G_1 - \G_2$ indicates the class $\{g_1 - g_2: g_1 \in \G_1, g_2 \in \G_2\}$.
Also, the following lemma, which will be proved in Appendix \ref{pf:svt-metric-entropy}, provides an upper bound of the metric entropy of $S(V, t)$ (under the norm $\|\cdot\|_n$).

\begin{lemma}\label{lem:svt-metric-entropy}
    There exists positive constant $C_{\rho, s}$ depending on $\rho$ and $s$ and $C_{\rho, d}$ depending on $\rho$ and $d$ such that
    \begingroup
    \allowdisplaybreaks
    \begin{align*}
        \log N(\epsilon, S(V, t), \|\cdot\|_n) &\le  2^d \log\Big(2 + C_{\rho, s} \cdot \frac{V + \sqrt{n} t}{\epsilon}\Big) \\
        &\quad + C_{\rho, d} \Big(2^{d + 1} + 1 + \frac{2^{d + 1} V }{\epsilon}\Big)^{\frac{1}{2}} \bigg[\log\Big(2^{d + 1} + 1 + \frac{2^{d + 1} V}{\epsilon}\Big)\bigg]^{\frac{3(2s - 1)}{4}}
    \end{align*}
    \endgroup
    for every $V > 0$, $t > 0$, and $\epsilon > 0$.
\end{lemma}

By Lemma \ref{lem:svt-metric-entropy} and the inequality $(x + y)^{1/2} \le x^{1/2} + y^{1/2}$, we can obtain the following upper bound of the square root of the metric entropy of $B_{\F_{\text{disc}}(V)}(t, f^*_{\text{disc}}, \|\cdot\|_n)$:
\begin{align}\label{sqrt-svt-upper-bound}
\begin{split}
    &\sqrt{\log N\big(\epsilon, B_{\F_{\text{disc}}(V)}(t, f^*_{\text{disc}}, \|\cdot\|_n), \|\cdot\|_n \big)} \le 2^{\frac{d}{2}} \cdot \sqrt{\log\Big(2 + C_{\rho, s} \cdot \frac{2V + \sqrt{n} t}{\epsilon}\Big)} \\
    &\qquad \qquad \qquad \qquad \qquad + C_{\rho, d} \Big(2^{d + 1} + 1 + \frac{2^{d + 2} V }{\epsilon}\Big)^{\frac{1}{4}} \bigg[\log\Big(2^{d + 1} + 1 + \frac{2^{d + 2} V}{\epsilon}\Big)\bigg]^{\frac{3(2s - 1)}{8}}.
\end{split}
\end{align}
To integrate the right-hand side of \eqref{sqrt-svt-upper-bound}, we construct the lemma below. 
We omit the proof of \eqref{eq:integration-helper1}, which is basically the result of simple application of integration by parts, and our proof of \eqref{eq:integration-helper2} is presented in Appendix \ref{pf:integration-helper2}.

\begin{lemma}\label{lem:integration-helper}
    For every $u > t$, 
    \begin{equation}\label{eq:integration-helper1}
        \int_{0}^{t} \sqrt{\log \frac{u}{\epsilon}} \, d\epsilon = \frac{t}{2\sqrt{\tau}} (1 + 2\tau),
    \end{equation}
    where $\tau = \log (u/t)$.
    Also, for every $u > t$ and $k > 0$,
    \begin{equation}\label{eq:integration-helper2}
        \int_{0}^{t} \Big(\frac{u}{\epsilon}\Big)^{\frac{1}{4}} \Big[\log \frac{u}{\epsilon}\Big]^k \, d\epsilon \le C_k u^{\frac{1}{4}} t^{\frac{3}{4}} (1 + \tau^k),
    \end{equation}
    where $C_k$ is a constant depending on $k$.
\end{lemma}

By \eqref{eq:integration-helper1}, 
\begingroup
\allowdisplaybreaks
\begin{align*}
    &\int_{0}^{t} \sqrt{\log\Big(2 + C_{\rho, s} \cdot \frac{2V + \sqrt{n} t}{\epsilon}\Big)} \, d\epsilon \le \int_{0}^{t} \sqrt{\log\Big(\frac{2 t + C_{\rho, s} (2V + \sqrt{n} t)}{\epsilon}\Big)} \, d\epsilon \\
    &\qquad \le C t\bigg[1 + 2\log\Big(2 + C_{\rho, s} \Big(\sqrt{n} + \frac{2V}{t}\Big)\Big)\bigg] \le C_{\rho, s} t \log n + C t\log\Big(1 + \frac{2V}{\sqrt{n}t}\Big). 
\end{align*}
\endgroup
Also, by \eqref{eq:integration-helper2} and the inequality $(x + y)^{1/4} \le x^{1/4} + y^{1/4}$,
\begingroup
\allowdisplaybreaks
\begin{align*}
    \int_{0}^{t} &\Big(2^{d + 1} + 1 + \frac{2^{d + 2} V }{\epsilon} \Big)^{\frac{1}{4}} \bigg[\log\Big(2^{d + 1} + 1 + \frac{2^{d + 2} V }{\epsilon}\Big)\bigg]^{\frac{3(2s - 1)}{8}} \, d\epsilon \\
    & \le \int_{0}^{t} \Big(\frac{(2^{d + 1} + 1)t + 2^{d + 2} V}{\epsilon} \Big)^{\frac{1}{4}} \bigg[\log\Big(\frac{(2^{d + 1} + 1)t + 2^{d + 2} V}{\epsilon}\Big)\bigg]^{\frac{3(2s - 1)}{8}} \, d\epsilon \\
    & \le C_d \big[(2^{d + 1} + 1) t + 2^{d + 2} V\big]^{\frac{1}{4}} t^{\frac{3}{4}} \bigg[1 + \Big\{\log\Big(2^{d + 1} + 1 + \frac{2^{d + 2} V }{t}\Big)\Big\}^{\frac{3(2s - 1)}{8}} \bigg] \\ 
    & \le C_d \big[\{(2^{d + 1} + 1) t\}^{\frac{1}{4}} + (2^{d + 2} V)^{\frac{1}{4}}\big] t^{\frac{3}{4}} \bigg[\log\Big(2 +  \frac{ V }{t}\Big)\bigg]^{\frac{3(2s - 1)}{8}} \\
    & \le C_d t \bigg[\log\Big(2 +  \frac{V}{t}\Big)\bigg]^{\frac{3(2s - 1)}{8}} + C_d V^{\frac{1}{4}} t^{\frac{3}{4}}  \bigg[\log\Big(2 +  \frac{V}{t}\Big)\bigg]^{\frac{3(2s - 1)}{8}}.
\end{align*}
\endgroup
Combining these results, we can derive that
\begin{align*}
    &\int_{0}^{t} \sqrt{\log N\big(\epsilon, B_{\F_{\text{disc}}(V)}(t, f^*_{\text{disc}}, \|\cdot\|_n), \|\cdot\|_n\big)} \, d\epsilon \le C_{\rho, d} t \log n + C_{d} t \log\Big(1 + \frac{2V}{\sqrt{n}t}\Big) \\
    &\qquad \qquad \qquad \qquad \qquad + C_{\rho, d} t \bigg[\log\Big(2 +  \frac{V}{t}\Big)\bigg]^{\frac{3(2s - 1)}{8}} + C_{\rho, d} V^{\frac{1}{4}} t^{\frac{3}{4}} \bigg[\log\Big(2 +  \frac{V}{t}\Big)\bigg]^{\frac{3(2s - 1)}{8}}. 
\end{align*}

Now, consider the function $\Psi$ defined as the right-hand side of the above inequality, i.e., 
\begin{align*}
    &\Psi(t) = C_{\rho, d} t \log n + C_{d} t \log\Big(1 + \frac{2V}{\sqrt{n}t}\Big) \\
    &\qquad \quad+ C_{\rho, d} t \bigg[\log\Big(2 +  \frac{V}{t}\Big)\bigg]^{\frac{3(2s - 1)}{8}} + C_{\rho, d} V^{\frac{1}{4}} t^{\frac{3}{4}} \bigg[\log\Big(2 +  \frac{V}{t}\Big)\bigg]^{\frac{3(2s - 1)}{8}}
\end{align*}
for $t > 0$.
Then, $t \mapsto \Psi(t)/t^2$ is monotonically decreasing and 
\begin{equation*}
    \Psi(t) \ge \max\bigg\{t, \int_{0}^{t}\sqrt{\log N\big(\epsilon, B_{\F_{\text{disc}}(V)}(t, f^*_{\text{disc}}, \|\cdot\|_n), \|\cdot\|_n \big)} d\epsilon \bigg\},
\end{equation*}
by definition.
Also, note that
\begingroup
\allowdisplaybreaks
\begin{align*}
    C_{\rho, d} t \log n \le \frac{\sqrt{n} t^2}{4c \sigma} \mbox{\quad if \;}& t \ge C_{\rho, d} \frac{\sigma \log n}{\sqrt{n}}, \\
    C_{d} t \log\Big(1 + \frac{2V}{\sqrt{n}t}\Big) \le \frac{\sqrt{n} t^2}{4c \sigma} \mbox{\quad if \;}& t \ge \max\Big\{\frac{\sigma}{\sqrt{n}}, C_{d} \frac{\sigma}{\sqrt{n}} \log\Big(1 + \frac{2V}{\sigma }\Big)\Big\}, \\
    C_{\rho, d} t \bigg[\log\Big(2 + \frac{V}{t}\Big)\bigg]^{\frac{3(2s - 1)}{8}} \le \frac{\sqrt{n} t^2}{4c \sigma} \mbox{\quad if \;}& t \ge \max\bigg\{\frac{\sigma}{\sqrt{n}}, C_{\rho, d} \frac{\sigma}{\sqrt{n}} \bigg[\log\Big(2 + \frac{V n^{\frac{1}{2}} }{\sigma}\Big)\bigg]^{\frac{3(2s - 1)}{8}}\bigg\},
\end{align*}
\endgroup
and 
\begin{align*}
    &C_{\rho, d} V^{\frac{1}{4}} t^{\frac{3}{4}} \bigg[\log\Big(2 + \frac{V}{t}\Big)\bigg]^{\frac{3(2s - 1)}{8}} \le \frac{\sqrt{n} t^2}{4c \sigma} \\
    &\qquad \qquad \qquad \mbox{\quad if \;} t \ge \max \bigg\{\frac{\sigma}{\sqrt{n}}, C_{\rho, d} \frac{\sigma^{\frac{4}{5}} V^{\frac{1}{5}}}{n^{\frac{2}{5}}} \bigg[\log\Big(2 + \frac{V n^{\frac{1}{2}}}{\sigma}\Big)\bigg]^{\frac{3(2s - 1)}{10}}\bigg\}.
\end{align*}
Therefore, if we let 
\begin{align*}
    &t_n = \max\bigg\{\frac{\sigma}{\sqrt{n}}, C_{\rho, d} \frac{\sigma \log n}{\sqrt{n}}, C_{d} \frac{\sigma}{\sqrt{n}} \log\Big(1 + \frac{2V}{\sigma }\Big), \\
    &\qquad \qquad \qquad C_{\rho, d}  \frac{\sigma}{\sqrt{n}} \bigg[\log\Big(2 + \frac{V n^{\frac{1}{2}} }{\sigma}\Big)\bigg]^{\frac{3(2s - 1)}{8}}, C_{\rho, d} \frac{\sigma^{\frac{4}{5}} V^{\frac{1}{5}}}{n^{\frac{2}{5}}} \bigg[\log\Big(2 + \frac{V n^{\frac{1}{2}}}{\sigma}\Big)\bigg]^{\frac{3(2s - 1)}{10}}\bigg\},
\end{align*}
then the inequality 
\begin{equation*}
        \sqrt{n} t_n^2 \ge c \sigma \Psi(t_n)
\end{equation*}
holds. 
By Theorem \ref{thm:least-squares-estimator} (and Remark \ref{rmk:least-squares-estimator-well-specified}), it follows that 
\begin{align}\label{eq:raw_risk_upper_bound}
\begin{split}
    &\mathcal{R}_F\big(\hat{f}^{d, s}_{n, V}, f^{*}\big) \le \frac{C \sigma^2}{n} + t_n^2 \\ 
    &\qquad \le C_{\rho, d}  \Big(\frac{\sigma^2 V^{\frac{1}{2}}}{n}\Big)^{\frac{4}{5}} \bigg[\log\Big(2 + \frac{V n^{\frac{1}{2}}}{\sigma}\Big)\bigg]^{\frac{3(2s - 1)}{5}} + C_{\rho, d} \frac{\sigma^2}{n} \bigg[\log\Big(2 + \frac{V n^{\frac{1}{2}}}{\sigma}\Big)\bigg]^{\frac{3(2s - 1)}{4}} \\
    &\qquad \qquad + C_{d} \frac{\sigma^2}{n} \Big[\log\Big(1 + \frac{2V }{\sigma}\Big)\Big]^2 + C_{\rho, d} \frac{\sigma^2}{n} [\log n]^2.
\end{split}
\end{align}
We can further observe that the second term and the third term on the right-hand side of \eqref{eq:raw_risk_upper_bound} can be absorbed into the other terms.
Note that there exist some positive constants $C_s$ and $C$ such that 
\begin{equation*}
    [\log(2 + x)]^{\frac{3(2s - 1)}{4}} \le x^{\frac{2}{5}} \qt{for $x \ge C_s$}
\end{equation*}
and
\begin{equation*}
    [\log(1 + 2x)]^2 \le x^{\frac{2}{5}} \qt{for $x \ge C$}.
\end{equation*}
Hence, the second term can be bounded by
\begin{equation*}
    C_{\rho, d} \frac{\sigma^2}{n} \bigg[\log\Big(2 + \frac{V n^{\frac{1}{2}}}{\sigma}\Big)\bigg]^{\frac{3(2s - 1)}{4}} \le C_{\rho, d} \frac{\sigma^2}{n} + C_{\rho, d} \frac{\sigma^2}{n} \Big(\frac{V n^{\frac{1}{2}}}{\sigma}\Big)^{\frac{2}{5}} = C_{\rho, d} \frac{\sigma^2}{n} + C_{\rho, d} \Big(\frac{\sigma^2 V^{\frac{1}{2}}}{n}\Big)^{\frac{4}{5}},
\end{equation*}
and the third term can be bounded by
\begin{align*}
    &C_{d} \frac{\sigma^2}{n} \Big[\log\Big(1 + \frac{2V }{\sigma}\Big)\Big]^2 \le C_{d} \frac{\sigma^2}{n} \Big[\log\Big(1 + \frac{2Vn^{\frac{1}{2}} }{\sigma}\Big)\Big]^2 \\
    &\qquad \qquad \ \le C_{d} \frac{\sigma^2}{n} + C_{d} \frac{\sigma^2}{n} \Big(\frac{V n^{\frac{1}{2}}}{\sigma}\Big)^{\frac{2}{5}} = C_{d} \frac{\sigma^2}{n} + C_{d} \Big(\frac{\sigma^2 V^{\frac{1}{2}}}{n}\Big)^{\frac{4}{5}}.
\end{align*}
Consequently, removing both terms from \eqref{eq:raw_risk_upper_bound}, we can derive that
\begin{equation*}
    \mathcal{R}_F\big(\hat{f}^{d, s}_{n, V}, f^{*}\big) \le C_{\rho, d}  \Big(\frac{\sigma^2 V^{\frac{1}{2}}}{n}\Big)^{\frac{4}{5}} \bigg[\log\Big(2 + \frac{V n^{\frac{1}{2}}}{\sigma}\Big)\bigg]^{\frac{3(2s - 1)}{5}} + C_{\rho, d} \frac{\sigma^2}{n} [\log n]^2.
\end{equation*}
\end{proof}

\subsubsection{Proof of Theorem \ref{thm:d-metric-entropy-main-text}} \label{pf:d-metric-entropy}
We first briefly recall what Brownian sheet is and introduce integrated Brownian sheet. 
An $m$-dimensional Brownian sheet $B_m = (B_m(t, \cdot): t \in [0, 1]^m)$ is a centered Gaussian process on $[0, 1]^m$ whose covariance is given by
\begin{equation*}
    \E\big[B_m(t, \cdot) B_m(s, \cdot)\big] = \prod_{j = 1}^{m} \min(t_j, s_j).
\end{equation*}
Let $\{\psi_{k_1} \otimes \dots \otimes \psi_{k_m}: k_1, \dots, k_m \in \mathbb{N}\}$ be an orthonormal basis of $L^2([0,1]^m)$.
Here $\psi_{k_1} \otimes \dots \otimes \psi_{k_m}$ indicates the
tensor product of the functions $\psi_{k_1}, \dots, \psi_{k_m}$,
which is defined by  
\begin{equation*}
    (\psi_{k_1} \otimes \dots \otimes \psi_{k_m})(x) = \psi_{k_1}(x_1) \cdot \dots \cdot \psi_{k_m}(x_m) 
\end{equation*}
for $x = (x_1, \dots, x_m) \in [0, 1]^m$.
Then, it is well-known that the Brownian sheet $B_m$ can be represented by 
\begin{equation*}
    B_m(\cdot, \omega) = \sum_{k_1, \dots, k_m \in \mathbb{N}} \xi_{k_1, \dots, k_m}(\omega) \cdot \intoper_m(\psi_{k_1} \otimes \dots \otimes \psi_{k_m})(\cdot),
\end{equation*}
where $\xi_{k_1, \dots, k_m}$ are independent standard normal random variables, $\intoper_m: L^2([0, 1]^m) \rightarrow C([0, 1]^m)$ is the $m$-dimensional integral operator defined by 
\begin{equation*}
    \intoper_m f(t) = \int_{[\zerovec, t]} f(x) \, dx \mbox{\quad for } t \in [0, 1]^m, 
\end{equation*}
and the series converges a.s. in $C([0, 1]^m)$ (see, e.g., \cite{giambartolomei2016loeve}, Section 5).
An $m$-dimensional integrated Brownian sheet $X_m = (X_m(t, \cdot): t \in [0, 1]^m)$ is then simply defined as 
\begin{equation*}
    X_m(\cdot, \omega) = \intoper_m \big(B_m(\cdot, \omega)\big). 
\end{equation*}
Since $\intoper_m$ is a continuous linear operator on $C([0, 1]^m)$, it follows that 
\begin{align*} 
    X_m(\cdot, \omega) &= \intoper_m\bigg(\sum_{k_1, \dots, k_m \in \mathbb{N}} \xi_{k_1, \dots, k_m}(\omega) \cdot \intoper_m(\psi_{k_1} \otimes \dots \otimes \psi_{k_m})(\cdot) \bigg) \\
    &= \sum_{k_1, \dots, k_m \in \mathbb{N}} \xi_{k_1, \dots, k_m}(\omega) \cdot \intoper_m^2(\psi_{k_1} \otimes \dots \otimes \psi_{k_m})(\cdot) 
\end{align*}
almost surely.
We will exploit this series expansion of the integrated Brownian sheet $X_m$ in our proof of Theorem \ref{thm:d-metric-entropy-main-text}.

We also utilize the following two results for our proof of Theorem \ref{thm:d-metric-entropy-main-text}. The first result is
\cite{Li1999}, Theorem 1.2, which is restated as Theorem
\ref{thm:li-linde} below. This result implies that upper bounds of
the metric entropy of $\D_m$ can be obtained from appropriate
lower bounds on the small ball probability of the integrated Brownian
sheet $X_m$. The second result is \cite{Chen2003}, Theorem 1.2, which is
restated as Theorem \ref{thm:chen-li} below. This result helps us
connect the small ball probability of the integrated Brownian sheet $X_m$ to
that of the Brownian sheet $B_m$ whose lower bounds are quite well-studied.

Recall that a random element $X$ in a Banach space $E$ is said to be
centered Gaussian if $\inner{A}{X} := A(X)$ is normally distributed
and has mean zero for all $A \in E^*$, where $E^*$ is the dual space of $E$. Also, its 
covariance operator $\Sigma: E^* \rightarrow E$ is defined by  
\begin{equation*}
    \inner{A}{\Sigma B} = \int_{E} \left<A, x \right> \left<B, x \right> \, d\mu(x) = \E\left[\left<A, X \right> \left<B, X \right> \right]
\end{equation*}
for all $A, B \in E^*$, where $\mu$ is the measure induced on $E$ by $X$ (see, e.g., \cite{lifshits1995gaussian}, Section 8). 
Moreover, recall that for the map $\Phi: E^* \rightarrow E$ defined by the Bochner integral 
\begin{equation*}
    \Phi(A) = \int_E \left<A, x \right> x \, d\mu(x),
\end{equation*} 
the reproducing kernel Hilbert space of $X$ in $E$ is defined as the completion of $\Phi(E^*)$ with respect to the inner product 
\begin{equation*}
    \left<\Phi(A), \Phi(B) \right> := \int_{E} \left<A, x \right> \left<B, x \right> \, d\mu(x) = \E\left[\left<A, X \right> \left<B, X \right>\right]
\end{equation*}
(see, e.g., \cite{van2008reproducing}, Section 2).

\begin{theorem}[\cite{Li1999}, Theorem 1.2]\label{thm:li-linde}
    Let $(E, \|\cdot\|)$ be a separable Banach space and $X$ be a centered Gaussian element in $E$.
    Also, let $K_{X}$ be the unit ball of the reproducing kernel Hilbert space of $X$ in $E$. 
    For $\alpha > 0$ and $\beta \in \R$, if there exists some constant $c > 0$ such that
    \begin{equation*}
        \log \P(\|X\| \le \epsilon) \ge -c \epsilon^{-\alpha}\Big[\log\frac{1}{\epsilon}\Big]^{\beta}
    \end{equation*}
    for sufficiently small $\epsilon > 0$, then there exists some constant $C > 0$ such that  
    \begin{equation*}
        \log N(\epsilon, K_{X}, \|\cdot\|) \le C \epsilon^{-\frac{2\alpha}{(2 + \alpha)}} \Big[\log\frac{1}{\epsilon}\Big]^{\frac{2\beta}{2 + \alpha}}
    \end{equation*}
for sufficiently small $\epsilon > 0$.
\end{theorem}

\begin{theorem}[\cite{Chen2003}, Theorem 1.2]\label{thm:chen-li}
    Let $(E, \|\cdot\|)$ be a separable Banach space and $(H, \|\cdot\|_H)$ be a Hilbert space. 
    Also, let $Y$ be a centered Gaussian element in $H$. 
    Then, for every linear operator $L: H \rightarrow E$ and a centered
    Gaussian element $X$ in $E$ with the covariance operator $LL^{*}$, 
    where $L^*: E^* \rightarrow H$ is the adjoint operator of $L$,  
    \begin{equation*}
        \P(\|LY\| \le \epsilon) \ge \P(\|X\| \le \lambda \epsilon) \cdot \E\Big[\exp\Big(-\frac{\lambda^2}{2}\|Y\|_H^2\Big)\Big]
    \end{equation*}
    for all $\lambda > 0$ and $\epsilon > 0$.
\end{theorem}

In our proof of Theorem \ref{thm:d-metric-entropy-main-text}, we use Theorem \ref{thm:li-linde} with $E = C([0, 1]^m)$ and $X = X_m$ and Theorem \ref{thm:chen-li} with $E = C([0, 1]^m)$, $H = L^2([0, 1]^m)$, $L = \intoper_m$, and $X = Y = B_m$.
Observe that the Brownian sheet $B_m$ is a centered Gaussian element both in $L^2([0, 1]^m)$ and $C([0, 1]^m)$, and the integrated Brownian sheet $X_m$ is a centered Gaussian element in $C([0, 1]^m)$.
It is also known that the covariance operator of the Brownian sheet $B_m$ is $\intoper_m \intoper_m^{*}$ (see, e.g., \cite{Dunker1999}, Section 3). 
Moreover, one can check from the series expansion of $X_m$ and the theorem below that the reproducing kernel Hilbert space $H_{X_m}$ of $X_m$ in $C([0, 1]^m)$ and its unit ball $K_{X_m}$ can be written as 
\begin{equation*}
    H_{X_m} = \bigg\{ \sum_{k_1, \dots, k_m \in \mathbb{N}} w_{k_1, \dots, k_m} \cdot \intoper_m^2(\psi_{k_1} \otimes \dots \otimes \psi_{k_m})(\cdot): (w_{k_1, \dots, k_m}, k_1, \dots, k_m \in
      \mathbb{N}) \in l^2 \bigg\}
\end{equation*}
and
\begin{equation*}
    K_{X_m} = \bigg\{ \sum_{k_1, \dots, k_m \in \mathbb{N}} w_{k_1, \dots, k_m} \cdot \intoper_m^2(\psi_{k_1} \otimes \dots \otimes \psi_{k_m})(\cdot) : (w_{k_1, \dots, k_m}, k_1, \dots, k_m \in
    \mathbb{N}) \in B_{l^2} \bigg\},
\end{equation*}
where $B_{l_2}$ is the unit ball in $l^2$.

\begin{theorem}[\cite{van2008reproducing}, Theorem 4.2]\label{thm:vaart-zanten}
Suppose that $\{h_k\}_{k \ge 1}$ is a sequence in a separable Banach space $E$ such that if $w \in l^2$ and
\begin{equation*}
    \sum_{k = 1}^{\infty} w_k h_k = 0,
\end{equation*}
where the convergence is in $E$, then $w = 0$.
Then, for a random element $X$ in $E$ defined by 
\begin{equation*}
    X = \sum_{k = 1}^{\infty} \xi_k h_k,
\end{equation*}
where $\xi_k$ are independent standard normal random variables, and the series converges a.s. in $E$, the reproducing kernel Hilbert space of $X$ in $E$ is given by
\begin{equation*}
    H_{X} = \bigg\{\sum_{k = 1}^{\infty} w_k h_k: w \in l^2 \bigg\}
\end{equation*}
with the norm 
\begin{equation*}
    \Big\|\sum_{k = 1}^{\infty} w_k h_k \Big\| = \|w\|_2.
\end{equation*}
\end{theorem}

\begin{proof}[Proof of Theorem \ref{thm:d-metric-entropy-main-text}]
For each $k \in \mathbb{N}$, let $\phi_k$ be the function on $[0, 1]$ defined by $\phi_k(x) = \psi_k(1-x)$ for $x \in [0, 1]$.
Then, clearly, $\{\phi_{k_1} \otimes \dots \otimes \phi_{k_m}: k_1, \dots, k_m \in \mathbb{N}\}$ is also an orthonormal basis of $L^2([0,1]^m)$.
Also, let   
\begin{align*}
    &Z = \big\{(z_{k_1, \dots, k_m}, k_1, \dots, k_m \in \mathbb{N}) :
    z_{k_1, \dots, k_m} = \inner{F}{\phi_{k_1} \otimes \dots \otimes
      \phi_{k_m}} \\ 
    &\qquad \qquad \qquad \qquad \qquad \qquad \qquad \qquad \mbox{ for } k_1, \dots, k_m \in \mathbb{N}, \mbox{
    where } F \in \D_m\big\}.  
\end{align*}
Note that $Z$ is convex and symmetric with respect to the origin, because $\D_m$ is convex and symmetric with respect to the zero-function.
Also, for $F \in \D_m$ of the form
\begin{equation*}
    F(x_1, \dots, x_m) = \int (x_1 - t_1)_+ \cdots (x_m - t_m)_+ \, d\nu(t) = \int_{[\zerovec, x]} (x_1 - t_1) \cdots (x_m - t_m) \, d\nu(t),
\end{equation*}
where $[\zerovec, x] = \{t \in \R^m : 0 \le t_j \le x_j \mbox{ for all } j \in [m]\}$ and $\nu$ is a signed measure on $[0, 1]^m$ with variation $|\nu|([0, 1]^m) \le 1$, we have
\begin{align} \label{F-inner-basis}
\begin{split}
    &\inner{F}{\phi_{k_1} \otimes \dots \otimes \phi_{k_m}} = \int_{[0, 1]^m} \int_{[\zerovec, x]} (x_1 - t_1) \cdots (x_m - t_m) \, d\nu(t) \cdot (\phi_{k_1} \otimes \dots \otimes \phi_{k_m})(x) \, dx \\
    &\qquad = \int_{[0, 1]^m} \int_{t_1}^1 \dots \int_{t_m}^1 (x_1 - t_1) \cdots (x_m - t_m) \cdot \phi_{k_1}(x_1) \cdots \phi_{k_m}(x_m) \, dx_m \cdots \, dx_1 \, d\nu(t)
\end{split}
\end{align}
for $k_1, \dots, k_m \in \mathbb{N}$.
By Parseval's indentity, $Z$ is a subset of $l^2$ and 
\begin{equation} \label{eq:parseval}
     \log N(\epsilon, \D_m, \|\cdot\|_2) = \log N(\epsilon, Z, \|\cdot\|_{l^2}).
\end{equation}
Also, since $Z$ is convex and symmetric, by the fundamental duality theorem of metric entropy (refer to, e.g., \cite{Artstein2004b}, Theorem 5), there exist universal constants $c, C > 0$ such that
\begin{equation}\label{eq:duality}
    \log N(\epsilon, Z, \|\cdot\|_{l^2}) \le C \log N(c \epsilon, B_{l_2}, \|\cdot\|_{Z^{\circ}}),
\end{equation}
where $B_{l_2}$ is the unit ball with respect to the $l^2$ norm and $\|\cdot\|_{Z^{\circ}}$ is the gauge of 
\begin{equation*}
    Z^{\circ} = \bigg\{ (w_{k_1, \dots, k_m}, k_1, \dots, k_m \in \mathbb{N}) \in l^2: \sup_{z \in Z} \bigg|\sum_{k_1, \dots, k_m \in \mathbb{N}}
    w_{k_1, \dots, k_m} z_{k_1, \dots, k_m}\bigg| \le 1 \bigg\}, 
\end{equation*}
i.e., 
\begin{equation*}
    \|w\|_{Z^{\circ}} = \inf \{r \ge 0: w \in r Z^{\circ}\} 
\end{equation*}
for $w = (w_{k_1, \dots, k_m}, k_1, \dots, k_m \in \mathbb{N}) \in l^2$. 
Due to \eqref{F-inner-basis}, we can also write $Z^{\circ}$ as 
\begin{align}\label{Z-circ-characterization}
\begin{split}
    &Z^{\circ} = \bigg\{ (w_{k_1, \dots, k_m}, k_1, \dots, k_m \in\mathbb{N}) \in l^2: \sup_{t \in [0, 1]^m} \bigg|\sum_{k_1, \dots, k_m \in \mathbb{N}} w_{k_1, \dots, k_m} \\
    &\qquad \qquad \cdot \int_{t_1}^1 \dots \int_{t_m}^1 (x_1 - t_1) \cdots (x_m - t_m) \phi_{k_1}(x_1) \cdots \phi_{k_m}(x_m) \, dx_m \cdots \, dx_1 \bigg| \le
      1 \bigg\}.
\end{split}
\end{align}

We now relate the covering number $N(\epsilon, B_{l_2}, \|\cdot\|_{Z^{\circ}})$ to the small ball probability of the integrated Brownian sheet $X_m$.
Recall that we can represent the integrated Brownian sheet $X_m$ as
\begin{equation} \label{integrated-brownian-sheet}
    X_m(\cdot, \omega) = \sum_{k_1, \dots, k_m \in \mathbb{N}} \xi_{k_1, \dots, k_m}(\omega) \cdot \intoper_m^2(\psi_{k_1} \otimes \dots \otimes \psi_{k_m})(\cdot), 
\end{equation}
where the series converges a.s. in $C([0, 1]^m)$.
Observe that 
\begin{align} \label{integrated-brownian-sheet-coefficients}
\begin{split}
    &\intoper_m^2(\psi_{k_1} \otimes \dots \otimes \psi_{k_m})(t) = \int_{[\zerovec, t]} \int_{[\zerovec, s]} (\psi_{k_1} \otimes \dots \otimes \psi_{k_m})(x) \, dx \, ds \\
    &\qquad = \int_{[\zerovec, t]} \int_{[x, t]} (\psi_{k_1} \otimes \dots \otimes \psi_{k_m})(x) \, ds \, dx \\
    &\qquad = \int_{[\zerovec, t]} (t_1 - x_1) \cdots (t_m - x_m) \psi_{k_1}(x_1) \cdots \psi_{k_m}(x_m) \, dx \\
    &\qquad = \int_{s_1}^1 \dots \int_{s_m}^1 (z_1 - s_1) \cdots (z_m - s_m) \phi_{k_1}(z_1) \cdots  \phi_{k_m}(z_m) \, dz_m \cdots \, dz_1
\end{split}
\end{align}
for $t \in [0, 1]^m$, where $[x, t] = \{u \in \R^m: x_i \le u_i \le t_i \mbox{ for all } i \in [m]\}$ and  $(s_1, \dots, s_m) = (1 - t_1, \dots, 1 - t_m)$.
By \eqref{Z-circ-characterization} and \eqref{integrated-brownian-sheet-coefficients}, 
\begingroup
\allowdisplaybreaks
\begin{align*}
    &\sup_{t \in [0, 1]^m} \bigg|\sum_{k_1, \dots, k_m \in \mathbb{N}} w_{k_1, \dots, k_m} \cdot \intoper_m^2(\psi_{k_1} \otimes \dots \otimes \psi_{k_m})(t) \bigg| \\
    &\quad= \sup_{s \in [0, 1]^m} \bigg|\sum_{k_1, \dots, k_m \in \mathbb{N}} w_{k_1, \dots, k_m} \\
    &\qquad \qquad \qquad \qquad \cdot \int_{s_1}^1 \dots \int_{s_m}^1 (z_1 - s_1) \cdots (z_m - s_m) \phi_{k_1}(z_1) \cdots  \phi_{k_m}(z_m) \, dz_m \cdots \, dz_1 \bigg| \\
    &\quad= \|w\|_{Z^{\circ}}
\end{align*}
\endgroup
for $w = (w_{k_1, \dots, k_m}, k_1, \dots, k_m \in \mathbb{N}) \in l^2$.
Since the unit ball $K_{X_m}$ of the reproducing kernel Hilbert space of $X_m$ in $C([0, 1]^m)$ is
\begin{equation*}
    K_{X_m} = \bigg\{ \sum_{k_1, \dots, k_m \in \mathbb{N}} w_{k_1, \dots, k_m} \cdot \intoper_m^2(\psi_{k_1} \otimes \dots \otimes \psi_{k_m})(\cdot) : (w_{k_1, \dots, k_m}, k_1, \dots, k_m \in
    \mathbb{N}) \in B_{l^2} \bigg\},
\end{equation*}
it thus follows that
\begin{equation*}
    N(\epsilon, K_{X_m}, \|\cdot\|_{\infty}) = N(\epsilon, B_{l^2}, \|\cdot\|_{Z^{\circ}}).
\end{equation*}
Hence, for the integrated Brownian sheet $X_m$, we can restate Theorem \ref{thm:li-linde} as follows:
for $\alpha > 0$ and $\beta \in \R$,
\begin{align}\label{lilindethm}
\begin{split}
    &\log \P\Big(\sup_{t \in [0, 1]^m} |X_m(t)| \le \epsilon \Big) \ge -c_m \epsilon^{-\alpha} \Big[\log\frac{1}{\epsilon}\Big]^{\beta} \mbox{ for all sufficiently small } \epsilon > 0 \\ 
    &\; \implies \log N(\epsilon, B_{l_2}, \|\cdot\|_{Z^{\circ}}) \le C_m \epsilon^{-\frac{2\alpha}{(2 + \alpha)}} \Big[\log\frac{1}{\epsilon}\Big]^{\frac{2\beta}{2 + \alpha}} \mbox{ for all sufficiently small } \epsilon > 0.
\end{split}
\end{align}
Combining \eqref{eq:parseval}, \eqref{eq:duality}, and \eqref{lilindethm}, we can see that it suffices to find lower bounds of $\log \P(\sup_{t \in [0, 1]^m} |X_m(t)| \le \epsilon)$ in order to find upper bounds of the metric entropy of $\D_m$.

We now turn to the small ball probability of $X_m$.
Recall that the covariance operator of the Brownian sheet $B_m$ is $\intoper_m \intoper_m^{*}$. 
By applying Theorem \ref{thm:chen-li} to $E = C([0, 1]^m)$, $H = L^2([0, 1]^m)$, $L = \intoper_m$, and $X = Y = B_m$, we can derive that 
\begin{equation*}
    \P\Big(\sup_{t \in [0, 1]^m} |X_m(t)| \le \epsilon \Big) \ge \P\Big(\sup_{t \in [0, 1]^m} |B_m(t)| \le \lambda \epsilon \Big) \cdot \E\Big[\exp \Big(-\frac{\lambda^2}{2} \int_{[0, 1]^m} B_m(t)^2 \, dt\Big) \Big].
\end{equation*}
for all $\lambda > 0$ and $\epsilon > 0$. 
Substituting $\lambda = \epsilon^{-\frac{2}{3}}$ and using Markov's inequality, we obtain
\begingroup
\allowdisplaybreaks
\begin{align*}
    &\P\Big(\sup_{t \in [0, 1]^m} |X_m(t)| \le \epsilon \Big)  \ge \P\Big(\sup_{t \in [0, 1]^m} |B_m(t)| \le \epsilon^{\frac{1}{3}} \Big) \cdot \E\Big[\exp \Big(-\frac{\epsilon^{-\frac{4}{3}}}{2} \int_{[0, 1]^m} B_m(t)^2 \, dt \Big) \Big] \\
    &\qquad \ge \P\Big(\sup_{t\in [0, 1]^m} |B_m(t)| \le \epsilon^{\frac{1}{3}}\Big) \cdot \Big[\P\Big(\sup_{t \in [0, 1]^m} |B_m(t)| \le \epsilon^{\frac{1}{3}}\Big) \cdot \exp\Big(-{\frac{\epsilon^{-\frac{2}{3}}}{2}}\Big)\Big] \\
    &\qquad \ge \Big[\P\Big(\sup_{t \in [0, 1]^m} |B_m(t)| \le \epsilon^{\frac{1}{3}}\Big)\Big]^2 \cdot \exp \Big( -{\frac{\epsilon^{-\frac{2}{3}}}{2}} \Big).
\end{align*}
\endgroup
In \cite{Dunker1999}, Theorem 6, it was proved that 
\begin{equation*}
    \log \P\Big(\sup_{t \in [0, 1]^m} |B_m(t)| \le \epsilon \Big) \ge -c_m \epsilon^{-2}\Big[\log \frac{1}{\epsilon}\Big]^{2m - 1}
\end{equation*}
for sufficiently small $\epsilon > 0$.
In particular, it is well-known that the logarithmic multiplicative factor of the right-hand side can be omitted when $m = 1$.
Therefore, it follows that 
\begin{align*}
    &\log \P\Big(\sup_{t \in [0, 1]^m} |X_m(t)| \le \epsilon\Big) \ge 2 \log \P\Big(\sup_{t \in [0, 1]^m} |B_m(t)| \le \epsilon^{\frac{1}{3}}\Big) - {\frac{\epsilon^{-\frac{2}{3}}}{2}} \\
    &\qquad \qquad \ge -c_m \epsilon^{-\frac{2}{3}} \Big[\log \frac{1}{\epsilon}\Big]^{2m - 1} - {\frac{\epsilon^{-\frac{2}{3}}}{2}} \ge -c_m \epsilon^{-\frac{2}{3}}\Big[\log \frac{1}{\epsilon}\Big]^{2m - 1}
\end{align*}
for sufficiently small $\epsilon > 0$.
This result together with \eqref{lilindethm} implies that
\begin{equation*}
    \log N(\epsilon, B_{l_2}, \|\cdot\|_{Z^{\circ}}) \le C_m \epsilon^{-\frac{1}{2}} \Big[\log \frac{1}{\epsilon}\Big]^{\frac{3(2m - 1)}{4}}
\end{equation*}
for sufficiently small $\epsilon > 0$.
Consequently, combining all the pieces, we can prove that
\begin{equation*}
    \log N(\epsilon, \D_m, \|\cdot\|_2) \le C_m \epsilon^{-\frac{1}{2}}\Big[\log \frac{1}{\epsilon}\Big]^{\frac{3(2m - 1)}{4}}
\end{equation*}
for sufficiently small $\epsilon > 0$. 
Note that we can omit the logarithmic factor when $m = 1$. 
\end{proof}

\subsubsection{Proof of Theorem \ref{thm:risk-upper-bound-fixed-approx}}\label{pf:risk-upper-bound-approx}
Our proof of Theorem \ref{thm:risk-upper-bound-fixed-approx} uses the following lemma, which we prove in Appendix \ref{pf:discrete-measures-approximation}. 
This lemma states that for every function $f \in \infmars^{d, s}$, we can find $\fazeromu$ which is close to $f$ and whose $\mu_{\alpha}$ are discrete signed measures concentrated on the lattices generated by $\tilde{\mathcal{U}}_k$.    
We will use this lemma again for our proof of Theorem \ref{thm:rate-of-convergence-random-approx}.

\begin{lemma}\label{lem:discrete-measures-approximation}
    Suppose we are given a real number $a_{\zerovec}$ and a collection of finite signed measures $\{\nu_{\alpha}\}$ where $\nu_{\alpha}$ is defined on $[0, 1)^{|\alpha|}$ for each $\alpha \in \{0, 1\}^d \setminus \{\zerovec\}$ with $|\alpha| \le s$.
    Then, there exists a collection of discrete signed measures $\{\mu_{\alpha}\}$ where $\mu_\alpha$ is concentrated on $(\prod_{k \in S(\alpha)} \tilde{\mathcal{U}}_k) \cap [0, 1)^{|\alpha|}$ for each $\alpha \in \{0, 1\}^d \setminus \{\zerovec\}$ with $|\alpha| \le s$ such that 
    \begin{longlist}   
        \item $\Vmars(\fazeromu) \le \Vmars(\fazeronu)$, 
        \item $\|\fazeromu - \fazeronu\|_{\infty} \le (2/N) \cdot \Vmars(\fazeronu)$, and
        \item $\|\fazeromu - \fazeronu\|_{p_0, 2} \le C_d (B/N^3)^{1/2} \cdot \Vmars(\fazeronu)$ 
    \end{longlist}
    for some positive constant $C_d$ depending on $d$,  where $N = \min_k N_k$.
  \end{lemma}

\begin{proof}[Proof of Theorem \ref{thm:risk-upper-bound-fixed-approx}]
Recall the definition of $\F_{\text{disc}}(V)$ from our proof of Theorem \ref{thm:risk-upper-bound-fixed}.
Also, let $\F_{\text{approx}}(V)$ be the collection of all the functions $\fazeronu \in \infmars^{d, s}$ with $\Vmars(\fazeronu) \le V$ where $\nu_\alpha$ is concentrated on $(\prod_{k \in S(\alpha)} \tilde{\mathcal{U}}_k) \cap [0, 1)^{|\alpha|}$ for each $\alpha \in \{0, 1\}^d \setminus \{\zerovec\}$ with $|\alpha| \le s$.
Then, clearly,
\begin{equation*}
    \tilde{f}^{d, s}_{n, V} \in \argmin_{f \in \F_{\text{approx}}(V)} \bigg\{\sum_{i = 1}^{n} \big(y_i - f(x^{(i)})\big)^2 \bigg\}. 
\end{equation*}
Also, Lemma \ref{lem:discrete-measures-approximation} guarantees that we have $f_0 \in \F_{\text{approx}}(V)$ for which
\begin{equation*}
    \| f_0 - f^* \|_{n} \le \| f_0 - f^* \|_{\infty} \le \frac{2V}{N}.
\end{equation*}

In order to apply Theorem \ref{thm:least-squares-estimator} with $f_0 \in \F_{\text{approx}}(V)$, we need to bound the metric entropy of $B_{\F_{\text{approx}}(V)}(t, f_0, \|\cdot\|_n)$ (under the norm $\| \cdot \|_n$).
By Lemma \ref{lem:reduction-to-discrete-measures}, there exists $g_0 \in \F_{\text{disc}}(V)$ such that 
\begin{equation*}
    g_0(x^{(i)}) = f_0(x^{(i)})
\end{equation*}
for all $i \in [n]$.
Also, by the same lemma and the fact that the norm $\| \cdot \|_n$ only depends on the values of functions at $x^{(i)}$, we have 
\begin{equation*}
    N\big(\epsilon, B_{\F_{\text{approx}}(V)}(t, f_0, \|\cdot\|_n \big), \|\cdot\|_n) \le N\big(\epsilon, B_{\F_{\text{disc}}(V)}(t, g_0, \|\cdot\|_n), \|\cdot\|_n \big)
\end{equation*}
for every $\epsilon > 0$.
Hence, we can repeat our argument in the proof of Theorem \ref{thm:risk-upper-bound-fixed} to bound $N(\epsilon, B_{\F_{\text{disc}}(V)}(t, g_0, \|\cdot\|_n), \|\cdot\|_n)$ and thereby show that the function $\Psi$ defined by
\begin{align*}
    &\Psi(t) = C_{\rho, d} t \log n + C_{d} t \log\Big(1 + \frac{2V}{\sqrt{n}t}\Big) \\
    &\qquad \quad+ C_{\rho, d} t \bigg[\log\Big(2 +  \frac{V}{t}\Big)\bigg]^{\frac{3(2s - 1)}{8}} + C_{\rho, d} V^{\frac{1}{4}} t^{\frac{3}{4}} \bigg[\log\Big(2 +  \frac{V}{t}\Big)\bigg]^{\frac{3(2s - 1)}{8}}
\end{align*}
for $t > 0$ and 
\begin{align*}
    &t_n = \max\bigg\{\| f_0 - f^* \|_{n}, \frac{\sigma}{\sqrt{n}}, C_{\rho, d} \frac{\sigma \log n}{\sqrt{n}}, C_{d} \frac{\sigma}{\sqrt{n}} \log\Big(1 + \frac{2V}{\sigma }\Big), \\
    &\qquad \qquad \quad C_{\rho, d}  \frac{\sigma}{\sqrt{n}} \bigg[\log\Big(2 + \frac{V n^{\frac{1}{2}} }{\sigma}\Big)\bigg]^{\frac{3(2s - 1)}{8}}, C_{\rho, d} \frac{\sigma^{\frac{4}{5}} V^{\frac{1}{5}}}{n^{\frac{2}{5}}} \bigg[\log\Big(2 + \frac{V n^{\frac{1}{2}}}{\sigma}\Big)\bigg]^{\frac{3(2s - 1)}{10}}\bigg\},
\end{align*}
satisfy the conditions in Theorem \ref{thm:least-squares-estimator}. 
As a result, Theorem \ref{thm:least-squares-estimator} (and Remark \ref{rmk:least-squares-estimator-fixed-design}) implies  
\begin{align*}
    &\mathcal{R}_F\big(\tilde{f}^{d, s}_{n, V}, f^{*}\big) \le \| f_0 - f^* \|_n^2 + \frac{C \sigma^2}{n} + t_n^2 \\
    &\qquad \le \frac{8 V^2}{N^2} + C_{\rho, d}  \Big(\frac{\sigma^2 V^{\frac{1}{2}}}{n}\Big)^{\frac{4}{5}} \bigg[\log\Big(2 + \frac{V n^{\frac{1}{2}}}{\sigma}\Big)\bigg]^{\frac{3(2s - 1)}{5}} \\
    &\qquad \quad + C_{\rho, d} \frac{\sigma^2}{n} \bigg[\log\Big(2 + \frac{V n^{\frac{1}{2}}}{\sigma}\Big)\bigg]^{\frac{3(2s - 1)}{4}} + C_{d} \frac{\sigma^2}{n} \Big[\log\Big(1 + \frac{2V }{\sigma}\Big)\Big]^2 + C_{\rho, d} \frac{\sigma^2}{n} [\log n]^2.
\end{align*}
As in the proof of Theorem \ref{thm:risk-upper-bound-fixed}, we can absorb the third and the fourth term into the other terms. 
Consequently, we can derive that
\begin{align*}
    \mathcal{R}_F\big(\tilde{f}^{d, s}_{n, V}, f^{*}\big) \le \frac{8 V^2}{N^2} + C_{\rho, d}  \Big(\frac{\sigma^2 V^{\frac{1}{2}}}{n}\Big)^{\frac{4}{5}} \bigg[\log\Big(2 + \frac{V n^{\frac{1}{2}}}{\sigma}\Big)\bigg]^{\frac{3(2s - 1)}{5}} + C_{\rho, d} \frac{\sigma^2}{n} [\log n]^2
\end{align*}
as desired.
\end{proof}

\subsubsection{Proof of Theorem \ref{thm:rate-of-convergence-random}}\label{pf:rate-of-convergence-random}
The following theorem plays a central role in our proof of Theorem \ref{thm:rate-of-convergence-random}.
This theorem is a simple extension of \cite{han2019convergence}, Proposition 2 to the misspecified case, where estimators are searched over the function class that may not contain the true underlying function.  
For Theorem \ref{thm:rate-of-convergence-random}, we only need the partial result of Theorem \ref{thm:han-wellner-prop2-variant} concerning the well-specified case. 
The more general result covering the misspecified case is only required for Theorem \ref{thm:rate-of-convergence-random-approx}.
However, here we state the full result in order to avoid redundancy.
We provide our proof of Theorem \ref{thm:han-wellner-prop2-variant} in Appendix \ref{pf:han-wellner-prop2-variant}.

\begin{theorem}\label{thm:han-wellner-prop2-variant}
    Suppose that data $(X^{(1)}, y_1), \dots, (X^{(n)}, y_n)$ are generated according to the model 
    \begin{equation}
        y_i = f^{*}(X^{(i)}) + \xi_i, 
    \end{equation}
    where $X^{(i)}$ are i.i.d. random variables with law $P$ on $\mathcal{X}$, $\xi_i$ are i.i.d. errors independent of $X^{(1)}, \dots, X^{(n)}$ with mean zero and with finite $L^q$ norm for some $q \ge 3$, and
    $f^{*}$ is an unknown real-valued regression function defined on $\mathcal{X}$.
    Let $\F$ be a countable collection of real-valued functions defined on $\mathcal{X}$ and $f_0$ be an element of $\F$, and assume that $\|f - f_0\|_{\infty} \le M$ for every $f \in \F$.
    Also, let $\hat{f}$ be an estimator of $f^{*}$ for which
    \begin{equation*}
        \hat{f} \in \argmin_{f \in \F} \bigg\{\sum_{i = 1}^{n} \big(y_i - f(X^{(i)})\big)^2 \bigg\} 
    \end{equation*}
    with probability at least $1 - \epsilon$ for some $\epsilon \ge 0$.
    Moreover, suppose that there exists $t_n \ge 4 \| f_0 - f^*\|_{P, 2}$ such that 
    \begin{align*}
        &\E\bigg[\sup_{\substack{f \in \F - \{f_0\}\\ \|f\|_{P, 2} \le rt_n} } \Big| \frac{1}{\sqrt{n}} \sum_{i=1}^{n} \xi_i f(X^{(i)}) \Big| \bigg] \le r \sqrt{n} t_n^2, \\
        &\E\bigg[\sup_{\substack{f \in \F - \{f_0\}\\ \|f \|_{P, 2} \le rt_n} } \Big| \frac{1}{\sqrt{n}} \sum_{i=1}^{n} \epsilon_i f(X^{(i)}) \Big| \bigg] \le r \sqrt{n} t_n^2,
    \end{align*}
    and 
    \begin{equation*}
        \E\bigg[\sup_{\substack{f \in \F - \{f_0\}\\ \|f\|_{P, 2} \le rt_n}} \Big| \frac{1}{\sqrt{n}} \sum_{i=1}^{n} \epsilon_i f(X^{(i)}) \cdot (f_0 - f^*)(X^{(i)}) \Big| \bigg] \le r \sqrt{n} t_n^2
    \end{equation*}
    for every $r \ge 1$.
    Here $\|\cdot\|_{P, 2}$ is the norm defined by 
    \begin{equation*}
        \| f \|_{P, 2} = \big(\E_{X \sim P}[f(X)^2 ]\big)^{\frac{1}{2}}
    \end{equation*}
    and $\epsilon_i$ are i.i.d. Rademacher variables independent of $X^{(1)}, \dots, X^{(n)}$.
    Then, there exists a universal positive constant $C$ such that
    \begin{align}\label{probability-bound-random-design}
    \begin{split}
        \P (\|\hat{f} - f_0 \|_{P, 2} > t) \le \epsilon &+ C \cdot \frac{1 + M^3}{t^3} \cdot t_n^3 + C \cdot \frac{\|\xi_1\|_2^3 + \|f_0 - f^*\|_{\infty}^3 + M^3}{n^{\frac{3}{2}} t^3} \\
        &+ C \cdot \frac{M^3(\|\xi_1\|_q^3 n^{\frac{3}{q}} + \|f_0 - f^*\|_{\infty}^3 + M^3)}{n^{3} t^6}
    \end{split}
    \end{align}
    for every $t \ge t_n$.
\end{theorem}

\begin{remark}
The reason why $\F$ is assumed to be countable in Theorem \ref{thm:han-wellner-prop2-variant} is to guarantee the measurability of the supremums. 
We can remove this condition when $\F$ is a subset of $C([0, 1])^m$. 

Suppose $\F$ is a subset of $C([0, 1])^m$ and $\F_0$ is a subset of $\F - \{f_0\}$. 
Because of our assumption that $f_0 \in \F$, $\F_0$ is a subset of $C([0, 1]^m)$ as well. 
Also, since $(C([0,1]^m), \|\cdot\|_{\infty})$ is a separable metric space, $\F_0$ is also separable with respect to the sup-norm. 
Let $\G_0$ be a countable dense subset of $\F_0$ with respect to the sup-norm.
Then, we have 
\begin{align*}
    \sup_{f \in \F_0} \Big|\frac{1}{\sqrt{n}} \sum_{i = 1}^{n} \xi_i f(X^{(i)}) \Big| &= \sup_{g \in \G_0} \Big|\frac{1}{\sqrt{n}} \sum_{i = 1}^{n} \xi_i g(X^{(i)}) \Big|, \\
    \sup_{f \in \F_0} \Big|\frac{1}{\sqrt{n}} \sum_{i = 1}^{n} \epsilon_i f(X^{(i)}) \Big| &= \sup_{g \in \G_0} \Big|\frac{1}{\sqrt{n}} \sum_{i = 1}^{n} \epsilon_i g(X^{(i)}) \Big|,
\end{align*}
and
\begin{equation*}
    \sup_{f \in \F_0} \Big|\frac{1}{\sqrt{n}} \sum_{i = 1}^{n} \epsilon_i f(X^{(i)}) \cdot (f_0 - f^*)(X^{(i)}) \Big| = \sup_{g \in \G_0} \Big|\frac{1}{\sqrt{n}} \sum_{i = 1}^{n} \epsilon_i g(X^{(i)}) \cdot (f_0 - f^*)(X^{(i)}) \Big|
\end{equation*}
so that all the supremums are measurable. 
For these reasons, in this case, Theorem \ref{thm:han-wellner-prop2-variant} is valid without the countability assumption.
\end{remark}

\begin{remark}[Well-specified Case]\label{rmk:han-wellner-prop2-variant-well-specified}
    If $f^*$ belongs to the function class $\F$, we can set $f_0 = f^*$. 
    In this case, the conditions $t_n \ge 4 \| f_0 - f^*\|_{P, 2}$ and 
    \begin{equation*}
        \E\bigg[\sup_{\substack{f \in \F - \{f_0\}\\ \|f\|_{P, 2} \le rt_n}} \Big| \frac{1}{\sqrt{n}} \sum_{i=1}^{n} \epsilon_i f(X^{(i)}) \cdot (f_0 - f^*)(X^{(i)}) \Big| \bigg] \le r \sqrt{n} t_n^2
    \end{equation*}
    always hold, and \eqref{probability-bound-random-design} is simplified as
    \begin{equation*}
        \P (\|\hat{f} - f^* \|_{P, 2} > t) \le \epsilon + C \cdot \frac{1 + M^3}{t^3} \cdot t_n^3 + C \cdot \frac{\|\xi_1\|_2^3 + M^3}{n^{\frac{3}{2}} t^3} + C \cdot \frac{M^3(\|\xi_1\|_q^3 n^{\frac{3}{q}} + M^3)}{n^{3} t^6}.
    \end{equation*}
    This simplified version will be used for our proof of Theorem \ref{thm:rate-of-convergence-random}.
\end{remark}

Recall that our estimator $\hat{f}^{d, s}_{n, V}$ is defined as a least squares estimator over
\begin{equation*}
    \big\{f \in \infmars^{d, s}: \Vmars(f) \leq V \big\},
\end{equation*}
which is not uniformly bounded as required for $\F$ in Theorem \ref{thm:han-wellner-prop2-variant}. 
Theorem \ref{thm:han-wellner-prop2-variant} is therefore not directly applicable to our estimator, but we can avoid this problem by using the fact that $\hat{f}^{d, s}_{n, V}$ is (uniformly) bounded in probability.
Due to the boundedness of $\hat{f}^{d, s}_{n, V}$ as stated in Lemma \ref{lem:sup-norm-bound} (proved in Appendix \ref{pf:sup-norm-bound}), it is guaranteed that $\hat{f}^{d, s}_{n, V}$ is a least squares estimator over a uniformly bounded function class with high probability.

\begin{lemma}\label{lem:sup-norm-bound}
    The estimator $\hat{f}^{d, s}_{n, V}$ satisfies that
    \begin{equation*}
        \| \hat{f}^{d, s}_{n, V} - f^* \|_{\infty} = O_p(1).
    \end{equation*}
\end{lemma}

We also use the following theorem from the same paper (\cite{han2019convergence}). 
This theorem will help us bound
\begin{equation*}
    \E\bigg[\sup_{\substack{f \in \F - \{f_0\}\\ \|f\|_{P, 2} \le rt_n} } \Big| \frac{1}{\sqrt{n}} \sum_{i=1}^{n} \xi_i f(X^{(i)}) \Big| \bigg]
\end{equation*}
in Theorem \ref{thm:han-wellner-prop2-variant} using upper bounds of 
\begin{equation*}
    \E\bigg[\sup_{\substack{f \in \F - \{f_0\}\\ \|f\|_{P, 2} \le rt_k} } \Big| \sum_{i=1}^{k} \epsilon_i f(X^{(i)}) \Big| \bigg] 
\end{equation*}
for suitably chosen $t_k$, for $1 \le k \le n$.

\begin{theorem}[\cite{han2019convergence}, Corollary 1]\label{thm:han-wellner}
Let $\F_1, \dots, \F_n$ be countable collections of real-valued functions defined on $\mathcal{X}$ for which $\F_k \supseteq \F_n$ for every $1 \le k \le n$.
Suppose $X^{(1)}, \dots, X^{(n)}$ are permutation invariant random variables on $\mathcal{X}$ and $\xi_1, \dots, \xi_n$ are i.i.d. mean zero random variables independent of $X^{(1)}, \dots, X^{(n)}$.
Assume that there exist positive constants $r \ge 1$ and $C$ such that
\begin{equation*}
    \E\bigg[\sup_{f \in \F_k} \Big| \sum_{i=1}^{k} \epsilon_i f(X^{(i)}) \Big| \bigg] \le C \cdot k^{\frac{1}{r}}
\end{equation*}
for every $1 \le k \le n$. 
Then, for every $q \ge 1$, we have
\begin{equation*}
    \E\bigg[\sup_{f \in \F_n} \Big| \sum_{i=1}^{n} \xi_i f(X^{(i)}) \Big| \bigg] \le 4C \cdot \|\xi_1\|_{\min\{q, r\}, 1} \cdot n^{1 / \min\{q, r\}}, 
\end{equation*}
where
\begin{equation*}
  \| \xi_1 \|_{p, 1} := \int_{0}^{\infty} (\P(|\xi_1| > t))^{\frac{1}{p}} \, dt
\end{equation*}
for each $p \ge 1$.
\end{theorem}

As in Theorem \ref{thm:han-wellner-prop2-variant}, the countability assumption in Theorem \ref{thm:han-wellner} can be dropped if $\F_k$ are subsets of $C([0, 1]^m)$.

\begin{proof}[Proof of Theorem \ref{thm:rate-of-convergence-random}]
Suppose we are given $\epsilon > 0$.
By Lemma \ref{lem:sup-norm-bound}, there exists $M > 0$ such that 
\begin{equation*}
    \P\big(\| \hat{f}^{d, s}_{n, V} - f^* \|_{\infty} \le M\big) \ge 1 - \epsilon
\end{equation*}
for sufficiently large $n$.
For $V > 0$, let
\begin{equation*}
    \F_M(V) = \big\{f \in \infmars^{d, s}: \Vmars(f) \le V \mbox{ and } \|f - f^*\|_{\infty} \le M \big\}.
\end{equation*}
Then, it is clear from the definition of $\hat{f}^{d, s}_{n, V}$ that for sufficiently large $n$,
\begin{equation*}
        \hat{f}^{d, s}_{n, V} \in \argmin_{f \in \F_M(V)} \bigg\{\sum_{i = 1}^{n} \big(y_i - f(X^{(i)})\big)^2 \bigg\} 
\end{equation*}
with probability at least $1 - \epsilon$.
Also, we can see that Theorem \ref{thm:han-wellner-prop2-variant} is valid for $\F_M(V)$ since $\F_M(V)$ is a subset of $C([0, 1]^d)$.

We first bound
\begin{equation*}
    \E\bigg[\sup_{\substack{f \in \F_M(V) - \{f^*\} \\ \|f\|_{p_0, 2} \le t} } \Big| \frac{1}{\sqrt{k}} \sum_{i=1}^{k} \epsilon_i f(X^{(i)}) \Big| \bigg]
\end{equation*}
for $1 \le k \le n$ and $t > 0$.
The following lemma (proved in Appendix \ref{pf:maximal-inequality-bracketing-fvm}) enables us to bound it using the bracketing entropy integral of $\F_M(V) - \{f^*\}$ (under the norm $\|\cdot\|_{p_0, 2}$). 
Recall that the bracketing entropy integral of a function class $\G$ under a norm $\| \cdot \|$ is defined as 
\begin{equation*}
    J_{[ \ ]}(t, \G, \|\cdot\|) = \int_{0}^{t} \sqrt{1 + \log N_{[ \ ]}( \epsilon, \G, \| \cdot \|)} d \epsilon
\end{equation*}
for $t > 0$.

\begin{lemma}\label{lem:maximal-inequality-bracketing-fvm} 
There exists a universal positive constant $C$ such that
\begin{align*}
    &\E\bigg[\sup_{\substack{f \in \F_M(V) - \{f^*\} \\ \|f\|_{p_0, 2} \le t} } \Big| \frac{1}{\sqrt{k}} \sum_{i=1}^{k} \epsilon_i f(X^{(i)}) \Big| \bigg] \\
    &\qquad  \le C J_{[ \ ]}(t, \F_M(V) - \{f^*\}, \|\cdot\|_{p_0, 2}) \cdot \bigg(1 + M \cdot \frac{J_{[ \ ]}(t, \F_M(V) - \{f^*\}, \|\cdot\|_{p_0, 2} )}{t^2 \sqrt{k}} \bigg)
\end{align*}
for $t > 0$.
\end{lemma}

Also, the following lemma provides an upper bound of the bracketing entropy of $\F_M(V) - \{f^*\}$ (under the norm $\| \cdot \|_{p_0, 2}$). We defer our proof of this lemma to Appendix \ref{pf:fvm-bracketing-entropy}.

\begin{lemma}\label{lem:fvm-bracketing-entropy}
There exist positive constants $C_s$ depending on $s$ and $C_{B, d}$ depending on $B$ and $d$ such that
\begin{align*}
    &\log N_{[ \ ]} (\epsilon, \F_M(V) - \{f^*\}, \|\cdot\|_{p_0, 2}) \le 2^d \log \Big( 2 + C_s \cdot \frac{M + V}{\epsilon} \Big) \\
    &\qquad \qquad \qquad \qquad \qquad \qquad + C_{B, d} \Big(2^{d+2} + \frac{2^{d+3} V}{\epsilon} \Big)^{\frac{1}{2}}\bigg[\log \Big(2^{d+2} + \frac{2^{d+3} V}{\epsilon}\Big)\bigg]^{2(s - 1)}
\end{align*}
for every $V > 0, M > 0$, and $\epsilon > 0$.
\end{lemma}

By Lemma \ref{lem:fvm-bracketing-entropy} and the inequality $(x + y)^{1/2} \le x^{1/2} + y^{1/2}$,
\begin{align*}
    J_{[ \ ]}(t, \F_M(V) - &\{f^*\}, \|\cdot\|_{p_0, 2}) = \int_{0}^{t} \sqrt{1 + \log N_{[ \ ]} (\epsilon, \F_M(V) - \{f^*\}, \|\cdot\|_{p_0, 2})} \, d\epsilon \\
    \le t &+ 2^{\frac{d}{2}} \int_{0}^{t} \sqrt{\log \Big( 2 + C_s \cdot \frac{M + V}{\epsilon} \Big)} \, d\epsilon \\
    &+ C_{B, d} \int_{0}^{t} \Big(2^{d+2} + \frac{2^{d+3} V}{\epsilon} \Big)^{\frac{1}{4}}\bigg[\log \Big(2^{d+2} + \frac{2^{d+3} V}{\epsilon}\Big)\bigg]^{s - 1} \, d\epsilon.
\end{align*}
We use Lemma \ref{lem:integration-helper} again for bounding the above integrals. 
Lemma \ref{lem:integration-helper} implies that 
\begin{align*}
    &\int_{0}^{t} \sqrt{\log \Big( 2 + C_s \cdot \frac{M + V}{\epsilon} \Big)} \, d\epsilon \le \int_{0}^{t} \sqrt{\log \Big(\frac{2t + C_s(M + V)}{\epsilon} \Big)} \, d\epsilon \\
    &\qquad \le Ct \bigg[1 + 2  \log \Big( 2 + C_s \cdot \frac{M + V}{t} \Big) \bigg] \le C_s t + Ct \log \Big( 2 + \frac{M + V}{t} \Big)
\end{align*}
and that 
\begin{align*}
    &\int_{0}^{t} \Big(2^{d+2} + \frac{2^{d+3} V}{\epsilon} \Big)^{\frac{1}{4}}\bigg[\log \Big(2^{d+2} + \frac{2^{d+3} V}{\epsilon}\Big)\bigg]^{s - 1} \, d\epsilon \\
    &\qquad \qquad \le \int_{0}^{t} \Big(\frac{2^{d+2} (t + 2V)}{\epsilon} \Big)^{\frac{1}{4}}\bigg[\log \Big(\frac{2^{d+2} (t +  2V)}{\epsilon}\Big)\bigg]^{s - 1} \, d\epsilon \\
    &\qquad \qquad \le C_s \big[2^{d+2} (t +  2V) \big]^{\frac{1}{4}} t^{\frac{3}{4}} \bigg[1 + \Big\{\log\Big(2^{d + 2} + \frac{2^{d + 3} V }{t}\Big)\Big\}^{s - 1} \bigg] \\
    &\qquad \qquad \le C_d (t + V^{\frac{1}{4}}t^{\frac{3}{4}}) \Big[\log \Big(2 + \frac{V}{t} \Big)\Big]^{s - 1} \\
    &\qquad \qquad \le C_d t \Big[\log \Big(2 + \frac{V}{t} \Big)\Big]^{s - 1} + C_d V^{\frac{1}{4}} t^{\frac{3}{4}} \Big[\log \Big(2 + \frac{V}{t} \Big)\Big]^{s - 1}.
\end{align*}
Hence, it follows that
\begin{align*}
    &J_{[ \ ]}(t, \F_M(V) - \{f^*\}, \|\cdot\|_{p_0, 2}) \le C_d t \log \Big( 2 + \frac{M + V}{t} \Big) + C_{B, d} t \Big[\log \Big(2 + \frac{V}{t} \Big)\Big]^{s - 1} \\ 
    &\qquad \qquad \qquad \qquad \qquad \qquad \quad + C_{B, d} V^{\frac{1}{4}} t^{\frac{3}{4}} \Big[\log \Big(2 + \frac{V}{t} \Big)\Big]^{s - 1}. 
\end{align*}
Note that if $t \le V$, 
\begin{equation*}
    C_{B, d} t \Big[\log \Big(2 + \frac{V}{t} \Big)\Big]^{s - 1} \le C_{B, d} V^{\frac{1}{4}} t^{\frac{3}{4}} \Big[\log \Big(2 + \frac{V}{t} \Big)\Big]^{s - 1}
\end{equation*}
and otherwise,
\begin{equation*}
    C_{B, d} t \Big[\log \Big(2 + \frac{V}{t} \Big)\Big]^{s - 1} \le C_{B, d} t \le C_{B, d} t \log \Big( 2 + \frac{M + V}{t} \Big). 
\end{equation*}
Therefore, we have
\begin{equation*}
    J_{[ \ ]}(t, \F_M(V) - \{f^*\}, \|\cdot\|_{p_0, 2}) \le C_{B, d} t \log \Big( 2 + \frac{M + V}{t} \Big) + C_{B, d} V^{\frac{1}{4}} t^{\frac{3}{4}} \Big[\log \Big(2 + \frac{V}{t} \Big)\Big]^{s - 1},
\end{equation*}
which leads to that
\begin{align*}
    &\E\bigg[\sup_{\substack{f \in \F_M(V) - \{f^*\} \\ \|f\|_{p_0, 2} \le t} } \Big| \frac{1}{\sqrt{k}} \sum_{i=1}^{k} \epsilon_i f(X_i) \Big| \bigg] \le C_{B, d} t \log \Big( 2 + \frac{M + V}{t} \Big) \\ 
    &\qquad \qquad \quad+ C_{B, d} V^{\frac{1}{4}} t^{\frac{3}{4}} \Big[\log \Big(2 + \frac{V}{t} \Big)\Big]^{s - 1} + C_{B, d}  M k^{-\frac{1}{2}} \Big[\log \Big( 2 + \frac{M + V}{t} \Big)\Big]^2 \\
    &\qquad \qquad \quad + C_{B, d}  M V^{\frac{1}{4}} t^{-\frac{1}{4}} k^{-\frac{1}{2}} \log \Big( 2 + \frac{M + V}{t} \Big) \Big[\log \Big(2 + \frac{V}{t} \Big)\Big]^{s - 1} \\
    &\qquad \qquad \quad + C_{B, d}  M V^{\frac{1}{2}} t^{-\frac{1}{2}} k^{-\frac{1}{2}} \Big[\log \Big(2 + \frac{V}{t} \Big)\Big]^{2(s - 1)} := \Psi_k(t)
\end{align*}
due to Lemma \ref{lem:maximal-inequality-bracketing-fvm}.

Note that $t \rightarrow \Psi_k(t) / t$ is monotonically decreasing.
Also, observe that for 
\begin{equation*}
    t \ge (1 + M^{\frac{1}{2}}) k^{-\frac{1}{2}},
\end{equation*}
we have
\begin{equation*}
    C_{B, d} t \log \Big( 2 + \frac{M + V}{t} \Big) \le \frac{1}{5} \sqrt{k} t^2 \mbox{\quad if \;} t \ge C_{B, d} k^{-\frac{1}{2}} \log\Big( 2 + \frac{(M + V) k^{\frac{1}{2}}}{1+ M^{\frac{1}{2}}}\Big), 
\end{equation*}
\begin{align*}
    &C_{B, d} M k^{-\frac{1}{2}} \Big[\log \Big( 2 + \frac{M + V}{t} \Big)\Big]^2 \le \frac{1}{5} \sqrt{k} t^2 \\ 
    &\qquad \qquad \qquad \qquad  \mbox{\quad if \;} t \ge C_{B, d} M^{\frac{1}{2}} k^{-\frac{1}{2}} \log\Big( 2 + \frac{(M + V) k^{\frac{1}{2}}}{1 + M^{\frac{1}{2}}}\Big),
\end{align*}
\begin{equation*}
    C_{B, d} V^{\frac{1}{4}} t^{\frac{3}{4}} \Big[\log \Big(2 + \frac{V}{t} \Big)\Big]^{s - 1} \le \frac{1}{5} \sqrt{k} t^2 \mbox{\quad if \;} t \ge C_{B, d} V^{\frac{1}{5}} k^{-\frac{2}{5}} \bigg[\log\Big( 2 + \frac{V k^{\frac{1}{2}}}{1 + M^{\frac{1}{2}}}\Big)\bigg]^{\frac{4(s - 1)}{5}},
\end{equation*}
\begin{align*}
    &C_{B, d} M V^{\frac{1}{4}} t^{-\frac{1}{4}} k^{-\frac{1}{2}} \log \Big( 2 + \frac{M + V}{t} \Big) \Big[\log \Big(2 + \frac{V}{t} \Big)\Big]^{s - 1} \le \frac{1}{5} \sqrt{k} t^2 \\
    &\qquad \mbox{\quad if \;} t \ge C_{B, d} M^{\frac{4}{9}} V^{\frac{1}{9}} k^{-\frac{4}{9}} \bigg[\log\Big( 2 + \frac{(M + V) k^{\frac{1}{2}}}{1 + M^{\frac{1}{2}}}\Big)\bigg]^{\frac{4}{9}} \bigg[\log\Big( 2 + \frac{V k^{\frac{1}{2}}}{1 + M^{\frac{1}{2}}}\Big)\bigg]^{\frac{4(s - 1)}{9}},
\end{align*}
and
\begin{align*}
    &C_{B, d} M V^{\frac{1}{2}} t^{-\frac{1}{2}} k^{-\frac{1}{2}} \Big[\log \Big(2 + \frac{V}{t} \Big)\Big]^{2(s - 1)} \le \frac{1}{5} \sqrt{k} t^2 \\ 
    &\qquad \qquad \qquad \qquad \mbox{\quad if \;} t \ge C_{B, d} M^{\frac{2}{5}} V^{\frac{1}{5}} k^{-\frac{2}{5}} \bigg[\log\Big( 2 + \frac{V k^{\frac{1}{2}}}{1 + M^{\frac{1}{2}}}\Big)\bigg]^{\frac{4(s - 1)}{5}}. 
\end{align*}
Clearly, there exist some positive constants $C_s$ and $C$ such that 
\begin{equation*}
    [\log(2 + x)]^{\frac{4s}{9}} \le x^{\frac{4}{45}} \qt{for $x \ge C_s$}
\end{equation*}
and
\begin{equation*}
    \log(2 + x) \le x^{\frac{1}{5}} \qt{for $x \ge C$}.
\end{equation*}
Thus, it follows that 
\begin{align*}
    &C_{B, d} (1 + M^{\frac{1}{2}}) k^{-\frac{1}{2}} \log\Big( 2 + \frac{(M + V) k^{\frac{1}{2}}}{1 + M^{\frac{1}{2}}}\Big) \\ 
    &\qquad \quad \le C_{B, d} (1 + M^{\frac{1}{2}}) k^{-\frac{1}{2}} + C_{B, d} (1 + M^{\frac{1}{2}}) k^{-\frac{1}{2}} \Big(\frac{(M + V) k^{\frac{1}{2}}}{1 + M^{\frac{1}{2}}}\Big)^{\frac{1}{5}} \\
    &\qquad \quad \le C_{B, d} (1 + M^{\frac{1}{2}}) k^{-\frac{1}{2}} + C_{B, d} (1 + M^{\frac{1}{2}})^{\frac{4}{5}} (M + V)^{\frac{1}{5}} k^{-\frac{2}{5}} \\
    &\qquad \quad \le C_{B, d} (1 + M^{\frac{1}{2}})^{\frac{4}{5}} \big[ (1 + M^{\frac{1}{2}})^{\frac{1}{5}} + (M + V)^{\frac{1}{5}} \big] k^{-\frac{2}{5}}
\end{align*}
and 
\begin{align*}
    &C_{B, d} M^{\frac{4}{9}} V^{\frac{1}{9}} k^{-\frac{4}{9}} \bigg[\log\Big( 2 + \frac{(M + V) k^{\frac{1}{2}}}{1 + M^{\frac{1}{2}}}\Big)\bigg]^{\frac{4}{9}} \bigg[\log\Big( 2 + \frac{V k^{\frac{1}{2}}}{1 + M^{\frac{1}{2}}}\Big)\bigg]^{\frac{4(s - 1)}{9}} \\
    &\qquad \quad \le C_{B, d} M^{\frac{4}{9}} V^{\frac{1}{9}} k^{-\frac{4}{9}} \bigg[\log\Big( 2 + \frac{(M + V) k^{\frac{1}{2}}}{1 + M^{\frac{1}{2}}}\Big)\bigg]^{\frac{4s}{9}} \\
    &\qquad \quad \le C_{B, d} M^{\frac{4}{9}} V^{\frac{1}{9}} k^{-\frac{4}{9}} + C_{B, d} M^{\frac{4}{9}} V^{\frac{1}{9}} k^{-\frac{4}{9}} \Big(\frac{(M + V) k^{\frac{1}{2}}}{1 + M^{\frac{1}{2}}}\Big)^{\frac{4}{45}} \\
    &\qquad \quad \le C_{B, d} M^{\frac{4}{9}} V^{\frac{1}{9}} \bigg[ 1 + \Big(\frac{M + V}{1 + M^{\frac{1}{2}}}\Big)^{\frac{4}{45}}\bigg] k^{-\frac{2}{5}}.
\end{align*}
For these reasons, for $k = 1, \dots, n$, if we let 
\begin{align*}
    &t_k = \bigg\{C_{B, d} (1 + M^{\frac{1}{2}})^{\frac{4}{5}} \big[ (1 + M^{\frac{1}{2}})^{\frac{1}{5}} + (M + V)^{\frac{1}{5}} \big] + C_{B, d} M^{\frac{4}{9}} V^{\frac{1}{9}} \bigg[ 1 + \Big(\frac{M + V}{1 + M^{\frac{1}{2}}}\Big)^{\frac{4}{45}}\bigg] \\
    &\qquad \qquad \qquad \qquad \qquad \qquad \quad + C_{B, d} (1 + M^{\frac{1}{2}})^{\frac{4}{5}} V^{\frac{1}{5}} \bigg[\log\Big( 2 + \frac{V n^{\frac{1}{2}}}{1 + M^{\frac{1}{2}}}\Big)\bigg]^{\frac{4(s - 1)}{5}} \bigg\} \cdot k^{-\frac{2}{5}}
\end{align*}
then
\begin{equation*}
    \Psi_k(t_k) \le \sqrt{k} t_k^2.
\end{equation*}
Hence, for $r \ge 1$, we have
\begin{equation*}
    \E\bigg[\sup_{\substack{f \in \F_M(V) - \{f^*\} \\ \|f\|_{p_0, 2} \le rt_k} } \Big| \sum_{i=1}^{k} \epsilon_i f(X^{(i)}) \Big| \bigg] \le \sqrt{k} \cdot \Psi_k(r t_k) \le r\sqrt{k} \cdot \Psi_k(t_k) \le r k t_k^2,
\end{equation*}
where the second inequality follows from that $t \rightarrow \Psi_k(t) / t$ is monotonically decreasing.

Recall that we assume $\| \xi_1 \|_{5, 1} < \infty$. 
Using Theorem \ref{thm:han-wellner}, we can thus show that 
\begin{equation*}
    \E\bigg[\sup_{\substack{f \in \F_M(V) - \{f^*\} \\ \|f\|_{p_0, 2} \le rt_n} } \Big| \sum_{i=1}^{n} \xi_i f(X^{(i)}) \Big| \bigg] \le 4  \|\xi_1\|_{5, 1} \cdot r n t_n^2 
\end{equation*}
for $r \ge 1$.
Therefore, if we redefine $t_n$ as $(1 + 4 \| \xi_1 \|_{5, 1}) t_n$, then
\begin{equation*}
    \E\bigg[\sup_{\substack{f \in \F_M(V) - \{f^*\} \\ \|f\|_{p_0, 2} \le rt_n} } \Big| \frac{1}{\sqrt{n}} \sum_{i=1}^{n} \epsilon_i f(X^{(i)}) \Big| \bigg] \le r \sqrt{n} t_n^2
\end{equation*}
and 
\begin{equation*}
    \E\bigg[\sup_{\substack{f \in \F_M(V) - \{f^*\} \\ \|f\|_{p_0, 2} \le rt_n} } \Big| \frac{1}{\sqrt{n}} \sum_{i=1}^{n} \xi_i f(X^{(i)}) \Big| \bigg] \le r \sqrt{n} t_n^2.
\end{equation*}
As a result, by Theorem \ref{thm:han-wellner-prop2-variant} (and Remark \ref{rmk:han-wellner-prop2-variant-well-specified}),
\begin{align*}
    \P \big(\|\hat{f}^{d, s}_{n, V} - f^* \|_{p_0, 2} > t\big) \le \epsilon &+ C \cdot \frac{1 + M^3}{t^3} \cdot t_n^3 \\
    &+ C \cdot \frac{\|\xi_1\|_2^3 + M^3}{n^{\frac{3}{2}} t^3} + C \cdot \frac{M^3(\|\xi_1\|_5^3 n^{\frac{3}{5}} + M^3)}{n^{3} t^6}
\end{align*}
provided that $n$ is sufficiently large and $t \ge t_n$. 

Observe that if $s = 1$, 
\begin{align*}
    t_n &\le (1 + 4 \| \xi_1 \|_{5, 1}) \cdot \bigg\{C_{B, d} (1 + M^{\frac{1}{2}})^{\frac{4}{5}} \big[ (1 + M^{\frac{1}{2}})^{\frac{1}{5}} + (M + V)^{\frac{1}{5}} \big] \\
    &\qquad \qquad \qquad \qquad \qquad \qquad \qquad \quad+ C_{B, d} M^{\frac{4}{9}} V^{\frac{1}{9}} \bigg[ 1 + \Big(\frac{M + V}{1 + M^{\frac{1}{2}}}\Big)^{\frac{4}{45}}\bigg] \bigg\} \cdot n^{-\frac{2}{5}} \\
    &\le (1 + 4 \| \xi_1 \|_{5, 1}) \cdot C_{B, d, V, M} \cdot n^{-\frac{2}{5}},
\end{align*} 
and otherwise ($s \ge 2$),
\begin{align*}
    t_n &= (1 + 4 \| \xi_1 \|_{5, 1}) \cdot C_{B, d} (1 + M^{\frac{1}{2}})^{\frac{4}{5}} V^{\frac{1}{5}} \bigg[\log\Big( 2 + \frac{V n^{\frac{1}{2}}}{1 + M^{\frac{1}{2}}}\Big)\bigg]^{\frac{4(s - 1)}{5}} \cdot n^{-\frac{2}{5}} + O(n^{-\frac{2}{5}}) \\
    &\le (1 + 4 \| \xi_1 \|_{5, 1}) \cdot C_{B, d, V, M} \bigg[\log\Big( 2 + \frac{V n^{\frac{1}{2}}}{1 + M^{\frac{1}{2}}}\Big)\bigg]^{\frac{4(s - 1)}{5}} 
    \cdot n^{-\frac{2}{5}} + O(n^{-\frac{2}{5}}),
\end{align*}
where the multiplicative constant underlying $O(\cdot)$ depends on $B, d, V, M$, and the moments of $\xi_i$.
Thus, if 
\begin{equation*}
    K \ge 2 (1 + 4 \| \xi_1 \|_{5, 1}) \cdot C_{B, d, V, M},
\end{equation*}
then  
\begin{equation*}
    K n^{-\frac{2}{5}} (\log n)^{\frac{4(s - 1)}{5}} \ge t_n
\end{equation*}
for sufficiently large $n$.
Hence, for such $K$, we have
\begin{align*}
    &\limsup_{n \rightarrow \infty} \P \big(\|\hat{f}^{d, s}_{n, V} - f^* \|_{p_0, 2} > K n^{-\frac{2}{5}} (\log n)^{\frac{4(s - 1)}{5}}\big) \\
    &\qquad \le \epsilon + \limsup_{n \rightarrow \infty} \bigg[C \cdot \frac{(1 + M^3) t_n^3}{K^3 n^{-\frac{6}{5}} (\log n)^{\frac{12(s - 1)}{5}}} + C \cdot \frac{\|\xi_1\|_2^3 + M^3}{K^3 n^{\frac{3}{10}} (\log n)^{\frac{12(s - 1)}{5}}} \\
    &\qquad \qquad \qquad \qquad \qquad \qquad \qquad \qquad \qquad \qquad \quad + C \cdot \frac{M^3(\|\xi_1\|_5^3 n^{\frac{3}{5}} + M^3)}{K^6 n^{\frac{3}{5}} (\log n)^{\frac{24(s - 1)}{5}}} \bigg] \\
    &\qquad \le \epsilon + (1 + 4 \| \xi_1 \|_{5, 1})^3 \cdot C_{B, d, V, M} \cdot \frac{1}{K^3} + \|\xi_1\|_5^3 \cdot C M^3 \cdot \frac{1}{K^6}.
\end{align*}
Consequently, we can find $K > 0$ for which
\begin{align*}
    \P \big(\|\hat{f}^{d, s}_{n, V} - f^* \|_{p_0, 2} > K n^{-\frac{2}{5}} (\log n)^{\frac{4(s - 1)}{5}}\big) < 2 \epsilon
\end{align*}
for sufficiently large $n$.
\end{proof}

\subsubsection{Proof of Theorem \ref{thm:d-bracket-entropy}}\label{pf:d-bracket-entropy}
We prove Theorem \ref{thm:d-bracket-entropy} by using \cite{gao2013bracketing}, Theorem 1.1
and ideas in the proof of 
Theorem 1.2 of the same paper. 
For a positive integer $m$ and $S > 0$, let $\H_m^{+}(S)$ be the collection of all the functions on $[0, 1]^m$ of the form 
\begin{equation*}
    (x_1, \dots, x_m) \mapsto \int \ind\{x_1 \ge t_1\} \cdots \ind\{x_m \ge
    t_m\} \, d\nu(t) = \nu([\zerovec, x])
\end{equation*}
where $\nu$ is a positive measure on $[0, 1]^m$ with $\nu([0, 1]^m) \le S$.
\cite{gao2013bracketing}, Theorem 1.1 (restated below) provides an upper bound of the bracketing entropy of $\H_m^{+}(S)$ under the $L^p$ norm for each $p \ge 1$.
Also, ideas in the proof of \cite{gao2013bracketing}, Theorem 1.2 enable us to obtain an upper bound of the bracketing entropy of $\D_m$ under the $L^2$ norm from those of $\H_m^{+}(S)$ under the $L^1$ and $L^2$ norms.

\begin{theorem}[\cite{gao2013bracketing}, Theorem 1.1]\label{thm:gao-bracketing-entropy}
Suppose that $p \ge 1$ and $m$ is a positive integer. Then, there exists a positive constant $C_{m, p}$ depending on $m$ and $p$ such that
    \begin{equation*}
        \log N_{[ \ ]}(\epsilon, \H_m^{+}(S), \|\cdot\|_p) \le C_{m, p} \cdot \frac{S}{\epsilon} \Big|\log \frac{S}{\epsilon} \Big|^{2(m - 1)}
    \end{equation*}
    for every $\epsilon > 0$.
\end{theorem}

\begin{proof}[Proof of Theorem \ref{thm:d-bracket-entropy}]
Assume that $0 < \epsilon < 2$. 
If $\epsilon \ge 2$, then 
\begin{equation*}
    \log N_{[ \ ]}(\epsilon, \D_m, \|\cdot\|_2) = 0
\end{equation*}
because 
\begin{equation*}
    -1 \le \int (x_1 - t_1)_+ \cdots (x_m - t_m)_+
    \, d\nu(t) \le 1
\end{equation*}
for every signed measure $\nu$ on $[0, 1]^m$ with variation $|\nu|([0, 1])^m \le 1$. 
Denote by $\D_m^+$ the subcollection of $\D_m$ consisting of all the functions on $[0, 1]^m$ of the form 
\begin{equation*}
    \int (x_1 - t_1)_+ \cdots (x_m - t_m)_+
    \, d\nu(t)
\end{equation*}
where $\nu$ is a positive measure on $[0, 1]^m$ with $\nu([0, 1]^m) \le 1$. 
Then, since 
\begin{equation*}
    \D_m \subseteq \D_m^+ - \D_m^+,
\end{equation*}
it follows that
\begin{equation}\label{bracket-entropy-dm+}
    \log N_{[ \ ]}(\epsilon, \D_m, \|\cdot\|_2) \le 2 \log N_{[ \ ]}\Big(\frac{\epsilon}{2}, \D_m^{+}, \|\cdot\|_2\Big).
\end{equation}
Recall that we denote by $\intoper_m: L^2([0, 1]^m) \rightarrow C([0, 1]^m)$ the $m$-dimensional integral operator defined by
\begin{equation*}
    \intoper_m f(x) = \int_{[\zerovec, x]} f(s) \, ds \mbox{\quad for } x \in [0, 1]^m. 
\end{equation*}
Since 
\begin{equation*}
    \int_{[0, 1]^m} (x_1 - t_1)_+ \cdots (x_m - t_m)_+
    \, d\nu(t) = \int_{[0, x]} \nu([\zerovec, s]) \, ds \mbox{\quad for } x \in [0, 1]^m
\end{equation*}
for every positive measure $\nu$ on $[0, 1]^m$ with $\nu([0, 1]^m) \le 1$, we have
\begin{equation*}
    \D_m^+ = \intoper_m (\H_m^+)
\end{equation*}
where $\H_m^+ := \H_m^+(1)$.

Now, fix $\delta > 0$, which will be specified later, and let $K = N(\delta, \H_m^+, \|\cdot\|_1)$.
Note that by Theorem \ref{thm:gao-bracketing-entropy},
\begin{equation}\label{upper-bound-of-net-size}
    \log K = \log N(\delta, \H_m^+, \|\cdot\|_1) \le \log N_{[ \ ]}(\delta, \H_m^+, \|\cdot\|_1) \le C_m \cdot \frac{1}{\delta} \Big|\log \frac{1}{\delta} \Big|^{2(m - 1)}.
\end{equation}
Also, let $h_1, \dots, h_K$ be a $\delta$-net of $H_m^+$ with respect to the $L^1$ norm, i.e.,
\begin{equation*}
    \H_m^+ \subseteq \bigcup_{k = 1}^{K} B_
    {\H_m^+}(\delta, h_k, \|\cdot\|_1)
\end{equation*}
where $B_{\H_m^+}(\delta, h_k, \|\cdot\|_1) := \{h \in \H_m^+: \|h - h_k\|_1 \le \delta \}$.
Next, for $k = 1, \dots, K$, let 
\begin{equation*}
    \I_k^+ = \Big\{\intoper_m((h - h_k)_+): h \in \H_m^+ \ \mbox{ and } \ \|h - h_k\|_1 \le \delta \Big\}
\end{equation*}
and 
\begin{equation*}
    \I_k^- = \Big\{\intoper_m((h_k - h)_+): h \in \H_m^+ \ \mbox{ and } \ \|h - h_k\|_1 \le \delta \Big\}.
\end{equation*}
Observe that $\I_k^+, \I_k^- \subseteq \H_m^+(\delta)$. 
Indeed, if $h \in \H_m^+$ and $\|h - h_k\|_1 \le \delta$, then for a positive measure $\nu$ on $[0, 1]^m$ defined by 
\begin{equation*}
    d\nu(t) = (h(t) - h_k(t))_+ \, dt, 
\end{equation*}
we have 
\begin{equation*}
    \nu([0, 1]^m) = \int_{[0, 1]^m} (h(t) - h_k(t))_+ \, dt \le \|h - h_k\|_1 \le \delta
\end{equation*}
and
\begin{equation*}
    \intoper_m((h - h_k)_+) (x) = \int_{[\zerovec, x]} (h(s) - h_k(s))_+ \, ds = \int_{[\zerovec, x]} \, d\nu(s) = \nu([\zerovec, x])
\end{equation*}
for all $x \in [0, 1]^m$.
Also, because $\D_m^+ = \intoper_m (\H_m^+)$, $ h_1, \dots, h_K$ form a $\delta$-net of $H_m^+$ with respect to the $L^1$ norm, and 
\begin{equation*}
    \intoper_m (h) = \intoper_m((h - h_k)_+ - (h_k - h)_+ + h_k) = \intoper_m((h - h_k)_+) - \intoper_m((h_k - h)_+) + \intoper_m(h_k)  
\end{equation*}
for $h \in \H_m^+$ and $k = 1, \dots, K$, it follows that
\begin{equation*}
    \D_m^+ \subseteq \bigcup_{k = 1}^{K} \Big(\I_k^+ - \I_k^- + \{\intoper_m (h_k)\}\Big).
\end{equation*}
Thus,
\begin{equation*}
    N_{[ \ ]} (\epsilon, \D_m^+ , \|\cdot\|_2) \le \sum_{k = 1}^{K} N_{[ \ ]} \Big(\frac{\epsilon}{2}, \I_k^+ , \|\cdot\|_2 \Big) \cdot N_{[ \ ]} \Big(\frac{\epsilon}{2}, \I_k^- , \|\cdot\|_2 \Big),
\end{equation*}
which together with \eqref{upper-bound-of-net-size} and Theorem \ref{thm:gao-bracketing-entropy} implies that
\begin{align*}
    \log N_{[ \ ]} (\epsilon, \D_m^+ , \|\cdot\|_2) &\le \log K + 2 \log N_{[ \ ]} \Big(\frac{\epsilon}{2}, \H_m^+ (\delta) , \|\cdot\|_2 \Big) \\
    &\le C_m \cdot \frac{1}{\delta} \Big|\log \frac{1}{\delta} \Big|^{2(m - 1)} + C_m \cdot \frac{2 \delta}{\epsilon} \Big|\log \frac{2 \delta}{\epsilon} \Big|^{2(m - 1)}.
\end{align*}
Therefore, with the choice of $\delta = (\epsilon/2)^{1/2}$, we can derive that
\begin{equation*}
    \log N_{[ \ ]} (\epsilon, \D_m^+ , \|\cdot\|_2) \le C_m \cdot \Big(\frac{2}{\epsilon}\Big)^{\frac{1}{2}} \Big|\log \frac{2}{\epsilon} \Big|^{2(m - 1)}.
\end{equation*}
By \eqref{bracket-entropy-dm+}, we can conclude that
\begin{equation*}
    \log N_{[ \ ]} (\epsilon, \D_m , \|\cdot\|_2) \le C_m \cdot \Big(\frac{4}{\epsilon}\Big)^{\frac{1}{2}} \Big|\log \frac{4}{\epsilon} \Big|^{2(m - 1)}.
\end{equation*}
\end{proof}

\subsubsection{Proof of Theorem \ref{thm:rate-of-convergence-random-approx}}\label{pf:rate-of-convergence-random-approx}

\begin{proof}[Proof of Theorem \ref{thm:rate-of-convergence-random-approx}]
Suppose we are given $\epsilon > 0$.
By repeating our argument in Lemma \ref{lem:sup-norm-bound}, we can show that 
there exists $M > 0$ independent of $N$ such that 
\begin{equation*}
    \P\big(\| \tilde{f}^{d, s}_{n, V} - f^* \|_{\infty} \le M\big) \ge 1 - \epsilon
\end{equation*}
for sufficiently large $n$.
Recall that we let $\F_{\text{approx}}(V)$ be the collection of all the functions $\fazeronu \in \infmars^{d, s}$ with $\Vmars(\fazeronu) \le V$ where $\nu_\alpha$ is concentrated on $(\prod_{k \in S(\alpha)} \tilde{\mathcal{U}}_k) \cap [0, 1)^{|\alpha|}$ for each $\alpha \in \{0, 1\}^d \setminus \{\zerovec\}$ with $|\alpha| \le s$. 
Due to Lemma \ref{lem:discrete-measures-approximation}, there exists $f_0 \in \F_{\text{approx}}(V)$ such that
\begin{equation*}
    \|f_0 - f^*\|_{\infty} \le \frac{2V}{N} \ \mbox{ and } \
    \|f_0 - f^*\|_{p_0, 2} \le C_d \Big(\frac{B}{N^3}\Big)^{\frac{1}{2}} V.
\end{equation*}
Thus, by the triangle inequality, we have
\begin{equation*}
    \P\Big(\| \tilde{f}^{d, s}_{n, V} - f_0 \|_{\infty} \le M + \frac{2V}{N}\Big) \ge 1 - \epsilon
\end{equation*}
for sufficiently large $n$.
For $V > 0$, let
\begin{equation*}
    \F_{M + 2V/N}(V) = \Big\{f \in \infmars^{d, s}: \Vmars(f) \le V \mbox{ and } \|f - f_0\|_{\infty} \le M + \frac{2V}{N} \Big\}
\end{equation*}
as in the proof of Theorem \ref{thm:rate-of-convergence-random}, 
and let 
\begin{equation*}
    \F_{\text{approx}, M + 2V/N}(V) = \Big\{f \in \F_{\text{approx}}(V): \|f - f_0 \|_{\infty} \le M + \frac{2V}{N}\Big\}.
\end{equation*}
Then, from the definition of $\tilde{f}^{d, s}_{n, V}$, we can see that
\begin{equation*}
    \tilde{f}^{d, s}_{n, V} \in \argmin_{f \in \F_{\text{approx}, M + 2V/N}(V)} \bigg\{\sum_{i = 1}^{n} \big(y_i - f(x^{(i)})\big)^2 \bigg\}. 
\end{equation*}
with probability at least $1 - \epsilon$.

Through the same argument as in the proof of Theorem \ref{thm:rate-of-convergence-random} with replacing $f^*$ with $f_0$ and $M$ with $\tilde{M} := M + 2V/N$, we can show that for
\begin{align*}
    &t_n = (1 + 4 \| \xi_1 \|_{5, 1}) \bigg\{C_{B, d} (1 + \tilde{M}^{\frac{1}{2}})^{\frac{4}{5}} \big[ (1 + \tilde{M}^{\frac{1}{2}})^{\frac{1}{5}} + (\tilde{M} + V)^{\frac{1}{5}} \big] \\
    &\qquad \qquad \qquad \qquad \qquad+ C_{B, d} \tilde{M}^{\frac{4}{9}} V^{\frac{1}{9}} \bigg[ 1 + \Big(\frac{\tilde{M} + V}{1 + \tilde{M}^{\frac{1}{2}}}\Big)^{\frac{4}{45}}\bigg] \\
    &\qquad \qquad \qquad \qquad \qquad + C_{B, d} (1 + \tilde{M}^{\frac{1}{2}})^{\frac{4}{5}} V^{\frac{1}{5}} \bigg[\log\Big( 2 + \frac{V n^{\frac{1}{2}}}{1 + \tilde{M}^{\frac{1}{2}}}\Big)\bigg]^{\frac{4(s - 1)}{5}} \bigg\} \cdot n^{-\frac{2}{5}},
\end{align*}
we have
\begin{align*}
    \E\bigg[\sup_{\substack{f \in \F_{\text{approx}, \tilde{M}}(V) - \{f_0\} \\ \|f\|_{p_0, 2} \le rt_n}} \Big| \frac{1}{\sqrt{n}} \sum_{i=1}^{n} \epsilon_i f(X^{(i)}) \Big| \bigg] &\le \E\bigg[\sup_{\substack{f \in \F_{\tilde{M}}(V) - \{f_0\} \\ \|f\|_{p_0, 2} \le rt_n} } \Big| \frac{1}{\sqrt{n}} \sum_{i=1}^{n} \epsilon_i f(X^{(i)}) \Big| \bigg] \\
    &\le r \sqrt{n} t_n^2
\end{align*}
and 
\begin{align*}
    \E\bigg[\sup_{\substack{f \in \F_{\text{approx}, \tilde{M}}(V) - \{f_0\} \\ \|f\|_{p_0, 2} \le rt_n} } \Big| \frac{1}{\sqrt{n}} \sum_{i=1}^{n} \xi_i f(X^{(i)}) \Big| \bigg] &\le \E\bigg[\sup_{\substack{f \in \F_{\tilde{M}}(V) - \{f_0\} \\ \|f\|_{p_0, 2} \le rt_n}} \Big| \frac{1}{\sqrt{n}} \sum_{i=1}^{n} \xi_i f(X^{(i)}) \Big| \bigg] \\
    &\le r \sqrt{n} t_n^2.
\end{align*}
Hence, to apply Theorem \ref{thm:han-wellner-prop2-variant}, it suffices to find an upper bound of 
\begin{equation*}
    \E\bigg[\sup_{\substack{f \in \F_{\text{approx}, \tilde{M}}(V) - \{f_0\} \\ \|f\|_{p_0, 2} \le t}} \Big| \frac{1}{\sqrt{n}} \sum_{i=1}^{n} \epsilon_i f(X^{(i)}) \cdot (f_0 - f^*)(X^{(i)})\Big| \bigg]
\end{equation*}
for $t > 0$.

First, by making minor modifications to the proof of Lemma \ref{lem:maximal-inequality-bracketing-fvm}, we can prove that
\begin{align*}
    &\E\bigg[\sup_{\substack{f \in \F_{\text{approx}, \tilde{M}}(V) - \{f_0\} \\ \|f\|_{p_0, 2} \le t} } \Big| \frac{1}{\sqrt{n}} \sum_{i=1}^{n} \epsilon_i f(X^{(i)}) \cdot (f_0 - f^*)(X^{(i)})\Big| \bigg] \\
    &\qquad \le \E\bigg[\sup_{\substack{f \in \F_{\tilde{M}}(V) - \{f_0\} \\ \|f\|_{p_0, 2} \le t} } \Big| \frac{1}{\sqrt{n}} \sum_{i=1}^{n} \epsilon_i f(X^{(i)}) \cdot (f_0 - f^*)(X^{(i)}) \Big| \bigg] \\
    &\qquad \le C \|f_0 - f^* \|_{\infty} \cdot J_{[ \ ]}\big(t, \F_{\tilde{M}}(V) - \{f_0\}, \|\cdot\|_{p_0, 2} \big) \\
    &\qquad \qquad \qquad \qquad \qquad \quad \cdot \bigg(1 + \tilde{M} \cdot \frac{J_{[ \ ]}\big(t, \F_{\tilde{M}}(V) - \{f_0\}, \|\cdot\|_{p_0, 2} \big)}{t^2 \sqrt{n}} \bigg) \\
    &\qquad \le C \cdot \frac{V}{N} \cdot J_{[ \ ]}\big(t, \F_{\tilde{M}}(V) - \{f_0\}, \|\cdot\|_{p_0, 2} \big) \\
    &\qquad \qquad \qquad \qquad \cdot \bigg(1 + \tilde{M} \cdot \frac{J_{[ \ ]}\big(t, \F_{\tilde{M}}(V) - \{f_0\}, \|\cdot\|_{p_0, 2} \big)}{t^2 \sqrt{n}} \bigg).
\end{align*}
Next, repeating the same computation as in the proof of Theorem \ref{thm:rate-of-convergence-random}, we can derive that 
\begin{align*}
    &\E\bigg[\sup_{\substack{f \in \F_{\text{approx}, \tilde{M}}(V) - \{f_0\} \\ \|f\|_{p_0, 2} \le t} } \Big| \frac{1}{\sqrt{n}} \sum_{i=1}^{n} \epsilon_i f(X^{(i)}) \cdot (f_0 - f^*)(X^{(i)})\Big| \bigg] \\
    &\qquad \quad \le \frac{V}{N }\bigg\{ C_{B, d} t \log \Big( 2 + \frac{\tilde{M} + V}{t} \Big) + C_{B, d} V^{\frac{1}{4}} t^{\frac{3}{4}} \Big[\log \Big(2 + \frac{V}{t} \Big)\Big]^{s - 1} \\ 
    &\qquad \qquad \qquad \qquad+ C_{B, d}  \tilde{M} n^{-\frac{1}{2}} \Big[\log \Big( 2 + \frac{\tilde{M} + V}{t} \Big)\Big]^2 \\
    &\qquad \qquad \qquad \qquad+ C_{B, d}  \tilde{M} V^{\frac{1}{4}} t^{-\frac{1}{4}} n^{-\frac{1}{2}} \log \Big( 2 + \frac{\tilde{M} + V}{t} \Big) \Big[\log \Big(2 + \frac{V}{t} \Big)\Big]^{s - 1} \\
    &\qquad \qquad \qquad \qquad+ C_{B, d}  \tilde{M} V^{\frac{1}{2}} t^{-\frac{1}{2}} n^{-\frac{1}{2}} \Big[\log \Big(2 + \frac{V}{t} \Big)\Big]^{2(s - 1)} \bigg\} := \tilde{\Psi}_n(t)
\end{align*}
for $t > 0$.
Note that for 
\begin{equation*}
    t \ge \Big(\frac{V}{N}\Big)^{\frac{1}{2}} \Big(\frac{V}{N} + \tilde{M} \Big)^{\frac{1}{2}} n^{-\frac{1}{2}},
\end{equation*}
we have
\begin{equation*}
    C_{B, d} \cdot \frac{V}{N} \cdot t \log \Big( 2 + \frac{\tilde{M} + V}{t} \Big) \le \frac{1}{5} \sqrt{n} t^2 \mbox{\quad if \;} t \ge C_{B, d} \cdot \frac{V}{N} \cdot n^{-\frac{1}{2}} \log\Big( 2 + \frac{N(\tilde{M} + V) n^{\frac{1}{2}}}{V^{\frac{1}{2}} (N\tilde{M} + V)^{\frac{1}{2}}} \Big), 
\end{equation*}
\begin{align*}
    &C_{B, d} \cdot \frac{V}{N} \cdot \tilde{M} n^{-\frac{1}{2}} \Big[\log \Big( 2 + \frac{\tilde{M} + V}{t} \Big)\Big]^2 \le \frac{1}{5} \sqrt{n} t^2 \\ 
    &\qquad \qquad \qquad \qquad \mbox{\quad if \;} t \ge C_{B, d} \Big(\frac{V}{N}\Big)^{\frac{1}{2}} \tilde{M}^{\frac{1}{2}} n^{-\frac{1}{2}} \log\Big( 2 + \frac{N(\tilde{M} + V) n^{\frac{1}{2}}}{V^{\frac{1}{2}} (N\tilde{M} + V)^{\frac{1}{2}}} \Big),
\end{align*}
\begin{align*}
    &C_{B, d} \cdot \frac{V}{N} \cdot V^{\frac{1}{4}} t^{\frac{3}{4}} \Big[\log \Big(2 + \frac{V}{t} \Big)\Big]^{s - 1} \le \frac{1}{5} \sqrt{n} t^2 \\
    &\qquad \qquad \qquad \qquad \mbox{\quad if \;} t \ge C_{B, d} \Big(\frac{V}{N}\Big)^{\frac{4}{5}} V^{\frac{1}{5}} n^{-\frac{2}{5}} \bigg[\log\Big( 2 + \frac{NV^{\frac{1}{2}} n^{\frac{1}{2}}}{(N\tilde{M} + V)^{\frac{1}{2}}} \Big)\bigg]^{\frac{4(s - 1)}{5}},
\end{align*}
\begin{align*}
    &C_{B, d} \cdot \frac{V}{N} \cdot \tilde{M} V^{\frac{1}{4}} t^{-\frac{1}{4}} n^{-\frac{1}{2}} \log \Big( 2 + \frac{\tilde{M} + V}{t} \Big) \Big[\log \Big(2 + \frac{V}{t} \Big)\Big]^{s - 1} \le \frac{1}{5} \sqrt{n} t^2 \\
    &\qquad \mbox{\quad if \;} t \ge C_{B, d} \Big(\frac{V}{N}\Big)^{\frac{4}{9}} \tilde{M}^{\frac{4}{9}} V^{\frac{1}{9}} n^{-\frac{4}{9}} \bigg[\log\Big( 2 + \frac{N(\tilde{M} + V) n^{\frac{1}{2}}}{V^{\frac{1}{2}} (N\tilde{M} + V)^{\frac{1}{2}}} \Big)\bigg]^{\frac{4}{9}} \\
    &\qquad \qquad \qquad \qquad \qquad \qquad \qquad \quad \qquad \qquad \cdot \bigg[\log\Big( 2 + \frac{NV^{\frac{1}{2}} n^{\frac{1}{2}}}{(N\tilde{M} + V)^{\frac{1}{2}}} \Big)\bigg]^{\frac{4(s - 1)}{9}},
\end{align*}
and
\begin{align*}
    &C_{B, d} \cdot \frac{V}{N} \cdot \tilde{M} V^{\frac{1}{2}} t^{-\frac{1}{2}} n^{-\frac{1}{2}} \Big[\log \Big(2 + \frac{V}{t} \Big)\Big]^{2(s - 1)} \le \frac{1}{5} \sqrt{n} t^2 \\ 
    &\qquad \qquad \qquad \qquad \mbox{\quad if \;} t \ge C_{B, d} \Big(\frac{V}{N}\Big)^{\frac{2}{5}} \tilde{M}^{\frac{2}{5}} V^{\frac{1}{5}} n^{-\frac{2}{5}} \bigg[\log\Big( 2 + \frac{NV^{\frac{1}{2}} n^{\frac{1}{2}}}{(N\tilde{M} + V)^{\frac{1}{2}}} \Big)\bigg]^{\frac{4(s - 1)}{5}}. 
\end{align*}
Thus, if we let 
\begin{align*}
    &\tilde{t}_n = \max\bigg\{C_{B, d} \cdot \Big(\frac{V}{N}\Big)^{\frac{1}{2}} \Big(\frac{V}{N} + \tilde{M} \Big)^{\frac{1}{2}} n^{-\frac{1}{2}} \log\Big( 2 + \frac{N(\tilde{M} + V) n^{\frac{1}{2}}}{V^{\frac{1}{2}} (N\tilde{M} + V)^{\frac{1}{2}}} \Big), \\
    &\quad C_{B, d} \Big(\frac{V}{N}\Big)^{\frac{4}{9}} \tilde{M}^{\frac{4}{9}} V^{\frac{1}{9}} n^{-\frac{4}{9}} \bigg[\log\Big( 2 + \frac{N(\tilde{M} + V) n^{\frac{1}{2}}}{V^{\frac{1}{2}} (N\tilde{M} + V)^{\frac{1}{2}}} \Big)\bigg]^{\frac{4}{9}} \bigg[\log\Big( 2 + \frac{NV^{\frac{1}{2}} n^{\frac{1}{2}}}{(N\tilde{M} + V)^{\frac{1}{2}}} \Big)\bigg]^{\frac{4(s - 1)}{9}} \\
    &\quad C_{B, d} \Big(\frac{V}{N}\Big)^{\frac{2}{5}} \Big(\frac{V}{N} + \tilde{M} \Big)^{\frac{2}{5}} V^{\frac{1}{5}} n^{-\frac{2}{5}} \bigg[\log\Big( 2 + \frac{NV^{\frac{1}{2}} n^{\frac{1}{2}}}{(N\tilde{M} + V)^{\frac{1}{2}}} \Big)\bigg]^{\frac{4(s - 1)}{5}} \bigg\},
\end{align*}
then 
\begin{equation*}
    \tilde{\Psi}_n(\tilde{t}_n) \le \sqrt{n} \tilde{t}_n^2.
\end{equation*}
This implies that
\begin{align*}
    &\E\bigg[\sup_{\substack{f \in \F_{\text{approx}, \tilde{M}}(V) - \{f_0\} \\ \|f\|_{p_0, 2} \le r\tilde{t}_n} } \Big| \frac{1}{\sqrt{n}} \sum_{i=1}^{n} \epsilon_i f(X^{(i)}) \cdot (f_0 - f^*)(X^{(i)})\Big| \bigg] \\
    &\qquad \quad \le \tilde{\Psi}_n(r\tilde{t}_n) \le r \tilde{\Psi}_n(\tilde{t}_n) \le r \sqrt{n} \tilde{t}_n^2
\end{align*}
for every $r \ge 1$.
Here the second inequality follows from the fact that $t \rightarrow \tilde{\Psi}_n(t) / t$ is monotonically decreasing.

Now, let 
\begin{equation*}
    \bar{t}_n = 4 \|f_0 - f^*\|_{p_0, 2} + t_n + \tilde{t}_n.
\end{equation*}
Then, it follows that
\begin{equation*}
    \E\bigg[\sup_{\substack{f \in \F_{\text{approx}, \tilde{M}}(V) - \{f_0\} \\ \|f\|_{p_0, 2} \le r\bar{t}_n} } \Big| \frac{1}{\sqrt{n}} \sum_{i=1}^{n} \epsilon_i f(X^{(i)}) \Big| \bigg] \le \Big(\frac{r\bar{t}_n}{t_n}\Big) \sqrt{n} t_n^2 = r\sqrt{n} t_n \bar{t}_n \le r \sqrt{n} \bar{t}_n^2
\end{equation*}
for every $r \ge 1$.
Similarly, we can show that for $r \ge 1$,
\begin{equation*}
    \E\bigg[\sup_{\substack{f \in \F_{\text{approx}, \tilde{M}}(V) - \{f_0\} \\ \|f\|_{p_0, 2} \le r\bar{t}_n} } \Big| \frac{1}{\sqrt{n}} \sum_{i=1}^{n} \xi_i f(X^{(i)}) \Big| \bigg] \le r \sqrt{n} \bar{t}_n^2.
\end{equation*}
and
\begin{equation*}
    \E\bigg[\sup_{\substack{f \in \F_{\text{approx}, \tilde{M}}(V) - \{f_0\} \\ \|f\|_{p_0, 2} \le r\bar{t}_n} } \Big| \frac{1}{\sqrt{n}} \sum_{i=1}^{n} \epsilon_i f(X^{(i)}) \cdot (f_0 - f^*)(X^{(i)})\Big| \bigg] \le r \sqrt{n} \bar{t}_n^2.
\end{equation*}
Therefore, by Theorem \ref{thm:han-wellner-prop2-variant} and the fact that 
\begin{equation*}
     \|f_0 - f^*\|_{\infty} \le \frac{2V}{N},
\end{equation*}
we have
\begin{align*}
    \P \big(\|\tilde{f}^{d, s}_{n, V} - f_0 \|_{p_0, 2} > t\big) \le \epsilon &+ C \cdot \frac{1 + \tilde{M}^3}{t^3} \cdot \bar{t}_n^3 + C \cdot \frac{\|\xi_1\|_2^3 + (2V/N)^3 + \tilde{M}^3}{n^{\frac{3}{2}} t^3} \\
    &+ C \cdot \frac{\tilde{M}^3(\|\xi_1\|_5^3 n^{\frac{3}{5}} + (2V/N)^3 + \tilde{M}^3)}{n^{3} t^6}
\end{align*}
provided that $n$ is sufficiently large and $t \ge \bar{t}_n$.
Using
\begin{equation*}
    \|f_0 - f^*\|_{p_0, 2} \le C_d \Big(\frac{B}{N^3}\Big)^{\frac{1}{2}} V,
\end{equation*}
we can derive from this result that
\begin{align*}
    &\P \big(\|\tilde{f}^{d, s}_{n, V} - f^* \|_{p_0, 2} > t\big) \le \P \bigg(\|\tilde{f}^{d, s}_{n, V} - f_0 \|_{p_0, 2} > t - C_d \Big(\frac{B}{N^3}\Big)^{\frac{1}{2}} V \bigg) \\
    &\qquad \quad \le \epsilon + C \cdot \frac{1 + (M + 2V/N)^3}{(t - C_d (B/N^3)^{\frac{1}{2}} V)^3} \cdot \bar{t}_n^3 + C \cdot \frac{\|\xi_1\|_2^3 + (M + 2V/N)^3}{n^{\frac{3}{2}} (t - C_d (B/N^3)^{\frac{1}{2}} V)^3} \\
    &\qquad \qquad \quad + C \cdot \frac{(M + 2V/N)^3(\|\xi_1\|_5^3 n^{\frac{3}{5}} + (M + 2V/N)^3)}{n^{3} (t - C_d (B/N^3)^{\frac{1}{2}} V)^6}
\end{align*}
provided that $n$ is sufficiently large and
\begin{equation*}
    t \ge C_d \Big(\frac{B}{N^3}\Big)^{\frac{1}{2}} V + \bar{t}_n.
\end{equation*}

Because $N = \Omega(n^{4/15})$,
\begin{equation*}
    C_d \Big(\frac{B}{N^3}\Big)^{\frac{1}{2}} V + \bar{t}_n \le (1 + 4 \| \xi_1 \|_{5, 1}) \cdot C_{B, d, V, M} \cdot n^{-\frac{2}{5}} + o(n^{-\frac{2}{5}}),
\end{equation*}
if $s = 1$, and otherwise ($s \ge 2$),
\begin{align*}
    C_d \Big(\frac{B}{N^3}\Big)^{\frac{1}{2}} V &+ \bar{t}_n \le (1 + 4 \| \xi_1 \|_{5, 1}) \cdot C_{B, d} (1 + M^{\frac{1}{2}})^{\frac{4}{5}} V^{\frac{1}{5}} \\
    &\qquad \qquad \quad \cdot \bigg[\log\Big( 2 + \frac{V n^{\frac{1}{2}}}{1 + M^{\frac{1}{2}}}\Big)\bigg]^{\frac{4(s - 1)}{5}} \cdot n^{-\frac{2}{5}} + O(n^{-\frac{2}{5}}) \\
    &\le (1 + 4 \| \xi_1 \|_{5, 1}) \cdot C_{B, d, V, M} \bigg[\log\Big( 2 + \frac{V n^{\frac{1}{2}}}{1 + M^{\frac{1}{2}}}\Big)\bigg]^{\frac{4(s - 1)}{5}} \cdot n^{-\frac{2}{5}} + O(n^{-\frac{2}{5}}),
\end{align*}
where the constants underlying $o(\cdot)$ and $O(\cdot)$ depend on $B,d,V,M$, and the moments of $\xi_i$.
Hence, if 

\begin{equation*}
    K \ge 2 (1 + 4 \| \xi_1 \|_{5, 1}) \cdot C_{B, d, V, M},
\end{equation*}
then 
\begin{equation*}
    K n^{-\frac{2}{5}} (\log n)^{\frac{4(s - 1)}{5}} \ge C_d \Big(\frac{B}{N^3}\Big)^{\frac{1}{2}} V + \bar{t}_n
\end{equation*}
for sufficiently large $n$.
For such $K$, again because $N = \Omega(n^{4/15})$, we have that
\begin{align*}
    &\limsup_{n \rightarrow \infty} \P \big(\|\tilde{f}^{d, s}_{n, V} - f^* \|_{p_0, 2} > K n^{-\frac{2}{5}} (\log n)^{\frac{4(s - 1)}{5}}\big) \\
    &\qquad \le \epsilon + \limsup_{n \rightarrow \infty} \bigg[ C \cdot \frac{1 + (M + 2V/N)^3}{(K n^{-\frac{2}{5}} (\log n)^{\frac{4(s - 1)}{5}} - C_d (B/N^3)^{\frac{1}{2}} V)^3} \cdot \bar{t}_n^3 \\
    &\qquad \qquad \qquad \qquad \qquad+ C \cdot \frac{\|\xi_1\|_2^3 + (M + 2V/N)^3}{n^{\frac{3}{2}} (K n^{-\frac{2}{5}} (\log n)^{\frac{4(s - 1)}{5}} - C_d (B/N^3)^{\frac{1}{2}} V)^3} \\
    &\qquad \qquad \qquad \qquad \qquad + C \cdot \frac{(M + 2V/N)^3(\|\xi_1\|_5^3 n^{\frac{3}{5}} + (M + 2V/N)^3)}{n^{3} (K n^{-\frac{2}{5}} (\log n)^{\frac{4(s - 1)}{5}} - C_d (B/N^3)^{\frac{1}{2}} V)^6}\bigg] \\
    &\qquad \le \epsilon + (1 + 4 \| \xi_1 \|_{5, 1})^3 \cdot C_{B, d, V, M} \cdot \frac{1}{K^3} + \|\xi_1\|_5^3 \cdot C M^3 \cdot \frac{1}{K^6}.
\end{align*}
As a result, we can find $K > 0$ that satisfies
\begin{align*}
    \P \big(\|\tilde{f}^{d, s}_{n, V} - f^* \|_{p_0, 2} > K n^{-\frac{2}{5}} (\log n)^{\frac{4(s - 1)}{5}}\big) < 2 \epsilon
\end{align*}
for sufficiently large $n$.
\end{proof}

\subsubsection{Proof of Theorem \ref{thm:lower-bound}}\label{pf:lower-bound}
Here we almost follow the proof of \cite{fang2021multivariate}, Theorem 4.6, which itself obtains the ideas from \cite{blei2007metric}, Section 4. 
As in \cite{fang2021multivariate}, we exploit Assouad's lemma in the following form.

\begin{lemma} [\cite{van2000asymptotic}, Lemma 24.3]\label{lem:assouad}
Let $q$ be a positive integer. 
Suppose for each $\eta \in \{-1, 1\}^q$, we have $f_{\eta} \in \infmars^{d, s}$ with $\Vmars(f_{\eta}) \le V$. 
Then, the minimax risk $\mathfrak{M}^{d, s}_{n, V}$ satisfies that 
\begin{equation*}
    \mathfrak{M}^{d, s}_{n, V} \ge \frac{q}{8} \cdot \min_{\eta \neq \eta'} \frac{\| f_\eta - f_{\eta'} \|_{p_0, 2}^2}{H(\eta, \eta')} \cdot \min_{H(\eta, \eta') = 1} \bigg(1 - 
    \sqrt{\frac{1}{2} \E \big[ K(\P_{f_\eta}, \P_{f_{\eta'}}) \big]}\bigg),
\end{equation*}
where $H$ denotes the Hamming distance $H(\eta, \eta') := \sum_{j = 1}^{q} 1\{\eta_j \neq \eta'_j\}$, $\P_f$ is the probability distribution of $(y_1, \dots, y_n)$ given $(X^{(1)}, \dots, X^{(n)})$ in \eqref{eq:model-random} when $f^{*} = f$, and $K(\cdot, \cdot)$ indicates the Kullback divergence between two probability distributions.
\end{lemma}

\begin{remark}
What \cite{van2000asymptotic}, Lemma 24.3 actually tells us in the setting of Lemma \ref{lem:assouad} is that
\begin{equation*}
    \max_{\eta} \E_{f_{\eta}} \| \hat{f}_n - f_{\eta} \|_{p_0, 2}^2 \ge \frac{q}{8} \cdot \min_{\eta \neq \eta'} \frac{\| f_\eta - f_{\eta'} \|_{p_0, 2}^2}{H(\eta, \eta')} \cdot \min_{H(\eta, \eta') = 1} \big\| Q_{f_\eta} \wedge Q_{f_\eta'} \big\|
\end{equation*}
for every estimator $\hat{f}_n$ based on $(y_1, X^{(1)}), \dots, (y_n, X^{(n)})$, where $Q_f$ is the joint distribution of $(y_i, X^{(i)})$ in \eqref{eq:model-random} when $f^{*} = f$ .
Here $\lVert Q_{f_{\eta}} \wedge Q_{f_{\eta'}} \rVert$ is defined by 
\[
    \big\lVert Q_{f_{\eta}} \wedge Q_{f_{\eta'}} \big\rVert = \int \min (q_{f_{\eta}}, q_{f_{\eta'}}) \, d\mu,
\]
where $\mu$ is every measure for which both $q_{f_{\eta}} := dQ_{f_{\eta}}/d\mu$ and $q_{f_{\eta'}} := dQ_{f_{\eta'}}/d\mu$ exist.
However, we can derive Lemma \ref{lem:assouad} directly from \cite{van2000asymptotic}, Lemma 24.3.
This is simply because we have
\begin{equation*}
    \mathfrak{M}^{d, s}_{n, V} = \inf_{\hat{f}_n} \sup_{\substack{f^* \in \infmars^{d, s} \\ \Vmars(f^*) \le V}} \E_{f^*} \| \hat{f}_n - f^{*} \|_{p_0, 2}^2 \ge \inf_{\hat{f}_n} \max_{\eta} \E_{f_{\eta}} \| \hat{f}_n - f_{\eta} \|_{p_0, 2}^2
\end{equation*}
and 
\begin{align*}
    \big\| Q_{f_{\eta}} \wedge Q_{f_{\eta'}} \big\| &= \int \min (q_{f_{\eta}}, q_{f_{\eta'}}) \, d\mu = 1 - \frac{1}{2}\int |q_{f_{\eta}} - q_{f_{\eta'}}| \, d\mu \\
    &\ge 1 - \sqrt{\frac{1}{2} K(Q_{f_\eta}, Q_{f_{\eta'}})} = 1 - \sqrt{\frac{1}{2} \E \big[ K(\P_{f_\eta}, \P_{f_{\eta'}}) \big]}.
\end{align*}
Here the inequality is due to the Pinsker--Csisz\'ar--Kullback inequality (see, e.g., \cite{wainwright2019high}, Lemma 15.2).
\end{remark}

\begin{proof} [Proof of Theorem \ref{thm:lower-bound}]
Let $l$ be a positive integer whose value will be specified momentarily. 
Consider the set 
\begin{equation*}
    P_l = \bigg\{(p_1, \dots, p_s) \in \mathbb{Z}_{\ge 0}^s: \sum_{j = 1}^{s} p_j = l\bigg\}
\end{equation*}
and for each $p \in P_l$, let 
\begin{equation*}
    I_p = \big\{(i_1, \dots, i_s): i_j \in [2^{p_j}] \mbox{ for each } j \in [s] \big\}.
\end{equation*}
Clearly, $|I_p| = 2^l$ for every $p \in P_l$. 
Also, we have  
\begin{equation*}
    |P_l| = \binom{s + l - 1}{s - 1} \ge \frac{l^{s- 1}}{(s - 1)!} := a_s l^{s - 1}.
\end{equation*}
Now, we choose $l$ as 
\begin{equation*}
    l = \bigg\lceil\frac{1}{5 \log 2}\Big\{\log\Big(\frac{C_{B, s} n V^2}{\sigma^2}\Big) - (s - 1) \log \log\Big(\frac{C_{B, s} n V^2}{\sigma^2}\Big)\Big\}\bigg\rceil,
\end{equation*}
where $C_{B, s} = B 2^{-12s + 1} (10 \log 2)^{s - 1} / a_s$. 
Here $\ceil{x}$ indicates the least integer greater than or equal to $x$.
The reason why we choose these particular $C_{B, s}$ and $l$ will become clear at the later stage of the proof. 
Note that 
\begin{equation} \label{2-minus-l}
    2^{-l} \le \Big(\frac{\sigma^2}{C_{B, s} n V^2}\Big)^{\frac{1}{5}}\bigg[\log\Big(\frac{C_{B, s} n V^2}{\sigma^2}\Big)\bigg]^{\frac{s - 1}{5}} < 2^{-l + 1}.
\end{equation} 

Let
\begin{equation*}
    Q = \big\{(p, i): p \in P_l \mbox{ and } i \in I_p \big\}
\end{equation*}
and let $q = |Q| = |P_l| \cdot 2^l$.
Throughout this proof, we will index the components of every $\eta \in \{-1, 1\}^q$ with the set $Q$.
For a positive integer $m$ and $k \in [2^m]$, we let $\phi_{m, k}$ be the real-valued function on $(0, 1)$ defined by
\begin{align*}
    \phi_{m, k}(x) = 
    \begin{cases}
        1 &\mbox{if } x \in \big((k-1)2^{-m}, (k-\frac{7}{8})2^{-m}\big) \cup \big((k-\frac{5}{8})2^{-m}, (k-\frac{1}{2})2^{-m}\big) \\
        &\qquad \qquad \cup \big((k-\frac{3}{8})2^{-m}, (k-\frac{1}{8})2^{-m}\big) \\
        -1 &\mbox{if } x \in \big((k-\frac{7}{8})2^{-m}, (k-\frac{5}{8})2^{-m}\big) \cup \big((k-\frac{1}{2})2^{-m}, (k-\frac{3}{8})2^{-m}\big) \\
        &\qquad \qquad \cup \big((k-\frac{1}{8})2^{-m}, k2^{-m}\big) \\
        0 &\mbox{otherwise}. 
    \end{cases}
\end{align*}
We now construct our function $f_{\eta}$ for each $\eta \in \{-1, 1\}^q$ based on these functions $\phi_{m, k}$. 
For each $\eta \in \{-1, 1\}^q$, we let $\nu_{\eta}$ be the signed measure on $(0, 1)^s$ defined by 
\begin{equation*}
    d\nu_{\eta}(t) = \frac{V}{\sqrt{|P_l|}} \sum_{p \in P_l} \sum_{i \in I_p} \eta_{p, i} \bigg(\prod_{j = 1}^s \phi_{p_j, i_j}(t_j)\bigg) dt,
\end{equation*}
and let $f_{\eta}$ denote the real-valued function on $[0, 1]^d$ defined by
\begin{equation*}
    f_{\eta}(x_1, \dots, x_d) = \int_{(0,1)^s} \prod_{j = 1}^{s} (x_j - t_j)_+ \, d\nu_{\eta}(t).
\end{equation*}
The following lemma, whose proof is given in Appendix \ref{pf:f-eta-zero-properties}, then presents the key properties of the functions $f_{\eta}$.

\begin{lemma}\label{lem:f-eta-zero-properties}
For each $\eta \in \{-1, 1\}^q$, $f_{\eta}$ satisfies
\begin{equation}\label{f-eta-zero-variation}
    \Vmars(f_\eta) \le V.
\end{equation}
Also, we have 
\begin{equation}\label{f-eta-zero-first-property}
    \max_{H(\eta, \eta') = 1} \| f_\eta - f_{\eta'} \|_{p_0, 2}^2 \le \frac{B V^2}{|P_l|} \cdot 2^{-5l - 12s + 2}.
\end{equation}
and 
\begin{equation}\label{f-eta-zero-second-property}
    \min_{\eta \neq \eta'} \frac{\| f_\eta - f_{\eta'}\|_{p_0, 2}^2}{H(\eta, \eta')} \ge \frac{b V^2}{|P_l|} \cdot 2^{-5l - 14s + 2}.
\end{equation}
\end{lemma}

Since the Kullback divergence between $\P_{f_\eta}$ and $\P_{f_{\eta'}}$ for $\eta, \eta' \in \{-1, 1\}^q$ can be computed by 
\begin{equation*}
    K(\P_{f_\eta}, \P_{f_{\eta'}}) = \frac{1}{2\sigma^2} \sum_{i=1}^{n} \big(f_{\eta}(X^{(i)}) - f_{\eta'}(X^{(i)})\big)^2,
\end{equation*}
we can obtain from \eqref{f-eta-zero-first-property} that
\begin{align}\label{f-eta-second-property}
    \max_{H(\eta, \eta') = 1} \E \big[ K(\P_{f_\eta}, \P_{f_{\eta'}}) \big] &= \frac{n}{2\sigma^2} \cdot  \max_{H(\eta, \eta') = 1} \| f_\eta - f_{\eta'} \|_{p_0, 2}^2 \le \frac{B n V^2}{\sigma^2|P_l|} \cdot 2^{-5l - 12s + 1}.
\end{align}

We are now ready to apply Lemma \ref{lem:assouad}.
By Lemma \ref{lem:assouad} together with \eqref{f-eta-zero-second-property} and \eqref{f-eta-second-property}, the minimax risk $\mathfrak{M}^{d, s}_{n, V}$ satisfies that   
\begin{align*}  
    \mathfrak{M}^{d, s}_{n, V} &\ge \frac{q}{8} \cdot \frac{b V^2}{|P_l|} \cdot 2^{-5l - 14s + 2} \bigg(1 - \sqrt{\frac{B n V^2}{\sigma^2|P_l|} \cdot 2^{-5l - 12s}}\bigg) \\
    &= b V^2 2^{-4l - 14s - 1} \bigg(1 - \sqrt{\frac{1}{2(10 \log 2)^{s - 1}} 
    \cdot \frac{C_{B, s} n V^2}{\sigma^2}\cdot \frac{2^{-5l}}{l^{s - 1}}}\bigg).
\end{align*}
Recall that we let $C_{B, s} = B 2^{-12s + 1} (10 \log 2)^{s - 1} / a_s$.
Our choice of $l$ then implies that
\begingroup
\allowdisplaybreaks
\begin{align*}
    \frac{1}{2(10 \log 2)^{s - 1}} \cdot \frac{C_{B, s} n V^2}{\sigma^2}\cdot \frac{2^{-5l}}{l^{s - 1}} &\le \frac{1}{2} \cdot \Bigg[\frac{\log\big(\frac{C_{B, s} n V^2}{\sigma^2}\big)}{2\Big(\log\big(\frac{C_{B, s} n V^2}{\sigma^2}\big) - (s - 1) \log \log \big(\frac{C_{B, s} n V^2}{\sigma^2}\big)\Big)}\Bigg]^{s - 1} \\
    &\le \frac{1}{2} \cdot \bigg[2\bigg(1 - (s-1) \frac{\log \log\big(\frac{C_{B, s} n V^2}{\sigma^2}\big)}{\log \big(\frac{C_{B, s} n V^2}{\sigma^2}\big)}\bigg)\bigg]^{-(s - 1)} \\
    &\le \frac{1}{2} \cdot \bigg[2\bigg(1 - (s - 1) \Big\{\log\Big(\frac{C_{B, s} n V^2}{\sigma^2}\Big)\Big\}^{-\frac{1}{2}} \bigg)\bigg]^{-(s - 1)}.
\end{align*}
\endgroup
Here we use 
\begin{equation*}
    l \ge \frac{1}{5 \log 2}\bigg\{\log\Big(\frac{C_{B, s} n V^2}{\sigma^2}\Big) - (s - 1) \log \log\Big(\frac{C_{B, s} n V^2}{\sigma^2}\Big)\bigg\}
\end{equation*}
and \eqref{2-minus-l} for the first inequality, and the last one is due to $\log \log x / \log x \le (\log x)^{-1/2}$, which holds for all $x > 1$. 
Thus, if 
\begin{equation*}
    n \ge \frac{e^{4s^2}}{C_{B, s}} \cdot \frac{\sigma^2}{V^2},
\end{equation*}
we have 
\begin{equation*}
    \frac{1}{2(10 \log 2)^{s - 1}} \cdot \frac{C_{B, s} n V^2}{\sigma^2}\cdot \frac{2^{-5l}}{l^{s - 1}} \le \frac{1}{2} \cdot \Big(1 + \frac{1}{s}\Big)^{-(s - 1)} \le \frac{1}{2},
\end{equation*}
and therefore,
\begin{align*}
    \mathfrak{M}^{d, s}_{n, V} &\ge \bigg(1 - \sqrt{\frac{1}{2}}\bigg) \cdot b V^2 2^{-4l - 14s - 1} \\
    &\ge \bigg(1 - \sqrt{\frac{1}{2}}\bigg) \cdot b V^2 2^{-14s - 5} \cdot \Big(\frac{\sigma^2}{C_{B, s} n V^2}\Big)^{\frac{4}{5}}\bigg[\log\Big(\frac{C_{B, s} n V^2}{\sigma^2}\Big)\bigg]^{\frac{4(s - 1)}{5}} \\
    &\ge C_{b, B, s}\Big(\frac{\sigma^2 V^{\frac{1}{2}}}{n }\Big)^{\frac{4}{5}}\bigg[\log\Big(\frac{C_{B, s} n V^2}{\sigma^2}\Big)\bigg]^{\frac{4(s - 1)}{5}},
\end{align*}
where $C_{b, B, s} = (1 - 1/\sqrt{2}) \cdot b  2^{-14s-5} C_{B, s}^{-4/5}$.
Here for the second inequality, we use \eqref{2-minus-l} again.
Now, note that 
\begin{equation*}
    \log\Big(\frac{C_{B, s} n V^2}{\sigma^2}\Big) \ge \frac{1}{2} \log\Big(\frac{n V^2}{\sigma^2}\Big)
\end{equation*}
if $n \ge (1/C_{B, s}^2) \cdot (\sigma^2/V^2)$.
Consequently, by further assuming that 
\begin{equation*}
    n \ge \max\Big\{\frac{e^{4s^2}}{C_{B, s}} \cdot \frac{\sigma^2}{V^2}, \frac{1}{C_{B, s}^2} \cdot \frac{\sigma^2}{V^2}\Big\},
\end{equation*}
we can derive the conclusion
\begin{equation*}
    \mathfrak{M}^{d, s}_{n, V} \ge C_{b, B, s}'\Big(\frac{\sigma^2 V^{\frac{1}{2}}}{n }\Big)^{\frac{4}{5}}\bigg[\log\Big(\frac{n V^2}{\sigma^2}\Big)\bigg]^{\frac{4(s - 1)}{5}},
\end{equation*}
where $C_{b, B, s}' = C_{b, B, s} \cdot 2^{-4(s-1)/5}$.
\end{proof}

\subsection{Proofs of Lemmas and Propositions in Section \ref{smoothness}}\label{pf:smoothness}
\subsubsection{Proof of Proposition \ref{prop:equivalent-formation}}\label{pf:equivalent-formation}
\begin{proof}[Proof of Proposition \ref{prop:equivalent-formation}]
Suppose $f \in \infmars^{d, s}$.
Then, there exist $a_{\zerovec} \in \R$ and a collection of finite signed measures $\{\nu_{\alpha}: \alpha \in \{0, 1\}^d \setminus \{\zerovec\} \mbox{ and } |\alpha| \le s\}$ as in \eqref{form-of-functions} such that $f = f_{a_{\zerovec}, \{\nu_{\alpha}\}}$.
For each $\alpha \in \{0, 1\}^{d} \setminus \{\zerovec\}$ with $|\alpha| \le s$, let $g_{\alpha}$ be the function on $[0, 1]^{|\alpha|}$ defined by 
\begin{equation*}
    g_{\alpha}(x^{(\alpha)}) = \nu_{\alpha}([\zerovec, x^{(\alpha)}] \cap [0, 1)^{|\alpha|})
\end{equation*}
for $x^{(\alpha)} \in [0, 1]^{|\alpha|}$.
Clearly, $g_{\alpha}$ is coordinate-wise right-continuous and coordinate-wise left-continuous at each point $x^{(\alpha)} = (x_j, j \in S(\alpha)) \in [0, 1]^{|\alpha|} \setminus [0, 1)^{|\alpha|}$ with respect to all the $j^{\text{th}}$ coordinates where $x_j = 1$.
Also, by Proposition \ref{prop:aistleitner-variant}, we can see that $g_{\alpha}$ has finite HK0 variation. 
Moreover, we have
\begingroup
\allowdisplaybreaks
\begin{align*}
    f(x_1, \dots, x_d) &= a_{\zerovec} + \sumoveralpha \int_{[0, 1)^{|\alpha|}} \prod_{j \in  S(\alpha)} (x_j - t_j)_{+} \, d\nu_{\alpha}(\talpha) \\ 
    &= a_{\zerovec} + \sumoveralpha \int_{\prod\limits_{j \in S(\alpha)} [0, x_j)} \prod_{j \in  S(\alpha)} (x_j - t_j) \, d\nu_{\alpha}(\talpha) \\
    &= a_{\zerovec} + \sumoveralpha \int_{\prod\limits_{j \in S(\alpha)} [0, x_j)} \int_{\prod\limits_{j \in S(\alpha)} [t_j, x_j)} 1 \, d\salpha \, d\nu_{\alpha}(\talpha) \\
    &= a_{\zerovec} + \sumoveralpha \int_{\prod\limits_{j \in S(\alpha)} [0, x_j)} \int_{\prod\limits_{j \in S(\alpha)} [0, s_j]} 1 \, d\nu_{\alpha}(\talpha) \, d\salpha \\
    &= a_{\zerovec} + \sumoveralpha \int_{\prod\limits_{j \in S(\alpha)} [0, x_j)} \nu_{\alpha}([\zerovec, \salpha]) \, d\salpha \\
    &= a_{\zerovec} + \sumoveralpha \int_{[\zerovec, \xalpha]} g_{\alpha}(\salpha) \, d\salpha
\end{align*}
\endgroup
for all $x = (x_1, \dots, x_d) \in [0, 1]^d$. 
Thus, $f$ can be represented as in \eqref{equivalent-formation} with $a_0 \in \R$ and the collection of functions $\{g_{\alpha}: \alpha \in \{0, 1\}^d \setminus \{\zerovec\} \mbox{ and } |\alpha| \le s\}$ that has the desired properties.

Conversely, suppose a real-valued function $f$ on $[0, 1]^d$ is of the form 
\begin{equation*}
    f(x_1, \dots, x_d) = a_{\zerovec} + \sumoveralpha \int_{[\zerovec, \xalpha]} g_{\alpha}(\talpha) \, d\talpha
\end{equation*}
for some $a_{\zerovec} \in \R$ and for some collection of functions $\{g_{\alpha}: \alpha \in \{0, 1\}^d \setminus \{\zerovec\} \mbox{ and } |\alpha| \le s\}$, where for each $\alpha \in \{0, 1\}^d \setminus \{\zerovec\}$ with $|\alpha| \le s$, $g_{\alpha}: [0, 1]^{|\alpha|} \rightarrow \R$ has finite HK0 variation, is coordinate-wise right-continuous, and is coordinate-wise left-continuous at each point $x^{(\alpha)} = (x_j, j \in S(\alpha)) \in [0, 1]^{|\alpha|} \setminus [0, 1)^{|\alpha|}$ with respect to all the $j^{\text{th}}$ coordinates where $x_j = 1$.
For each $\alpha \in \{0, 1\}^d \setminus \{\zerovec\}$ with $|\alpha| \le s$, by Proposition \ref{prop:aistleitner-variant} again, there exists a finite signed measure $\nu_{\alpha}$ on $[0, 1)^{|\alpha|}$ such that
\begin{equation*}
    g_{\alpha}(x^{(\alpha)}) = \nu_{\alpha}([\zerovec, x^{(\alpha)}] \cap [0, 1)^{|\alpha|}) \qt{for $x^{(\alpha)} \in [0, 1]^{|\alpha|}$}
\end{equation*}
and
\begin{equation*}
    |\nu_{\alpha}|([0, 1)^{|\alpha|} \setminus \{\zerovec\}) = \Vhk(g_{\alpha}).
\end{equation*}
Hence, the function $f$ can be written as 
\begin{align*}
    f(x_1, \dots, x_d) &= a_{\zerovec} + \sumoveralpha \int_{[\zerovec, \xalpha]} g_{\alpha}(\talpha) \, d\talpha \\
    &= a_{\zerovec} + \sumoveralpha \int_{\prod\limits_{j \in S(\alpha)} [0, x_j)} \nu_{\alpha}([\zerovec, \talpha]) \, d\talpha \\ 
    &= a_{\zerovec} + \sumoveralpha \int_{[0, 1)^{|\alpha|}} \prod_{j \in  S(\alpha)} (x_j - s_j)_{+} \, d\nu_{\alpha}(\salpha),
\end{align*} 
which means that $f \in \infmars^{d, s}$.
Moreover, we can represent the complexity measure of $f$ as 
\begin{equation*}
    \Vmars(f) = \sumoveralpha |\nu_{\alpha}|([0, 1)^{|\alpha|} \setminus \{\zerovec\}) = \sumoveralpha \Vhk(g_{\alpha}).
\end{equation*}
\end{proof}

\subsubsection{Proof of Proposition \ref{galpha-chracterization}}\label{pf:galpha-characterization}
Here we prove a stronger statement that directly implies Proposition \ref{galpha-chracterization}.
Here we show that for $f$ in the condition \eqref{equivalent-formation}, the derivatives of the interval functions associated with $f$ exist and equal to $g_{\alpha}$ in the same condition almost everywhere (with respect to the Lebesgue measure).
We start by introducing interval functions, their derivatives, and some related concepts. 
One can refer to \cite{lojasiewicz1988introduction} for more detailed introduction of them. 

Two closed real intervals are said to be adjoining if they have a common end point and do not share interior points. 
For example, $[0, 1]$ and $[1, 2]$ are two adjoining closed intervals.
Similarly, we say that two $m$-dimensional axis-aligned closed rectangles
\begin{equation*}
    [x, y] = \prod_{i = 1}^{m} [x_i, y_i] \ \mbox{ and } \ [z, w] = \prod_{i = 1}^{m} [z_i, w_i]
\end{equation*} 
are adjoining if $[x_k, y_k]$ and $[z_k, w_k]$ are adjoining for some $k \in [m]$ and $[x_i, y_i] = [z_i, w_i]$ for all $i \neq k$. 
For instance, $[0, 1] \times [0, 1]$ and $[0, 1] \times [1, 2]$ are two adjoining two-dimensional axis-aligned closed rectangles. 
Note that for two adjoining axis-aligned closed rectangles, their union is also an axis-aligned closed rectangle.
A real-valued function $F$ defined on a class of $m$-dimensional axis-aligned closed rectangles is then called an additive interval function if 
\begin{equation*}
    F(R \cup R') = F(R) + F(R')
\end{equation*}
for every pair of $m$-dimensional adjoining axis-aligned closed rectangles $R$ and $R'$. 

For an axis-aligned closed rectangle 
\begin{equation*}
    [x, y] = \prod_{i = 1}^{m} [x_i, y_i], 
\end{equation*} 
$|[x, y]|$ denotes the volume of the rectangle $|[x, y]| = (y_1 - x_1) \cdots (y_m - x_m)$.
Also, $[x, y]$ is called a cube if it is the product of closed intervals of equal length, i.e., $y_1 - x_1 = \cdots = y_m - x_m$.
The upper and the lower derivative of an additive interval function $F$ at $x \in \R^m$ are then defined by 
\begin{equation*}
    \overline{D}F(x) = \limsup_{x \in C, |C| \rightarrow 0} \frac{F(C)}{|C|}
    \ \text{ and } \ \underline{D}F(x) = \liminf_{x \in C, |C| \rightarrow 0} \frac{F(C)}{|C|}
\end{equation*}
where $\limsup$ and $\liminf$ are taken over $m$-dimensional cubes $C$ containing $x$.
If the two derivatives coincide, we say the derivative of $F$ at $x$ exists and define it as the common value
\begin{equation*}
    DF(x) = \lim_{x \in C, |C| \rightarrow 0} \frac{F(C)}{|C|}.
\end{equation*}

The Lebesgue differentiation theorem (see, e.g., \cite{Rudin:87book}, Theorem
7.10) can be stated in terms of interval functions and their derivatives as follows. 
The following theorem plays a key role in identifying $g_{\alpha}$ in the condition \eqref{equivalent-formation}.
\begin{theorem}[\cite{lojasiewicz1988introduction}, Theorem 7.1.7]\label{thm:Lebesgue-differentiation}
Let $f$ be an integrable real-valued function defined on an axis-aligned closed rectangle $R_0$. 
Suppose that $F$ is the additive interval function defined by 
\begin{equation*}
    F(R) = \int_{R} f(x) dx
\end{equation*}
for axis-aligned closed rectangles $R \subseteq R_0$. 
Then, it follows that $DF = f$ almost everywhere (with respect to the Lebesgue measure) on $R_0$.
\end{theorem}

We now define for a real-valued function $f$ on $[0, 1]^d$ the interval functions associated with $f$.
For each $\alpha \in \{0, 1\}^d \setminus \{\zerovec\}$, denote by $\Delta_{\alpha} f$ the additive interval function which is defined on the class of axis-aligned closed rectangles contained in $[0, 1]^{|\alpha|}$ and which is defined by 
\begin{equation*}
    \Delta_{\alpha} f([\xalpha, \yalpha]) = \sum_{\delta \in \prod\limits_{k \in S(\alpha)}{\{0, 1\}}} (-1)^{\sum\limits_{k \in S(\alpha)} \delta_k} f\big(\widetilde{xy}_{\delta}^{(\alpha)}\big)
\end{equation*}
where $\xalpha = (x_k, k \in S(\alpha))$, $\yalpha = (y_k, k \in S(\alpha))$, and $\widetilde{xy}_{\delta}^{(\alpha)}$ is the $d$-dimensional vector with 
\begin{equation*}
    \big(\widetilde{xy}_{\delta}^{(\alpha)}\big)_k = 
    \begin{cases}
        \delta_k x_k + (1 - \delta_k) y_k &\mbox{if } k \in S(\alpha)  \\
        0 &\mbox{otherwise}
    \end{cases}
\end{equation*}
for $k \in [d]$.
For example, if $d = 3$ and $\alpha = (1, 0, 1)$, $\Delta_{1, 0, 1} f$ is defined by
\begin{equation*}
    \Delta_{1, 0, 1} f([x_1, y_1] \times [x_3, y_3])
    = f(y_1, 0, y_3) - f(y_1, 0, x_3) - f(x_1, 0, y_3) + f(x_1, 0, x_3).
\end{equation*}
Also, if $\alpha = \onevec = (1, \dots, 1)$, $\Delta_{\onevec} f([x, y])$ coincides with the quasi-volume of $f$ on $[x, y]$ for all $x, y \in [0, 1]^d$.
We call these $\Delta_{\alpha} f$ the interval functions associated with $f$.

We are ready to state and prove the stronger statement describing the relationship between $f$ and $g_{\alpha}$ in the condition \eqref{equivalent-formation}.
It is clear from the definitions above that Proposition \ref{galpha-chracterization} directly follows from this result.

\begin{proposition}\label{galpha-chracterization-stronger}
Suppose that the condition \eqref{equivalent-formation} holds.
Then, for each $\alpha \in \{0, 1\}^d \setminus \{\zerovec\}$, $D \Delta_{\alpha} f = 0$ if $|\alpha| > s$, and $D\Delta_{\alpha} f = g_{\alpha}$ almost everywhere (with respect to the Lebesgue measure) on $[0, 1]^{|\alpha|}$ if $|\alpha| \le s$. 
\end{proposition}

\begin{proof}[Proof of Proposition \ref{galpha-chracterization-stronger}]
Recall that the condition \eqref{equivalent-formation} assumes that
\begin{equation*}
    f(x_1, \dots, x_d) = a_{\zerovec} + \sumoveralpha \int_{[\zerovec, \xalpha]} g_{\alpha}(\talpha) \, d\talpha
\end{equation*}
for $(x_1, \dots, x_d) \in [0, 1]^d$, where $a_0 \in \R$ and $g_{\alpha}$ is a real-valued function on $[0, 1]^{|\alpha|}$ for each $\alpha \in \{0, 1\}^d \setminus \{\zerovec\}$ with $|\alpha| \le s$.
By letting $g_{\alpha} = 0$ for each $\alpha \in \{0, 1\}^d$ with $|\alpha| > s$, we can rewrite it as 
\begin{equation*}
    f(x_1, \dots, x_d) = a_{\zerovec} + \sumoveralphafull \int_{[\zerovec, \xalpha]} g_{\alpha}(\talpha) \, d\talpha
\end{equation*}
for $(x_1, \dots, x_d) \in [0, 1]^d$.
Then, for each $\alpha \in \{0, 1\}^d \setminus \{\zerovec\}$, we have
\begin{align*}
    &\Delta_{\alpha} f([x^{(\alpha)}, y^{(\alpha)}]) = \sum_{\delta \in \prod\limits_{k \in S(\alpha)}{\{0, 1\}}} (-1)^{\sum\limits_{k \in S(\alpha)} \delta_k} f\big(\widetilde{xy}_{\delta}^{(\alpha)}\big) \\
    &\qquad \ = \sum_{\delta \in \prod\limits_{k \in S(\alpha)}{\{0, 1\}}} (-1)^{\sum\limits_{k \in S(\alpha)} \delta_k} \cdot \sum_{\substack{\alpha' \in \{0, 1\}^d \setminus \{\zerovec\} \\ \alpha' \le \alpha}} \int_{\prod\limits_{k \in S(\alpha')} [0, \delta_k x_k + (1 - \delta_k) y_k]} g_{\alpha'}(t^{(\alpha')}) \, dt^{(\alpha')} \\
    &\qquad \ = \sum_{\substack{\alpha' \in \{0, 1\}^d \setminus \{\zerovec\} \\ \alpha' \le \alpha}} \sum_{\delta \in \prod\limits_{k \in S(\alpha)}{\{0, 1\}}} (-1)^{\sum\limits_{k \in S(\alpha)} \delta_k} \cdot  \int_{\prod\limits_{k \in S(\alpha')} [0, \delta_k x_k + (1 - \delta_k) y_k]} g_{\alpha'}(t^{(\alpha')}) \, dt^{(\alpha')} \\
    &\qquad \ = \sum_{\delta \in \prod\limits_{k \in S(\alpha)}{\{0, 1\}}} (-1)^{\sum\limits_{k \in S(\alpha)} \delta_k} \cdot  \int_{\prod\limits_{k \in S(\alpha)} [0, \delta_k x_k + (1 - \delta_k) y_k]} g_{\alpha}(t^{(\alpha)}) \, dt^{(\alpha)} \\
    &\qquad \ = \int_{\prod\limits_{k \in S(\alpha)} [x_k, y_k]} g_{\alpha}(t^{(\alpha)}) \, dt^{(\alpha)} = \int_{[\xalpha, \yalpha]} g_{\alpha}(\talpha) \, d\talpha
\end{align*}
for $\xalpha, \yalpha \in [0, 1]^{|\alpha|}$.
Thus, by Theorem \ref{thm:Lebesgue-differentiation}, the derivative $D \Delta_{\alpha} f$ exists and equals to $g_{\alpha}$ almost everywhere (with respect to the Lebesgue measure) for each $\alpha \in \{0, 1\}^d \setminus \{\zerovec\}$.
Because we let $g_{\alpha} = 0$ for $\alpha \in \{0, 1\}^d$ with $|\alpha| > s$, this means that $D \Delta_{\alpha} f = 0$ if $|\alpha| > s$, and $D \Delta_{\alpha} f = g_{\alpha}$ almost everywhere if $|\alpha| \le s$. 
\end{proof}

\subsubsection{Proof of Lemma \ref{lem:smooth-functions}}\label{pf:smooth-functions}
We use the following lemma for proving Lemma \ref{lem:smooth-functions}. 
This lemma will be used again for proving Lemma \ref{lem:hardy-krause}.
We provide our proof of Lemma \ref{lem:smooth-functions-helper} right after the proof of Lemma \ref{lem:smooth-functions}. 

\begin{lemma}
\label{lem:smooth-functions-helper}
    If a real-valued function $g$ defined on $[0, 1]^{m}$ is smooth in the sense that $g^{(\alpha)}$ exists and is continuous on $[0, 1]^m$ for every $\alpha \in \{0, 1\}^m$, then $g$ satisfies 
    \begin{equation*}
        g(x_1, \dots, x_m) = g(\zerovec) + \sum_{\alpha \in \{0, 1\}^m \setminus \{\zerovec\}}  \int_{[\zerovec, \xalpha]} g^{(\alpha)} (\tilde{t}^{(\alpha)}) \, d\talpha
    \end{equation*}
    for all $x = (x_1, \dots, x_m) \in [0, 1]^m$. 
    Here $\tilde{t}^{(\alpha)}$ is an $m$-dimensional extension of $\talpha$ where
    \begin{align*}
    \tilde{t}^{(\alpha)}_k = 
    \begin{cases}
        t_k &\mbox{if } k \in S(\alpha) \\
        0 &\mbox{otherwise}.
    \end{cases}
    \end{align*}
\end{lemma}

\begin{proof}[Proof of Lemma \ref{lem:smooth-functions}]
Suppose $f$ is smooth in the sense that it satisfies the two conditions in the lemma and assume that $f^{(\alpha)} = 0$ for every $\alpha \in \{0, 1\}^d$ with $|\alpha| > s$.
For each $\alpha \in \{0, 1\}^d \setminus \{\zerovec\}$ with $|\alpha| \le s$, let $g_{\alpha}$ be the function on $[0, 1]^{|\alpha|}$ defined by 
\begin{equation*}
    g_{\alpha}(x^{(\alpha)}) = f^{(\alpha)}(\tilde{x}^{(\alpha)})
\end{equation*}
for $x^{(\alpha)} \in [0, 1]^{|\alpha|}$.  
Note that since $f^{(\alpha)}$ is continuous on $\bar{T}^{(\alpha)}$, $g_{\alpha}$ is continuous on $[0, 1]^{|\alpha|}$. 
Also, since $f^{(\beta)}$ exists and is continuous on $\bar{\Sbeta}^{(\alpha)}$ for every $\beta \in J_{\alpha}$, we can see that $g_{\alpha}^{(\delta)}$ exists and is continuous on $[0, 1]^{|\alpha|}$ for every $\delta \in \{0, 1\}^{|\alpha|}$.
Thus, by Lemma \ref{lem:hardy-krause}, it follows that $g_{\alpha}$ has finite HK0 variation and 
\begin{equation*}
    \Vhk(g_{\alpha}) = \sum_{\delta \in \{0, 1\}^{|\alpha|} \setminus \{\zerovec\}} \int_{\bar{\Sbeta}_{\alpha}^{(\delta)}} |g_{\alpha}^{(\delta)}|,
\end{equation*}
where 
\begin{align*}
    \bar{\Sbeta}_{\alpha}^{(\delta)} := \prod_{k \in S(\alpha)} \bar{\Sbeta}^{(\delta)}_k ~\text{ where }~ \bar{\Sbeta}^{(\delta)}_k = 
        \begin{cases}
            (0, 1) &\mbox{if } \delta_k = \max_j \delta_j \\
            \{0\} &\mbox{otherwise}.
        \end{cases}
\end{align*}
Moreover, by the assumption that $f^{(\alpha)}$ exists and is continuous on $[0, 1]^d$ for every $\alpha \in \{0, 1\}^d$, Lemma \ref{lem:smooth-functions-helper} implies that 
\begingroup
\allowdisplaybreaks
\begin{align*}
    f(x_1, \dots, x_d) &= f(\zerovec) + \sumoveralphafull  \int_{[\zerovec, \xalpha]} f^{(\alpha)} (\tilde{t}^{(\alpha)}) \, d\talpha \\
    &= f(\zerovec) + \sumoveralpha  \int_{[\zerovec, \xalpha]} g_{\alpha} (\talpha) \, d\talpha
\end{align*}
\endgroup
for all $x = (x_1, \dots, x_d) \in [0, 1]^d$.
For these reasons, by Proposition \ref{prop:equivalent-formation}, we can conclude that $f \in \infmars^{d, s}$ and that
\begingroup
\allowdisplaybreaks
\begin{align*}
    \Vmars(f) &= \sumoveralpha \Vhk (g_{\alpha}) = \sumoveralpha \sum_{\delta \in \{0, 1\}^{|\alpha|} \setminus \{\zerovec\}} \int_{\bar{\Sbeta}_{\alpha}^{(\delta)}} |g_{\alpha}^{(\delta)}| \\
    &= \sumoveralpha \sum_{\beta \in J_{\alpha}} \int_{\bar{\Sbeta}^{(\beta)}} |f^{(\beta)}|.
\end{align*}
\endgroup
If $s = d$, $\Vmars(f)$ can also be simplified as
\begin{equation*}
    \Vmars(f) = \sumoveralphafull \sum_{\beta \in J_{\alpha}} \int_{\bar{\Sbeta}^{(\beta)}} |f^{(\beta)}| = \sumoverbeta \int_{\bar{\Sbeta}^{(\beta)}} |f^{(\beta)}|.
\end{equation*}
\end{proof}

\begin{proof}[Proof of Lemma \ref{lem:smooth-functions-helper}]
Note that for $\alpha \in \{0, 1\}^m \setminus \{\zerovec\}$, 
\begin{equation*}
    \int_{[\zerovec, \xalpha]} g^{(\alpha)} (\tilde{t}^{(\alpha)}) \, d\talpha = \sum_{\delta \in \prod\limits_{k \in S(\alpha)} \{0, 1\}} (-1)^{\sum\limits_{k \in S(\alpha)} \delta_k} \cdot g\big(\tilde{x}_{\delta}^{(\alpha)}\big),
\end{equation*}
where $\tilde{x}_{\delta}^{(\alpha)}$ is the $m$-dimensional vector with 
\begin{equation*}
    \big(\tilde{x}_{\delta}^{(\alpha)}\big)_k = 
    \begin{cases}
        (1 - \delta_k) x_k &\mbox{if } k \in S(\alpha) \\
        0 &\mbox{otherwise} 
    \end{cases}
\end{equation*}
for $k \in [m]$.
It can be easily proved by the fundamental theorem of calculus, and we use the smoothness assumption on $g$ here.
For example, if $m = 2$ and $\alpha = (1, 1)$,
\begin{align*}
    &\int_{[0, x_1] \times [0, x_2]} g^{(1, 1)}(t_1, t_2) \, d(t_1, t_2) = \int_{0}^{x_1} \int_{0}^{x_2} g^{(1, 1)}(t_1, t_2) \, dt_2 \, dt_1 \\
    &\quad = \int_{0}^{x_1} \big(g^{(1, 0)} (t_1, x_2) - g^{(1, 0)} (t_1, 0)\big) \, dt_1 = \int_{0}^{x_1} g^{(1, 0)} (t_1, x_2) \, dt_1 - \int_{0}^{x_1} g^{(1, 0)} (t_1, 0) \, dt_1 \\
    &\quad= g(x_1, x_2) - g(0, x_2) - g(x_1, 0) + g(0, 0) = \sum_{\delta \in \{0, 1\}^2} (-1)^{\delta_1 + \delta_2} \cdot g\big(\tilde{x}_{\delta}^{(1, 1)}\big),
\end{align*}
and we use that $g, g^{(1,0)}, g^{(0,1)}$, and $g^{(1,1)}$ are continuous on $[0, 1]^2$ for this argument.
Thus, 
\begingroup
\allowdisplaybreaks
\begin{align*}
    &g(\zerovec) + \sum_{\alpha \in \{0, 1\}^m \setminus \{\zerovec\}} \int_{[\zerovec, \xalpha]} g^{(\alpha)} (\tilde{t}^{(\alpha)}) \, d\talpha \\
    &\quad= g(\zerovec) + \sum_{\alpha \in \{0, 1\}^m \setminus \{\zerovec\}} \sum_{\delta \in \prod\limits_{k \in S(\alpha)} \{0, 1\}} (-1)^{\sum\limits_{k \in S(\alpha)} \delta_k} \cdot g\big(\tilde{x}_{\delta}^{(\alpha)}\big) \\
    &\quad = \sum_{\alpha \in \{0, 1\}^m} \sum_{\delta \in \prod\limits_{k \in S(\alpha)} \{0, 1\}} (-1)^{\sum\limits_{k \in S(\alpha)} \delta_k} \cdot g\big(\tilde{x}_{\delta}^{(\alpha)}\big) = \sum_{\alpha' \in \{0, 1\}^m} g(\tilde{x}^{(\alpha')}) \cdot \sum_{\substack{\alpha \in \{0, 1\}^m \\ \alpha \ge \alpha'}} (-1)^{|\alpha| - |\alpha'|}.
\end{align*}
\endgroup
Here the last equality results from counting the number of $g(\tilde{x}^{(\alpha')})$ for each $\alpha' \in \{0, 1\}^m$.
Note that if $\alpha' \neq \onevec$,
\begin{equation*}
    \sum_{\substack{\alpha \in \{0, 1\}^m \\ \alpha \ge \alpha'}} (-1)^{|\alpha| - |\alpha'|} = \sum_{\delta \in \prod\limits_{k \in S_0(\alpha')} \{0, 1\}} (-1)^{\sum\limits_{k \in S_0(\alpha')} \delta_k} = (1 - 1)^{|S_0(\alpha')|} = 0,
\end{equation*}
where $S_0(\alpha') = \{j \in [m]: \alpha'_j = 0\}$, and that if $\alpha' = \onevec$,
\begin{equation*}
    \sum_{\substack{\alpha \in \{0, 1\}^m \\ \alpha \ge \alpha'}} (-1)^{|\alpha| - |\alpha'|} = 1.
\end{equation*}
As a result, it follows that 
\begin{equation*}
    g(\zerovec) + \sum_{\alpha \in \{0, 1\}^m \setminus \{\zerovec\}} \int_{[\zerovec, \xalpha]} g^{(\alpha)} (\tilde{t}^{(\alpha)}) \, d\talpha = g(\tilde{x}^{(\onevec)}) = g(x_1, \dots, x_m)
\end{equation*}
as desired.
\end{proof}

\subsection{Proofs of Lemmas and Propositions in Appendix \ref{hkreview}}\label{pf:hkreview}
\subsubsection{Proof of Proposition \ref{prop:aistleitner-variant}}\label{pf:aistleitner-variant}

Our proof of Proposition \ref{prop:aistleitner-variant} is based on \cite{aistleitner2015functions}, Theorem 3, which we restate as Theorem \ref{thm:aistleitner} below.
\cite{aistleitner2015functions}, Theorem 3 connects functions on $[0, 1]^m$ that has finite HK0 variation and is coordinate-wise right-continuous to the cumulative distribution functions of finite signed measures on $[0, 1]^m$.
We will use this theorem again for proving Lemma \ref{lem:hardy-krause}.

\begin{theorem}[\cite{aistleitner2015functions}, Theorem 3]\label{thm:aistleitner}
Suppose a real-valued function $g$ defined on $[0, 1]^m$ has finite HK0 variation and is coordinate-wise right-continuous. 
Then, there exists a unique finite signed measure $\nu$ on $[0, 1]^m$ such that
\begin{equation}\label{eq:aistleitner}
    g(x) = \nu([\zerovec, x]) \qt{for every $x \in [0, 1]^m$}.
\end{equation}
Also, if a real-valued function $g$ defined on $[0, 1]^m$ is of the form \eqref{eq:aistleitner} for some finite signed measure $\nu$ on $[0, 1]^m$, then $g$ has finite HK0 variation and 
\begin{equation*}
    |\nu|([0, 1]^m) = \Vhk(g) + |g(\zerovec)|.
\end{equation*}
\end{theorem}

\begin{proof}[Proof of Proposition \ref{prop:aistleitner-variant}]
Assume that a real-valued function $g$ defined on $[0, 1]^m$ has finite HK0 variation, is coordinate-wise right-continuous, and is coordinate-wise left-continuous at each point $x \in [0, 1]^{m} \setminus [0, 1)^{m}$ with respect to all the $j^{\text{th}}$ coordinates where $x_j = 1$. 
Since $g$ has finite HK0 variation and is coordinate-wise right-continuous, by Theorem \ref{thm:aistleitner}, there exists a finite signed measure $\bar{\nu}$ on $[0, 1]^m$ such that 
\begin{equation*}
    g(x) = \bar{\nu}([\zerovec, x])
\end{equation*}
for all $x \in [0, 1]^m$.
Let $\mu$ be the finite signed measure on $[0, 1]^m$ for which 
\begin{equation*}
    \mu(E) = \bar{\nu}(E \cap [0, 1)^m).
\end{equation*}
Then, it follows that 
\begin{equation*}
    g(x) = \mu([\zerovec, x])
\end{equation*}
for all $x \in [0, 1]^m$.
Indeed, if $x \in [0, 1)^m$, then 
\begin{equation*}
    g(x) = \bar{\nu}([\zerovec, x]) = \bar{\nu}([\zerovec, x] \cap [0, 1)^m) = \mu([\zerovec, x]).
\end{equation*}
Otherwise, without loss of generality, let $x = (1, \dots, 1, x_{k+1}, \dots, x_{m})$ where $x_{k+1}, \dots, x_{m}$ $< 1$ for some $k \in [m]$.
Then, by the left-continuity of $g$, 
\begin{align*}
    g(x) &= \lim_{y_1 \uparrow 1} \cdots \lim_{y_k \uparrow 1} g(y_1, \dots, y_k, x_{k+1}, \dots, x_m) \\
    &= \lim_{y_1 \uparrow 1} \cdots \lim_{y_k \uparrow 1} \bar{\nu}([\zerovec, (y_1, \dots, y_k, x_{k+1}, \dots, x_m)]) \\
    &= \bar{\nu}([0, 1)^k \times [0, x_{k+1}] \times \dots \times [0, x_m]) \\
    &= \bar{\nu}([\zerovec, x] \cap [0, 1)^m) = \mu([\zerovec, x]),
\end{align*}
as well. 
We thus have the two signed measure $\bar{\nu}$ and $\mu$ satisfying \eqref{eq:aistleitner}, and the uniqueness in Theorem \ref{thm:aistleitner} implies that $\bar{\nu} = \mu$, i.e., 
\begin{equation*}
    \bar{\nu}(E) = \mu(E) = \bar{\nu}(E \cap [0, 1)^m).
\end{equation*}
Therefore, if we let $\nu$ be the finite signed measure on $[0, 1)^m$ for which 
\begin{equation*}
    \nu(E) = \bar{\nu}(E),
\end{equation*}
then for every $x \in [0, 1]^m$, 
\begin{equation*}
    g(x) = \bar{\nu}([\zerovec, x]) = \bar{\nu}([\zerovec, x] \cap [0, 1)^m) = \nu([\zerovec, x] \cap [0, 1)^m)
\end{equation*}
as desired.
The uniqueness of such signed measure directly follows from the uniqueness of the signed measure in Theorem \ref{thm:aistleitner}.

Next, suppose a real-valued function $g$ on $[0, 1]^m$ is of the form \eqref{eq:aistleitner-variant} for some signed measure $\nu$ on $[0, 1)^m$.
Let $\bar{\nu}$ be the finite signed measure on $[0, 1]^m$ for which
\begin{equation*}
    \bar{\nu}(E) = \nu(E \cap [0, 1)^m).
\end{equation*}
Then, clearly, 
\begin{equation*}
    g(x) = \nu([\zerovec, x] \cap [0, 1)^m) = \bar{\nu}([\zerovec, x])
\end{equation*}
for every $x \in [0, 1]^m$.
By Theorem \ref{thm:aistleitner}, $g$ thus has finite HK0 variation and 
\begin{equation*}
    |\bar{\nu}|([0, 1]^m) = \Vhk(g) + |g(\zerovec)|.
\end{equation*}
Also, since $|g(\zerovec)| = |\nu(\{0\})| = |\nu|(\{0\})$ and $|\bar{\nu}|([0, 1]^m) = |\nu|([0, 1)^m)$, it follows that 
\begin{equation*}
    \Vhk(g) = |\nu|([0, 1)^m \setminus \{\zerovec\}).
\end{equation*}
\end{proof}

\subsubsection{Proof of Lemma \ref{lem:hardy-krause}}\label{pf:hardy-krause}
\begin{proof}[Proof of Lemma \ref{lem:hardy-krause}]
First, note that since $g$ is smooth in the sense of Lemma \ref{lem:smooth-functions-helper}, we have 
\begin{equation}\label{eq:smooth-functions-helper-half-closed}
\begin{split}
        g(x_1, \dots, x_m) &= g(\zerovec) + \sum_{\alpha \in \{0, 1\}^m \setminus \{\zerovec\}} \int_{[\zerovec, \xalpha]} g^{(\alpha)} (\tilde{t}^{(\alpha)}) \, d\talpha
\end{split}
\end{equation}
for all $x = (x_1, \dots, x_m) \in [0, 1]^m$.
For each $\alpha \in \{0, 1\}^m \setminus \{\zerovec\}$, let $\mu_{\alpha}$ be the signed measure on $(0, 1]^{|\alpha|}$ defined by 
\begin{equation*}
    d\mu_{\alpha} = g^{(\alpha)} (\tilde{t}^{(\alpha)}) \, d\talpha.
\end{equation*}
Also, for each $\alpha \in \{0, 1\}^m \setminus \{\zerovec\}$, consider the projection $\pi_{\alpha}: \R^m \rightarrow \R^{|\alpha|}$ for which  
\begin{equation*}
    \pi_{\alpha}\big((t_k, k \in [m])\big) = (t_k, k \in S(\alpha)) 
\end{equation*}
and let $R^{(\alpha)} = \prod_{k = 1}^{m} R^{(\alpha)}_k$ where 
\begin{align*}
    R^{(\alpha)}_k = 
    \begin{cases}
        (0, 1] &\mbox{if } \alpha_k = \max_j \alpha_j \\
        \{0\} &\mbox{otherwise}.
    \end{cases}
\end{align*}
Next, we patch together these signed measures and form the signed measure $\mu$ on $[0, 1]^m$ by 
\begin{equation*}
    \mu(E) = g(\zerovec) \cdot 1\{\zerovec \in E\} + \sum_{\alpha \in \{0, 1\}^m \setminus \{\zerovec\}} \mu_{\alpha}(\pi_{\alpha}(E \cap R^{(\alpha)})).
\end{equation*}
Then, we can observe that 
\begin{align}\label{eq:total-variation-of-nu-in-g-alpha}
\begin{split}
    |\mu|([0, 1]^m) &= |g(\zerovec)| + \sum_{\alpha \in \{0, 1\}^m \setminus \{\zerovec\}} |\mu_{\alpha}|((0, 1]^{|\alpha|}) \\
    &= |g(\zerovec)| + \sum_{\alpha \in \{0, 1\}^m \setminus \{\zerovec\}} \int_{(0, 1]^{|\alpha|}} |g^{(\alpha)} (\tilde{t}^{(\alpha)})| \, d\talpha \\
    &= |g(\zerovec)| + \sum_{\alpha \in \{0, 1\}^m \setminus \{\zerovec\}} \int_{\bar{\Sbeta}^{(\alpha)}} |g^{(\alpha)}|.
\end{split}
\end{align}
By \eqref{eq:smooth-functions-helper-half-closed}, we also have 
\begingroup
\allowdisplaybreaks
\begin{align*}
    &g(x_1, \dots, x_m) = g(\zerovec) + \sum_{\alpha \in \{0, 1\}^m \setminus \{\zerovec\}}  \int_{\prod\limits_{k \in S(\alpha)} (0, x_k]} g^{(\alpha)} (\tilde{t}^{(\alpha)}) \, d\talpha \\
    &\qquad= g(\zerovec) + \sum_{\alpha \in \{0, 1\}^m \setminus \{\zerovec\}} \int_{\prod\limits_{k \in S(\alpha)} (0, x_k]} \, d\mu_{\alpha} = g(\zerovec) + \sum_{\alpha \in \{0, 1\}^m \setminus \{\zerovec\}} \mu_{\alpha} \Big(\prod\limits_{k \in S(\alpha)} (0, x_k] \Big) \\
    &\qquad= g(\zerovec) \cdot 1\{\zerovec \in [\zerovec, x]\} + \sum_{\alpha \in \{0, 1\}^m \setminus \{\zerovec\}} \mu_{\alpha} \Big(\pi_{\alpha}\big([\zerovec, x] \cap R^{(\alpha)}\big)\Big) = \mu([\zerovec, x])
\end{align*}
\endgroup
for all $x = (x_1, \cdots, x_m) \in [0, 1]^m$.
Therefore, Theorem \ref{thm:aistleitner} implies that $g$ has finite HK0 variation and 
\begin{equation*}
    |\mu|([0, 1]^m) = \Vhk(g) + |g(\zerovec)|.
\end{equation*}
Combining this with the equation \eqref{eq:total-variation-of-nu-in-g-alpha}, we can derive that 
\begin{equation*}
    \Vhk(g) = \sum_{\alpha \in \{0, 1\}^m \setminus \{\zerovec\}} \int_{\bar{\Sbeta}^{(\alpha)}} |g^{(\alpha)}|.
\end{equation*}
\end{proof}

\subsection{Proof of Proposition \ref{prop:reduction-to-discrete-analogue}}\label{pf:reduction-to-discrete-analogue}
Before proving Proposition \ref{prop:reduction-to-discrete-analogue}, we extend the definition of the operator $H^{(2)}$ to general orders.
This helps understand connection between $\theta$ and $H^{(2)} \theta$ for vectors $\theta$ indexed with the set $I_0 = [0: (n_1 - 1)] \times \dots \times [0: (n_d - 1)]$.

Suppose that $\theta$ is an $n$-dimensional vector indexed with the set $I_0$.
We let $H^{(0)} \theta = \theta$ and for each positive integer $r$, let $H^{(r)} \theta$ be the $n$-dimensional vector indexed by $i \in I_0$, where 
\begin{align*}
    (H^{(r)} \theta)_i = 
    \begin{cases}
        (D^{(\alpha)} \theta)_i  &\mbox{if } i \in I_0^{(\alpha)} \mbox{ for some } \alpha \in [0: r]^d \mbox{ with } \max_j \alpha_j = r \\
        (D^{(i)} \theta)_i &\mbox{if } i \in [0: (r - 1)]^d.
    \end{cases}
\end{align*}
Recall that
\begin{equation*}
    \biguplus_{\substack{\alpha \in [0: r]^d \\ \max_j \alpha_j = r}} I_0^{(\alpha)} = I_0 \setminus [0: (r - 1)]^d
\end{equation*}
and note that the definition agrees with the former one when $r = 2$.
Next, let $A^{(0)} = L_{I_0}$ and for each positive integer $r$, let $A^{(r)}$ be the block matrix that is indexed with the set $I_0$ and consists of the identity block matrix indexed with $[0: (r-1)]^d$ and all the blocks $L_{I_0^{(\alpha)}}$ for $\alpha \in [0: r]^d$ with $\max_j \alpha_j = r$.
For an index set $I$, we use the notation $L_I$ to indicate the $|I| \times |I|$ matrix whose rows and columns are indexed by $i, j \in I$, where 
\begin{equation*}
    (L_I)_{ij} = \ind\{i \ge j\} := \prod_{k} \ind\{i_k \ge j_k\}
\end{equation*}
for $i, j \in I$.
Then, one can verify that for each $r \ge 0$,
\begin{equation*}
    (H^{(r)} \theta)_i = \sum_{j \in I_0^{(\alpha)}} \ind \{i \ge j\} (H^{(r + 1)} \theta)_j
\end{equation*}
for every $i \in I_0^{(\alpha)}$ for $\alpha \in [0: r]^d$ with $\max_j \alpha_j = r$. 
We thus have
\begin{equation*}
    H^{(r)} \theta = A^{(r)} (H^{(r + 1)} \theta)
\end{equation*}
for all $r \ge 0$, and this leads to $\theta = A^{(0)} A^{(1)} \cdots A^{(r - 1)} H^{(r)} \theta$ for all $r \ge 0$. 
Especially, when $r = 2$, we have 
\begin{equation}\label{eq:theta-and-H2theta}
    \theta = A^{(0)} A^{(1)} (H^{(2)} \theta ),
\end{equation}
which can be explicitly written as
\begin{equation}\label{eq:representation-of-theta}
    \theta_i = (H^{(2)} \theta)_{\zerovec} + \sumoveralphafull \sum_{l \in I_0^{(\alpha)}}\bigg[\bigg(\prod_{k = 1}^d \binom{i_k - l_k + \alpha_k}{\alpha_k}\bigg) \cdot \ind\{i \ge l\} \cdot (H^{(2)} \theta)_l\bigg] 
\end{equation}
for $i \in I_0$.
Moreover, here we additionally note that $V_2(\theta)$ can be represented in $H^{(2)} \theta$ as 
\begingroup
\allowdisplaybreaks
\begin{align}\label{discrete-variation-in-H-theta}
    V_2(\theta) &= \sumoverbeta \bigg[\bigg(\prod_{k = 1}^{d} n_k^{\beta_k - \ind\{\beta_k = 2\}}\bigg) \cdot  \sum_{i \in I_0^{(\beta)}} \big|(D^{(\beta)} \theta)_i\big|\bigg] \nonumber \\ 
    &= \sumoveralphafull \sum_{\beta \in J_{\alpha}} \bigg[\bigg(\prod_{k = 1}^{d} n_k^{\beta_k - \ind\{\beta_k = 2\}}\bigg) \cdot \sum_{i \in I_0^{(\beta)}} \big|(D^{(\beta)} \theta)_i\big|\bigg] \\ 
    &= \sumoveralphafull \bigg[\bigg(\prod_{k = 1}^{d} n_k^{\alpha_k}\bigg) \cdot \sum_{\beta \in J_{\alpha}} \sum_{i \in I_0^{(\beta)}} \big|(D^{(\beta)} \theta)_i\big|\bigg] \nonumber \\
    &= \sumoveralphafull \bigg[\bigg(\prod_{k = 1}^{d} n_k^{\alpha_k}\bigg) \cdot \sum_{i \in I_0^{(\alpha)} \setminus \{\alpha\}} \big|(H^{(2)} \theta)_i\big|\bigg]. \nonumber
\end{align}
\endgroup

\begin{proof}[Proof of Proposition \ref{prop:reduction-to-discrete-analogue}]
We first re-index the rows of $M$ (defined in Proposition \ref{prop:reduction-to-lasso}) with the set $I_0 = [0: (n_1 - 1)] \times \dots \times [0: (n_d - 1)]$, i.e., we let
\begin{equation*}
    M_{i, (\alpha, l)} = \prod_{k \in S(\alpha)} \Big(\frac{i_k}{n_k} - \frac{l_k}{n_k}\Big)_{+}
\end{equation*}
for $i \in I_0$ and $(\alpha, l) \in J$.
We can then observe that
\begin{equation*}
    M_{i, (\alpha, l)} = 0    
\end{equation*}
if $l_k = n_k - 1$ for some $k \in [d]$. 
This suggests that we can remove such all-zero columns from $M$ and the corresponding components from $\gamma$ in the problem \eqref{finite-dimensional-lasso}.
Consider the submatrix $\tilde{M}$ of $M$ consisting of the components that have column indices in $\tilde{J}$.
Then, from the observation above, it can be easily checked that for a solution $(\hat{a}_{\zerovec}, \hat{\omega}^{d}_{n, V}) \in \R\times\R^{|\tilde{J}|}$ to the new lasso problem
\begin{equation}\label{lasso-lattice}
    (\hat{a}_{\zerovec}, \hat{\omega}^{d}_{n, V}) \in
    \underset{a_{\zerovec} \in \R, \omega \in \R^{|\tilde{J}|}}{\argmin}
    \Bigg\{\big\lVert y - a_{\zerovec} \onevec - \tilde{M}\omega \big\rVert_2^2: \sum_{\substack{(\alpha, l) \in
        \tilde{J} \\ l \neq \zerovec}} |\omega_{\alpha, l}| \le
    V \Bigg\}, 
\end{equation}
if we let $\gamma$ be the $|J|$-dimensional vector indexed by $(\alpha, l) \in J$ for which 
\begin{align*}
    \gamma_{\alpha, l} = 
    \begin{cases}
        \big(\hat{\omega}^{d}_{n, V}\big)_{\alpha, l} &\mbox{if } (\alpha, l) \in \tilde{J} \\
        0 &\mbox{otherwise},
    \end{cases}
\end{align*}
then $(\hat{a}_{\zerovec}, \gamma)$ is a solution to the original lasso problem \eqref{finite-dimensional-lasso}.

Now, let $\check{M}$ be the $n \times (1 + |\tilde{J}|)$ matrix with columns indexed with the set $\{\zerovec\} \cup \tilde{J}$ such that
\begin{equation*}
    \check{M}_{i, 0} = 1 \qt{for $i \in [n]$}
\end{equation*}
and 
\begin{equation*}
\check{M}_{i, (\alpha, l)} = \tilde{M}_{i, (\alpha, l)} \qt{for $i \in [n]$ and $(\alpha, l) \in \tilde{J}$}. 
\end{equation*}
Since
\begin{equation*}
    1 + |\tilde{J}| = 1 + \sumoveralphafull \prod_{k \in S(\alpha)} (n_k - 1) = \prod_{k=1}^{d} \big[1 + (n_k - 1)\big] = n = |I_0|,
\end{equation*}
$\check{M}$ is a square matrix.
If we concatenate $a_0$ and $\omega$ in the problem \eqref{lasso-lattice} and form a $(1 + |\tilde{J}|)$-dimensional vector $\check{\omega}$ indexed with the set $\{\zerovec\} \cup \tilde{J}$, the problem \eqref{lasso-lattice} can be rewritten as 
\begin{equation}\label{problem-concatenated}
    \underset{\check{\omega} \in \R^{1 + |\tilde{J}|}}{\argmin}
    \Bigg\{\big\lVert y - \check{M}\check{\omega} \big\rVert_2^2: \sum_{\substack{(\alpha, l) \in
        \tilde{J} \\ l \neq \zerovec}} |\check{\omega}_{\alpha, l}| \le
    V\Bigg\}.
\end{equation}

Next, fix an $n$-dimensional vector $\theta$ indexed with the set $I_0$ and let $\check{\omega}$ be the $(1 + |\tilde{J}|)$-dimensional vector indexed with the set $\{\zerovec\} \cup \tilde{J}$ for which 
\begin{equation*}
    \check{\omega}_{\zerovec} = \theta_{\zerovec} \mbox{\quad and \quad} \check{\omega}_{\alpha, l} = \bigg(\prod_{k=1}^{d} n_k^{\alpha_k}\bigg) \cdot (H^{(2)} \theta)_{\tilde{l} + \alpha} \qt{for $(\alpha, l) \in \tilde{J},$}
\end{equation*}
where $\tilde{l}$ is the $d$-dimensional vector with 
\begin{equation*}
    \tilde{l}_k = 
    \begin{cases}
        l_k  &\mbox{if } k \in S(\alpha) \\
        0 &\mbox{otherwise}.
    \end{cases}
\end{equation*}
Then, by \eqref{eq:representation-of-theta}, 
\begingroup
\allowdisplaybreaks
\begin{align*}
    &(\check{M} \check{\omega})_i = \theta_{\zerovec} + \sumoveralphafull  
    \sum_{l \in \prod\limits_{k \in S(\alpha)}[0: (n_k - 2)]} \bigg[ \bigg(\prod_{k \in S(\alpha)} \Big(\frac{i_k}{n_k} - \frac{l_k}{n_k}\Big)_{+} \bigg)\\
    &\qquad \qquad \qquad \qquad \qquad \qquad \qquad \qquad  \qquad \qquad \cdot \bigg(\prod_{k=1}^{d} n_k^{\alpha_k}\bigg) \cdot (H^{(2)} \theta)_{\tilde{l} + \alpha} \bigg] \\
    &\quad= (H^{(2)} \theta)_{\zerovec} + \sumoveralphafull 
    \sum_{l \in \prod\limits_{k \in S(\alpha)}[n_k - 1]} \bigg[ \bigg(\prod_{k \in S(\alpha)} (i_k - l_k + 1)_{+} \bigg) \cdot (H^{(2)} \theta)_{\tilde{l}} \bigg] \\
    &\quad= (H^{(2)} \theta)_{\zerovec} + \sumoveralphafull \sum_{l \in I_0^{(\alpha)}}\bigg[\bigg(\prod_{k = 1}^d \binom{i_k - l_k + \alpha_k}{\alpha_k}\bigg) \cdot \ind\{i \ge l\} \cdot (H^{(2)} \theta)_l\bigg] = \theta_i
\end{align*}
\endgroup
for each $i \in I_0$, and thus, $\check{M} \check{\omega} = \theta$.
From this result, we can also deduce that $\check{M}$ is invertible.   
Moreover, by \eqref{discrete-variation-in-H-theta}, we have 
\begin{equation*}
    \sum_{\substack{(\alpha, l) \in \tilde{J} \\ l \neq \zerovec}} |\check{\omega}_{\alpha, l}| = \sumoveralphafull \bigg[\bigg(\prod_{k = 1}^{d} n_k^{\alpha_k}\bigg) \cdot \sum_{i \in I_0^{(\alpha)} \setminus \{\alpha\}} \big|(H^{(2)} \theta)_i\big|\bigg] = V_2(\theta).
\end{equation*}

For these reasons, for the unique solution $\hat{\theta}^{d}_{n, V}$ to the problem \eqref{discrete-analogue}, if we let $\check{\omega}$ be the $(1 + |\tilde{J}|)$-dimensional vector indexed with the set $\{\zerovec\} \cup \tilde{J}$ for which
\begin{equation*}
    \check{\omega}_{\zerovec} = \big(\hat{\theta}^{d}_{n, V}\big)_{\zerovec} \mbox{\quad and \quad} \check{\omega}_{\alpha, l} = \bigg(\prod_{k=1}^{d} n_k^{\alpha_k}\bigg) \cdot \big(H^{(2)} \hat{\theta}^{d}_{n, V}\big)_{\tilde{l} + \alpha} \qt{for $(\alpha, l) \in \tilde{J}$},
\end{equation*}
then $\check{\omega}$ is the unique solution to the problem \eqref{problem-concatenated}.
Due to the connection between the problem \eqref{lasso-lattice} and the problem \eqref{problem-concatenated}, this implies that if 
\begin{equation*}
    a_{\zerovec} = \big(\hat{\theta}^{d}_{n, V}\big)_{\zerovec}
\end{equation*}
and $\omega$ is the $|\tilde{J}|$-dimensional vector indexed by $(\alpha, l) \in \tilde{J}$ for which
\begin{equation*}
    \omega_{\alpha, l} = \bigg(\prod_{k=1}^{d} n_k^{\alpha_k}\bigg) \cdot \big(H^{(2)} \hat{\theta}^{d}_{n, V}\big)_{\tilde{l} + \alpha} \qt{for $(\alpha, l) \in \tilde{J}$},
\end{equation*}
then $(a_0, \omega)$ becomes the unique solution to the problem \eqref{lasso-lattice}. 
If we additionally let $\gamma$ be the $|J|$-dimensional vector indexed by $(\alpha, l) \in J$ for which 
\begin{align*}
    \gamma_{\alpha, l} = 
    \begin{cases}
        \big(\prod_{k=1}^{d} n_k^{\alpha_k}\big) \cdot \big(H^{(2)} \hat{\theta}^{d}_{n, V}\big)_{\tilde{l} + \alpha} &\mbox{if } (\alpha, l) \in \tilde{J} \\
        0 &\mbox{otherwise},
    \end{cases}
\end{align*}
the connection between the problem \eqref{finite-dimensional-lasso} and the problem \eqref{lasso-lattice} then shows that $(a_0, \gamma)$ is a solution to \eqref{finite-dimensional-lasso}.
Consequently, by Proposition \ref{prop:reduction-to-lasso}, we can see that the function $f$ on $[0, 1]^d$ defined by 
\begin{equation*}
    f(x_1, \dots, x_d) = \big(\hat{\theta}^{d}_{n, V}\big)_{\zerovec} + \sum_{(\alpha, l) \in \tilde{J}}  \big(H^{(2)} \hat{\theta}^{d}_{n, V}\big)_{\tilde{l} + \alpha} \cdot \prod_{k \in S(\alpha)} \big(n_k x_k - l_k\big)_{+},
\end{equation*} 
is a solution to the original optimization problem \eqref{our-problem-restated}. 
In addition, Proposition \ref{prop:reduction-to-lasso} also implies that for every solution $\hat{f}^{d, d}_{n, V}$ to \eqref{our-problem-restated}, the values of $\hat{f}^{d, d}_{n, V}$ at the design points equal $\hat{\theta}^{d}_{n, V}$, i.e.,
\begin{equation*}
    \hat{f}^{d, d}_{n, V}\Big(\Big(\frac{i_k}{n_k}, k \in [d]\Big)\Big) = \big(\hat{\theta}^{d}_{n, V}\big)_i
\end{equation*}
for all $i \in I_0$.
\end{proof}

\subsection{Proofs of Lemmas in Appendix \ref{subsec:proof-of-risk-results}}\label{subsec:proof-of-lemmas-risk-results}
\subsubsection{Proof of Lemma \ref{lem:svt-metric-entropy}}\label{pf:svt-metric-entropy}
\begin{proof} [Proof of Lemma \ref{lem:svt-metric-entropy}]
Suppose $f = \fazeronu \in S(V, t)$ and let $a_{\alpha} = \nu_{\alpha}(\{\zerovec\})$ for each $\alpha \in \{0, 1\}^d \setminus \{\zerovec\}$ with $|\alpha| \le s$. 
Then, we can write $f$ as  
\begin{align*}
    &f(x_1, \dots, x_d) = a_{\zerovec} + \sumoveralpha \int_{[0, 1)^{|\alpha|}} \prod_{j \in S(\alpha)} (x_j - s_j)_+ \, d\nu_{\alpha}(\salpha) \\   
    &\qquad \quad= \sum_{\substack{\alpha \in \{0, 1\}^d \\ |\alpha| \le s}} a_{\alpha} \prod_{j \in S(\alpha)} x_j + \sumoveralpha \int_{[0, 1)^{|\alpha|} \setminus \{\zerovec\}} \prod_{j \in S(\alpha)} (x_j - s_j)_+ \, d\nu_{\alpha}(\salpha)
\end{align*}
for $(x_1, \dots, x_d) \in [0, 1]^d$.
Note that 
\begin{align*}
    t^2 \ge \|f\|_n^2 = \frac{1}{n} \sum_{i = 1}^{n} f(x^{(i)})^2 = \frac{1}{n} \sum_{i \in I_0} f\big(\big(u^{(k)}_{i_k}, k \in [d]\big)\big)^2
\end{align*}
where $I_0 := [0 : (n_1 - 1)] \times \dots \times [0 : (n_d - 1)]$.
By the Cauchy inequality, we have 
\begingroup
\allowdisplaybreaks
\begin{align*}
    f(x_1, \dots, x_d)^2 &\ge \frac{1}{2} \bigg[\sum_{\substack{\alpha \in \{0, 1\}^d \\ |\alpha| \le s}} a_{\alpha} \prod_{j \in S(\alpha)} x_j \bigg]^2 \\ 
    & \qquad - \Bigg[\sumoveralpha \int_{[0, 1)^{|\alpha|} \setminus \{\zerovec\}} \prod_{j \in S(\alpha)} (x_j - s_j)_+ \, d\nu_{\alpha}(\salpha)\Bigg]^2 \\
    &\ge \frac{1}{2} \bigg[\sum_{\substack{\alpha \in \{0, 1\}^d \\ |\alpha| \le s}} a_{\alpha} \prod_{j \in S(\alpha)} x_j \bigg]^2 - \Bigg[\sumoveralpha \int_{[0, 1)^{|\alpha|} \setminus \{\zerovec\}} \, d|\nu_{\alpha}|(\salpha)\Bigg]^2 \\
    &\ge \frac{1}{2} \bigg[\sum_{\substack{\alpha \in \{0, 1\}^d \\ |\alpha| \le s}} a_{\alpha} \prod_{j \in S(\alpha)} x_j \bigg]^2 - V^2
\end{align*}
\endgroup
for every $(x_1, \dots, x_d) \in [0, 1]^d$.
Using this inequality, we will bound $a_{\alpha}$ for each $\alpha \in \{0, 1\}^d$ with $|\alpha| \le s$ first. 
Observe that $|a_{\zerovec}| = |f(\zerovec)| \le \sqrt{n} t \le V + \sqrt{n} t$. 
The above inequality then implies that
\begingroup
\allowdisplaybreaks
\begin{align*}
    n t^2 &\ge \sum_{i \in I_0} f\big(\big(u^{(k)}_{i_k}, k \in [d]\big)\big)^2 \ge \sum_{\substack{i_1 \ge 0\\ i_2, \dots, i_d = 0}} f\big(\big(u^{(k)}_{i_k}, k \in [d]\big)\big)^2 \\
    &\ge \sum_{i_1 = 0}^{n_1 - 1} \Big[\frac{1}{2}\big(a_{1, 0, \dots, 0} \cdot u^{(1)}_{i_1} + a_{\zerovec}\big)^2 - V^2\Big] = \frac{1}{2}\sum_{i_1 = 0}^{n_1 - 1} \big(a_{1, 0, \dots, 0} \cdot u^{(1)}_{i_1} + a_{\zerovec}\big)^2 - n_1 V^2 \\
    &\ge \frac{1}{2}\sum_{i_1 = 0}^{n_1 - 1} \Big(\frac{1}{2} a_{1, 0, \dots, 0}^2 \cdot \big(u^{(1)}_{i_1}\big)^2 - a_{\zerovec}^2\Big) - n_1 V^2 \\
    &\ge \frac{1}{4} a_{1, 0, \dots, 0}^2 \sum_{i_1 = 0}^{n_1 - 1} \big(u^{(1)}_{i_1}\big)^2 - \frac{n_1}{2} (nt^2) - n_1 V^2 \\
    &\ge \frac{1}{4} a_{1, 0, \dots, 0}^2 \sum_{i_1 = 0}^{n_1 - 1} {\rho}^2 \Big(\frac{i_1}{n_1}\Big)^2 - \frac{n_1}{2} (nt^2) - n_1 V^2 \\
    &= {\rho}^2 \cdot \frac{(n_1 - 1)(2n_1 - 1)}{24 n_1} \cdot a_{1, 0, \dots, 0}^2 - \frac{n_1}{2} (nt^2) - n_1 V^2.
\end{align*}
\endgroup
Here the last inequality follows from 
\begin{equation*}
    u^{(1)}_{i_1} = u^{(1)}_0 + \sum_{j_1 = 1}^{i_1} \big(u^{(1)}_{j_1} - u^{(1)}_{j_1 - 1}\big) \ge 0 + \sum_{j_1 = 1}^{i_1} \frac{\rho}{n_1} = \rho \cdot \frac{i_1}{n_1}
\end{equation*}
for $i_1 \in [n_1 - 1]$.
Thus, further applying the inequality $(x + y)^{1/2} \le x^{1/2} + y^{1/2}$, we can obtain
\begin{equation*}
    |a_{1, 0, \dots, 0}| \le C_{\rho} \cdot (V + \sqrt{n} t).
\end{equation*}
By the same argument, it can be shown that
\begin{equation*}
    |a_{\alpha}| \le C_{\rho} \cdot (V + \sqrt{n}t) 
\end{equation*}
for every $\alpha \in \{0, 1\}^d$ with $|\alpha| = 1$.
Also, since 
\begingroup
\allowdisplaybreaks
\begin{align*}
    nt^2 &\ge \sum_{i \in I_0} f\big(\big(u^{(k)}_{i_k}, k \in [d]\big)\big)^2 \\
    & \ge \sum_{\substack{i_1, i_2 \ge 0\\ i_3, \dots, i_d = 0}} \Big[\frac{1}{2}\big(a_{1, 1, 0, \dots, 0} \cdot u^{(1)}_{i_1} u^{(2)}_{i_2} + a_{1, 0, \dots, 0} \cdot u^{(1)}_{i_1} + a_{0, 1, 0, \dots, 0} \cdot u^{(2)}_{i_2} + a_{\zerovec}\big)^2 - V^2\Big] \\
    & =\frac{1}{2}\sum_{i_1 = 0}^{n_1 - 1} \sum_{i_2 = 0}^{n_2 - 1} \big(a_{1, 1, 0, \dots, 0} \cdot u^{(1)}_{i_1} u^{(2)}_{i_2} + a_{1, 0, \dots, 0} \cdot u^{(1)}_{i_1} + a_{0, 1, 0, \dots, 0} \cdot u^{(2)}_{i_2} + a_{\zerovec}\big)^2 - n_1 n_2 V^2 \\
    & \ge \frac{1}{2}\sum_{i_1 = 0}^{n_1 - 1} \sum_{i_2 = 0}^{n_2 - 1} \Big[\frac{1}{2} a_{1, 1, 0, \dots, 0}^2 \cdot \big(u^{(1)}_{i_1} u^{(2)}_{i_2})^2 \\
    &\qquad \qquad \qquad \quad- \big(a_{1, 0, \dots, 0} \cdot u^{(1)}_{i_1} + a_{0, 1, 0, \dots, 0} \cdot u^{(2)}_{i_2} + a_{\zerovec}\big)^2\Big] - n_1 n_2 V^2 \\
    & \ge {\rho}^4 \Big[\frac{(n_1 - 1)(2n_1 - 1)}{12n_1} \cdot \frac{(n_2 - 1)(2n_2 - 1)}{12n_2}\Big] \cdot a_{1, 1, 0, \dots, 0}^2 - C_{\rho} n_1 n_2 (V^2 + nt^2), 
\end{align*}
\endgroup
it follows that 
\begin{equation*}
    |a_{1, 1, 0, \dots, 0}| \le C_{\rho} \cdot (V + \sqrt{n}t).
\end{equation*}
The same argument implies 
\begin{equation*}
    |a_{\alpha}| \le C_{\rho} \cdot (V + \sqrt{n}t)
\end{equation*}
for all $\alpha \in \{0, 1\}^d$ with $|\alpha| = 2$.
Repeating this argument, we can show inductively that 
\begin{equation*}
    |a_{\alpha}| \le C_{\rho, s} \cdot (V + \sqrt{n}t) := \tilde{t}
\end{equation*}
for all $\alpha \in \{0, 1\}^d$ with $|\alpha| \le s$.

Now, fix $\delta > 0$, which will be specified later, and let 
\begin{equation*}
    \K = \big\{(k_{\alpha}, \alpha \in \{0, 1\}^d \mbox{ and }  |\alpha| \le s): -(K + 1) \le k_{\alpha} \le K \mbox{ for all } \alpha \in \{0, 1\}^d \mbox{ with } |\alpha| \le s \big\}, 
\end{equation*}
where $K = \floor{\tilde{t}/\delta}$. 
Here $\floor{\tilde{t}/\delta}$ indicates the greatest integer less than or equal to $\tilde{t}/\delta$.
Also, for each $k \in \K$, let 
\begin{equation*}
    M(k) = \{\fazeronu \in S(V, t): k_{\alpha} \delta \le a_{\alpha} \le (k_{\alpha} + 1) \delta \mbox{ for each } \alpha \in \{0, 1\}^d \mbox{ with } |\alpha| \le s\}.
\end{equation*}
Then,
\begin{equation*}
    S(V, t) = \bigcup_{k \in \K} M(k),
\end{equation*}
and thereby, 
\begin{align}\label{svt-decomposition}
\begin{split}
    \log N(\epsilon, S(V, t), \|\cdot\|_n) &\le \log\Big(\sum_{k \in \K} N(\epsilon, M(k), \|\cdot\|_n)\Big) \\
    &\le \log |\K| + \sup_{k \in \K} \log N(\epsilon, M(k), \|\cdot\|_n) \\
    &\le 2^d \log\Big(2 + \frac{2 \tilde{t}}{\delta}\Big) + \sup_{k \in \K} \log N(\epsilon, M(k), \|\cdot\|_n).
\end{split}
\end{align}

Fix $k \in \K$. Let $M_{\zerovec}(k)$ be the collection of all the constant functions on $[0, 1]^d$
\begin{equation*}
    (x_1, \dots, x_d) \mapsto a_{\zerovec}
\end{equation*}
where $k_{\zerovec} \delta \le a_{\zerovec} \le (k_{\zerovec} + 1) \delta$, and for each $\alpha \in \{0, 1\}^d \setminus \{\zerovec\}$ with $|\alpha| \le s$, let $M_{\alpha}(k)$ be the collection of the functions on $[0, 1]^d$ of the form
\begin{equation*}
    (x_1, \dots, x_d) \mapsto a_{\alpha} \prod_{j \in S(\alpha)} x_j + \int_{[0, 1)^{|\alpha|} \setminus \{\zerovec\}} \prod_{j \in S(\alpha)} (x_j - s_j)_+ \, d\nu_{\alpha}(\salpha),
\end{equation*}
where $k_{\alpha} \delta \le a_{\alpha} \le (k_{\alpha} + 1) \delta$ and $\nu_{\alpha}$ is a  signed measure concentrated on $(\prod_{k \in S(\alpha)} \mathcal{U}_k) \cap [0, 1)^{|\alpha|}$ with $|\nu_{\alpha}|([0, 1)^{|\alpha|} \setminus \{\zerovec\}) \le V$.
Then, by definition,
\begin{equation*}
    M(k) \subseteq \bigoplus_{\substack{\alpha \in \{0, 1\}^d \\ |\alpha| \le s}} M_{\alpha}(k), 
\end{equation*}
which implies that 
\begin{equation}\label{mk-decomposition}
    \log N(\epsilon, M(k), \|\cdot\|_n) \le \sum_{\substack{\alpha \in \{0, 1\}^d \\ |\alpha| \le s}} \log N\Big(\frac{\epsilon}{2^d}, M_{\alpha}(k), \|\cdot\|_n\Big).
\end{equation}
Here for a collection of sets $\{X_i\}_{i \in I}$, $\bigoplus_{i \in I} X_i$ indicates the set $\{\sum_{i \in I} x_i: x_i \in X_i \mbox{ for } i \in I\}$.

Now, we will bound each term of the right-hand side of \eqref{mk-decomposition}.
It is easy to show that
\begin{equation}\label{m0k-metric-entropy}
    N(\epsilon, M_{\zerovec}(k), \|\cdot\|_n) \le 1 + \frac{\delta_{\zerovec}}{\epsilon}.
\end{equation}
Hence, it suffices to bound $\log N(\epsilon, M_{\alpha}(k), \|\cdot\|_n)$ for each $\alpha \in \{0, 1\}^d \setminus \{\zerovec\}$ with $|\alpha| \le s$.
Fix $\alpha \in \{0, 1\}^d \setminus \{\zerovec\}$ with $|\alpha| \le s$ and let  $\tilde{M}_{\alpha}(k)$ be the collection of all the functions on $[0, 1]^{|\alpha|}$ of the form
\begin{equation*}
    (x_j, j \in S(\alpha)) \mapsto a_{\alpha} \prod_{j \in S(\alpha)} x_j + \int_{[0, 1)^{|\alpha|} \setminus \{\zerovec\}} \prod_{j \in S(\alpha)} (x_j - s_j)_+ \, d\nu_{\alpha}(\salpha),
\end{equation*}
where $a_{\alpha}$ and $\nu_{\alpha}$ satisfy the same conditions as in the definition of $M_{\alpha}(k)$.
With slight abusing of notation, let $\|\cdot\|_{n}$ also denote the norm on $[0, 1]^{|\alpha|}$ defined by 
\begin{equation*}
    \|f\|_{n}^2 = \frac{1}{\prod\limits_{j \in S(\alpha)} n_j} \sum_{i \in \prod\limits_{j \in S(\alpha)} [0: (n_j - 1)]}  \Big(f\big(\big(u^{(j)}_{i_j}, j \in S(\alpha)\big)\big)\Big)^2.
\end{equation*}
Then, clearly, 
\begin{equation}\label{m-alpha-k-and-m-tilde-alpha-k}
    \log N(\epsilon, M_{\alpha}(k), \|\cdot\|_n) \le \log N(\epsilon, \tilde{M}_{\alpha}(k), \|\cdot\|_{n}).
\end{equation}
Also, since $\tilde{M}_{\alpha}(k)$ can be obtained by shifting $\tilde{M}_{\alpha}(\zerovec)$ by some suitable function, we have 
\begin{equation}\label{m-tilde-alpha-k-and-0}
    \log N(\epsilon, \tilde{M}_{\alpha}(k), \|\cdot\|_{n}) = \log N(\epsilon, \tilde{M}_{\alpha}(\zerovec), \|\cdot\|_{n}).
\end{equation}
Moreover, if we let $\tilde{M}_{\alpha}$ be the collection of all the functions on $[0, 1]^{|\alpha|}$ of the form
\begin{equation*}
    (x_j, j \in S(\alpha)) \mapsto \int_{[0, 1)^{|\alpha|}} \prod_{j \in S(\alpha)} (x_j - s_j)_+ \, d\nu_{\alpha}(\salpha)
\end{equation*}
where $\nu_{\alpha}$ is a signed measure concentrated on $(\prod_{k \in S(\alpha)} \mathcal{U}_k) \cap [0, 1)^{|\alpha|}$ with $|\nu_{\alpha}|([0, 1)^{|\alpha|}) \le V + \delta$, then 
\begin{equation}\label{m-tilde-alpha-and-0}
    \log N(\epsilon, \tilde{M}_{\alpha}(\zerovec), \|\cdot\|_{n}) \le \log N(\epsilon, \tilde{M}_{\alpha}, \|\cdot\|_{n})
\end{equation}
because $\tilde{M}_{\alpha}(\zerovec) \subseteq \tilde{M}_{\alpha}$.
For these reasons, we can see that it is enough to bound the metric entropy of $\tilde{M}_{\alpha}$ in order to bound the metric entropy of $M_{\alpha}(k)$.

For $m \in [d]$ and $S > 0$, we denote by $\genmalpha_m(S)$ the collection of the functions on $[0, 1]^{m}$ of the form
\begin{equation*}
    (x_1, \dots, x_m) \mapsto \int_{[0, 1)^{m}} (x_1 - s_1)_+ \cdots (x_m - s_m)_+ \, d\nu(s)
\end{equation*}
where $\nu$ is a signed measure concentrated on $(\prod_{k =1}^{m} \mathcal{U}_k) \cap [0, 1)^{m}$ with variation $|\nu|([0, 1)^m) \le S$.
We can see that 
\begin{equation*}
    \tilde{M}_{\alpha} = \genmalpha_{|\alpha|}(V + \delta)
\end{equation*}
after suitable re-indexing. 
From the following lemma, which provides an upper bound of the metric entropy of $\genmalpha_{|\alpha|}(S)$, we can therefore obtain an upper bound of the metric entropy of $\tilde{M}_{\alpha}$.
We defer our proof of this lemma to Appendix \ref{pf:tas-metric-entropy}.

\begin{lemma}\label{lem:tas-metric-entropy}
    There exist positive constants $C_{\rho, m}$ and $\kappa_{\rho, m}$ depending on $\rho$ and $m$ such that for every $S > 0$,
    \begin{equation*}
        \log N(\epsilon, \genmalpha_m(S), \|\cdot\|_{n}) \le C_{\rho, m} \Big(\frac{S }{\epsilon}\Big)^{\frac{1}{2}} \bigg[\log\Big(\frac{S }{\epsilon}\Big)\bigg]^{\frac{3(2m - 1)}{4}}
    \end{equation*}
    provided that $0 < \epsilon / S \le \kappa_{\rho, m}$. 
    The log term can be omitted when $m = 1$.
\end{lemma}
By Lemma \ref{lem:tas-metric-entropy}, 
\begin{align}\label{mtilde-alpha-metric-entropy}
\begin{split}
    \log N(\epsilon, \tilde{M}_{\alpha}, \|\cdot\|_n) &\le C_{\rho, |\alpha|} \Big(\frac{V + \delta}{\epsilon}\Big)^{\frac{1}{2}} \cdot \bigg[\log\Big(\frac{V + \delta}{\epsilon}\Big)\bigg]^{\frac{3(2|\alpha| - 1)}{4}} \\
    &\le C_{\rho, |\alpha|} \Big(\frac{2(V + \delta) }{\epsilon}\Big)^{\frac{1}{2}} \cdot \bigg[\log\Big(\frac{2(V + \delta)}{\epsilon}\Big)\bigg]^{\frac{3(2|\alpha| - 1)}{4}}
\end{split}
\end{align}
provided that $0 < \epsilon \le \kappa_{\rho, |\alpha|}(V + \delta) := \epsilon_0 $. 
Note that the logarithmic multiplicative factor can be omitted when $|\alpha| = 1$. For $\epsilon_0 < \epsilon \le V + \delta$, 
\begin{equation*}
    \log N(\epsilon, \tilde{M}_{\alpha}, \|\cdot\|_n) \le \log N(\epsilon_0, \tilde{M}_{\alpha}, \|\cdot\|_n) \le \tilde{C}_{\rho, |\alpha|} \Big(\frac{2(V + \delta)}{\epsilon}\Big)^{\frac{1}{2}} \cdot \bigg[\log\Big(\frac{2(V + \delta)}{\epsilon}\Big)\bigg]^{\frac{3(2|\alpha| - 1)}{4}}
\end{equation*}
for 
\begin{equation*}
    \tilde{C}_{\rho, |\alpha|} = C_{\rho, |\alpha|} \Big(\frac{1}{\kappa_{\rho, |\alpha|}}\Big)^{\frac{1}{2}} \Big[\log\Big(\frac{2}{\kappa_{\rho, |\alpha|}}\Big)\Big]^{\frac{3(2|\alpha| - 1)}{4}} \Big[2^{\frac{1}{2}}(\log 2)^{\frac{3(2|\alpha| - 1)}{4}}\Big]^{-1}, 
\end{equation*}
where $C_{\rho, |\alpha|}$ and $\kappa_{\rho, |\alpha|}$ are the constants in \eqref{mtilde-alpha-metric-entropy}.
Moreover, for $f \in \tilde{M}_{\alpha}$, we have
\begin{equation*}
    \|f\|_n \le V + \delta,
\end{equation*}
which implies that $\log N(\epsilon, \tilde{M}_{\alpha}, \|\cdot\|_n) = 0$ for every $\epsilon > V + \delta$. 
Combining those pieces, we can deduce that 
\begin{equation*}
    \log N(\epsilon, \tilde{M}_{\alpha}, \|\cdot\|_n) \le C_{\rho, |\alpha|} \Big(1 + \frac{2(V + \delta)}{\epsilon}\Big)^{\frac{1}{2}} \cdot \bigg[\log\Big(1 + \frac{2(V + \delta) }{\epsilon}\Big)\bigg]^{\frac{3(2|\alpha| - 1)}{4}}
\end{equation*}
for every $\epsilon > 0$. 
This together with \eqref{m-alpha-k-and-m-tilde-alpha-k}, \eqref{m-tilde-alpha-k-and-0}, and \eqref{m-tilde-alpha-and-0} implies that 
\begin{align}\label{m-alpha-k-metric-entropy}
    \log N(\epsilon, M_{\alpha}(k), \|\cdot\|_n) \le C_{\rho, |\alpha|} \Big(1 + \frac{2(V + \delta)}{\epsilon}\Big)^{\frac{1}{2}} \cdot \bigg[\log\Big(1 + \frac{2(V + \delta) }{\epsilon}\Big)\bigg]^{\frac{3(2|\alpha| - 1)}{4}}
\end{align}
Note that we can omit the logarithmic multiplicative factor when $|\alpha| = 1$.

By \eqref{mk-decomposition}, \eqref{m0k-metric-entropy}, and \eqref{m-alpha-k-metric-entropy}, we have
\begingroup
\allowdisplaybreaks
\begin{align*}
    &\log N(\epsilon, M(k), \|\cdot\|_n) \le \log\Big(1 + \frac{2^d\delta}{\epsilon}\Big) + C_{\rho} \sum_{\substack{\alpha \in \{0, 1\}^d \\ |\alpha| = 1}} \bigg[1 + \frac{2^{d + 1} (V + \delta) }{\epsilon}\bigg]^{\frac{1}{2}} \\
    &\qquad \qquad \qquad \quad + C_{\rho, d} \sum_{\substack{\alpha \in \{0, 1\}^d \\ 1 < |\alpha| \le s}} \bigg[1 + \frac{2^{d + 1} (V + \delta)}{\epsilon}\bigg]^{\frac{1}{2}} \cdot \bigg[\log\Big(1 + \frac{2^{d + 1} (V + \delta)}{\epsilon}\Big)\bigg]^{\frac{3(2|\alpha| - 1)}{4}}.
\end{align*}
\endgroup
As a result, with the choice of $\delta = \epsilon$, \eqref{svt-decomposition} derives the conclusion
\begingroup
\allowdisplaybreaks
\begin{align*}
    \log N(\epsilon, S(V, t), \|\cdot\|_n) &\le  2^d \log\Big(2 + C_{\rho, s} \cdot \frac{V + \sqrt{n} t}{\epsilon}\Big) + C_{\rho, d} + C_{\rho, d}\Big(\frac{V}{\epsilon}\Big)^{\frac{1}{2}} \\
    &\quad + C_{\rho, d} \Big(2^{d + 1} + 1 + \frac{2^{d + 1} V }{\epsilon}\Big)^{\frac{1}{2}} \bigg[\log\Big(2^{d + 1} + 1 + \frac{2^{d + 1} V}{\epsilon}\Big)\bigg]^{\frac{3(2s - 1)}{4}} \\
    &\le  2^d \log\Big(2 + C_{\rho, s} \cdot \frac{V + \sqrt{n} t}{\epsilon}\Big) \\
    &\quad + C_{\rho, d} \Big(2^{d + 1} + 1 + \frac{2^{d + 1} V }{\epsilon}\Big)^{\frac{1}{2}} \bigg[\log\Big(2^{d + 1} + 1 + \frac{2^{d + 1} V}{\epsilon}\Big)\bigg]^{\frac{3(2s - 1)}{4}}.
\end{align*}
\endgroup
\end{proof}

\subsubsection{Proof of Lemma \ref{lem:integration-helper}}\label{pf:integration-helper2}
\begin{proof}[Proof of Lemma \ref{lem:integration-helper}]
We basically follows the proof of \cite{fang2021multivariate}, Lemma C.5 here.
First, note that 
\begin{align*}
    \int_{0}^{t} \Big(\frac{u}{\epsilon}\Big)^{\frac{1}{4}} \Big[\log \frac{u}{\epsilon}\Big]^k \, d\epsilon &= u \int_{\tau}^{+\infty} e^{-\frac{3}{4}v} v^k \, dv, & v = \log\frac{u}{\epsilon}
\end{align*}
where $\tau = \log(u/t)$. 
If $\tau \le 1$, 
\begin{equation*}
    \int_{0}^{t} \Big(\frac{u}{\epsilon}\Big)^{\frac{1}{4}} \Big[\log \frac{u}{\epsilon}\Big]^k \, d\epsilon \le u \int_{0}^{+\infty} e^{-\frac{3}{4}v} v^k \, dv \le C u e^{-\frac{3}{4}\tau} = C u^{\frac{1}{4}} t^{\frac{3}{4}},
\end{equation*}
so the lemma directly follows.
Otherwise, applying integration by parts iteratively, we obtain
\begingroup
\allowdisplaybreaks
\begin{align*}
    \int_{\tau}^{+\infty} e^{-\frac{3}{4}v} v^k \, dv &= \frac{4}{3} e^{-\frac{3}{4}\tau} \tau^k + \frac{4}{3} k \int_{\tau}^{+\infty} e^{-\frac{3}{4}v} v^{k-1} \, dv \\
    &= \dots = \frac{4}{3} e^{-\frac{3}{4}\tau} \tau^k + \dots + \Big(\frac{4}{3}\Big)^{\floor{k}+1} k \cdots (k - \floor{k} + 1) e^{-\frac{3}{4}\tau} \tau^{k-\floor{k}} \\
    &\qquad \qquad + \Big(\frac{4}{3}\Big)^{\floor{k}+1} k \cdots (k - \floor{k}) \int_{\tau}^{+\infty} e^{-\frac{3}{4}v} v^{k-\floor{k}-1} \, dv \\
    &\le C_k e^{-\frac{3}{4}\tau} \tau^k + C_k \int_{\tau}^{+\infty} e^{-\frac{3}{4}v} v^{k-\floor{k}-1} \, dv \\
    &\le C_k e^{-\frac{3}{4}\tau} \tau^k + C_k e^{-\frac{3}{4}\tau} = C_k u^{-\frac{3}{4}} t^{\frac{3}{4}} (1 + \tau^k).
\end{align*}
\endgroup
Here $\floor{k}$ indicates the greatest integer less than or equal to $k$, and the inequalities are due to the fact that $\tau > 1$.
From this result, we can derive the conclusion 
\begin{equation*}
    \int_{0}^{t} \Big(\frac{u}{\epsilon}\Big)^{\frac{1}{4}} \Big[\log \frac{u}{\epsilon}\Big]^k \, d\epsilon \le C_k u^{\frac{1}{4}} t^{\frac{3}{4}} (1 + \tau^k).
\end{equation*}
\end{proof}

\subsubsection{Proof of Lemma \ref{lem:discrete-measures-approximation}}\label{pf:discrete-measures-approximation}
\begin{proof}[Proof of Lemma \ref{lem:discrete-measures-approximation}]
For $\alpha \in \{0, 1\}^d$ with $|\alpha| \le s$ and $l \in \prod_{k \in S(\alpha)} [0: (N_k - 1)]$, let $R_l^{(\alpha)} = \prod_{k \in S(\alpha)} [l_k/N_k, (l_k + 1)/N_k)$ and define discrete signed measures $\tilde{\mu}_{\alpha, l}$ and $\mu_{\alpha}$ as in the proof of Lemma \ref{lem:reduction-to-discrete-measures}. 
We will show that $\fazeromu$ constructed from these signed measures $\mu_{\alpha}$ satisfies all the desired conditions. 

First, as in the proof of Lemma \ref{lem:reduction-to-discrete-measures}, it can be shown that 
\begin{equation}\label{eq:agree-at-lattice}
    \fazeromu\Big(\frac{i_1}{N_1}, \dots, \frac{i_d}{N_d}\Big) = \fazeronu\Big(\frac{i_1}{N_1}, \dots, \frac{i_d}{N_d}\Big)
\end{equation}
for every $i \in \prod_{k \in [d]} [0: N_k]$ and 
\begin{equation*}
    \Vmars(\fazeromu) \le \Vmars(\fazeronu).
\end{equation*}
Also, by repeating the arguments in the proof of Lemma \ref{lem:reduction-to-discrete-measures}, we can show that
\begin{equation}\label{eq:mualphal-nualpha}
    \int_{[0, 1]^{|\alpha|}} \prod_{k \in S(\alpha)} (x_k - t_k)_+ \, d\tilde{\mu}_{\alpha, l}(t^{(\alpha)}) = \int_{R_l^{(\alpha)}} \prod_{k \in S(\alpha)} (x_k - t_k)_+ \, d\nu_{\alpha}(t^{(\alpha)}) 
\end{equation}
if $x_k \ge (l_k + 1)/ N_k$ for all $k \in S(\alpha)$, and that 
\begin{equation}\label{eq:mualphal-nualpha-total-varitation}
    |\tilde{\mu}_{\alpha, l}|([0, 1]^{|\alpha|}) \le |\nu_{\alpha}|\big(R_l^{(\alpha)}\big).
\end{equation}

Fix $x = (x_1, \dots, x_d) \in [0, 1)^d$ and suppose that $i_k/N_k \le x_k < (i_k + 1)/N_k$ for $k \in [d]$.
We then have
\begin{align*}
    &\big|\fazeromu(x_1, \dots, x_d) - \fazeronu(x_1, \dots, x_d)\big| \\
    &\qquad \le \sumoveralpha \sum_{l \in \prod\limits_{k \in S(\alpha)}[0: i_k]} \bigg| \int_{[0, 1]^{|\alpha|}} \prod_{k \in S(\alpha)} (x_k - t_k)_+ \, d\tilde{\mu}_{\alpha, l}(t^{(\alpha)}) \\
    &\qquad \qquad \qquad \qquad \qquad \qquad \qquad \qquad- \int_{R_l^{(\alpha)}} \prod_{k \in S(\alpha)} (x_k - t_k)_+ \, d\nu_{\alpha}(t^{(\alpha)}) \bigg|.  
\end{align*}
If $l \in \prod_{k \in S(\alpha)}[0:(i_k - 1)]$, then 
\begin{equation*}
    \bigg| \int_{[0, 1]^{|\alpha|}} \prod_{k \in S(\alpha)} (x_k - t_k)_+ \, d\tilde{\mu}_{\alpha, l}(t^{(\alpha)}) - \int_{R_l^{(\alpha)}} \prod_{k \in S(\alpha)} (x_k - t_k)_+ \, d\nu_{\alpha}(t^{(\alpha)}) \bigg| = 0
\end{equation*}
because of \eqref{eq:mualphal-nualpha}. 
Otherwise, if $l \in \prod_{k \in S(\alpha)}[0:i_k] \setminus \prod_{k \in S(\alpha)}[0:(i_k - 1)]$, there exists $j \in S(\alpha)$ such that $l_j = i_j$. 
Hence, by \eqref{eq:mualphal-nualpha-total-varitation}, 
\begin{align*}
    &\bigg| \int_{[0, 1]^{|\alpha|}} \prod_{k \in S(\alpha)} (x_k - t_k)_+ \, d\tilde{\mu}_{\alpha, l}(t^{(\alpha)}) - \int_{R_l^{(\alpha)}} \prod_{k \in S(\alpha)} (x_k - t_k)_+ \, d\nu_{\alpha}(t^{(\alpha)}) \bigg| \\
    &\quad \le \frac{1}{N_j} \cdot \Big(|\tilde{\mu}_{\alpha, l}|\big([0, 1]^{|\alpha|}\big) +  |\nu_{\alpha}|\big(R_l^{(\alpha)}\big) \Big) \le \frac{2}{N} \cdot |\nu_{\alpha}|\big(R_l^{(\alpha)}\big) = \frac{2}{N} \cdot |\nu_{\alpha}|\big(R_l^{(\alpha)} \setminus \{\zerovec\}\big)
\end{align*}
if $l \neq \zerovec$, and 
\begin{align*}
    &\bigg| \int_{[0, 1]^{|\alpha|}} \prod_{k \in S(\alpha)} (x_k - t_k)_+ \, d\tilde{\mu}_{\alpha, \zerovec}(t^{(\alpha)}) - \int_{R_{\zerovec}^{(\alpha)}} \prod_{k \in S(\alpha)} (x_k - t_k)_+ \, d\nu_{\alpha}(t^{(\alpha)}) \bigg| \\
    &\quad \le \bigg| \int_{[0, 1]^{|\alpha|} \setminus \{\zerovec\}} \prod_{k \in S(\alpha)} (x_k - t_k)_+ \, d\tilde{\mu}_{\alpha, \zerovec}(t^{(\alpha)})\bigg| + \bigg| \int_{R_{\zerovec}^{(\alpha)} \setminus \{\zerovec\}} \prod_{k \in S(\alpha)} (x_k - t_k)_+ \, d\nu_{\alpha}(t^{(\alpha)})\bigg| \\
    &\quad \qquad + \bigg| \prod_{k \in S(\alpha)} x_k \cdot \big(|\tilde{\mu}_{\alpha, \zerovec}|(\{\zerovec\}) -  |\nu_{\alpha}|(\{\zerovec\}) \big) \bigg| \\
    &\quad \le \frac{1}{N_j} \cdot \Big(|\tilde{\mu}_{\alpha, \zerovec}|\big([0, 1]^{|\alpha|} \setminus \{\zerovec\}\big) +  |\nu_{\alpha}|\big(R_{\zerovec}^{(\alpha)} \setminus \{\zerovec\}\big) + \big||\tilde{\mu}_{\alpha, \zerovec}| (\{\zerovec\}) -  |\nu_{\alpha}|(\{\zerovec\}) \big| \Big) \\ 
    &\quad \le \frac{1}{N} \cdot \bigg( |\nu_{\alpha}|\big(R_{\zerovec}^{(\alpha)} \setminus \{\zerovec\}\big) \\
    &\quad \quad + \sum_{\delta \in \prod\limits_{k \in S(\alpha)} \{0, 1\} \setminus \{\zerovec\}} \bigg|\int_{R^{(\alpha)}_\zerovec \setminus \{\zerovec\}} \prod_{k \in S(\alpha)} \bigg[\bigg(\frac{1/N_k - t_k}{1/N_k - 0}\bigg)^{1 - \delta_k} \cdot \bigg(\frac{t_k - 0}{1/N_k - 0}\bigg)^{\delta_k}\bigg] \,  d\nu_{\alpha}(\talpha)\bigg| \\
    &\quad \quad +  \bigg|\int_{R^{(\alpha)}_\zerovec \setminus \{\zerovec\}} \prod_{k \in S(\alpha)} \bigg(\frac{1/N_k - t_k}{1/N_k - 0}\bigg) \,  d\nu_{\alpha}(\talpha)\bigg|\bigg) \\
    &\quad \le \frac{1}{N} \cdot \bigg( |\nu_{\alpha}|\big(R_{\zerovec}^{(\alpha)} \setminus \{\zerovec\}\big) \\
    &\quad \quad + \sum_{\delta \in \prod\limits_{k \in S(\alpha)} \{0, 1\}} \bigg|\int_{R^{(\alpha)}_\zerovec \setminus \{\zerovec\}} \prod_{k \in S(\alpha)} \bigg[\bigg(\frac{1/N_k - t_k}{1/N_k - 0}\bigg)^{1 - \delta_k} \cdot \bigg(\frac{t_k - 0}{1/N_k - 0}\bigg)^{\delta_k}\bigg] \,  d\nu_{\alpha}(\talpha)\bigg|\bigg) \\
    &\quad \le \frac{2}{N} \cdot  |\nu_{\alpha}|\big(R_{\zerovec}^{(\alpha)} \setminus \{\zerovec\}\big), 
\end{align*}
where the third inequality is from the definition of $\tilde{\mu}_{\alpha, \zerovec}$.
Using these results, we can first show that 
\begin{align*}
    &\big|\fazeromu(x_1, \dots, x_d) - \fazeronu(x_1, \dots, x_d)\big| \\
    &\qquad \le \sumoveralpha \sum_{l \in \prod\limits_{k \in S(\alpha)}[0: i_k]} \bigg| \int_{[0, 1]^{|\alpha|}} \prod_{k \in S(\alpha)} (x_k - t_k)_+ \, d\tilde{\mu}_{\alpha, l}(t^{(\alpha)}) \\
    &\qquad \qquad \qquad \qquad \qquad \qquad \qquad \qquad- \int_{R_l^{(\alpha)}} \prod_{k \in S(\alpha)} (x_k - t_k)_+ \, d\nu_{\alpha}(t^{(\alpha)}) \bigg| \\
    &\qquad \le \sumoveralpha \sum_{l \in \prod\limits_{k \in S(\alpha)}[0: i_k]} \frac{2}{N} \cdot  |\nu_{\alpha}|\big(R_{l}^{(\alpha)} \setminus \{\zerovec\}\big) \\
    &\qquad \le \frac{2}{N} \cdot  \sumoveralpha |\nu_{\alpha}|\big([0, 1)^{|\alpha|} \setminus \{\zerovec\}\big) = \frac{2}{N} \cdot \Vmars(\fazeronu),
\end{align*}
which directly leads to that
\begin{equation*}
    \| \fazeromu - \fazeronu \|_{\infty} \le \frac{2}{N} \cdot \Vmars(\fazeronu).
\end{equation*}
Also, we can show that
\begin{align*}
    &\big|\fazeromu(x_1, \dots, x_d) - \fazeronu(x_1, \dots, x_d)\big| \\
    &\qquad \le \sumoveralpha \sum_{j \in S(\alpha)} \sum_{\substack{l \in \prod\limits_{k \in S(\alpha)}[0: i_k] \\ l_j = i_j}} \bigg| \int_{[0, 1]^{|\alpha|}} \prod_{k \in S(\alpha)} (x_k - t_k)_+ \, d\tilde{\mu}_{\alpha, l}(t^{(\alpha)}) \\
    &\qquad \qquad \qquad \qquad \qquad \qquad \qquad \qquad \qquad - \int_{R_l^{(\alpha)}} \prod_{k \in S(\alpha)} (x_k - t_k)_+ \, d\nu_{\alpha}(t^{(\alpha)}) \bigg| \\
    &\qquad \le \sumoveralpha \sum_{j \in S(\alpha)} \sum_{\substack{l \in \prod\limits_{k \in S(\alpha)}[0: (N_k - 1)] \\ l_j = i_j}} \frac{2}{N} \cdot  |\nu_{\alpha}|\big(R_{l}^{(\alpha)} \setminus \{\zerovec\}\big) \\
    &\qquad \le \frac{2}{N} \cdot \sumoveralpha \sum_{j \in S(\alpha)} |\nu_{\alpha}|\bigg( \Big(\Big[\frac{i_j}{N_j}, \frac{i_j + 1}{N_j}\Big) \times \prod_{\substack{k \in S(\alpha) \\ k \neq j}} [0, 1) \Big) \setminus \{\zerovec\} \bigg).
\end{align*}
From this result, we can derive that
\begin{align*}
    &\| \fazeronu - \fazeromu \|_{p_0, 2}^2 = \int_{[0, 1]^d} \big(\fazeromu(x_1, \dots, x_d) - \fazeronu(x_1, \dots, x_d)\big)^2 p_0(x) dx \\
    &\quad \le B \cdot \sum_{i \in \prod\limits_{k = 1}^{d}[0: (N_k - 1)]} \bigg( \prod_{k = 1}^{d} \frac{1}{N_k}\bigg) \\
    &\quad \qquad \qquad \qquad \cdot \Bigg[ \frac{2}{N} \cdot \sumoveralpha \sum_{j \in S(\alpha)} |\nu_{\alpha}|\bigg( \Big(\Big[\frac{i_j}{N_j}, \frac{i_j + 1}{N_j}\Big) \times \prod_{\substack{k \in S(\alpha) \\ k \neq j}} [0, 1) \Big) \setminus \{\zerovec\} \bigg) \Bigg]^2 \\
    &\quad \le \frac{4B}{N^2} \cdot \sum_{i \in \prod\limits_{k = 1}^{d}[0: (N_k - 1)]} \bigg( \prod_{k = 1}^{d} \frac{1}{N_k}\bigg) \cdot \bigg(\sumoveralpha |\alpha|\bigg) \\
    &\quad \qquad \qquad \qquad \cdot \sumoveralpha \sum_{j \in S(\alpha)} \Bigg[ |\nu_{\alpha}|\bigg( \Big(\Big[\frac{i_j}{N_j}, \frac{i_j + 1}{N_j}\Big) \times \prod_{\substack{k \in S(\alpha) \\ k \neq j}} [0, 1) \Big) \setminus \{\zerovec\} \bigg) \Bigg]^2 \\
    &\quad \le \frac{C_d B}{N^2} \cdot \sumoveralpha \sum_{j \in S(\alpha)} \sum_{i \in \prod\limits_{k = 1}^{d}[0: (N_k - 1)]} \bigg( \prod_{k = 1}^{d} \frac{1}{N_k}\bigg)\\
    &\quad \qquad \qquad \qquad \qquad \qquad \qquad \qquad \cdot \Bigg[ |\nu_{\alpha}|\bigg( \Big(\Big[\frac{i_j}{N_j}, \frac{i_j + 1}{N_j}\Big) \times \prod_{\substack{k \in S(\alpha) \\ k \neq j}} [0, 1) \Big) \setminus \{\zerovec\} \bigg) \Bigg]^2 \\
    &\quad \le \frac{C_d B}{N^3} \cdot \sumoveralpha \sum_{j \in S(\alpha)} \sum_{i_j = 0}^{N_j - 1} \Bigg[ |\nu_{\alpha}|\bigg( \Big(\Big[\frac{i_j}{N_j}, \frac{i_j + 1}{N_j}\Big) \times \prod_{\substack{k \in S(\alpha) \\ k \neq j}} [0, 1) \Big) \setminus \{\zerovec\} \bigg) \Bigg]^2 \\
     &\quad \le \frac{C_d B}{N^3} \cdot \sumoveralpha \sum_{j \in S(\alpha)} \Bigg[ \sum_{i_j = 0}^{N_j - 1} |\nu_{\alpha}|\bigg( \Big(\Big[\frac{i_j}{N_j}, \frac{i_j + 1}{N_j}\Big) \times \prod_{\substack{k \in S(\alpha) \\ k \neq j}} [0, 1) \Big) \setminus \{\zerovec\} \bigg) \Bigg]^2 \\ 
     &\quad \le \frac{C_d B}{N^3} \cdot \sumoveralpha |\alpha| \cdot \Big[ |\nu_{\alpha}|
     \big([0, 1)^{|\alpha|} \setminus \{\zerovec\} \big) \Big]^2 \\
     &\quad \le \frac{C_d B}{N^3} \cdot \bigg[ \sumoveralpha |\nu_{\alpha}|
     \big([0, 1)^{|\alpha|} \setminus \{\zerovec\} \big) \bigg]^2 = \frac{C_d B}{N^3} \cdot \big( \Vmars(\fazeronu) \big)^2,
\end{align*}
where the second inequality is from the Cauchy inequality.
\end{proof}

\subsubsection{Proof of Theorem \ref{thm:han-wellner-prop2-variant}}\label{pf:han-wellner-prop2-variant}
Here we almost follow the proof of \cite{han2019convergence}, Proposition 2 and make a few modifications to the proof.
As in the proof of \cite{han2019convergence}, Proposition 2, we also use the standard peeling argument along with the following moment inequality for empirical processes.

\begin{lemma}[\cite{gine2000exponential}, Proposition 3.1]\label{lem:Hoffmann-moment-inequality}
Let $\F$ be a countable collection of real-valued functions defined on $\mathcal{X}$.
Suppose $X^{(1)}, \dots, X^{(n)}$ are i.i.d. random variables with law $P$ on $\mathcal{X}$ and $\xi_1, \dots, \xi_n$ are independent mean zero random variables independent of $X^{(1)}, \dots, X^{(n)}$. 
Then, there exists a universal positive constant $C$ such that
\begin{align*}
    &\E \bigg[\sup_{f \in \F} \Big| \sum_{i = 1}^{n} \xi_i f(X^{(i)}) \Big|^p \bigg] \le C^p \bigg[ \bigg( \E \bigg[\sup_{f \in \F} \Big| \sum_{i = 1}^{n} \xi_i f(X^{(i)}) \Big| \bigg] \bigg)^p \\
    &\qquad \qquad \qquad + p^{\frac{p}{2}} n^{\frac{p}{2}} \Big( \sup_{f \in \F} \| f \|_{P, 2} \Big)^{p} \cdot \max_i \| \xi_i \|_2^p + p^p \E \Big[ \max_i |\xi_i|^p \cdot \sup_{f \in \F} |f(X^{(i)})|^p \Big]\bigg]
\end{align*}
for every $p \ge 1$.
\end{lemma}

\begin{proof}[Proof of Theorem \ref{thm:han-wellner-prop2-variant}]
Let $A$ be the event where $\hat{f}$ becomes a least squares estimator over $\F$.
By the assumption, we have $\P(A) \ge 1 - \epsilon$.
Fix $r \ge 1$ and let
\begin{equation*}
    \F_j = \{f \in \F: 2^{j - 1} r t_n < \| f - f_0 \|_{P, 2} \le 2^{j} r t_n\}
\end{equation*} 
for each positive integer $j$. 
Then, clearly,
\begin{equation*}
    \{f \in \F: \| f - f_0 \|_{P, 2} > r t_n\} = \biguplus_{j = 1}^{\infty} \F_j,
\end{equation*}
and thus
\begin{align}\label{eq:decomposition-into-fj}
\begin{split}
    \P(\| \hat{f} - f_0 \|_{P, 2} > rt_n) &\le \P(A^c) + \P(A \ \text{and} \ \|\hat{f} - f_0 \|_{P, 2} > rt_n) \\
    &\le \epsilon + \sum_{j = 1}^{\infty} \P(A \ \text{and} \ \hat{f} \in \F_j).
\end{split}
\end{align}
Recall that $\biguplus$ indicates disjoint union.
Let $(\mathbb{M}_n(f): f \in \F)$ denote the stochastic process defined by
\begin{equation*}
    \mathbb{M}_n(f) = \frac{2}{n} \sum_{i = 1}^{n} \xi_i (f - f^*)(X^{(i)}) - \frac{1}{n} \sum_{i = 1}^{n} \big((f - f^*) (X^{(i)})\big)^2
\end{equation*}
and let $(M(f): f \in \F)$ denote the deterministic process defined by 
\begin{equation*}
    M(f) = \E[\mathbb{M}_n(f)] = - \|f - f^* \|_{P, 2}^2.
\end{equation*}
Note that $\mathbb{M}_n(f)$ can be alternatively represented as 
\begin{equation*}
    \mathbb{M}_n(f) = - \frac{1}{n} \sum_{i = 1}^{n} \big(y_i - f(X^{(i)})\big)^2 + \frac{1}{n} \sum_{i = 1}^{n} \xi_i^2.
\end{equation*}
Because $f_0 \in \F$, if $\hat{f}$ is a least squares estimator over $\F$, then it follows that
\begin{equation*}
    \mathbb{M}_n(\hat{f}) - \mathbb{M}_n(f_0) \ge 0.
\end{equation*}
Now, observe that 
\begin{align*}
    \mathbb{M}_n(f) - \mathbb{M}_n(f_0) &= \frac{2}{n} \sum_{i = 1}^{n} \xi_i (f - f_0)(X^{(i)}) - \frac{1}{n} \sum_{i = 1}^{n} \big((f - f_0) (X^{(i)})\big)^2 \\ 
    &\qquad - \frac{2}{n} \sum_{i = 1}^{n} (f - f_0) (X^{(i)}) \cdot (f_0 - f^*) (X^{(i)}) 
\end{align*}
and that 
\begin{align*}
    M(f) - M(f_0) &= - \|f - f^* \|_{P, 2}^2 + \|f_0 - f^* \|_{P, 2}^2 \\
    &= - \|f - f_0 \|_{P, 2}^2 - 2 \E_{X \sim P} \big[(f - f_0)(X) \cdot (f_0 - f^*) (X)\big],
\end{align*}
where $X$ in the expectation is independent of all $X^{(i)}$.
By the assumption that $\| f_0 - f^* \|_{P, 2} \le t_n/4$, we thus have that for every $f \in \F_j$
\begin{align*}
  - (M(f) - M(f_0)) &= \|f - f_0 \|_{P, 2}^2 + 2 \E_{X \sim P} \big[(f - f_0)(X) \cdot (f_0 - f^*) (X)\big] \\
  &\ge \|f - f_0 \|_{P, 2} \cdot \big(\|f - f_0 \|_{P, 2} - 2 \|f_0 - f^* \|_{P, 2}\big) \\
  &\ge 2^{j - 1} rt_n \Big(2^{j - 1} rt_n - \frac{t_n}{2}\Big) \ge 2^{2j - 3} r^2 t_n^2,
\end{align*}
where the first inequality is from the Cauchy inequality.
Hence, for each positive integer $j$, 
\begin{align}\label{peeling-results}
    \P(A \ \text{and} \ &\hat{f} \in \F_j) \le \P\Big(\sup_{f \in \F_j} 
    \big(\mathbb{M}_n(f) - \mathbb{M}_n(f_0)\big) \ge 0 \Big) \nonumber \\
    &\le \P\Big(\sup_{f \in \F_j} 
    \big((\mathbb{M}_n(f) - \mathbb{M}_n(f_0)) - (M(f) - M(f_0))\big) \ge 2^{2j - 3} r^2 t_n^2 \Big) \nonumber \\
    &\le \P \bigg(\sup_{f \in \F_j} \Big| \frac{1}{n} \sum_{i = 1}^{n} \xi_i (f - f_0)(X^{(i)}) \Big| \ge 2^{2j - 5} r^2 t_n^2 \bigg) \nonumber \\
    &\quad + \P \bigg(\sup_{f \in \F_j} \Big| \frac{1}{n} \sum_{i = 1}^{n} \big((f - f_0)(X^{(i)})\big)^2 - \| f - f_0 \|_{P, 2}^2 \Big| \ge 2^{2j - 5} r^2 t_n^2 \bigg) \nonumber \\
    &\quad + \P \bigg(\sup_{f \in \F_j} \Big| \frac{1}{n} \sum_{i = 1}^{n} (f - f_0) (X^{(i)}) \cdot (f_0 - f^*) (X^{(i)}) \nonumber \\
    &\qquad \qquad \qquad \qquad - \E_{X \sim P} \big[(f - f_0)(X) \cdot (f_0 - f^*) (X)\big] \Big| \ge 2^{2j - 6} r^2 t_n^2 \bigg) \nonumber \\
    &\le \P \bigg(\sup_{\substack{f \in \F - \{f_0\}\\ \|f\|_{P, 2} \le 2^j r t_n}} \Big| \frac{1}{\sqrt{n}} \sum_{i = 1}^{n} \xi_i f(X^{(i)}) \Big| \ge 2^{2j - 5} r^2 \sqrt{n} t_n^2 \bigg) \nonumber \\
    &\quad + \P \bigg(\sup_{\substack{f \in \F - \{f_0\} \\ \|f\|_{P, 2} \le 2^j r t_n}} \Big| \frac{1}{\sqrt{n}} \sum_{i = 1}^{n} \big(f(X^{(i)})^2 - \| f \|_{P, 2}^2 \big)\Big| \ge 2^{2j - 5} r^2 \sqrt{n} t_n^2 \bigg) \\
    &\quad + \P \bigg(\sup_{\substack{f \in \F -\{f_0\} \\ \|f\|_{P, 2} \le 2^j r t_n}} \Big| \frac{1}{\sqrt{n}} \sum_{i = 1}^{n} \Big(f(X^{(i)}) \cdot (f_0 - f^*) (X^{(i)}) \nonumber \\
    &\qquad \qquad \qquad \qquad \qquad \quad- \E_{X \sim P} \big[f(X) \cdot (f_0 - f^*)(X)\big] \Big) \Big| \ge 2^{2j - 6} r^2 \sqrt{n} t_n^2 \bigg) \nonumber.
\end{align}
Here the second inequality is due to the fact that 
\begin{equation*}
    - (M(f) - M(f_0)) \ge 2^{2j - 3} r^2 t_n^2
\end{equation*}
for $f \in \F_j$, and the third inequality is from the triangle inequality.

We first bound the first term of \eqref{peeling-results}. 
Fix a positive integer $j$.
By Markov's inequality, 
\begin{align}\label{first-term-Markov-inequality}
\begin{split}
    &\P \bigg(\sup_{\substack{f \in \F - \{f_0\}\\ \|f\|_{P, 2} \le 2^j r t_n}} \Big| \frac{1}{\sqrt{n}} \sum_{i = 1}^{n} \xi_i f(X^{(i)}) \Big| \ge 2^{2j - 5} r^2 \sqrt{n} t_n^2 \bigg) \\
    &\qquad \quad \le \frac{1}{(2^{2j - 5} r^2 \sqrt{n} t_n^2)^3} \cdot \E \bigg[\sup_{\substack{f \in \F - \{f_0\}\\ \|f\|_{P, 2} \le 2^j r t_n}} \Big| \frac{1}{\sqrt{n}} \sum_{i = 1}^{n} \xi_i f(X^{(i)}) \Big|^3 \bigg].
\end{split}
\end{align}
Also, by Lemma \ref{lem:Hoffmann-moment-inequality} (the moment inequality for empirical processes), we have 
\begin{align*}
    &\E \bigg[\sup_{\substack{f \in \F - \{f_0\} \\ \|f\|_{P, 2} \le 2^j r t_n}} \Big| \frac{1}{\sqrt{n}} \sum_{i = 1}^{n} \xi_i f(X^{(i)}) \Big|^3 \bigg] \le C \bigg( \E \bigg[\sup_{\substack{f \in \F - \{f_0\} \\ \|f\|_{P, 2} \le 2^j r t_n}} \Big| \frac{1}{\sqrt{n}} \sum_{i = 1}^{n} \xi_i f(X^{(i)}) \Big| \bigg] \bigg)^3 \\
    &\qquad \qquad \qquad \qquad \quad + C 2^{3j} r^3 t_n^3 \cdot \| \xi_1 \|_2^3 + C n^{-\frac{3}{2}} \cdot \E \bigg[ \max_{i} |\xi_i|^3 \cdot \sup_{\substack{f \in \F - \{f_0\} \\ \|f\|_{P, 2} \le 2^j r t_n}} |f(X^{(i)})|^3 \bigg] \\
    &\qquad \quad \le C 2^{3j} r^3 n^{\frac{3}{2}} t_n^6 + C 2^{3j} r^3 t_n^3 \cdot \| \xi_1 \|_2^3 + C M^3 n^{-\frac{3}{2}} \cdot \E\big[\max_{i} |\xi_i|^3 \big].
\end{align*}
Since $\xi_i$ have finite $L^q$ norm, 
\begin{equation*}
    \E\big[\max_{i} |\xi_i|^3 \big] = \E\big[\big(\max_{i} |\xi_i|^q \big)^{\frac{3}{q}}\big] \le \Big(\E\big[\max_{i} |\xi_i|^q \big] \Big)^{\frac{3}{q}} \le \bigg(\sum_{i = 1}^{n} \E\big[|\xi_i|^q \big] \bigg)^{\frac{3}{q}} = n^{\frac{3}{q}} \| \xi_1 \|_{q}^3,
\end{equation*}
where the first inequality is from Jensen's inequality.
Hence, we have 
\begin{align*}
    &\E \bigg[\sup_{\substack{f \in \F - \{f_0\} \\ \|f\|_{P, 2} \le 2^j r t_n}} \Big| \frac{1}{\sqrt{n}} \sum_{i = 1}^{n} \xi_i f(X^{(i)}) \Big|^3 \bigg] \\
    &\qquad \qquad \le C 2^{3j} r^3 n^{\frac{3}{2}} t_n^6 + C 2^{3j} r^3 t_n^3 \cdot \| \xi_1 \|_2^3 + C M^3 n^{3(-\frac{1}{2} + \frac{1}{q})} \| \xi_1 \|_{q}^3,
\end{align*}
which together with \eqref{first-term-Markov-inequality} implies that 
\begin{align*}
    &\P \bigg(\sup_{\substack{f \in \F - \{f_0\}\\ \|f\|_{P, 2} \le 2^j r t_n}} \Big| \frac{1}{\sqrt{n}} \sum_{i = 1}^{n} \xi_i f(X^{(i)}) \Big| \ge 2^{2j - 5} r^2 \sqrt{n} t_n^2 \bigg) \\
    &\qquad \qquad \le
    C \cdot \frac{1}{2^{3j} r^3} + C \cdot \frac{\|\xi_1 \|_2^3}{2^{3j} r^3 n^{\frac{3}{2}} t_n^3} + C \cdot \frac{M^3 \| \xi_1 \|_{q}^3}{2^{6j} r^6 n^{3(1 - \frac{1}{q})} t_n^6}.
\end{align*}

Now, we bound the second term of \eqref{peeling-results}.
We start by applying Markov's inequality as in the case of the first term: 
\begin{align}\label{second-term-Markov-inequality}
\begin{split}
    &\P \bigg(\sup_{\substack{f \in \F - \{ f_0 \}\\ \|f\|_{P, 2} \le 2^j r t_n}} \Big| \frac{1}{\sqrt{n}} \sum_{i = 1}^{n} \big(f(X^{(i)})^2 - \| f \|_{P, 2}^2 \big)\Big| \ge 2^{2j - 5} r^2 \sqrt{n} t_n^2 \bigg) \\
    &\qquad \quad \le \frac{1}{(2^{2j - 5} r^2 \sqrt{n} t_n^2)^3} \cdot \E \bigg[\sup_{\substack{f \in \F - \{f_0\} \\ \|f\|_{P, 2} \le 2^j r t_n}} \Big| \frac{1}{\sqrt{n}} \sum_{i = 1}^{n} \big(f(X^{(i)})^2 - \| f \|_{P, 2}^2 \big)\Big|^3 \bigg].
\end{split}
\end{align}
Using the standard argument of symmetrization (see, e.g., \cite{vaartwellner96book}, Lemma 2.3.1 or \cite{van2016estimation}, Theorem 16.1), we can show that 
\begin{equation*}
    \E \bigg[\sup_{\substack{f \in \F - \{f_0\} \\ \|f\|_{P, 2} \le 2^j r t_n}} \Big| \frac{1}{\sqrt{n}} \sum_{i = 1}^{n} \big(f(X^{(i)})^2 - \| f \|_{P, 2}^2 \big)\Big|^3 \bigg] \le 8 \E \bigg[\sup_{\substack{f \in \F - \{f_0\} \\ \|f\|_{P, 2} \le 2^j r t_n}} \Big| \frac{1}{\sqrt{n}} \sum_{i = 1}^{n} \epsilon_i f(X^{(i)})^2 \Big|^3 \bigg].
\end{equation*}
Also, by Lemma \ref{lem:Hoffmann-moment-inequality}, we have  
\begin{align*}
    &\E \bigg[\sup_{\substack{f \in \F - \{f_0\} \\ \|f\|_{P, 2} \le 2^j r t_n}} \Big| \frac{1}{\sqrt{n}} \sum_{i = 1}^{n} \epsilon_i f(X^{(i)})^2 \Big|^3 \bigg] \le C \bigg( \E \bigg[\sup_{\substack{f \in \F - \{f_0\} \\ \|f\|_{P, 2} \le 2^j r t_n}} \Big| \frac{1}{\sqrt{n}} \sum_{i = 1}^{n} \epsilon_i f(X^{(i)})^2 \Big| \bigg] \bigg)^3 \\
    &\qquad \qquad \qquad \qquad + C \bigg(\sup_{\substack{f \in \F - \{f_0\} \\ \|f\|_{P, 2} \le 2^j r t_n}} \|f^2\|_{P,2}\bigg)^3 + C n^{-\frac{3}{2}} \cdot \E \bigg[  \sup_{\substack{f \in \F - \{f_0\} \\ \|f\|_{P, 2} \le 2^j r t_n}} |f (X^{(i)})|^6 \bigg].
\end{align*}
Due to the contraction principle (see, e.g., \cite{vaartwellner96book}, Proposition A.3.2), 
\begin{align*}
    \E \bigg[\sup_{\substack{f \in \F - \{f_0\} \\ \|f\|_{P, 2} \le 2^j r t_n}} \Big| \frac{1}{\sqrt{n}} \sum_{i = 1}^{n} \epsilon_i f(X^{(i)})^2 \Big| \bigg] &\le 4 M \cdot \E \bigg[\sup_{\substack{f \in \F - \{f_0\} \\ \|f\|_{P, 2} \le 2^j r t_n}} \Big| \frac{1}{\sqrt{n}} \sum_{i = 1}^{n} \epsilon_i f(X^{(i)}) \Big| \bigg] \\
    &\le 2^{j + 2} r M \sqrt{n} t_n^2.
\end{align*}
Moreover, since $\|f\|_{\infty} \le M$ for every $f \in \F - \{f_0\}$, 
\begin{equation*}
    \sup_{\substack{f \in \F - \{f_0\} \\ \|f\|_{P, 2} \le 2^j r t_n}} \|f^2\|_{P,2} \le M \cdot \sup_{\substack{f \in \F - \{f_0\} \\ \|f\|_{P, 2} \le 2^j r t_n}} \|f\|_{P,2} \le 2^{j} r M t_n.
\end{equation*}
Therefore, we have 
\begin{equation*}
    \E \bigg[\sup_{\substack{f \in \F - \{f_0\} \\ \|f\|_{P, 2} \le 2^j r t_n}} \Big| \frac{1}{\sqrt{n}} \sum_{i = 1}^{n} \epsilon_i f(X^{(i)})^2 \Big|^3 \bigg] \le C 2^{3j} r^3 M^3 n^{\frac{3}{2}} t_n^6 + C 2^{3j} r^3 M^3 t_n^3 + C M^6 n^{-\frac{3}{2}}, 
\end{equation*}
which directly leads to   
\begin{align*}
    &\P \bigg(\sup_{\substack{f \in \F - \{f_0\} \\ \|f\|_{P, 2} \le 2^j r t_n}} \Big| \frac{1}{\sqrt{n}} \sum_{i = 1}^{n} \big(f(X^{(i)})^2 - \| f \|_{P, 2}^2 \big)\Big| \ge 2^{2j - 5} r^2 \sqrt{n} t_n^2 \bigg) \\
    &\qquad \quad \le C \cdot \frac{M^3}{2^{3j} r^3} + C \cdot \frac{M^3}{2^{3j} r^3 n^{\frac{3}{2}} t_n^3} + C \cdot \frac{M^6}{2^{6j} r^6 n^3 t_n^6},
\end{align*}
because of \eqref{second-term-Markov-inequality}.

The third term of \eqref{peeling-results} can be bounded in the same way as the second term. 
Indeed, by repeating the same argument (except for the application of the contraction principle), we can show that
\begin{align*}
    &\P \bigg(\sup_{\substack{f \in \F - \{f_0\}\\ \|f\|_{P, 2} \le 2^j r t_n}} \Big| \frac{1}{\sqrt{n}} \sum_{i = 1}^{n} \Big(f(X^{(i)}) \cdot (f_0 - f^*) (X^{(i)}) \\
    &\qquad \qquad \qquad \qquad \qquad \quad - \E_{X \sim P} \big[f(X) \cdot (f_0 - f^*)(X)\big] \Big) \Big| \ge 2^{2j - 6} r^2 \sqrt{n} t_n^2 \bigg) \\
    &\qquad \quad \le C \cdot \frac{1}{2^{3j} r^3} + C \cdot \frac{\|f_0 - f^* \|_{\infty}^3}{2^{3j} r^3 n^{\frac{3}{2}} t_n^3} + C \cdot \frac{M^3 \cdot \|f_0 - f^* \|_{\infty}^3}{2^{6j} r^6 n^3 t_n^6}.
\end{align*}

As a result, by \eqref{eq:decomposition-into-fj} and \eqref{peeling-results},
\begin{align*}
    \P(\| \hat{f} - f_0 \|_{P, 2} > r t_n) \le \epsilon &+ \sum_{j = 1}^{\infty} \bigg[ C \cdot \frac{1 + M^3}{2^{3j} r^3} + C \cdot \frac{\|\xi_1\|_2^3 + \|f_0 - f^*\|_{\infty}^3 + M^3}{2^{3j} r^3 n^{\frac{3}{2}} t_n^3} \\
    &\qquad \qquad \qquad + C \cdot \frac{M^3(\|\xi_1\|_q^3 n^{\frac{3}{q}} + \|f_0 - f^*\|_{\infty}^3 + M^3)}{2^{6j} r^6 n^3 t_n^6}  \bigg] \\
    \le \epsilon &+ C \cdot \frac{1 + M^3}{r^3} + C \cdot \frac{\|\xi_1\|_2^3 + \|f_0 - f^*\|_{\infty}^3 + M^3}{r^3 n^{\frac{3}{2}} t_n^3} \\
    &+ C \cdot \frac{M^3(\|\xi_1\|_q^3 n^{\frac{3}{q}} + \|f_0 - f^*\|_{\infty}^3 + M^3)}{r^6 n^3 t_n^6},
\end{align*}
and replacing $r$ with $t/t_n$ completes the proof.
\end{proof}

\subsubsection{Proof of Lemma \ref{lem:sup-norm-bound}}\label{pf:sup-norm-bound}
\begin{proof}[Proof of Lemma \ref{lem:sup-norm-bound}]
We first observe that every function $f \in \infmars^{d, s}$ can be uniquely split into a multi-affine term and the remaining term as follows.
Let $\A$ be the collection of all the multi-affine functions 
\begin{equation*}
    a(x_1, \dots, x_d) = \sum_{\substack{\alpha \in \{0, 1\}^d \\ |\alpha| \le s}} a_{\alpha} \prod_{j \in S(\alpha)} x_j,
\end{equation*}
where $a_{\alpha} \in \R$ for each $\alpha \in \{0, 1\}^d$ with $|\alpha| \leq s$, 
and let $\H$ be the collection of all the functions
\begin{equation*}
    h(x_1, \dots, x_d) = \sumoveralpha
    \int_{[0, 1)^{|\alpha|} \setminus \{\zerovec\}} \prod_{j \in S(\alpha)} (x_j -
    t_j)_+ \, d\nu_{\alpha}(\talpha),
\end{equation*}
where $\nu_{\alpha}$ is a finite signed measure on $[0, 1)^{|\alpha|} \setminus \{\zerovec\}$ for each $\alpha \in \{0, 1\}^d \setminus \{\zerovec\}$ with $|\alpha| \leq s$.
By the uniqueness of representation (see Lemma \ref{lem:uniqueness-of-representation}), clearly, we have
\begin{equation*}
    \infmars^{d, s} = \A \oplus \H,
\end{equation*}
where $\oplus$ indicates direct sum.
Also, note that if $f = a + h$ where $a \in \A$ and $h \in \H$ are in the forms above, then
\begin{equation*}
    \| h \|_{\infty} \le \sumoveralpha |\nu_{\alpha}|([0, 1)^{|\alpha|} \setminus \{\zerovec\}) = \Vmars(f).
\end{equation*}

Write $\hat{f}^{d, s}_{n, V} = \hat{a}^{d, s}_{n, V} + \hat{h}^{d, s}_{n, V}$ and $f^* = a^* + h^*$ where $\hat{a}^{d, s}_{n, V}, a^* \in \A$ and $\hat{h}^{d, s}_{n, V}, h^* \in \H$.
Because
\begin{equation*}
\| \hat{h}^{d, s}_{n, V} \|_{\infty} \le   \Vmars\big(\hat{f}^{d, s}_{n, V}\big) = V \ \text{ and } \ \| h^* \|_{\infty} \le   \Vmars(f^*) = V, 
\end{equation*}
it follows that
\begin{align*}
    \|\hat{f}^{d, s}_{n, V} - f^* \|_{\infty} &\le \|\hat{a}^{d, s}_{n, V} - a^* \|_{\infty} + \|\hat{h}^{d, s}_{n, V} \|_{\infty} + \| h^* \|_{\infty} \\
    &\le \|\hat{a}^{d, s}_{n, V} - a^* \|_{\infty} + 2V.
\end{align*}
Hence, in order to prove the lemma, it suffices to show that 
\begin{equation*}
    \| \hat{a}^{d, s}_{n, V} - a^* \|_{\infty} = O_p(1).
\end{equation*}

Recall from \eqref{our-problem-restated} that there is no constraint on multi-affine terms in our definition of $\hat{f}^{d, s}_{n, V}$.  
Thus, we can characterize $\hat{a}^{d, s}_{n, V}$ as 
\begin{align}\label{eq:characterization-of-a-function}
\begin{split}
    \hat{a}^{d, s}_{n, V} &= \argmin_{a \in \A} \bigg\{\sum_{i = 1}^{n} \big(y_i - \hat{h}^{d, s}_{n, V}(X^{(i)}) -a(X^{(i)})\big)^2 \bigg\} \\
    &= \argmin_{a \in \A} \bigg\{\sum_{i = 1}^{n} \big(\xi_i - (\hat{h}^{d, s}_{n, V} - h^*)(X^{(i)}) - (a - a^*)(X^{(i)})\big)^2 \bigg\}.
\end{split}
\end{align}
Let $\underline{X}$ be the matrix whose rows and columns are indexed by $i \in [n]$ and $\alpha \in \{0, 1\}^d$ with $|\alpha| \le s$, where $\underline{X}_{i\zerovec} = 1$ and 
\begin{equation*}
    \underline{X}_{i \alpha} = \prod_{j \in S(\alpha)} X^{(i)}_j
\end{equation*}
for $\alpha \neq \zerovec$.
Also, consider the coefficient vectors $\hat{\underline{a}}$ and $\underline{a}^*$ of the multi-affine functions $\hat{a}^{d, s}_{n, V}$ and $a^*$.
Specifically, 
\begin{equation*}
    \hat{\underline{a}} = (\hat{a}_{\alpha}, \alpha \in \{0, 1\}^d \text{ and } |\alpha| \le s)
\end{equation*}
and 
\begin{equation*}
    \underline{a}^* = (a^*_{\alpha}, \alpha \in \{0, 1\}^d \text{ and } |\alpha| \le s),
\end{equation*}
where 
\begin{equation*}
    \hat{a}^{d, s}_{n, V}(x_1, \dots, x_d) = \sum_{\substack{\alpha \in \{0, 1\}^d \\ |\alpha| \le s}} \hat{a}_{\alpha} \prod_{j \in S(\alpha)} x_j
\end{equation*}
and
\begin{equation*}
    a^*(x_1, \dots, x_d) = \sum_{\substack{\alpha \in \{0, 1\}^d \\ |\alpha| \le s}} a^*_{\alpha} \prod_{j \in S(\alpha)} x_j.
\end{equation*}
Moreover, let $\hat{\underline{h}}$ and $\underline{h}^*$ be the $n$-dimensional vectors for which
\begin{equation*}
    \hat{\underline{h}}_i = \hat{h}^{d, s}_{n, V}(X^{(i)}) \ \text{ and } \ \underline{h}^*_i = h^*(X^{(i)})
\end{equation*}
for $i \in  [n]$.
Then, it follows from \eqref{eq:characterization-of-a-function} that
\begin{equation}\label{eq:characterization-of-a-vector}
    \hat{\underline{a}} = \argmin_{\underline{a}} \big\{\| \xi - (\hat{\underline{h}} - \underline{h}^*) - \underline{X} (\underline{a} - \underline{a}^*) \|_2^2:  \underline{a} = (a_{\alpha}, \alpha \in \{0, 1\}^d \text{ and } |\alpha| \le s) \big\},
\end{equation}
where $\xi := (\xi_i, i \in [n])$.

Now, let $\Pi_{\underline{X}}$ be the projection operator onto the column space of $\underline{X}$.
Then, \eqref{eq:characterization-of-a-vector} implies that
\begin{equation*}
    \underline{X} (\hat{\underline{a}} - \underline{a}^*) = \Pi_{\underline{X}}\big(\xi - (\hat{\underline{h}} - \underline{h}^*)\big).
\end{equation*}
Since $\Pi_{\underline{X}}$ is a projection operator, 
\begin{align*}
    \big\|\Pi_{\underline{X}}\big(\xi - (\hat{\underline{h}} - \underline{h}^*)\big)\big\|_2 &\le \|\xi - (\hat{\underline{h}} - \underline{h}^*)\|_2 \le \bigg( \sum_{i = 1}^{n} \xi_i^2 \bigg)^{\frac{1}{2}} + \sqrt{n} \big(\|\hat{h}^{d, s}_{n, V}\|_{\infty} + \| h^*\|_{\infty}\big) \\
    &\le \bigg( \sum_{i = 1}^{n} \xi_i^2 \bigg)^{\frac{1}{2}} + 2V \cdot \sqrt{n}.
\end{align*}
Also, we have that
\begin{equation*}
    \|\underline{X} (\hat{\underline{a}} - \underline{a}^*)\|_2^2 = (\hat{\underline{a}} - \underline{a}^*)^T \underline{X}^T \underline{X} (\hat{\underline{a}} - \underline{a}^*) \ge \lambda_{\text{min}} (\underline{X}^T \underline{X}) \cdot \|\hat{\underline{a}} - \underline{a}^*\|_2^2,
\end{equation*}
where $\lambda_{\text{min}} (\underline{X}^T \underline{X})$ is the smallest eigenvalue of $\underline{X}^T \underline{X}$.
Hence, combining these results, we can derive that 
\begin{equation}\label{eq:a-vector-inequality}
    \big(\lambda_{\text{min}} (\underline{X}^T \underline{X} / n)\big)^{\frac{1}{2}} \cdot \|\hat{\underline{a}} - \underline{a}^*\|_2 \le \bigg( \frac{1}{n} \sum_{i = 1}^{n} \xi_i^2 \bigg)^{\frac{1}{2}} + 2V.
\end{equation}

Consider the symmetric matrix $\Sigma$ whose rows and columns are indexed by $\alpha \in \{0, 1\}^d$ with $|\alpha| \le s$, where 
\begin{equation*}
    \Sigma_{\alpha \alpha'} = \E\bigg[ \prod_{j \in S(\alpha)} X^{(1)}_j \cdot \prod_{j' \in S(\alpha')} X^{(1)}_{j'} \bigg].
\end{equation*}
By the law of large numbers, for every $\epsilon > 0$,
\begin{equation*}
    \lim_{n \rightarrow \infty} \P \big(\| \underline{X}^T \underline{X} / n - \Sigma \|_F > \epsilon \big) = 0,
\end{equation*}
where $\| \cdot \|_F$ denotes the Frobenius norm.
Also,  
\begin{equation*}
    \big|\lambda_{\text{min}}(\underline{X}^T \underline{X} / n) - \lambda_{\text{min}} (\Sigma)\big| \le \| \underline{X}^T \underline{X} / n - \Sigma \|_F
\end{equation*}
due to Weyl's inequality, and the following lemma, which we prove in Appendix \ref{pf:smallest-eigenvalue}, ensures that $\lambda_{\text{min}} (\Sigma) > 0$.

\begin{lemma}\label{lem:smallest-eigenvalue}
The smallest eigenvalue of $\Sigma$ is positive, i.e., $\lambda_{\text{min}}(\Sigma) > 0$.
\end{lemma}

For these reasons,  
\begin{align*}
    \P\Big(\lambda_{\text{min}}(\underline{X}^T \underline{X} / n) < \frac{\lambda_{\text{min}}(\Sigma)}{2}\Big) &\le \P\Big(\big|\lambda_{\text{min}}(\underline{X}^T \underline{X} / n) - \lambda_{\text{min}} (\Sigma)\big| > \frac{\lambda_{\text{min}}(\Sigma)}{2} \Big) \\
    &\le \P\Big(\| \underline{X}^T \underline{X} / n - \Sigma \|_F > \frac{\lambda_{\text{min}}(\Sigma)}{2} \Big),
\end{align*}
which leads to that 
\begin{equation*}
    \lim_{n \rightarrow \infty} \P\Big(\lambda_{\text{min}}(\underline{X}^T \underline{X} / n) < \frac{\lambda_{\text{min}}(\Sigma)}{2}\Big) = 0.
\end{equation*}
Consequently, by \eqref{eq:a-vector-inequality}, we have
\begin{align*}
    \limsup_{n \rightarrow \infty} \P\big(\|\hat{\underline{a}} - \underline{a}^* \|_2 > K\big) &\le \lim_{n \rightarrow \infty} \P\Big(\lambda_{\text{min}}(\underline{X}^T \underline{X} / n) < \frac{\lambda_{\text{min}}(\Sigma)}{2}\Big) \\
    &\qquad + \limsup_{n \rightarrow \infty} \P\bigg(\frac{1}{n} \sum_{i = 1}^{n} \xi_i^2 > \Big\{ \Big( \frac{\lambda_{\text{min}}(\Sigma)}{2} \Big)^{\frac{1}{2}} \cdot K - 2V \Big\}^2 \bigg) \\
    &\le \Big\{ \Big( \frac{\lambda_{\text{min}}(\Sigma)}{2} \Big)^{\frac{1}{2}} \cdot K - 2V \Big\}^{-2} \cdot \|\xi_1\|_2^2
\end{align*}
provided that $(\lambda_{\text{min}}(\Sigma)/2)^{1/2} \cdot K - 2V > 0$.
Here Markov's inequality is used for the second inequality.
This result implies that for each $\epsilon > 0$, there exists $K > 0$ such that 
\begin{equation*}
    \P\big(\|\hat{\underline{a}} - \underline{a}^* \|_2 > K\big) < \epsilon,
\end{equation*}
i.e., $\|\hat{\underline{a}} - \underline{a}^* \|_2 = O_p(1)$.
By the Cauchy inequality, 
\begin{equation*}
    \|\hat{a}^{d, s}_{n, V} - a^*\|_{\infty} \le \|\hat{\underline{a}} - \underline{a}^* \|_1 \le C_d \cdot \|\hat{\underline{a}} - \underline{a}^* \|_2,
\end{equation*}
and thus, it follows that $\|\hat{a}^{d, s}_{n, V} - a^* \|_{\infty} = O_p(1)$.
\end{proof}

\subsubsection{Proof of Lemma \ref{lem:maximal-inequality-bracketing-fvm}}\label{pf:maximal-inequality-bracketing-fvm}
Our proof of Lemma \ref{lem:maximal-inequality-bracketing-fvm} is based on the following maximal inequality for empirical processes involving the Bernstein norm. 
For a random variable $X$ with law $P$ on $\mathcal{X}$ and a real-valued function $f$ on $\mathcal{X}$, the Bernstein norm $\|f\|_{P, B}$ of $f$ is defined as 
\begin{equation*}
    \| f \|_{P, B}  = \big(2 \E_{P} \big[\exp(|f(X)|) - 1 - |f(X)|\big] \big)^{\frac{1}{2}}.
\end{equation*}
In fact, the Bernstein norm is not a norm because it is not homogeneous and does not satisfy the triangle inequality. 
However, we can still use it for measuring the size of functions.

\begin{lemma}[\cite{vaartwellner96book}, Lemma 3.4.3]\label{lem:maximal-inequality-empirical-processes-vaart}
Suppose $X^{(1)}, \dots, X^{(k)}$ are i.i.d. random variables with law $P$ on $\mathcal{X}$, and $\F$ is a countable collection of real-valued functions on $\mathcal{X}$. 
Also, assume that $\|f\|_{P, B} \le \delta$ for every $f \in \F$. 
Then, we have
\begin{equation*}
    \E_{P} \Big[\sup_{f \in \F} \Big|\frac{1}{\sqrt{k}} \sum_{i = 1}^{k} f(X^{(i)}) \Big| \Big] \le C J_{[ \ ]}(\delta, \F, \|\cdot\|_{P, B}) \bigg(1 + \frac{J_{[ \ ]}(\delta, \F, \|\cdot\|_{P, B})}{\delta^2 \sqrt{k}} \bigg)
\end{equation*}
for some universal positive constant $C$.
\end{lemma}

\begin{remark}
As in Theorem \ref{thm:han-wellner-prop2-variant}, the countability assumption on $\F$ in Lemma \ref{lem:maximal-inequality-empirical-processes-vaart} is just for the measurability of the supremum.
This condition can be removed if $\F$ is pointwise measurable; 
that is, $\F$ has a countable subset $\G$ for which for every $f \in \F$, there exists a sequence $\{g_m\}_{m \ge 1}$ in $\G$ such that $g_m(x) \rightarrow f(x)$ for every $x \in \mathcal{X}$. 
In this case, 
\begin{equation*}
    \sup_{f \in \F} \Big|\frac{1}{\sqrt{k}} \sum_{i = 1}^{k} f(X^{(i)}) \Big| = \sup_{g \in \G} \Big|\frac{1}{\sqrt{k}} \sum_{i = 1}^{k} g(X^{(i)}) \Big|,
\end{equation*}
so that the measurability issue can be avoided, and the lemma is still valid without the assumption of countability.
\end{remark}

\begin{proof}[Proof of Lemma \ref{lem:maximal-inequality-bracketing-fvm}]
Let $P$ be the law of $(X^{(i)}, \epsilon_i)$ on $[0, 1]^d \times \R$. 
Also, let $\mathcal{A}_t$ be the collection of all the functions on $[0, 1]^d \times \R$ of the form
\begin{equation*}
    \Phi_f(X, \epsilon) = \frac{1}{2 M} \cdot \epsilon f(X) \qt{ for $X \in [0, 1]^d$ and $\epsilon \in \R$,}
\end{equation*}
where $f \in \F_M(V) - \{f^*\}$ and $\|f\|_{p_0, 2} \le t$, i.e., $f \in B_{\F_M(V) - \{f^*\}}(t, 0, \| \cdot \|_{p_0, 2}) := B(t, 0, \| \cdot \|_{p_0, 2})$.
Since $B(t, 0, \| \cdot \|_{p_0, 2})$ is a subset of $C([0, 1]^d)$, and $(C([0,1]^d), \|\cdot\|_{\infty})$ is a separable metric space, $B(t, 0, \| \cdot \|_{p_0, 2})$ is also separable with respect to the sup-norm. 
Let $\G$ be a countable dense subset of $B(t, 0, \| \cdot \|_{p_0, 2})$ with respect to the sup-norm. 
For every $\Phi_{f} \in \mathcal{A}_t$, if we let $\{g_m\}_{m \ge 1}$ be a sequence in $\G$ for which $\|g_m - f\|_{\infty} \rightarrow 0$, then we have
\begin{equation*}
    \Phi_{g_m}(X, \epsilon) - \Phi_f(X, \epsilon) = \frac{1}{2 M} \cdot \epsilon (g_m - f)(X) \rightarrow 0 
\end{equation*}
for every $(X, \epsilon) \in [0, 1]^d \times \R$.
Hence, $\mathcal{A}_t$ is pointwise measurable, and Lemma \ref{lem:maximal-inequality-empirical-processes-vaart} is valid for $\mathcal{A}_t$.

Note that for $\Phi_f \in \mathcal{A}_t$,
\begin{align}\label{eq:bernstein-norm-bound}
\begin{split}
    \| \Phi_f \|_{P, B} &= \bigg(2 \E_{P} \Big[\exp\Big(\Big|\frac{1}{2 M}\cdot \epsilon f(X)\Big|\Big) - 1 - \Big|\frac{1}{2M}\cdot \epsilon f(X)\Big|\Big] \bigg)^{\frac{1}{2}} \\
    &= \bigg(2 \sum_{m = 2}^{\infty} \frac{1}{m!} \cdot \E_{P} \Big[ \Big(\frac{|f(X)|}{2M} \Big)^{m} \Big] \bigg)^{\frac{1}{2}} \\
    &\le \bigg(2 \sum_{m = 2}^{\infty} \frac{1}{m!} \cdot \E_{P} \Big[ \Big(\frac{|f(X)|}{2M} \Big)^{2} \Big] \bigg)^{\frac{1}{2}} \le \frac{c_0 t}{2M},
\end{split}
\end{align}
where $c_0 := (2(e - 2))^{1/2}$.
Here the first inequality is due to that 
\begin{equation*}
    \| f \|_{\infty} \le M 
\end{equation*}
for every $f \in \F_M(V) - \{f^*\}$, and the second inequality is from that $\|f\|_{p_0, 2} \le t$.
Thus, by Lemma \ref{lem:maximal-inequality-empirical-processes-vaart}, we have
\begin{align}\label{eq:maximal-inequality-bernstein-norm-result}
\begin{split}
    &\E\bigg[\sup_{\substack{f \in \F_M(V) - \{f^*\} \\ \|f\|_{p_0, 2} \le t} } \Big| \frac{1}{\sqrt{k}} \sum_{i=1}^{k} \epsilon_i f(X^{(i)}) \Big| \bigg] = 2 M \cdot \E\bigg[\sup_{\Phi_f \in \mathcal{A}_t} \Big| \frac{1}{\sqrt{k}} \sum_{i=1}^{k} \Phi_f (X^{(i)}, \epsilon_i) \Big| \bigg] \\
    &\qquad \quad \le C \cdot 2 M \cdot J_{[ \ ]}\Big(\frac{c_0 t}{2 M}, \mathcal{A}_t, \|\cdot\|_{P, B}\Big) \bigg(1 + \frac{J_{[ \ ]}(\frac{c_0 t}{2 M}, \mathcal{A}_t, \|\cdot\|_{P, B})}{(\frac{c_0 t}{2 M})^2 \sqrt{k}} \bigg).
\end{split}
\end{align}

Now, we relate the bracketing entropy of $\mathcal{A}_t$ with respect to the Bernstein norm to that of $\F_M(V) - \{f^*\}$ with respect to the norm $\| \cdot \|_{p_0, 2}$. 
Suppose $\Phi_f \in \mathcal{A}_t$ and let $[f_1, f_2]$ be a bracket containing $f$. 
Since $\|f\|_{\infty} \le M$, we can assume that $\|f_1\|_{\infty} \le M$ and $\|f_2\|_{\infty} \le M$.
Also, let $p$ and $q$ denote the functions 
\begin{equation*}
    p(\epsilon) = \epsilon_+ = \max \{\epsilon, 0\} \ \mbox{ and } \ q(\epsilon) = \epsilon_- = \max \{-\epsilon, 0\}.
\end{equation*}
Then, clearly, 
\begin{equation*}
    \Phi_f \in \Big[\frac{1}{2 M}\{ f_1 \otimes p - f_2 \otimes q \}, \frac{1}{2 M}\{ f_2 \otimes p - f_1 \otimes q \}\Big].
\end{equation*}
Recall that $\otimes$ denotes the tensor product between functions. 
For example, $f_1 \otimes p$ indicates the function defined by
\begin{equation*}
    (f_1 \otimes p) (X, \epsilon) = f_1(X) \cdot p(\epsilon) 
\end{equation*}
for $X \in [0, 1]^d$ and $\epsilon \in \R$.
Also, observe that by the same argument as in \eqref{eq:bernstein-norm-bound},
\begin{align*}
    &\Big\| \frac{1}{2 M}\{ f_2 \otimes p - f_1 \otimes q \} - \frac{1}{2 M}\{ f_1 \otimes p - f_2 \otimes q \} \Big\|_{P, B} \\
    &\qquad = \Big\| \frac{1}{2 M} \cdot (f_2 - f_1) \otimes (p + q)\Big\|_{P, B} \\
    &\qquad = \bigg(2 \E_{P} \Big[\exp\Big(\Big|\frac{1}{2 M}\cdot (f_2 - f_1)(X)\Big|\Big) - 1 - \Big|\frac{1}{2 M}\cdot (f_2 - f_1)(X)\Big|\Big] \bigg)^{\frac{1}{2}} \\
    &\qquad \le \frac{c_0}{2  M} \cdot \|f_2 - f_1\|_{p_0, 2}.
\end{align*}
For these reasons, for every $\epsilon > 0$, we have
\begin{equation*}
    N_{[ \ ]} \Big(\frac{c_0 \epsilon}{2 M}, \mathcal{A}_t, \| \cdot \|_{P, B} \Big) \le N_{[ \ ]}(\epsilon, \F_M(V) - \{f^*\}, \| \cdot \|_{p_0, 2}),
\end{equation*}
which implies that
\begin{align*}
\begin{split}
    J_{[ \ ]} \Big( \frac{c_0 t}{2  M}, \mathcal{A}_t, \| \cdot \|_{P, B} \Big) &= \int_{0}^{\frac{c_0 t}{2 M}} \sqrt{1 + N_{[ \ ]}(\epsilon, \mathcal{A}_t, \| \cdot \|_{P, B})} \, d\epsilon \\
    &= \frac{c_0}{2 M} \int_{0}^{t} \sqrt{1 + N_{[ \ ]} \Big(\frac{c_0\epsilon}{2 M}, \mathcal{A}_t, \| \cdot \|_{P, B} \Big)} \, d\epsilon \\
    &\le \frac{c_0}{2 M} \int_{0}^{t} \sqrt{1 + N_{[ \ ]}(\epsilon, \F_M(V) - \{f^*\}, \| \cdot \|_{p_0, 2} )} \, d\epsilon \\ 
    &= \frac{c_0}{2 M} \cdot J_{[ \ ]} (t, \F_M(V) - \{f^*\}, \| \cdot \|_{p_0, 2}).
\end{split}
\end{align*}
As a result, we can deduce from \eqref{eq:maximal-inequality-bernstein-norm-result} that 
\begin{align*}
    &\E\bigg[\sup_{\substack{f \in \F_M(V) - \{f^*\} \\ \|f\|_{p_0, 2} \le t} } \Big| \frac{1}{\sqrt{k}} \sum_{i=1}^{k} \epsilon_i f(X^{(i)}) \Big| \bigg] \\
    &\qquad \le C \cdot J_{[ \ ]} (t, \F_M(V) - \{f^*\}, \| \cdot \|_{p_0, 2}) \bigg(1 + M \cdot \frac{J_{[ \ ]} (t, \F_M(V) - \{f^*\}, \| \cdot \|_{p_0, 2})}{t^2 \sqrt{k}} \bigg).
\end{align*}
\end{proof}

\subsubsection{Proof of Lemma \ref{lem:fvm-bracketing-entropy}}\label{pf:fvm-bracketing-entropy}
\begin{proof}[Proof of Lemma \ref{lem:fvm-bracketing-entropy}]
Suppose $f = \fazeronu \in \F_M(V) - \{f^*\}$. 
By the definition of $\F_M(V)$, we have $\|f\|_{\infty} \le M$ and $\Vmars(f) \le V + \Vmars(f^*) \le 2V$.
Let $a_{\alpha} = \nu_{\alpha}(\{\zerovec\})$ for each $\alpha \in \{0, 1\}^d \setminus \{\zerovec\}$ with $|\alpha| \le s$. 
Then, the function $f$ can be written as  
\begin{align*}
    &f(x_1, \dots, x_d) = a_{\zerovec} + \sum_{\substack{\alpha \in \{0, 1\}^d \setminus \{\zerovec\} \\ |\alpha| \le s}} \int_{[0, 1)^{|\alpha|}} \prod_{j \in S(\alpha)} (x_j - s_j)_+ \, d\nu_{\alpha}(\salpha) \\   
    &\qquad \quad= \sum_{\substack{\alpha \in \{0, 1\}^d \\ |\alpha| \le s}} a_{\alpha} \prod_{j \in S(\alpha)} x_j + \sum_{\substack{\alpha \in \{0, 1\}^d \setminus \{\zerovec\} \\ |\alpha| \le s}} \int_{[0, 1)^{|\alpha|} \setminus \{\zerovec\}} \prod_{j \in S(\alpha)} (x_j - s_j)_+ \, d\nu_{\alpha}(\salpha).
\end{align*}
As in the proof of Lemma \ref{lem:svt-metric-entropy}, we will bound $a_{\alpha}$ for each $\alpha \in \{0, 1\}^d \setminus \{\zerovec\}$ with $|\alpha| \le s$ first.

Note that $|a_{\zerovec}| = |f(\zerovec)| \le M$.
Also, since 
\begin{equation*}
    f(1, 0, \dots, 0) = a_{\zerovec} + a_{1, 0, \dots, 0} + \int_{(0, 1)} (1 - s_1)_+ \, d\nu_{1, 0, \dots, 0}(s_1),
\end{equation*}
we have
\begin{align*}
    |a_{1, 0, \dots, 0}| &\le |f(1, 0, \dots, 0)| + |a_{\zerovec}| + \Big|\int_{(0, 1)} (1 - s_1)_+ \, d\nu_{1, 0, \dots, 0}(s_1)\Big| \\
    &\le |f(1, 0, \dots, 0)| + |a_{\zerovec}| + |\nu_{1, 0, \dots, 0}|((0, 1)) \le 2M + 2V.
\end{align*}
By the same argument, we can show that
\begin{equation*}
    |a_{\alpha}| \le 2M + 2V
\end{equation*}
for every $\alpha \in \{0, 1\}^d \setminus \{\zerovec\}$ with $|\alpha| = 1$.
Moreover, since 
\begin{align*}
    &f(1, 1, 0, \dots, 0) = a_{\zerovec} + a_{1, 0, \dots, 0} + a_{0, 1, 0, \dots, 0} + a_{1, 1, 0, \dots, 0} \\ 
    &\qquad \qquad \qquad \quad+ \int_{(0, 1)} (1 - s_1)_+ \, d\nu_{1, 0, \dots, 0}(s_1) + \int_{(0, 1)} (1 - s_2)_+ \, d\nu_{0, 1, 0, \dots, 0}(s_2) \\
    &\qquad \qquad \qquad \quad + \int_{[0, 1)^2 \setminus \{\zerovec\}} (1 - s_1)_+ (1 - s_2)_+ \, d\nu_{1, 1, 0, \dots, 0}(s_1, s_2), 
\end{align*}
it follows that
\begin{align*}
    |a_{1, 1, 0, \dots, 0}| &\le |f(1, 1, 0, \dots, 0)| + |a_{\zerovec}| + |a_{1, 0, \dots, 0}| + |a_{0, 1, 0, \dots, 0}| \\
    &\quad+ \Big|\int_{(0, 1)} (1 - s_1)_+ \, d\nu_{1, 0, \dots, 0}(s_1)\Big| + \Big|\int_{(0, 1)} (1 - s_2)_+ \, d\nu_{0, 1, 0, \dots, 0}(s_2)\Big| \\
    &\quad+ \Big|\int_{[0, 1)^2 \setminus \{\zerovec\}} (1 - s_1)_+ (1 - s_2)_+ \, d\nu_{1, 1, 0, \dots, 0}(s_1, s_2)\Big| \\
    &\le |f(1, 1, 0, \dots, 0)| + |a_{\zerovec}| + |a_{1, 0, \dots, 0}| + |a_{0, 1, 0, \dots, 0}| \\
    &\quad+ |\nu_{1, 0, \dots, 0}|((0, 1)) + |\nu_{0, 1, 0, \dots, 0}|((0, 1)) + |\nu_{1, 1, 0, \dots, 0}|\big([0, 1)^2 \setminus \{\zerovec\}\big) \\
    &\le M + M + (2M + 2V) + (2M + 2V) + 2V = 6M + 6V.
\end{align*}
Similarly, it can be shown that 
\begin{equation*}
    |a_{\alpha}| \le 6M + 6V
\end{equation*}
for all $\alpha \in \{0, 1\}^d \setminus \{\zerovec\}$ with $|\alpha| = 2$.
Repeating this argument, we can show inductively that 
\begin{equation*}
    |a_{\alpha}| \le C_s(M + V) := \tilde{M}
\end{equation*}
for every $\alpha \in \{0, 1\}^d \setminus \{\zerovec\}$ with $|\alpha| \le s$.

Now, we repeat our arguments in the proof of Lemma \ref{lem:svt-metric-entropy}. 
We cover $\F_M(V) - \{f^*\}$ with function classes whose upper bounds of bracketing entropy can be obtained from those of $\D_m$ for $m \ge 1$. 
Fix $\delta > 0$, which we will specify later, and let 
\begin{equation*}
    \K = \big\{(k_{\alpha}, \alpha \in \{0, 1\}^d \mbox{ and } |\alpha| \le s): -(K + 1) \le k_{\alpha} \le K \mbox{ for all } \alpha \in \{0, 1\}^d \mbox{ with } |\alpha| \le s \big\} 
\end{equation*}
where $K = \floor{\tilde{M}/\delta}$. 
For each $k \in \K$, we also let
\begin{equation*}
    M(k) = \{\fazeronu \in \F_M(V) - \{f^*\}: k_{\alpha} \delta \le a_{\alpha} \le (k_{\alpha} + 1) \delta \mbox{ for each } {\alpha} \in \{0, 1\}^d \mbox{ with } |\alpha| \le s\}.
\end{equation*}
Then, by definition,
\begin{equation*}
    \F_M(V) - \{f^*\} = \bigcup_{k \in \K} M(k),
\end{equation*}
from which we can obtain 
\begin{align}\label{fvm-decomposition}
\begin{split}
    \log N_{[ \ ]}(\epsilon, \F_M(V) - \{f^*\}, \|\cdot\|_{p_0, 2}) &\le \log\Big(\sum_{k \in \K} N_{[ \ ]}(\epsilon, M(k), \|\cdot\|_{p_0, 2})\Big) \\
    &\le \log |\K| + \sup_{k \in \K} \log N_{[ \ ]}(\epsilon, M(k), \|\cdot\|_{p_0, 2}) \\
    &\le 2^d \log\Big(2 + \frac{2 \tilde{M}}{\delta}\Big) + \sup_{k \in \K} \log N_{[ \ ]}(\epsilon, M(k), \|\cdot\|_{p_0, 2}).
\end{split}
\end{align}

Fix $k \in \K$. 
Let $M_{\zerovec}(k)$ be the collection of all the constant functions on $[0, 1]^d$
\begin{equation*}
    (x_1, \dots, x_d) \mapsto a_{\zerovec}
\end{equation*}
where $k_{\zerovec} \delta \le a_{\zerovec} \le (k_{\zerovec} + 1) \delta$, and for each $\alpha \in \{0, 1\}^d \setminus \{\zerovec\}$ with $|\alpha| \le s$, let $M_{\alpha}(k)$ be the collection of all the functions on $[0, 1]^d$ of the form
\begin{equation*}
    (x_1, \dots, x_d) \mapsto a_{\alpha} \prod_{j \in S(\alpha)} x_j + \int_{[0, 1)^{|\alpha|} \setminus \{\zerovec\}} \prod_{j \in S(\alpha)} (x_j - s_j)_+ \, d\nu_{\alpha}(\salpha),
\end{equation*}
where $k_{\alpha} \delta \le a_{\alpha} \le (k_{\alpha} + 1) \delta$ and $\nu_{\alpha}$ is a  signed measure on $[0, 1)^{|\alpha|}$ with  $|\nu_{\alpha}|([0, 1)^{|\alpha|} \setminus \{\zerovec\}) \le 2V$.
It then follows that
\begin{equation*}
    M(k) \subseteq \bigoplus_{\substack{\alpha \in \{0, 1\}^d \\ |\alpha| \le s}} M_{\alpha}(k), 
\end{equation*}
which leads to 
\begin{equation}\label{mk-decomposition-bracketing}
    \log N_{[ \ ]}(\epsilon, M(k), \|\cdot\|_{p_0, 2}) \le \sum_{\substack{\alpha \in \{0, 1\}^d \\ |\alpha| \le s}} \log N_{[ \ ]}\Big(\frac{\epsilon}{2^d}, M_{\alpha}(k), \|\cdot\|_{p_0, 2} \Big).
\end{equation}

It is clear that 
\begin{equation}\label{m0k-bracketing-entropy}
    N_{[ \ ]}(\epsilon, M_{\zerovec}(k), \|\cdot\|_{p_0, 2}) \le 1 + \frac{\delta_{\zerovec}}{\epsilon}.
\end{equation} 
Hence, it is enough to bound $\log N_{[ \ ]}(\epsilon, M_{\alpha}(k), \|\cdot\|_{p_0, 2})$ for each $\alpha \in \{0, 1\}^d \setminus \{\zerovec\}$ with $|\alpha| \le s$.
For each $\alpha \in \{0, 1\}^d \setminus \{\zerovec\}$ with $|\alpha| \le s$, let $\tilde{M}_{\alpha}$ be the collection of all the functions on $[0, 1]^{|\alpha|}$ of the form
\begin{equation*}
    (x_j, j \in S(\alpha)) \mapsto \int_{[0, 1)^{|\alpha|}} \prod_{j \in S(\alpha)} (x_j - s_j)_+ \, d\nu_{\alpha}(\salpha)
\end{equation*}
where $\nu_{\alpha}$ is a signed measure on $[0, 1)^{|\alpha|}$ with variation $|\nu_{\alpha}|([0, 1)^{|\alpha|}) \le 2V + \delta$.
As in the proof of Lemma \ref{lem:svt-metric-entropy}, we can easily show that
\begin{equation}\label{m-alpha-k-and-m-tilde-alpha}
    N_{[ \ ]}(\epsilon, M_{\alpha}(k), \|\cdot\|_{p_0, 2}) \le N_{[ \ ]}(\epsilon, \tilde{M}_{\alpha}, \|\cdot\|_{\tilde{p}_0, 2}),
\end{equation}
where $\tilde{p}_0$ is the probability density function on $[0, 1]^{|\alpha|}$ defined by 
\begin{equation*}
    \tilde{p}_0(x^{(\alpha)}) = \int_{[0, 1]^{|\onevec - \alpha|}} p_0(x_1, \dots, x_d) \, dx^{(\onevec - \alpha)}
\end{equation*}
for $x^{(\alpha)} \in [0, 1]^{|\alpha|}$.
Since $p_0$ is bounded by $B$, $\tilde{p}_0$ is also bounded by $B$, i.e., $\|\tilde{p}_0 \|_{\infty} \le B$.
Hence, from Theorem \ref{thm:d-bracket-entropy}, we can deduce that  
\begin{align}\label{bracketing-entropy-m-tilde-alpha}
\begin{split}
    \log N_{[ \ ]}(\epsilon, \tilde{M}_{\alpha}, \|\cdot\|_{\tilde{p}_0, 2}) &\le \log N_{[ \ ]}\Big( \frac{\epsilon}{\sqrt{B}}, \tilde{M}_{\alpha}, \|\cdot\|_{2}\Big) \\
    &\le C_{|\alpha|} \Big(\frac{4 \sqrt{B} (2V + \delta)}{ \epsilon}\Big)^{\frac{1}{2}}\Big|\log \Big( \frac{4 \sqrt{B}(2V + \delta)}{\epsilon} \Big) \Big|^{2(|\alpha| - 1)}.
\end{split}
\end{align}
Combining \eqref{mk-decomposition-bracketing}, \eqref{m0k-bracketing-entropy}, \eqref{m-alpha-k-and-m-tilde-alpha}, and
\eqref{bracketing-entropy-m-tilde-alpha}, 
we can derive
\begin{align*}
    &\log N_{[ \ ]}(\epsilon, M(k), \|\cdot\|_{p_0, 2}) \le \log\Big(1 + \frac{2^d\delta}{\epsilon}\Big) \\
    &\qquad \qquad \qquad + C_s \sum_{\substack{\alpha \in \{0, 1\}^d \\ |\alpha| \le s}} \bigg[\frac{2^{d + 2} \sqrt{B} (2V + \delta)}{ \epsilon}\bigg]^{\frac{1}{2}} \cdot \bigg[\log\Big(\frac{2^{d + 2} \sqrt{B} (2V + \delta)}{ \epsilon}\Big)\bigg]^{2(|\alpha| - 1)}.
\end{align*}
With the choice of $\delta = \epsilon$, this together with \eqref{fvm-decomposition} implies that
\begin{align*}
    &\log N_{[ \ ]} (\epsilon, \F(V, M), \|\cdot\|_{p_0, 2}) \le 2^d \log \Big( 2 + C_s \cdot \frac{M + V}{\epsilon} \Big) \\
    &\qquad \qquad \qquad \qquad \qquad \qquad + C_{B, d} \Big(2^{d+2} + \frac{2^{d+3} V}{\epsilon} \Big)^{\frac{1}{2}}\bigg[\log \Big(2^{d+2} + \frac{2^{d+3} V}{\epsilon}\Big)\bigg]^{2(s - 1)}
\end{align*}
as desired.
\end{proof}

\subsubsection{Proof of Lemma \ref{lem:f-eta-zero-properties}}
\label{pf:f-eta-zero-properties}
\begin{proof} [Proof of \eqref{f-eta-zero-variation}]
For each $\eta \in \{-1, 1\}^q$, we have
\begingroup
\allowdisplaybreaks
\begin{align*}
    |\nu_{\eta}|\big((0, 1)^s\big) &= \frac{V}{\sqrt{|P_l|}} \int_{(0,1)^s} \Bigg|\sum_{p \in P_l} \sum_{i \in I_p} \eta_{p, i} \bigg(\prod_{j = 1}^s \phi_{p_j, i_j}(t_j)\bigg) \Bigg| \, dt \\
    &\le \frac{V}{\sqrt{|P_l|}} \Bigg(\int_{(0,1)^s} \bigg(\sum_{p \in P_l} \sum_{i \in I_p} \eta_{p, i} \bigg(\prod_{j = 1}^s \phi_{p_j, i_j}(t_j)\bigg) \bigg)^2 \, dt \Bigg)^{\frac{1}{2}} \\
    &= \frac{V}{\sqrt{|P_l|}} \Bigg(\int_{(0,1)^s} \sum_{p \in P_l} \bigg(\sum_{i \in I_p} \eta_{p, i} \bigg(\prod_{j = 1}^s \phi_{p_j, i_j}(t_j)\bigg) \bigg)^2 \, dt \Bigg)^{\frac{1}{2}} \\ 
    &= \frac{V}{\sqrt{|P_l|}} \Bigg(\sum_{p \in P_l} \sum_{i \in I_p} \int_{(0,1)^s} \bigg(\prod_{j = 1}^s \phi_{p_j, i_j}(t_j)\bigg)^2 \, dt \Bigg)^{\frac{1}{2}} \\
    &= \frac{V}{\sqrt{|P_l|}} \Bigg(\sum_{p \in P_l} \sum_{i \in I_p} \prod_{j = 1}^s \int_{0}^{1} (\phi_{p_j, i_j}(t_j))^2 \, dt_j \Bigg)^{\frac{1}{2}} \\
    &= \frac{V}{\sqrt{|P_l|}} \bigg(\sum_{p \in P_l} \sum_{i \in I_p} \prod_{j = 1}^s 2^{-p_j} \bigg)^{\frac{1}{2}} = V.
\end{align*}
\endgroup
Here the inequality follows from the Cauchy inequality, the second equality is due to the fact that 
\begin{equation*}
    \int_{0}^{1} \phi_{m,k}(x) \phi_{m', k'}(x) \, dx = 0
\end{equation*}
for distinct $m$ and $m'$, and the third equality is from that $\phi_{m,k} \phi_{m, k'}$ $\equiv 0$ if $k$ and $k'$ are different.  
This directly implies that
$\Vmars(f_\eta) \le V$ for every $\eta \in \{-1, 1\}^q$. 
\end{proof}

\begin{proof} [Proof of \eqref{f-eta-zero-first-property}]
We first represent our function $f_{\eta}$ in a simpler form.
For a positive integer $m$ and $k \in [2^m]$, we let $\Phi_{m, k}$ denote the real-valued function on $[0, 1]$ defined by  
\begin{equation*}
    \Phi_{m, k}(x) = \int_{0}^{x} \int_{0}^{u} \phi_{m, k}(t) \, dt \, du.
\end{equation*}
Then, it can be easily checked that $\Phi_{m, k}(x) = 0$ if $x \le (k-1)2^{-m}$ or $x \ge k2^{-m}$, $\Phi_{m, k}(x + 2^{-m - 1}) = -\Phi_{m, k}(x)$ for $x \in [(k-1) 2^{-m}, (k - 1/2) 2^{-m}]$, and $|\Phi_{m, k}(x)| \le 2^{-2m-6}$ for all $x \in [0, 1]$.
Using these $\Phi_{m, k}$, we can represent $f_{\eta}$ as
\begingroup
\allowdisplaybreaks
\begin{align*}
    f_{\eta}(x_1, \dots, x_d) &= \int_{(0,1)^s} \prod_{j = 1}^{s} (x_j - t_j)_+ \, d\nu_{\eta}(t) \\
    &= \frac{V}{\sqrt{|P_l|}} \sum_{p \in P_l} \sum_{i \in I_p} \eta_{p, i} \int_{(0,1)^s} \prod_{j = 1}^{s} \big((x_j - t_j)_+ \cdot \phi_{p_j, i_j}(t_j)\big) \, dt \\
    &= \frac{V}{\sqrt{|P_l|}} \sum_{p \in P_l} \sum_{i \in I_p} \eta_{p, i} \prod_{j = 1}^{s} \int_{0}^{1} (x_j - t_j)_+ \cdot \phi_{p_j, i_j}(t_j) \, dt_j \\
    &= \frac{V}{\sqrt{|P_l|}} \sum_{p \in P_l} \sum_{i \in I_p} \eta_{p, i} \prod_{j = 1}^{s} \int_{0}^{x_j} \int_{0}^{u_j} \phi_{p_j, i_j}(t_j) \, dt_j \, du_j \\
    &= \frac{V}{\sqrt{|P_l|}} \sum_{p \in P_l} \sum_{i \in I_p} \eta_{p, i} \prod_{j = 1}^{s} \Phi_{p_j, i_j}(x_j)
\end{align*}
\endgroup
for each $\eta \in \{-1, 1\}^q$ .

Now, we prove \eqref{f-eta-zero-first-property} using the above representation. 
Suppose $H(\eta, \eta') = 1$ for $\eta, \eta' \in \{-1, 1\}^q$. 
Since $H(\eta, \eta') = 1$, there exists a unique $(p, i) \in Q$ such that $\eta_{p, i} \neq \eta'_{p, i}$, and thus 
\begin{equation*}
    (f_{\eta} - f_{\eta'})(x_1, \dots, x_d) = \frac{V}{\sqrt{|P_l|}} (\eta_{p, i} - \eta'_{p, i}) \prod_{j=1}^{s} \Phi_{p_j, i_j} (x_j).
\end{equation*}
Hence, 
\begin{align*}
    &\| f_\eta - f_{\eta'}\|_{p_0, 2}^2 = \frac{4 V^2}{|P_l|} \int_{[0, 1]^d} \bigg( \prod_{j=1}^{s} \Phi_{p_j, i_j} (x_j) \bigg)^2 p_0(x) \, dx \\
    &\qquad\le \frac{4 B V^2}{|P_l|} \cdot \prod_{j = 1}^{s} \int_{0}^{1} \big( \Phi_{p_j, i_j} (x_j) \big)^2 \, dx_j \le \frac{4 B V^2}{|P_l|} \cdot \prod_{j = 1}^{s} 2^{-p_j} \cdot 2^{-4p_j -12} \\
    &\qquad= \frac{BV^2}{|P_l|} \cdot 2^{-5l - 12s + 2}.
\end{align*}
\end{proof}

\begin{proof} [Proof of \eqref{f-eta-zero-second-property}]
Fix $\eta \neq \eta' \in \{-1, 1\}^q$.
For each positive integer $m$ and $k \in [2^m]$, let $h_{m, k}$ be the real-valued function on $[0, 1]$ defined by 
\begin{equation*}
    h_{m, k}(x) = 
    \begin{cases}
        2^{\frac{m}{2}} &\mbox{if } (k-1)2^{-m} < x < \big(k - \frac{1}{2}\big) 2^{-m} \\
        -2^{\frac{m}{2}} &\mbox{if } \big(k-\frac{1}{2}\big)2^{-m} < x < k 2^{-m} \\
        0 &\mbox{otherwise}, 
    \end{cases}
\end{equation*}
and for each $(p, i) \in Q$, let $H_{p, i}$ be the real-valued function on $[0, 1]^s$ defined by 
\begin{equation*}
    H_{p, i}(x_1, \dots, x_s) = \prod_{j = 1}^{s} h_{p_j, i_j}(x_j).
\end{equation*}
One can check that $\{H_{p, i}: (p, i) \in Q\}$ is an orthonormal set with respect to the $L^2$ inner product.
Also, let $g_{\eta, \eta'}$ denote the function on $[0, 1]^s$ defined by
\begin{equation*}
    g_{\eta, \eta'}(x_1, \dots, x_s) = \frac{V}{\sqrt{|P_l|}} \sum_{p \in P_l} \sum_{i \in I_p} \big(\eta_{p, i} - \eta'_{p, i}\big) \prod_{j=1}^{s} \Phi_{p_j, i_j} (x_j).
\end{equation*}
Then, by Bessel's inequality, 
\begin{equation}\label{bessel-inequality}
    \| f_{\eta} - f_{\eta'} \|_{p_0, 2}^2 \ge b \|g_{\eta, \eta'}\|_2^2 \ge b\sum_{p' \in P_l} \sum_{i' \in I_p} \inner{g_{\eta, \eta'}}{H_{p', i'}}^2,
\end{equation}
where $\|\cdot\|_2$ and $\inner{\cdot}{\cdot}$ denote the $L^2$ norm and the $L^2$ inner product, respectively. 
Note that for each $(p', i') \in Q$,  
\begin{equation}\label{bessel-inequality-application}
    \inner{g_{\eta, \eta'}}{H_{p', i'}} = \frac{V}{\sqrt{|P_l|}} \sum_{p \in P_l} \sum_{i \in I_p} \big(\eta_{p, i} - \eta'_{p, i}\big) \prod_{j=1}^{s} \inner{\Phi_{p_j, i_j}}{h_{p'_j, i'_j}}. 
\end{equation}
Now, we claim that for $(p, i), (p', i') \in Q$, 
\begin{equation}\label{inner-product-claim}
    \prod_{j = 1}^{s} \inner{\Phi_{p_j, i_j}}{h_{p'_j, i'_j}} =
    \begin{cases}
        2^{-\frac{5}{2}l - 7s} &\mbox{if } (p, i) = (p', i') \\
        0 &\mbox{if } (p, i) \neq (p', i'). 
    \end{cases}
\end{equation}
We first assume $(p, i) \neq (p', i')$. 
If $p \neq p'$, then since $\sum_{j = 1}^{s} p_j = l = \sum_{j = 1}^{s} p'_j$, there exists $j \in [s]$ such that $p_j > p'_j$, and thus $h_{p'_j, i'_j}$ is constant on $((i_j - 1)2^{-p_j}, i_j 2^{-p_j})$. 
Also, recall that $\Phi_{p_j, i_j}(x) = 0$ if $x \le (i_j-1)2^{-p_j}$ or $x \ge i_j2^{-p_j}$ and $\Phi_{p_j, i_j}(x + 2^{-p_j - 1}) = -\Phi_{p_j, i_j}(x)$ for $x \in [(i_j-1) 2^{-p_j}, (i_j - 1/2) 2^{-p_j}]$.
For these reasons, we have $\inner{\Phi_{p_j, i_j}}{h_{p'_j, i'_j}} = 0$.
Otherwise, if $p = p'$, then $i$ and $i'$ should be distinct. 
Let $j \in [s]$ be an index for which $i_j \neq i'_j$. 
Then, $\Phi_{p_j, i_j}(x) h_{p'_j, i'_j}(x) = 0$  for all $x \in [0, 1]$, which directly implies that $\inner{\Phi_{p_j, i_j}}{h_{p'_j, i'_j}} = 0$.
Next, we assume $(p, i) = (p', i')$. 
In this case, the claim follows from that 
\begin{equation*}
    \inner{\Phi_{p_j, i_j}}{h_{p_j, i_j}} = \int_{(i_j - 1)2^{-p_j}}^{i_j2^{-p_j}} \Phi_{p_j, i_j}(x) h_{p_j, i_j}(x) \, dx = 2^{\frac{p_j}{2}} \int_{(i_j - 1)2^{-p_j}}^{i_j2^{-p_j}} |\Phi_{p_j, i_j}(x)| \, dx = 2^{-\frac{5}{2} p_j - 7} 
\end{equation*}
for each $j \in [s]$.

Combining \eqref{bessel-inequality}, \eqref{bessel-inequality-application}, and \eqref{inner-product-claim}, we obtain
\begin{equation*}
    \| f_{\eta} - f_{\eta'}\|_{p_0, 2}^2 \ge \frac{b V^2}{|P_l|} \cdot 2^{-5l -14s} \sum_{p \in P_l} \sum_{i \in I_p} (\eta_{p, i} - \eta'_{p, i})^2 = \frac{b V^2}{|P_l|} \cdot 2^{-5l -14s + 2} \cdot H(\eta, \eta'). 
\end{equation*}
From this result, we can derive the conclusion
\begin{equation*}
    \min_{\eta \neq \eta'} \frac{\| f_\eta - f_{\eta'}\|_{p_0, 2}^2}{H(\eta, \eta')} \ge \frac{b V^2}{|P_l|} \cdot 2^{-5l - 14s + 2}.
\end{equation*}
\end{proof}

\subsection{Proofs of Lemmas in Appendix \ref{subsec:proof-of-lemmas-risk-results}}
\subsubsection{Proof of Lemma \ref{lem:tas-metric-entropy}}\label{pf:tas-metric-entropy}
Recall that we define $\D_m$ as the collection of all the functions $F$ on $[0, 1]^m$ of the form
\begin{equation*}
    F(x_1, \dots, x_m) = \int (x_1 - s_1)_+ \cdots (x_m - s_m)_+ \, d\mu(s) = \int_{[\zerovec, x]} (x_1 - s_1) \cdots (x_m - s_m) \, d\mu(s),
\end{equation*}
where $[\zerovec, x] = \{s \in \R^m : 0 \le s_i \le x_i \mbox{ for all } i \in [m]\}$ and $\mu$ is a signed measure on $[0, 1]^m$ with variation $|\mu|([0, 1]^m) \le 1$.
Lemma \ref{lem:tas-metric-entropy} is then a direct consequence of  Theorem \ref{thm:d-metric-entropy-main-text} and the following lemma, which connects the metric entropy of $\genmalpha_m(S)$ and $\D_m$.
Here the covering number of $\D_m$ in the lemma is under the $L^2$ norm as in Theorem \ref{thm:d-metric-entropy-main-text}.

\begin{lemma}\label{lem:t+-metric-entropy}
    There exists a positive constant $c_{\rho, m}$ depending on $\rho$ and $m$ such that 
    \begin{equation*}
        N(\epsilon, \genmalpha_m(S), \|\cdot\|_n) \le N\Big(\frac{c_{\rho, m} \epsilon}{S}, \D_m, \|\cdot\|_2 \Big)
    \end{equation*}
    for every $S > 0$ and $\epsilon > 0$.
  \end{lemma}

\begin{proof}[Proof of Lemma \ref{lem:t+-metric-entropy}]
Define a map $\Phi: \genmalpha_m(S) \rightarrow \D_m$ as follows. 
Given $f \in \genmalpha_m(S)$ of the form 
\begin{equation*}
    f(x_1, \dots, x_m) = \int_{[0, 1)^{m}} (x_1 - s_1)_+ \cdots (x_m - s_m)_+ \, d\nu(s)
\end{equation*}
where $\nu$ is a signed measure concentrated on $(\prod_{k =1}^{m} \mathcal{U}_k) \cap [0, 1)^{m}$ with variation $|\nu|([0, 1)^m) \le S$,
we first consider the signed measure $\mu$ on $[0, 1]^m$ satisfying the followings:
\begin{itemize}
    \item $\mu$ is concentrated on the set 
    \begin{equation*}
        \prod_{k = 1}^{m} \Big\{u^{(k)}_{i_k} + \big(1 - u^{(k)}_{n_k - 1}\big): i_k \in [0:(n_k - 1)]\Big\}.
    \end{equation*}
    \item For $(i_1, \dots, i_m) \in I_0 := [0 : (n_1 - 1)] \times \dots \times [0 : (n_d - 1)]$,
    \begin{equation*}
        \mu\Big(\Big\{\Big(u^{(k)}_{i_k} + \big(1 - u^{(k)}_{n_k - 1}\big), k \in [m] \Big)\Big\}\Big) = \frac{1}{S} \cdot \nu\big(\big\{\big(u^{(k)}_{i_k}, k \in [m] \big)\big\}\big).
    \end{equation*}
\end{itemize}
Note that $|\mu|([0, 1]^m) \le 1$ since $|\nu|([0, 1)^m) \le S$.
Next, let $F$ denote the function on $[0, 1]^m$ of the form 
\begin{equation*}
    F(x_1, \dots, x_m) = \int_{[0, 1]^m} (x_1 - s_1)_+ \cdots (x_m - s_m)_+ \, d\mu(s) 
\end{equation*}
and let $\Phi(f) = F$. 
Clearly, for $(x_1, \dots, x_m) \in [0, 1]^m$,
\begin{equation*}
    F(x_1, \dots, x_m) = 0
\end{equation*}
if $x_k \le 1 - u^{(k)}_{n_k - 1}$ for some $k \in [m]$, and 
\begin{equation*}
    F(x_1, \dots, x_m) = \frac{1}{S} \cdot f\Big(\Big(x_k - \big(1 - u^{(k)}_{n_k - 1}\big), k \in [m]\Big)\Big)
\end{equation*}
otherwise.
Also, we can note that for
\begin{equation*}
    x = \Big((1 - z_k) \cdot u^{(k)}_{i_k - 1} + z_k \cdot u^{(k)}_{i_k} + \big(1 - u^{(k)}_{n_k - 1}\big), k \in [m]\Big),
\end{equation*}
where $i = (i_1, \dots, i_m) \in \prod_{k = 1}^{m} [n_k - 1]$ and $0 \le z_1, \dots, z_m < 1$, 
\begingroup
\allowdisplaybreaks
\begin{align*}
    F(x) &= \int_{[\zerovec, x]} (x_1 - s_1) \cdots (x_m - s_m) \, d\mu(s) \\
    &= \frac{1}{S} \sum_{j \in I_0} \prod_{k = 1}^{m} \Big(x_k - u^{(k)}_{j_k} - \big(1 - u^{(k)}_{n_k - 1}\big)\Big) \cdot \ind\{i - \onevec \ge j\} \cdot \beta_j \\
    &= \frac{1}{S} \sum_{j \in I_0} \prod_{k = 1}^{m} \Big[(1 - z_k)\big(u^{(k)}_{i_k - 1} - u^{(k)}_{j_k}\big) + z_k\big(u^{(k)}_{i_k} - u^{(k)}_{j_k}\big)\Big] \cdot \ind\{i - \onevec \ge j\} \cdot \beta_j \\
    &= \frac{1}{S} \sum_{j \in I_0} \sum_{\delta \in \{0, 1\}^m} \prod_{k = 1}^{m} \big[(1 - z_k)^{1 - \delta_k} z_k^{\delta_k} \big(u^{(k)}_{i_k - 1 + \delta_k} - u^{(k)}_{j_k}\big)\big] \cdot \ind\{i - \onevec \ge j\} \cdot \beta_j \\
    &= \frac{1}{S} \sum_{\delta \in \{0, 1\}^m} \bigg[\prod_{k = 1}^{m} \big[(1 - z_k)^{1 - \delta_k} z_k^{\delta_k}\big] \cdot \sum_{j \in I_0} \prod_{k = 1}^{m} \big(u^{(k)}_{i_k - 1 + \delta_k} - u^{(k)}_{j_k}\big) \cdot \ind\{i - \onevec \ge j\} \cdot \beta_{j}\bigg] \\
    &= \frac{1}{S} \sum_{\delta \in \{0, 1\}^m} \bigg[\prod_{k = 1}^{m} \big[(1 - z_k)^{1 - \delta_k} z_k^{\delta_k}\big] \cdot f\big(\big(u^{(k)}_{i_k - 1 + \delta_k}, k \in [m]\big)\big)\bigg].
\end{align*}
\endgroup
Here we let 
\begin{equation*}
    \beta_j = \nu\big(\big\{\big(u^{(k)}_{j_k}, k \in [m]\big)\big\}\big) \qt{for $j \in I_0$}
\end{equation*}
just for convenience.
Hence, for $f, \tilde{f} \in \genmalpha_m(S)$,
\begingroup
\allowdisplaybreaks
\begin{align*}
    &\big\| \Phi(f) - \Phi(\tilde{f}) \big\|_2^2 = \int_{[0, 1]^m} \big|\Phi(f)(s) - \Phi(\tilde{f})(s)\big|^2 \, ds \\
    &\qquad = \sum_{i \in \prod\limits_{k} [n_k - 1]} \int_{u^{(1)}_{i_1 - 1} + (1 - u^{(1)}_{n_1 - 1})}^{u^{(1)}_{i_1} + (1 - u^{(1)}_{n_1 - 1})} \dots \int_{u^{(m)}_{i_m - 1} + (1 - u^{(m)}_{n_m - 1})}^{u^{(m)}_{i_m} + (1 - u^{(m)}_{n_m - 1})} \big|\Phi(f)(s) - \Phi(\tilde{f})(s)\big|^2 \, ds_m \cdots \, ds_1 \\
    &\qquad = \frac{1}{S^2} \sum_{i \in \prod\limits_{k} [n_k - 1]} \prod_{k = 1}^{m} \big(u^{(k)}_{i_k} - u^{(k)}_{i_k - 1}\big) \\
    &\qquad \qquad \quad \cdot  \int_{[0, 1]^m} \bigg[ \sum_{\delta \in \{0, 1\}^m} \bigg[\prod_{k = 1}^{m} \big[(1 - z_k)^{1 - \delta_k} z_k^{\delta_k}\big] \cdot \Delta f\big(\big(u^{(k)}_{i_k - 1 + \delta_k}, k \in [m]\big)\big)\bigg] \bigg]^2 \, dz \\
    &\qquad \ge \frac{1}{S^2} \cdot \frac{{\rho}^m}{n_1 \cdots n_m} \sum_{i \in \prod\limits_{k} [n_k - 1]} \int_{[0, 1]^m} \bigg[ \sum_{\delta \in \{0, 1\}^m} \bigg[\prod_{k = 1}^{m} \big[(1 - z_k)^{1 - \delta_k} z_k^{\delta_k}\big] \\
    &\qquad \qquad \qquad \qquad \qquad \qquad \qquad \qquad \qquad \qquad \qquad \cdot \Delta f\big(\big(u^{(k)}_{i_k - 1 + \delta_k}, k \in [m]\big)\big)\bigg] \bigg]^2 \, dz, 
\end{align*}
\endgroup
where $\Delta f = f - \tilde{f}$.
For the last equality, we change the variables as 
\begin{equation*}
    s_k = (1 - z_k) \cdot u^{(k)}_{i_k - 1} + z_k \cdot u^{(k)}_{i_k} + \big(1 - u^{(k)}_{n_k - 1}\big)
\end{equation*}
for $k \in [m]$.
For a fixed constant $0 < r < 1$, by the Cauchy inequality, 
\begingroup
\allowdisplaybreaks
\begin{align*}
    &\int_{[0, 1]^m} \bigg[ \sum_{\delta \in \{0, 1\}^m} \bigg[\prod_{k = 1}^{m} \big[(1 - z_k)^{1 - \delta_k} z_k^{\delta_k}\big] \cdot \Delta f\big(\big(u^{(k)}_{i_k - 1 + \delta_k}, k \in [m]\big)\big) \bigg] \bigg]^2 \, dz \\
    &\; \ge \int_{[r, 1]^m} \bigg[ \sum_{\delta \in \{0, 1\}^m} \bigg[\prod_{k = 1}^{m} \big[(1 - z_k)^{1 - \delta_k} z_k^{\delta_k}\big] \cdot \Delta f\big(\big(u^{(k)}_{i_k - 1 + \delta_k}, k \in [m]\big)\big) \bigg] \bigg]^2 \, dz \\
    &\; \ge \frac{1}{2} \int_{[r, 1]^m} z_1^2 \cdots z_m^2 \cdot \Big(\Delta f\big(\big(u^{(k)}_{i_k}, k \in [m]\big)\big)\Big)^2 \, dz \\
    &\; \; - \int_{[r, 1]^m} \bigg[ \sum_{\delta \in \{0, 1\}^m \setminus \{\onevec\}} \bigg[\prod_{k = 1}^{m} \big[(1 - z_k)^{1 - \delta_k} z_k^{\delta_k}\big] \cdot \Delta f\big(\big(u^{(k)}_{i_k - 1 + \delta_k}, k \in [m]\big)\big) \bigg] \bigg]^2 \, dz \\
    &\; \ge \frac{1}{2} \int_{[r, 1]^m} z_1^2 \cdots z_m^2 \cdot \Big(\Delta f\big(\big(u^{(k)}_{i_k}, k \in [m]\big)\big)\Big)^2 \, dz \\
    &\; \; - (2^m - 1) \sum_{\delta \in \{0, 1\}^m \setminus \{\onevec\}} \int_{[r, 1]^m} \prod_{k = 1}^{m} \big[(1 - z_k)^{2(1 - \delta_k)} s_k^{2\delta_k}\big] \cdot \Big(\Delta f\big(\big(u^{(k)}_{i_k - 1 + \delta_k}, k \in [m]\big)\big)\Big)^2 \, dz \\
    &\; =\frac{(1 - r^3)^m}{2 \cdot 3^m} \Big(\Delta f\big(\big(u^{(k)}_{i_k}, k \in [m]\big)\big)\Big)^2 \\
    &\; \; - \frac{2^m - 1}{3^m} \sum_{\delta \in \{0, 1\}^m \setminus \{\onevec\}} \Big[(1 - r^3)^{\sum_{k} \delta_k}(1 - r)^{3(m- \sum_{k} \delta_k)} \Big(\Delta f\big(\big(u^{(k)}_{i_k - 1 + \delta_k}, k \in [m]\big)\big)\Big)^2\Big]
\end{align*}
\endgroup
for all $(i_1, \dots, i_m) \in \prod_{k = 1}^{m} [n_k - 1]$. 
Thus, for $f, \tilde{f} \in \genmalpha_m(S)$,
\begingroup
\allowdisplaybreaks
\begin{align*}
    \big\| \Phi(f) - \Phi(\tilde{f}) \big\|_2^2 &\ge \bigg[\frac{(1 - r^3)^m}{2 \cdot 3^m} - \frac{2^m - 1}{3^m} \cdot \sum_{\delta \in \{0, 1\}^m \setminus \{\onevec\}} \big[(1 - r^3)^{\sum_{k} \delta_k}(1 - r)^{3(m- \sum_{k} \delta_k)}\big] \bigg] \\
    &\qquad \cdot \frac{1}{S^2} \cdot \frac{{\rho}^m}{n_1 \cdots n_m} \sum_{i \in \prod\limits_{k} [n_k - 1]} \Big(\Delta f\big(\big(u^{(k)}_{i_k}, k \in [m]\big)\big)\Big)^2 \\
    &= \bigg[\frac{(1 - r^3)^m}{2 \cdot 3^m} - \frac{2^m - 1}{3^m} \cdot \sum_{\delta \in \{0, 1\}^m \setminus \{\onevec\}} \big[(1 - r^3)^{\sum_{k} \delta_k}(1 - r)^{3(m- \sum_{k} \delta_k)}\big] \bigg] \\
    &\qquad \cdot \frac{{\rho}^m}{S^2}  \cdot \| f - \tilde{f} \|_n^2 
\end{align*}
\endgroup
since $f((u^{(k)}_{i_k}, k \in [m])) = \tilde{f}((u^{(k)}_{i_k}, k \in [m])) = 0$ if $i_k = 0$ for some $k \in [m]$.
Note that 
\begingroup
\allowdisplaybreaks
\begin{align*}
    &\frac{(1 - r^3)^m}{2 \cdot 3^m} - \frac{2^m - 1}{3^m} \cdot \sum_{\delta \in \{0, 1\}^m \setminus \{\onevec\}} \big[(1 - r^3)^{\sum_{k} \delta_k}(1 - r)^{3(m- \sum_{k} \delta_k)}\big] \\
    &\qquad = \frac{(1 - r)^m}{3^m} \cdot \bigg[\frac{1}{2}\cdot \Big(\frac{1-r^3}{1-r}\Big)^m - (2^m - 1)(1 - r)^2 \\
    &\qquad \qquad \qquad \qquad \qquad \quad \cdot \sum_{\delta \in \{0, 1\}^m \setminus \{\onevec\}} \Big[\Big(\frac{1-r^3}{1-r}\Big)^{\sum_{k} \delta_k}(1 - r)^{2(m - 1 - \sum_{k} \delta_k)}\Big] \bigg].
\end{align*}
\endgroup
Therefore, by choosing $r$ close enough to 1, we can show that there exists a positive constant $c_{\rho, m} $ depending on $\rho$ and $m$ such that
\begin{equation}\label{norm-inequality}
    \big\| \Phi(f) - \Phi(\tilde{f}) \big\|_2 \ge \frac{c_{\rho, m}}{S} \cdot \| f - \tilde{f} \|_n
\end{equation}
for all $f, \tilde{f} \in \genmalpha_m(S)$.
Consequently, 
\begin{equation*}
    N(\epsilon, \genmalpha_m(S), \|\cdot\|_n) \le M(\epsilon, \genmalpha_m(S), \|\cdot\|_n) \le M\Big(\frac{c_{\rho, m} \epsilon}{S}, \D_m, \|\cdot\|_2\Big) \le N\Big(\frac{c_{\rho, m} \epsilon}{2S}, \D_m, \|\cdot\|_2\Big).
\end{equation*}
Here $M(\epsilon, X, \|\cdot\|_n)$ indicates the $\epsilon$-packing number of a set $X$ under the norm $\|\cdot\|_n$, the second inequality follows from \eqref{norm-inequality}, and the first and the last inequality are from the general theory of packing numbers and covering numbers (see, e.g., \cite{wainwright2019high}, Lemma 5.5).
\end{proof}

\subsubsection{Proof of Lemma \ref{lem:smallest-eigenvalue}}\label{pf:smallest-eigenvalue}
\begin{proof}[Proof of Lemma \ref{lem:smallest-eigenvalue}]
Note that for every $v = (v_{\alpha}, \alpha \in \{0, 1\}^d \text{ and } |\alpha| \le s)$,
\begin{equation*}
    v^T \Sigma v = \E\bigg[\sum_{\alpha, \alpha'} \Big(v_{\alpha} \prod_{j \in S(\alpha)} X^{(1)}_j \Big) \Big(v_{\alpha'} \prod_{j' \in S(\alpha')} X^{(1)}_{j'} \Big) \bigg] = \E\bigg[\bigg(\sum_{\substack{\alpha \in \{0, 1\}^d \\ |\alpha| \le s}} \prod_{j \in S(\alpha)} X^{(1)}_j \bigg)^2\Big] \ge 0.
\end{equation*}
Hence, it suffices to show that $v^T \Sigma v > 0$ for every $v \neq 0$ in order to prove the lemma.
Suppose there exists $v = (v_{\alpha}, \alpha \in \{0, 1\}^d \text{ and } |\alpha| \le s)$ such that $v^T \Sigma v = 0$.
From the computation above, we can see that 
\begin{equation*}
    \E\bigg[\bigg(\sum_{\substack{\alpha \in \{0, 1\}^d \\ |\alpha| \le s}} \prod_{j \in S(\alpha)} X^{(1)}_j \Big)^2\Big] = 0,
\end{equation*}
which directly implies
\begin{equation}\label{eq:multi-affine-term-as-zero}
    \P\bigg(\sum_{\substack{\alpha \in \{0, 1\}^d \\ |\alpha| \le s}} v_{\alpha} \prod_{j \in S(\alpha)} X^{(1)}_j = 0\bigg) = 1.
\end{equation}
For notational convenience, let $a_d$ be the multi-affine function defined by 
\begin{equation*}
    a_d(x_1, \dots, x_d) = \sum_{\substack{\alpha \in \{0, 1\}^d \\ |\alpha| \le s}} \prod_{j \in S(\alpha)} x_j.
\end{equation*}
Then, we can restate \eqref{eq:multi-affine-term-as-zero} as 
\begin{equation*}
    \P\Big(a_d\big(X^{(1)}_1, \dots, X^{(1)}_d\big) = 0\Big) = 1.
\end{equation*}

Now, observe that we can write $a_d$ as 
\begin{equation}\label{eq:multi-affine-function-split}
    a_d(x_1, \dots, x_d) = b_{d - 1}(x_1, \dots, x_{d - 1}) + a_{d - 1}(x_1, \dots, x_{d - 1}) \cdot x_d
\end{equation}
for some multi-affine functions $a_{d - 1}$ and $b_{d - 1}$ of $x_1, \dots, x_{d - 1}$.
We can then show that the evaluations of $a_{d - 1}$ and $b_{d - 1}$ at $X^{(1)}$ should also be zero almost surely. 
More precisely, we can show that 
\begin{equation*}
    \P\Big(a_{d - 1}\big(X^{(1)}_1, \dots, X^{(1)}_{d - 1}\big) = 0\Big) = 1 \ \text{ and } \ \P\Big(b_{d - 1}\big(X^{(1)}_1, \dots, X^{(1)}_{d - 1}\big) = 0\Big) = 1.
\end{equation*}
This is because
\begin{align*}
    1 &= \P\Big(a_d\big(X^{(1)}_1, \dots, X^{(1)}_d\big) = 0\Big) \\
    &\le \P\Big(a_{d - 1}\big(X^{(1)}_1, \dots, X^{(1)}_{d - 1}\big) = 0 \Big) \\
    &\quad+ \P\Big(a_{d - 1}\big(X^{(1)}_1, \dots, X^{(1)}_{d - 1}\big) \neq 0 \text{ and } X^{(1)}_d = -(b_{d - 1}/a_{d - 1})\big(X^{(1)}_1, \dots, X^{(1)}_{d - 1}\big) \Big) \\
    &= \P\Big(a_{d - 1}\big(X^{(1)}_1, \dots, X^{(1)}_{d - 1}\big) = 0 \Big),
\end{align*}
where the second equality is derived from the assumption that the $X^{(1)}$ is a continuous random variable with a probability density function.

We can repeat this argument. 
Suppose we have a multi-affine function of the first $i$ variables $x_1, \dots, x_i$ whose evaluation at $X^{(1)}$ is zero almost surely. 
Then, we can represent it as in \eqref{eq:multi-affine-function-split} with two multi-affine functions of one less variable, $x_1, \dots, x_{i - 1}$ and show as before that the evaluations of those two functions at $X^{(1)}$ are zero almost surely as well.
As the number of variables of multi-affine functions decreases by one at each iteration, we end up with constant functions that are zero almost surely, which means that they are identically zero.
If we recursively plug in them back to the recurrence relations between multi-affine functions, we can show that $a_d$ is also identically zero, and this proves that the coefficient vector $v$ of $a_d$ is a zero vector.
\end{proof}

\end{document}